\documentclass[10pt]{iopart}
\usepackage[colorlinks=true,linkcolor=blue,citecolor=blue,urlcolor=blue,linktocpage=true]{hyperref}
\usepackage{times}
\usepackage{mathptmx}
\usepackage{graphicx}
\usepackage{tikz}
\usepackage{amsthm,amssymb}
\usepackage{units}
\usepackage{caption,subcaption}

\usepackage{mathabx}

\usepackage{setspace} 
\usepackage{float}
\floatstyle{ruled}
\newfloat{algorithm}{h}{loa}
\floatname{algorithm}{Algorithm}
\usepackage{algpseudocode}

\usepackage{url}

\makeatletter
\renewcommand\tableofcontents{%
  \section*{\contentsname
  }%
  \@starttoc{toc}%
}
\makeatother

\usepackage{scalerel}[2014/03/10]
\usepackage[usestackEOL]{stackengine}
\def\dashint{\,\ThisStyle{\ensurestackMath{%
  \stackinset{c}{.2\LMpt}{c}{.5\LMpt}{\SavedStyle-}{\SavedStyle\phantom{\int}}}%
  \setbox0=\hbox{$\SavedStyle\int\,$}\kern-\wd0}\int}

\theoremstyle{plain}
\newtheorem{lemma}{Lemma}[section]
\newtheorem{proposition}[lemma]{Proposition}
\newtheorem{theorem}[lemma]{Theorem}
\newtheorem{corollary}[lemma]{Corollary}

\theoremstyle{definition}
\newtheorem{definition}[lemma]{Definition}
\newtheorem{example}[lemma]{Example}

\theoremstyle{remark}
\newtheorem{remark}[lemma]{Remark}

\newcommand{\text}[1]{\mbox{#1}}
\newcommand{\eqref}[1]{{(\ref{#1})}}

\newcommand{\mod}{\mathop{\rm mod}\nolimits}

\newcommand{\dom}{\mathop{\rm dom}\nolimits}
\newcommand{\id}{\mathop{\rm id}\nolimits}

\newcommand{\diag}{\mathop{\rm diag}\nolimits}
\newcommand{\sgn}{\mathop{\rm sgn}\nolimits}
\newcommand{\divergence}{\mathop{\rm div}\nolimits}

\newcommand{\supp}{\mathop{\rm supp}\nolimits}

\newcommand{\lspan}{\mathop{\rm span}\nolimits}

\newcommand{\TV}{\mathop{\rm TV}\nolimits}
\newcommand{\NLTV}{\mathop{\rm NLTV}\nolimits}

\newcommand{\BV}{\mathop{\rm BV}\nolimits}
\newcommand{\mrg}{\mathop{\rm rg}\nolimits}
\newcommand{\TGV}{\mathop{\rm TGV}\nolimits}
\newcommand{\ICTGV}{\mathop{\rm ICTGV}\nolimits}

\newcommand{\ITGV}{\mathop{\rm ITGV}\nolimits}
\newcommand{\NLTGV}{\mathop{\rm NLTGV}\nolimits}
\newcommand{\BGV}{\mathop{\rm BGV}\nolimits}

\newcommand{\TD}{\mathop{\rm TD}\nolimits}
\newcommand{\BD}{\mathop{\rm BD}\nolimits}
\newcommand{\Per}{\mathop{\rm Per}\nolimits}

\newcommand{\Sym}{\mathop{\rm Sym}\nolimits}
\newcommand{\sym}{\mathop{\rm sym}\nolimits}

\newcommand{\trace}{\tr}
\newcommand{\prox}{\mathop{\rm prox}\nolimits}
\newcommand{\proj}{\mathop{\rm proj}\nolimits}

\newcommand{\Imag}{\mathop{\rm Imag}\nolimits}

\newcommand{\argmin}{\mathop{\rm arg\,min}\limits}
\newcommand{\KL}{\mathop{\rm KL}\nolimits}
\newcommand{\kl}{g}
\newcommand{\mS}{\mathcal{S}}

\newcommand{\osci}{\mathop{\rm osci}\nolimits}

\newcommand{\kpet}{K_{\mathop{\rm PET}}}
\newcommand{\kmr}{K_{\mathop{\rm MR}}}
\newcommand{\kpeth}{K_{\mathop{\rm PET}, h}}
\newcommand{\kmrh}{K_{\mathop{\rm MR}, h}}
\newcommand{\ktem}{K_{\mathop{\rm TEM}}}
\newcommand{\Mc}{\mathcal{M}}

\newcommand{\Gap}{\mathfrak{G}}

\newcommand{\unwrap}{\mathop{\rm unwrap}\nolimits}

\newcommand{\fg}{\mathop{\rm fg}\nolimits}

\newcommand{\doublehookrightarrow}%
{\DOTSB\lhook\joinrel\relbar\!\!\!\!\lhook\joinrel\rightarrow}

\newcommand{\longrightharpoonup}%
{\relbar\joinrel\rightharpoonup}

\newcommand{\conditionalcomma}[1]{\ifx#1\empty\else,\fi}

\newcommand{\CC}{\mathbf{C}}
\newcommand{\RR}{\mathbf{R}}
\newcommand{\NN}{\mathbf{N}}
\newcommand{\ZZ}{\mathbf{Z}}

\newcommand{\poly}{\mathbf{P}}

\newcommand{\mI}{\mathcal{I}}
\newcommand{\mN}{\mathcal{N}}
\newcommand{\mB}{\mathcal{B}}
\newcommand{\mR}{\mathcal{R}}
\newcommand{\mF}{\mathcal{F}}
\newcommand{\mG}{\mathcal{G}}
\newcommand{\mK}{\mathcal{K}}

\newcommand{\mO}{\mathcal{O}}

\newcommand{\mH}{\mathcal{H}}
\newcommand{\mX}{\mathcal{X}}
\newcommand{\mY}{\mathcal{Y}}
\newcommand{\mL}{\mathcal{L}}

\newcommand{\smallset}[2]{\{{#1}|{#2}\}}
\newcommand{\set}[2]{\{{#1} \ \bigl| \ {#2}\}}

\newcommand{\Bigset}[2]{\Bigl\{{#1} \ \Bigl| \ {#2}\Bigr\}}
\newcommand{\sett}[1]{\{{#1}\}}

\newcommand{\without}{\backslash}

\newcommand{\kernel}[1]{\ker({#1})}
\newcommand{\infconv}{\triangle}
\newcommand{\placeholder}{\,\cdot\,}

\newcommand{\bdry}{\partial}
\newcommand{\laplace}{\Delta}
\newcommand{\tensor}{\otimes}
\newcommand{\grad}{\nabla}

\newcommand{\abs}[2][]{{|{#2}|_{#1}}}
\newcommand{\bigabs}[1]{\bigl|{#1}\bigr|}
\newcommand{\Bigabs}[1]{\Bigl|{#1}\Bigr|}

\newcommand{\expE}{\mathrm{e}}
\newcommand{\inprod}{\cdot}
\newcommand{\ones}{\mathbf{1}}
\newcommand{\interleave}{|\!|\!|}
\newcommand{\scp}[3][]{\langle{#2},\, {#3}\rangle_{#1}}

\newcommand{\norm}[2][]{\|{#2}\|_{#1}}
\newcommand{\bignorm}[2][]{\bigl\|{#2}\bigr\|_{#1}}
\newcommand{\Bignorm}[2][]{\Bigl\|{#2}\Bigr\|_{#1}}

\newcommand{\dd}[1]{\ \mathrm{d}{#1}}
\newcommand{\conv}{\ast}
\newcommand{\transp}{\mathrm{T}}
\newcommand{\wrightarrow}{\rightharpoonup}
\newcommand{\wstarrightarrow}{\stackrel{*}{\rightharpoonup}}

\newcommand{\compactin}{\subset\subset}
\newcommand{\closure}[1]{\overline{#1}}

\newcommand{\compose}{\circ}

\newcommand{\range}[1]{\mrg({#1})}
\newcommand{\fourier}{\mathcal{F}}
\newcommand{\imag}{\mathrm{i}}

\newcommand{\seq}[1]{\{{#1}\}}

\newcommand{\deriv}{\mathrm{D}}
\newcommand{\embed}{\hookrightarrow}

\newcommand{\ball}[3][]{B^{#1}_{#2}({#3})}

\newcommand{\subgrad}{\partial}

\DeclareMathDelimiter{\llcornernew}{\mathopen}{AMSa}{"78}{AMSa}{"78}
\newcommand{\restricted}{\:\llcornernew\:}

\newcommand{\hausdorff}[1]{\mathcal{H}^{#1}}
\newcommand{\lebesgue}[1]{\mathcal{L}^{#1}}
\DeclareMathSymbol{\squarenew}{\mathord}{AMSa}{"03}
\newcommand{\wave}{\squarenew}
\newcommand{\breg}[2]{D^{#1}_{#2}}
\newcommand{\kronO}{\mO}

\newcommand{\Cspace}[3][]{\mathcal{C}_{#1}^{#2}({#3})}
\newcommand{\Ccspace}[2]{\mathcal{C}_{\mathrm{c}}^{#1}({#2})}

\newcommand{\lebesgueL}[1]{L^{#1}}

\newcommand{\LPspace}[3][]{\lebesgueL{#2}_{#1}({#3})}
\newcommand{\LPlocspace}[2]{\lebesgueL{#1}_{\mathrm{loc}}({#2})}

\newcommand{\Hspace}[3][]{H^{#2\conditionalcomma{#1}{#1}}(#3)}
\newcommand{\Hcspace}[3][]{H_0^{#2\conditionalcomma{#1}{#1}}(#3)}

\newcommand{\linspace}[3][]{\mathcal{L}^{#1}\bigl({#2},{#3}\bigr)}

\newcommand{\radon}{\mathcal{M}}
\newcommand{\radonspace}[1]{\mathcal{M}({#1})}
\newcommand{\tensorspace}[2][]{\mathcal{T}^{#1}(#2)}
\newcommand{\symgrad}{\mathcal{E}}
\newcommand{\borelalg}[1]{\mathcal{B}({#1})}
\newcommand{\binom}[2]{{{#1} \choose {#2}}}

\newcommand{\R}{\RR}
\newcommand{\N}{\NN}

\newcommand{\M}{\mathcal{M}}

\newlength{\formulaindentwidth}

\usepackage{pgfplots}
\pgfplotsset{compat=newest}

\usepgfplotslibrary{polar}

\begin{document}

\topical[Higher-order TV approaches and generalisations] {Higher-order total variation approaches and
  generalisations}

\author{Kristian Bredies and Martin Holler}

\address{Institute of Mathematics and Scientific Computing, University
of Graz, Heinrichstra\ss{}e 36, A-8010 Graz, Austria}
\ead{\href{mailto:kristian.bredies@uni-graz.at}{kristian.bredies@uni-graz.at}, \href{mailto:martin.holler@uni-graz.at}{martin.holler@uni-graz.at}}
\vspace{10pt}
\begin{indented}
\item[]November 2019
\end{indented}

\begin{abstract}
  Over the last decades, the total variation (TV) evolved to one of
  the most broadly-used regularisation functionals for inverse
  problems, in particular for imaging applications. When first
  introduced as a regulariser, higher-order generalisations of TV were
  soon proposed and studied with increasing interest, which led to a
  variety of different approaches being available today. We review
  several of these approaches, discussing aspects ranging from
  functional-analytic foundations to regularisation theory for linear
  inverse problems in Banach space, and provide a unified framework
  concerning well-posedness and convergence for vanishing noise level
  for respective Tikhonov regularisation. This includes general higher
  orders of TV, additive and infimal-convolution multi-order total
  variation, total generalised variation (TGV), and beyond. Further,
  numerical optimisation algorithms are developed and discussed that
  are suitable for solving the Tikhonov minimisation problem for all
  presented models. Focus is laid in particular on covering the whole
  pipeline starting at the discretisation of the problem and ending at
  concrete, implementable iterative procedures. A major part of this
  review is finally concerned with presenting examples and
  applications where higher-order TV approaches turned out to be
  beneficial. These applications range from classical inverse problems
  in imaging such as denoising, deconvolution, compressed sensing,
  optical-flow estimation and decompression, to image reconstruction
  in medical imaging and beyond, including magnetic resonance imaging
  (MRI), computed tomography (CT), magnetic-resonance positron
  emission tomography (MR-PET), and electron tomography.
\end{abstract}

\tableofcontents

\section{Introduction}

In this paper we give a review of higher-order regularisation
functionals of total-variation type, encompassing their development
from their origins to generalisations and most recent approaches.
Research in this field %
has in particular been triggered by the success of the total variation
(TV) as a regularisation functional for inverse problems on the one
hand, but on the other hand by the insight that tailored
regularisation approaches are indispensable for solving ill-posed
inverse problems in theory and in practice. The last decades comprised
active development of the latter topic which resulted in a variety of
different strategies for TV-based regularisation functionals that
model data with some inherent smoothness, possibly of
higher order or multiple orders. For these functionals, this
paper especially aims at providing a unified presentation of the
underlying regularisation aspects, giving an overview of numerical
algorithms suitable to solve associated regularised inverse problems
as well as showing the breadth of respective applications.

Let us put classical and higher-order total-variation regularisation
into an inverse problems context.
From the inverse problems point of view,
the central theme of
regularisation is the stabilisation of the inversion of an ill-posed
operator equation, which is commonly phrased as finding a $u \in X$
such that
\[
  K(u) = f
\]
for given $K: X \to Y$ and $f \in Y$, where $X$ and $Y$ are usually
Banach spaces.
Various approaches for regularisation exist, e.g., iterative regularisation, Tikhonov regularisation, regularisation based on spectral theory in Hilbert spaces, or regularisation by discretization. Being a %
regularisation and providing a stable inversion %
is mathematically well-formalised \cite{Engl96_book_regularization_ip_mh}, and usually comprises \emph{regularisation parameters}. Essentially, stable inversion means that each regularised inverse mapping from data to solution space is continuous in some topology, and being a regularisation requires in addition that, in case the measured data approximates the noiseless situation, a suitable choice of the regularisation parameters allows to approximate a solution that is meaningful and matches the noiseless data. %
These properties are %
typically referred to as \emph{stability} and %
\emph{convergence for vanishing noise}, respectively. For general non-linear
inverse problems, they usually depend on an interplay between the selected regularisation strategy and the forward operator $K$, where often, derivative-based assumptions on the local behaviour around the sought solution are made \cite{Engl96_book_regularization_ip_mh,kazufumi2014inverse}. In contrast, for linear forward operators, unified statements are commonly available such that regularisation properties solely depend on the regularisation strategy. We therefore consider linear inverse problems throughout the paper, i.e., the solution of $Ku =f$ where $K: X \to Y$ is always assumed to be linear and continuous.

Variational regularisation, which is %
the stabilised solution of such an inverse problems via energy minimisation methods, then encompasses --- and is often identified with --- Tikhonov regularisation (but comprises, for instance, also Morozov regularisation \cite{Morozov67_morozov_regularization_mh} or Ivanov regularisation \cite{Ivanov65_ivanov_regularization}). Driven by its success in practical applications, it has become a major direction of research in inverse problems. Part of its success may be explained by the fact that variational regularisation allows to incorporate a \emph{modelling} of expected solutions via \emph{regularisation functionals}. In a Tikhonov framework, this means that the solution of the operator equation $Ku=f$ is obtained via solving
\[ \min_{u \in X} \ S_f(Ku) + \mR_\alpha(u), \]
where $S_f:Y \rightarrow [0,\infty]$ is an energy that measures the discrepancy between $Ku$ and the measured data $f$, and $\mR_\alpha:X \rightarrow [0,\infty]$ is the regularisation functional that depends on regularisation parameters $\alpha$. 
From the analytical perspective, two main features of $\mR_\alpha$ are important: First, it needs to possess %
properties that allow to guarantee that the corresponding solution map enjoys the stability and convergence properties as mentioned above (typically, lower semi-continuity and coercivity in some topology). Second, it needs to provide a good model of reasonable/expected solutions of $Ku=f$ in the sense that $\mR_\alpha(u)$ is small for such reasonable solutions and $\mR_\alpha(u)$ is large for unreasonable solutions that suffer, for instance, from artefacts or noise.

While the first requirement is purely qualitative and known to be true
for a wide range of norms and seminorms, the second requirement
involves the modelling of expected solutions as well as suitable
quantification, having in particular in mind that the outcome should
be simple enough to be amenable to numerical solution algorithms.
Suitable models are for instance provided by various classical
smoothness measures such as Hilbert scales of smooth functions, i.e.,
by $H^s$-norms where $s \geq 0$, but also reflexive Banach-space norms
such as $L^p$-norms, associated Sobolev-space seminorms in $H^{k,p}$ for
$1 < p < \infty$, and Besov-space seminorms based on
wavelet-coefficient expansions \cite{Hofmann12_regularization_methods_mh,bredies2018imageprocessing,daubechies2004ista_mh}. The reflexivity of the underlying
spaces then helps to turn an ill-posed equation into a well-posed one,
since the direct method in the calculus of variations can be employed
with weak convergence.

However, there are reasons to consider Banach spaces that lack reflexivity,
with $L^1$-spaces and spaces of Radon measures being prominent examples.
Indeed, $L^1$-type norms as penalties in variational energies have seen a tremendous 
rise in popularity %
in the past two decades, %
most notably in the theory of \emph{compressed sensing} \cite{donoho2006compressedsensing_mh}. This is due to their property of
favouring \emph{sparsity} in solutions, which allows to model more specific
a-priori assumptions on the expected solutions than generic smoothness,
for instance.
While sparsity in $L^1$-type spaces over discrete domains, such as spaces of wavelet coefficients, is directly amenable to analysis, sparsity for continuous domains
requires to consider spaces
of \emph{Radon measures} and corresponding Radon-norm-type energies which are
natural generalisations of $L^1$-type norms.
Being the dual of a separable normed space then mitigates the non-reflexivity
of these spaces.
As a consequence, they %
play a major role in continuous models for sparsity-promoting variational regularisation strategies. %

A particular example %
is the total variation functional \cite{rudin1992tv_mh,chambolle1997inftv_mh}, see Section~\ref{sec:tv_reg} below for a precise definition,
which %
can be interpreted as %
the Radon norm %
realised as a dual norm
on the distributional derivative of $u$. As such, $\TV(u)$ is finite if and only if the distributional derivative of $u$ can be represented by a finite Radon measure.
The TV functional then penalises variations of $u$ via a norm on its derivative while still being finite in the case of jump discontinuities, i.e., %
when $u$ is piecewise smooth.
In particular, its minimisation
realises sparsity of the derivative which is often considered
a suitable model for piecewise constant functions.
In addition, %
it is convex and lower semi-continuous with respect to $L^p$-convergence for any $p \in [1,\infty]$, and coercive up to constants in suitable $L^p$-norm topologies.
These features make TV a reasonable model for piecewise constant %
solutions
and allow to obtain well-posedness of TV regularisation for a broad class of inverse problems. They can be considered as
some of the main reasons for the overwhelming popularity of TV in inverse problems, imaging sciences and beyond.

Naturally, the simplicity and desirable properties of TV come with a cost. As previously mentioned, interpreting TV as a functional that generalises the $L^1$-norm %
of the image gradient, compressed sensing theory suggests that this enforces sparsity of the gradient and hence piecewise constancy, i.e., one might expect that a TV-regularised function is non-constant only on low-dimensional subsets of its domain. While this might in fact be a feature if the sought solution is piecewise constant, it is not appropriate for general piecewise smooth data. Indeed, for non-piecewise-constant data, TV has the defect of producing artificial plateau-like structures in the reconstructions which became known as the \emph{staircasing effect} of TV. This effect is nowadays well-understood analytically in the case of denoising \cite{Nikolova00_mh,Caselles07_mh,Ring00_mh}, and recent results also provide an analytical confirmation of this fact in the context of inverse problems with finite-dimensional measurement data \cite{Carioni18sparsity_mh,Boyer19representer_mh}. %
The appearance of staircasing artefacts is in particular problematic since jump discontinuities are features which are, on the one hand, very prominent in visual perception and typically associated with relevant structures, and, on the other hand, important for automatic post-processing or interpretation of the data. 
As a result, it became an important research question in the past two decades how to improve upon this defect of TV regularisation while maintaining its desirable features, especially the sparsity-enforcing properties.

This review is concerned with the developments undertaken in this direction
that are related to the incorporation of higher-order derivatives, while maintaining the sparsity concepts realised by the Radon norm and the underlying
spaces of Radon measures. This resulted in a variety of different variational regularisation strategies, for which some are very successful in achieving the goal of providing an amenable model for piecewise smooth solutions.
It is also a central message of this review that %
the %
success of higher-order TV model in terms of %
modelling and regularisation effect depends very much on the structure and the functional-analytic setting in which the higher-order derivatives are included. Following this insight, we will discuss different higher-order regularisation functionals such as higher-order total variation, the infimal-convolution of higher-order TV as well as the total generalised variation (TGV), which carries out a cascadic decomposition to different orders of differentiation. Starting form the analytical framework of the total-variation functional and functions of bounded variation, we will introduce and analyse several higher-order approaches in a continuous setting, discuss their regularisation properties in a Tikhonov regularisation framework, introduce appropriate discretizations as well as numerical solution strategies for the resulting energy minimisation problems, and present various applications in image processing, computer vision, biomedical imaging and beyond.

Nevertheless, due to the broad range of the topic as well as the many
works published in its environment, it is impossible to give a complete
overview. The various references to the literature given throughout
the paper therefore only represent a selection.  Let us also point out
that we selected the presented material in particular on a basis that,
one the one hand, enables a treatment that is a unified as
possible. On the other hand, a clear focus is put on approaches for
which the whole pipeline ranging from mathematical modelling,
embedding into a functional-analytic context, proof of regularisation
properties, numerical discretization, optimisation algorithms and
efficient implementation can be covered. In addition, extensions and
further developments will shortly be pointed out when appropriate.
Especially, many of the applications in image processing, computer
vision, medical imaging and image reconstruction refer to these
extensions. The applications were further chosen to represent a wide
spectrum of inverse problems, their variational modelling and
higher-order TV-type regularisation, and, not negligible, successful
realisation of the presented
theory.  %
We finally aimed at providing a maximal amount of useful information
regarding theory and practical realisation in this context.

\section{Total-variation (TV) regularisation}
\label{sec:tv_reg}

Before discussing higher-order total variation and how it may be used
to regularise ill-posed inverse problems, let us begin with an
overview of first-order total variation. Throughout the review, we
mainly adapt a continuous viewpoint which means that the objects of
interest are usually functions on some fixed domain $\Omega$, i.e., an
non-empty, open and connected subset $\Omega \subset \RR^d$ in the
$d$-dimensional Euclidean space. This requires in particular a common
functional-analytic context for which we assume that the reader is
familiar with and refer to the books \cite{Adams2003sobolev_mh,evans1992measure_mh,Ziemer12weakly_mh} for further information.
In the following, we will make, for instance, use of the
Lebesgue spaces $\LPspace{p}{\Omega,H}$ for $H$-valued functions where
$H$ is a finite-dimensional real Hilbert space as well as their
measure-theoretic and functional-analytic properties. Also, concepts
of weak differentiability and properties of the associated Sobolev
spaces $\Hspace[p]{k}{\Omega,H}$ will be utilised without further
introduction. This moreover applies to the classical spaces such as
$\Cspace{}{\closure{\Omega},H}$, $\Ccspace{}{\Omega, H}$ and
$\Cspace[0]{}{\Omega,H}$, i.e., the spaces of uniformly continuous
functions on $\closure{\Omega}$, of compactly supported continuous
functions on $\Omega$ and its closure with respect to the supremum
norm. As usual, the respective spaces of $k$-times continuously
differentiable functions are denoted by
$\Cspace{k}{\closure{\Omega}, H}$, $\Ccspace{k}{\Omega,H}$ and
$\Cspace[0]{k}{\Omega, H}$ where $k$ could also be infinity, leading to
spaces of test functions. 

We further employ, throughout this section, basic
concepts from convex analysis and optimisation. At this point, we
would like to recall that for a convex function
$F: X \to {]{-\infty,\infty}]}$ defined on a Banach space $X$, the
\emph{subgradient} $\subgrad F (x)$ at a point $u \in X$ is the collection of
all $w \in X^*$ that satisfy the \emph{subgradient inequality}
\[
  F(u) + \scp[X^* \times X]{w}{v - u} \leq F(v) \qquad
  \text{for all} \quad v \in X.
\]
For $F$ proper, the \emph{Fenchel dual} or \emph{Fenchel
  conjugate} of $F$ is the function
$F^*: X^* \to {]{-\infty,\infty}]}$ defined by
\[
F^*(w) = \sup_{u \in X} \ \scp[X^* \times X]{w}{u} - F(u).
\]
The \emph{Fenchel inequality} then states that
$\scp[X^* \times X]{w}{u} \leq F(u) + F^*(w)$ for all $u \in X$ and
$w \in X^*$ with equality if and only if $w \in \subgrad F(u)$.
For more details regarding these notions and convex analysis in
general, we refer to research monographs covering this subject, for
instance \cite{Ekeland_book_mh,zalinescu2002convex}.

\subsection{Functions of bounded variation}

Generally, when solving a specific ill-posed inverse problem with, for
instance, Tikhonov regularisation, one usually has many choices
regarding the regularisation functional.  Now, while %
functionals associated with Hilbertian norms or seminorms possess
several advantages such as smoothness and allow, in addition, for
regularisation strategies that can be computed by solving a linear
equation, they are often not able to provide a good model for
piecewise smooth functions. This can, for instance, be illustrated as
follows.

\begin{example}
  \label{ex:sobolev-no-jump}
  Classical Sobolev spaces cannot contain non-trivial 
  piecewise constant 
  functions. 
  Let $\Omega \subset \RR^d$
  be a domain and $\Omega' \subset \Omega$ be non-empty, open with
  $\bdry\Omega'$ a null set.
  Then, the characteristic function $u = \chi_{\Omega'}$, i.e.,
  $u(x) = 1$ if $x \in \Omega'$ and $0$ otherwise, is not contained
  in $\Hspace[p]{1}{\Omega}$ for any $p \in [1,\infty]$.
  To see this, suppose that $v \in \LPspace{p}{\Omega,\RR^d}$ is
  the weak derivative of $u$.
  Let $\varphi \in \Ccspace{\infty}{\Omega'}$ be a test function.
  Clearly,
  \[
  \int_\Omega v \inprod \varphi \dd{x} = 
  -\int_\Omega u \divergence 
  \varphi \dd{x} = - \int_{\Omega'} \divergence  \varphi \dd{x} = 0.
  \]
  Hence, $v = 0$ on $\Omega'$. Likewise, one sees that also $v = 0$
  on $\Omega \without \closure{\Omega'}$. In total, 
  $v=0$ almost everywhere and as $v$ is the weak derivative
  of $u$, $u$ must be constant which is a contradiction.
\end{example}

The defect which is responsible for the failure of characteristic
function being (classical) Sobolev function can, however, be remedied
by allowing weak derivatives to be Radon measures. These are in
particular able to concentrate on Lebesgue null-sets; a property that
is necessary as the previous example just showed. In the following, we
introduce some basic notions and results about vector-valued Radon
measures, in particular, with an eye of embedding them into a
functional-analytic framework.  Moreover, we would like to have these
notions readily available when dealing with higher-order derivatives
and the associated higher-order total variation.

Throughout this section, let $\Omega \subset \RR^d$ be a domain and
$H$ a non-trivial finite-dimensional real Hilbert space with $\inprod$
and $\abs{\placeholder}$ denoting the associated scalar product and
norm, respectively. As usual, the case $H = \RR$ corresponds to the
scalar case and $H = \RR^d$ to the vector-field case, but, as we will
see later, $H$ could also be a space of higher-order tensors.
The following definitions and statements regarding basic measure theory
and can, for instance, be found in \cite{Ambrosio_mh}.
 
\begin{definition} 
  A \emph{vector-valued Radon measure} or \emph{$H$-valued
    Radon measure} on $\Omega$ is a function 
  $\mu: \borelalg{\Omega} \to H$ 
  on the Borel $\sigma$-algebra $\borelalg{\Omega}$ associated 
  with the standard topology on
  $\Omega$ satisfying the following properties:
  \begin{enumerate}
  \item 
    \label{item:radon_measure_i}
    It holds that $\mu(\emptyset) = 0$,
  \item 
    \label{item:radon_measure_ii}
    for each pairwise disjoint countable 
    collection $A_1,A_2,\ldots$ in $\borelalg{\Omega}$ it holds that
    $\mu(\bigcup_{i \in \NN} A_i) = \sum_{i=1}^\infty \mu(A_i)$ in 
    $H$.
  \end{enumerate}

  A \emph{positive Radon measure} is a function $\mu: \borelalg{\Omega} 
  \to [0,\infty]$ satisfying~\eqref{item:radon_measure_i},
  \eqref{item:radon_measure_ii} (with $H$ replaced by
  $[0,\infty]$) 
  as well as $\mu(K) < \infty$ for
  each compact $K \compactin \Omega$. It is called \emph{finite}, if
  $\mu(\Omega) < \infty$.
\end{definition}

Naturally, vector-valued Radon measures can be associated to an integral.
For $\mu$ an $H$-valued Radon measure and step functions
$u = \sum_{j=1}^N c_j \chi_{A_j}$, $v = \sum_{j=1}^N v_j \chi_{A_j}$
with $c_1,\ldots,c_N \in \RR$, $v_1,\ldots,v_N \in H$ and
$A_1,\ldots,A_N \in \borelalg{\Omega}$, the following integrals make
sense:
\[
\int_\Omega u \dd{\mu} = \sum_{j=1}^N c_j \mu(A_j) \in H, \qquad
\int_\Omega v \inprod \dd{\mu} = \sum_{j=1}^N v_j \inprod \mu(A_j) \in
\RR.
\]
For uniformly
continuous functions $u: \closure{\Omega} \to \RR$ and
$v: \closure{\Omega} \to H$, the integrals are given as
\[
\int_\Omega u \dd{\mu} = \lim_{n \to \infty} \int_\Omega u^n \dd\mu,
\qquad \int_\Omega v \inprod \dd{\mu} = \lim_{n \to \infty} \int_\Omega
v^n \inprod \dd{\mu}
\]
where $\seq{u^n}$ and $\seq{v^n}$ are sequences of step functions
converging uniformly to $u$ and $v$, respectively. Of course, the
above integrals are well-defined, meaning that there are approximating
sequences as stated and the above limits exist independently of the
specific choice of the approximating sequences.  The following
definition is the basis for introducing a norm for $H$-valued Radon
measures.

\begin{definition}
  For a vector-valued Radon measure $\mu$ 
  on $\Omega$ the positive Radon measure $\abs{\mu}$ given by
  \[
  \fl
  \abs{\mu}(A) = \sup \
  \Bigset{\sum_{i=1}^\infty \abs{\mu(A_i)} }{A_1,A_2,\ldots \in 
    \borelalg{\Omega} \ \text{pairwise disjoint}, \ 
    A_i \subset A \ \text{for all} \ i \in \NN
  }
  \]
  is called the \emph{total-variation measure} of $\mu$.
\end{definition}

The total-variation measure is always positive and finite, i.e.,
$0 \leq \abs{\mu}(A) < \infty$ for all $A \in \borelalg{\Omega}$.  By
construction, $\mu$ is absolutely continuous with respect to
$\abs{\mu}$, i.e., $\mu(A) = 0$ whenever $\abs{\mu}(A) = 0$ for a
$A \in \borelalg{\Omega}$. By Radon--Nikod\'ym's theorem, we thus have
that each $H$-valued Radon measure $\mu$ can be written as
$\mu = \sigma_\mu \abs{\mu}$ with
$\sigma_\mu \in \LPspace[\abs{\mu}]{\infty}{\Omega,H}$ such that
$\norm[\infty]{\sigma_\mu} \leq 1$ and $\abs{\sigma_\mu} = 1$ almost
everywhere with respect to $\abs{\mu}$.  In this light, integration
can also be phrased as
\[
\int_\Omega u \dd{\mu} = \int_\Omega u \sigma_\mu \dd{\abs{\mu}},
\qquad
\int_\Omega v \inprod \dd{\mu} = \int_\Omega v \inprod \sigma_\mu \dd{\abs{\mu}}
\]
for %
$u: \closure{\Omega} \to \RR$, $v: \closure{\Omega} \to H$ uniformly
continuous. The following theorem, which is a direct consequence of \cite[Theorem 6.19]{rudin2006real_complex_analysis_mh}, 
provides a useful characterisation of the space of vector valued measures as the dual of a separable space.

\begin{proposition}
  \label{prop:duality_radon_measures}
  The space $\radonspace{\Omega, H}$ of all vector-valued
  Radon measures equipped with the norm 
  $\norm[\radon]{\mu} = \abs{\mu}(\Omega)$ for $\mu \in 
  \radonspace{\Omega, H}$
  is a Banach space.

  It can be identified with the dual space $\Cspace[0]{}{\Omega, H}^*$
  as follows. For each $T \in \Cspace[0]{}{\Omega, H}^*$ there exists
  a unique $\mu \in \radonspace{\Omega, H}$ such that 
  \[
  \norm[\mathcal{C}_0^*]{T} = \norm[\radon]{\mu}, \qquad
  T(\varphi) = \int_\Omega \varphi \inprod \dd{\mu} \quad \text{for all}
  \quad \varphi \in \Cspace[0]{}{\Omega,H}.
  \]
\end{proposition}

In particular, one has a notion of weak*-convergence of Radon
measures.  For a sequence $\seq{\mu^n}$ and an element $\mu^*$ in
$\radonspace{\Omega, H}$ we have that $\mu^n \wstarrightarrow \mu^*$
in $\radonspace{\Omega,H}$ if 
\[
\text{for all} \ \ \varphi \in \Cspace[0]{}{\Omega, H}: \quad
\int_\Omega
\varphi \inprod \dd{\mu^n} \to \int_\Omega \varphi \inprod \dd{\mu^*}
\ \ \text{as} \ \ n \to \infty.
\]
As the predual space $\Cspace[0]{}{\Omega,H}$ is separable, the
Banach--Alaoglu theorem yields in particular the sequential
relative weak*-compactness of bounded sets. That means for instance that
a bounded sequence always admits a weakly*-convergent subsequence, a 
property that may compensate for the lack of reflexivity of 
$\radonspace{\Omega,H}$.

The interpretation as a dual space as well as the density of 
test functions in $\Cspace[0]{}{\Omega,H}$ 
also allows to conclude that in
order for a linear functional $T$ defining a Radon measure, it
suffices to test against $\varphi \in \Ccspace{\infty}{\Omega, H}$ and
to establish $\abs{T(\varphi)} \leq C \norm[\infty]{\varphi}$ for all
$\varphi \in \Ccspace{\infty}{\Omega,H}$ and $C > 0$ independent
of $\varphi$.
This is useful for  derivatives, i.e., the 
derivative of a $u \in \LPlocspace{1}{\Omega,H}$ defines a Radon measure 
in $\radonspace{\Omega,H^d}$ if
\begin{equation}
  \label{eq:weak_derivative_radon_measure}
  \Bigabs{\int_\Omega u \inprod \divergence \varphi \dd{x}} \leq
  C \norm[\infty]{\varphi} \qquad
  \text{for all} \quad \varphi \in \Ccspace{\infty}{\Omega,H^d}.
\end{equation}
In this case, we denote by $\grad u \in \radonspace{\Omega,H^d}$ the
unique $H^d$-valued Radon measure for which
$\int_\Omega \varphi \inprod \dd{\grad u} = -\int_\Omega u \divergence \varphi 
\dd{x}$ for all $\varphi \in \Ccspace{\infty}{\Omega,H^d}$.
Here, $H^d$ is equipped with the scalar product $x \inprod y = 
\sum_{i=1}^d x_i \inprod y_i$ for $x,y \in H^d$.
In the case where~\eqref{eq:weak_derivative_radon_measure} fails, 
there exists a sequence $\seq{\varphi^n}$ in 
$\Ccspace{\infty}{\Omega,\RR^d}$ with 
$\norm[\infty]{\varphi^n} = 1$ and $\abs{\int_\Omega u \inprod \divergence \varphi
  \dd{x}} \to \infty$ as $n\to \infty$. Thus, allowing the supremum to take
the value $\infty$, this yields following definition.

\begin{definition}
  \label{def:tv}
  The \emph{total variation} of a $u \in \LPlocspace{1}{\Omega,H}$ is 
  the value
  \[
  \TV(u) = \sup \ \Bigset{\int_\Omega u \inprod \divergence \varphi 
    \dd{x}}{\varphi \in \Ccspace{\infty}{\Omega,H^d}, \ 
    \norm[\infty]{\varphi} \leq 1}.
  \]
\end{definition}

Clearly, in case $\TV(u) < \infty$, we have
$\grad u \in \radonspace{\Omega,H^d}$ with
$\norm[\radon]{\grad u} = \TV(u)$. Trivially, for scalar functions,
i.e., $H = \RR$, one recovers the well-known definition \cite{Ambrosio_mh,rudin1992tv_mh}.
Also, one immediately sees that $\TV$ is invariant to translations and
rotations, or, more generally, to Euclidean-distance preserving
transformations. This is the reason that this definition is also
referred to as the \emph{isotropic total variation}.

\begin{example}
  Piecewise constant functions may have a Radon measure as derivative.
  Let $\Omega' \subset \Omega$ be a subdomain such that
  $\bdry\Omega' \cap \Omega$ can be parameterised by finitely many Lipschitz
  mappings. Then, the outer normal $\nu$ exists
  almost everywhere in $\bdry\Omega' \cap \Omega$ with respect to
  the Hausdorff $\hausdorff{d-1}$ measure and one can employ the
  divergence theorem. This yields, for $u = \chi_{\Omega'}$ and
  $\varphi \in \Ccspace{\infty}
  {\Omega,\RR^d}$ with $\norm[\infty]{\varphi} \leq 1$ that
  \[
  \fl
  \int_\Omega u \divergence \varphi \dd{x} = \int_{\bdry \Omega' \cap \Omega}
  \varphi \inprod \nu \dd{\hausdorff{d-1}} 
  = \int_\Omega \varphi \inprod \dd{\nu\hausdorff{d-1} \restricted 
    (\bdry\Omega' \cap \Omega)}
  \leq \hausdorff{d-1}(\bdry\Omega' \cap \Omega)
  \]
  so
  $\grad u = - \nu \hausdorff{d-1} \restricted (\bdry\Omega'
    \cap \Omega)$
  is a Radon measure. One sees, for instance via approximation, that
  $\norm[\radon]{\grad u} = \hausdorff{d-1}(\bdry\Omega' \cap
  \Omega)$.
\end{example}

The class of sets $\Omega' \subset \Omega$ for which $\chi_{\Omega'}$
possesses a Radon measure as weak derivative is actually much greater
than the class of bounded Lipschitz domains.  These are the sets of
\emph{finite perimeter}, denoted by
$\Per(\Omega') = \norm[\radon]{\grad \chi_{\Omega'}}$.
One the other hand, for $u \in \Hspace[1]{1}{\Omega}$, we have
$\TV(u) = \int_\Omega \abs{\grad u} \dd{x}$ and the weak derivative
as Radon measure is just $\grad u \lebesgue{d}$, i.e., the Sobolev
derivative interpreted as a weight on the Lebesgue measure. Collecting all 
functions whose weak derivative is a Radon measure, we arrive at the following space.

\begin{definition}
  The space
  \[
  \BV(\Omega,H) = \set{u \in \LPspace{1}{\Omega,H}}{\TV(u) < \infty},
  \qquad
  \norm[\BV]{u} = \norm[1]{u} + \TV(u)
  \]
  is the space of $H$-valued \emph{functions of bounded variation}.
  In case $H=\RR$, we denote by $\BV(\Omega) = \BV(\Omega,\RR)$ and
  just refer to functions of bounded variation.
\end{definition}

\begin{proposition}
  The space $\BV(\Omega,H)$ with the associated norm 
  is a Banach space. The total variation functional
  $\TV$ is a continuous seminorm on 
  $\BV(\Omega,H)$ which vanishes exactly at the constant functions,
  i.e., $\ker(\TV) = H\ones$, with $H\ones$ being the set of constant, $H$-valued functions.
\end{proposition}

The total variation functional is just designed to possess many
convenient properties \cite{Ambrosio_mh}.

\begin{proposition}
  \mbox{}
  \begin{itemize}
  \item 
    The functional $\TV$ is proper, convex and lower semi-continuous
    on each $\LPspace{p}{\Omega,H}$, i.e., for
    $1 \leq p \leq \infty$.
  \item
    For $1 \leq p < \infty$, each
    $u \in \BV(\Omega,H) \cap \LPspace{p}{\Omega,H}$ can smoothly 
    be approximated as
    follows: For $\varepsilon > 0$, there exists $u^\varepsilon \in 
    \Cspace{\infty}{\Omega,H} \cap \BV(\Omega,H) \cap \LPspace{p}{\Omega,H}$
    such that
    \[
    \norm[p]{u - u^\varepsilon} \leq \varepsilon, 
    \qquad \abs{\TV(u) - \TV(u^\varepsilon)} \leq \varepsilon.
    \]
  \item If $\Omega$ is a bounded Lipschitz domain, then there exists a
    constant $C > 0$ such that for each $u \in \BV(\Omega,H)$ with
    $\int_\Omega u \dd{x} = 0$, the \emph{Poincar\'e--Wirtinger
      estimate}
    \[
    \norm[d/(d-1)]{u} \leq C \TV(u)
    \]
    holds.
  \end{itemize}
\end{proposition}

From the regularisation-theoretic point of view, the fact that $\TV$
is proper, convex and lower semi-continuous on Lebesgue spaces is
relevant, a property that fails for the Sobolev-seminorm
$\norm[1]{\grad \placeholder}$.  The Poincar\'e--Wirtinger estimate
can be interpreted as a coercivity property on a subspace with
codimension 1.  Also note that this estimate is the same as for
$\Hspace[1]{1}{\Omega,H}$-functions and the respective constants $C$
coincide. Consequently, the embedding properties of the latter space
transfer immediately.

\begin{proposition}
  \label{prop:bv_embeddings}
  Let $\Omega$ is a bounded Lipschitz domain.
  Then, 
  \begin{itemize}
  \item 
    the embedding $\BV(\Omega,H)
    \embed \LPspace{d/(d-1)}{\Omega,H}$ (with $d/(d-1) = \infty$ for
    $d = 1$) exists and is continuous,
  \item
    the embedding $\BV(\Omega,H) \embed \LPspace{p}{\Omega,H}$
    is compact for each $1 \leq p < d/(d-1)$,
  \item
    each bounded sequence $\seq{u^n}$ in $\BV(\Omega,H)$ possesses
    a subsequence $\seq{u^{n_k}}$ which converges to a $u \in \BV(\Omega,H)$
	weak* in $\BV(\Omega,H)$, which we define as $u^{n_k} \to u$ in $\LPspace{1}{\Omega,H}$, 
    $\grad u^{n_k} \wstarrightarrow \grad u$ in $\radonspace{\Omega,H^d}$
    as $k \to \infty$.
  \end{itemize}
\end{proposition}

Consequently, the total variation is suitable for regularising
ill-posed inverse problems in certain $\lebesgueL{p}$-spaces.

\subsection{Tikhonov regularisation} \label{sec:tikhonov_tv}

Let us now turn to solving ill-posed inverse problems with
Tikhonov regularisation and $\BV$-based penalty, i.e., solving \[
Ku = f
\]
for some data $f$ in a Banach space $Y$. As mentioned in the introduction,
since the focus of this review is on regularisation terms rather than tackling inverse problems in the most possible generality, we restrict ourselves here to linear and continuous forward operators $K: \LPspace{d/(d-1)}{\Omega} \to Y$.
Nevertheless we note that, building on the results developed here for the linear setting, an extension to non-linear operators typically boils down to ensuring additional requirements on the non-linear forward model rather than the regularisation term, see for instance \cite{Tikhonov98_nonlinear_ill_posed_mh,Engl96_book_regularization_ip_mh,Hofmann07_nonlinear_tikhonov_banach_mh}.

Measuring the discrepancy in terms of the norm in $Y$, the 
problem is then to solve
\[
\min_{u \in \BV(\Omega)} \ \frac{\norm[Y]{Ku - f}^q}{q} + \alpha 
\int_\Omega \dd{\abs{\grad u}}
\]
for some exponent $q \geq 1$. Usually, $Y$ is some Hilbert
space and $q = 2$, resulting in a quadratic discrepancy, which is often used in
case of Gaussian noise. For impulsive noise (or salt-and-pepper
noise), the space $Y = \LPspace{1}{\Omega'}$, with $\Omega'$ a domain, turns out to be useful. In
case of Poisson noise, however, it is not advisable to take the norm
but rather the \emph{Kullback--Leibler divergence} between $Ku$ and
$f$, i.e.~$\KL(Ku, f)$, where $\KL$ is given, for
$f \in \LPspace{1}{\Omega'}$ with $f \geq 0$ almost everywhere,
according to the non-negative integral
\begin{equation} \label{eq:kl_definition}
\KL(v, f) = \int_{\Omega'} f \Bigl( 
\frac{v}{f} - \log\Bigl( \frac{v}{f} \Bigr) - 1 \Bigr) \dd{x}
\end{equation}
provided that $v \geq 0$ a.e., and $\infty$ else. In particular, in
this context, we agree to set the integrand to $v$ where $f = 0$ and
to $\infty$ where $v = 0$ and $f > 0$.

In the following, we assume to have given a discrepancy functional
$S_f: Y \to [0,\infty]$ that is proper, convex, lower
semi-continuous and coercive. 
This is not the most general case but will be
sufficient for us in order to ensure existence of minimizers of the
Tikhonov functional.

\begin{theorem}
  \label{thm:tv_reg_existence}
  Let $\Omega$ be a bounded Lipschitz domain, $Y$ be a Banach space, $K: \LPspace{d/(d-1)}{\Omega} \to Y$ linear and continuous
  (weak*-to-weak-continuous in case $d=1$),
  $S_f: Y \to {[{0,\infty}]}$ a proper, convex, lower
  semi-continuous and coercive discrepancy functional associated with
  some data $f$ and $\alpha > 0$. Then, there exist solutions of
  \begin{equation}
    \label{eq:general_tv_min}
    \min_{u \in \LPspace{d/(d-1)}{\Omega} } \ S_f(Ku) + \alpha \TV(u). %
  \end{equation}
  If $S_f$ is strictly convex and $K$ is injective, the solution is
  unique whenever the minimum is finite.
\end{theorem}

We provide the proof for the sake of completeness and as a prototype
for the generalisation to higher-order functionals.

\begin{proof}
  Assume that the objective functional in~\eqref{eq:general_tv_min} is
  proper, otherwise, there is nothing to show. For a minimising
  sequence $\seq{u^n}$, the Poincar\'e--Wirtinger inequality gives
  boundedness of
  $\seq{u^n - \abs{\Omega}^{-1} \int_\Omega u^n \dd{x}}$ in
  $\LPspace{d/(d-1)}{\Omega}$ while the coercivity of $S_f$ yields the
  boundedness of $\seq{Ku^n}$. By continuity,
  $\seq{K(u^n - \abs{\Omega}^{-1}\int_\Omega u^n \dd{x})}$ must be
  bounded, so if $K \ones \neq 0$, then $\seq{\int_\Omega u^n \dd{x}}$
  is bounded as otherwise, $\seq{Ku^n}$ would be unbounded.  In the
  case that $K \ones = 0$, we can without loss of generality assume
  that $\int_\Omega u^n \dd{x} = 0$ for all $n$ as shifting along
  constants does not change the functional value. In each case,
  $\seq{\int_\Omega u^n \dd{x}}$ is bounded, so $\seq{u^n}$ must be
  bounded in $\LPspace{d/(d-1)}{\Omega}$. Hence, by compact embedding
  (Proposition~\ref{prop:bv_embeddings}) we have $u^{n_k} \to u^*$ in
  $\LPspace{1}{\Omega}$ as $k \to \infty$ for a subsequence $\seq{u^{n_k}}$ and $u^* \in \BV(\Omega)$.
  Reflexivity and continuity of $K$ (weak* sequential compactness and
  weak*-to-weak continuity in case $d=1$) give
  $Ku^{n_k} \wrightarrow Ku^*$ in $Y$ for another subsequence (not
  relabelled). By lower semi-continuity, $u^*$ has to be a solution
  to~\eqref{eq:general_tv_min}.

  Finally, if $S_f$ is strictly convex and $K$ is injective,
  then $S_f \compose K$ is already strictly convex, so minimizers have to
  be unique.
\end{proof}

\pagebreak
\begin{example}
  \label{ex:discrepancies_existence}
  \mbox{}
  \begin{itemize}
  \item The discrepancy functional
    $S_f(v) = \frac1q \norm[Y]{v - f}^q$ for some $f \in Y$
    is obviously proper, convex, lower semi-continuous and coercive.
  \item It follows from Lemma \ref{lem:kl_basic_properties} in the appendix that the
    discrepancy $S_f(v) = \KL(v,f)$ defined on
    $Y = \LPspace{1}{\Omega'}$ for $f \in \LPspace{1}{\Omega'}$ with
    $f \geq 0$ almost everywhere is proper, convex and coercive in $L^1(\Omega')$.
    Lower semi-continuity in turn follows as special case of Lemma \ref{lem:kl_lim_inf_sup_estimates}.
  \end{itemize}
\end{example}

\begin{remark}
  Note that if the inversion of $K: \LPspace{p}{\Omega} \to Y$ is
  well-posed for some $p \in [1,\infty]$, then solutions
  of~\eqref{eq:general_tv_min} still exist (even for $\alpha = 0$).
  Clearly, the $\TV$ penalty is not necessary for obtaining a
  regularising effect for these problems. In this case, minimising the
  Tikhonov function with $\TV$ penalty may the interpreted as
  denoising. The most prominent example might be the
  \emph{Rudin-Osher-Fatemi} problem \cite{rudin1992tv_mh}
  which reads as
  \[
  \min_{u \in \LPspace{2}{\Omega}} \ \frac12 \int_{\Omega} \abs{u - f}^2 \dd{x}
  + \alpha \TV(u)
  \]
  for $f \in \LPspace{2}{\Omega}$.  Here, as the identity is
  ``inverted'', the effect of total-variation regularisation can be studied
  in detail.
  Minimisation problem of this type with other regularisation
  functionals are thus a good benchmark test for the properties of
  this functional.
\end{remark}

The stability of solutions in case of varying $f$ depends, of course,
on the dependence of $S_f$ on $f$. The appropriate notion here is the
convergence of the discrepancy functional, i.e., 
for a sequence $\seq{f^n}$ %
and limit $f$, %
we say that
$S_{f^n}$ %
converges to $S_f$ if
\begin{equation} \label{eq:discrepancy_convergence}
  \left\{
\begin{array}{rlrl}
  S_f(v) & \displaystyle
           \leq \liminf_{n \to \infty} 
           \ S_{f^n}(v^n)
  & \text{whenever}\ & v^n \wrightarrow
                       v \ \text{in}\ Y, \\
  S_f(v) 
         &\displaystyle
           \geq \limsup_{n \to \infty}
           \ S_{f^n} (v) & %
         \text{for each} & v \in Y.                                          
\end{array}
\right.
\end{equation}
Moreover, we say that $\seq{S_{f^n}}$ is \emph{equi-coercive} if there
is a coercive function $S_0: Y \to {[{0,\infty}]}$ such that
$S_{f^n} \geq S_0$ in $Y$ for each $n$.

\begin{theorem} \label{thm:tv_reg_stability}
  In the situation of Theorem \ref{thm:tv_reg_existence}, assume that 
  $S_{f^n}$ converges to $S_f$ in the sense of \eqref{eq:discrepancy_convergence} and $\seq{S_{f^n}}$ is
  equi-coercive. Then, for each sequence of minimizers $\seq{u^n}$
  of~\eqref{eq:general_tv_min} with discrepancy $S_{f^n}$,

  \begin{itemize}
  \item
    either $S_{f^n}(Ku^n) + \alpha \TV(u^n) \to \infty$ as
    $n \to \infty$ and~\eqref{eq:general_tv_min} with discrepancy
    $f$ does not admit a finite solution,
  \item or
    $S_{f^n}(Ku^n) + \alpha \TV(u^n) \to \min_{u \in
      \LPspace{d/(d-1)}{\Omega}} S_f(u) + \alpha \TV(u)$
    as $n \to \infty$ and there is, possibly up to constant shifts, a
    weak accumulation point $u \in \LPspace{d/(d-1)}{\Omega}$ (weak*
    accumulation point for $d=1$) that
    minimises~\eqref{eq:general_tv_min} with discrepancy $S_f$.
  \end{itemize}  
  For each subsequence $\seq{u^{n_k}} $ weakly converging to some $u$ in
  $\LPspace{d/(d-1)}{\Omega}$ ($u^{n_k} \wstarrightarrow u$ in case
  $d=1$), it holds that
  $\TV(u^{n_k}) \to \TV(u)$ as $k \to \infty$ and
  $u$ solves~\eqref{eq:general_tv_min} with discrepancy $S_f$.
  If solutions to the latter are unique, %
  we have $u^n \wrightarrow u$ in
  $\LPspace{d/(d-1)}{\Omega}$ ($u^n \wstarrightarrow u$ in case
  $d=1$).
\end{theorem}

\begin{proof}
  Let, in the following $\int_\Omega u^n \dd{x} = 0$ for all $n$ if
  $K\ones = 0$ and denote by $F = S_f \compose K + \alpha \TV$ as well
  as $F_n = S_{f^n} \compose K + \alpha \TV$.  First of all, suppose
  that $\seq{F_n(u^n)}$ is bounded. As $\seq{S_{f^n}}$ is
  equi-coercive, we can conclude as in the proof of
  Theorem~\ref{thm:tv_reg_existence} that $\seq{u^n}$ is
  bounded. Therefore, a weak accumulation point (weak* in case $d=1$)
  exists.

  Suppose that $u^{n_k} \wrightarrow u$ as $k \to \infty$. Then,
  \[
  S_f(Ku) \leq \liminf_{k \to \infty} \ S_{f^{n_k}}(Ku^{n_k}), \qquad \TV(u)
  \leq \liminf_{k \to \infty} \ \TV(u^{n_k})
  \]
  as well as, for each $u' \in \LPspace{d/(d-1)}{\Omega}$
  \[
  \fl 
  F(u) 
  \leq \liminf_{k \to \infty}
  S_{f^{n_k}}(Ku^{n_k}) + \alpha \TV(u^{n_k}) \leq \limsup_{k \to
    \infty} S_{f^{n_k}}(Ku') + \alpha \TV(u') \leq 
  F(u')
  \]
  Thus, $u$ is a minimizer for $F$ and plugging in $u'=u$ we see that
  $\lim_{k \to \infty} F_{n_k}(u^{n_k}) = F(u)$. In order to obtain
  $\lim_{k \to \infty} \TV(u^{n_k}) = \TV(u)$, suppose that
  $\limsup_{k \to \infty} \TV(u^{n_k}) > \TV(u)$, such that
  \[
  \liminf_{k \to \infty} \ S_{f^{n_k}}(Ku^{n_k}) 
  \leq \lim_{k\to\infty} F_{n_k}(u^{n_k}) - \alpha \limsup_{k \to \infty}
  \ \TV(u^{n_k}) < S_f(Ku)
  \]
  which is a contradiction. Thus, $\lim_{k \to \infty} \TV(u^{n_k}) = \TV(u)$.
  Finally, if $u$ is the unique minimizer for~\eqref{eq:general_tv_min} with
  discrepancy $S_f$, then $u^n \wrightarrow u$ as $n \to \infty$
  for the whole sequence ($u^n \wstarrightarrow u$ in case $d=1$) as
  any subsequence has to contain another subsequence that converges
  weakly (weakly*) to $u$.

  In order to conclude the proof, suppose that
  $\liminf_{n \to \infty} F_n(u^n) < \infty$. In that case, the above
  arguments yield an accumulation point as stated as well as a
  minimizer $u \in \BV(\Omega)$ of $F$ with
  $F(u) \leq \liminf_{n \to \infty} F_n(u^n)$. In particular, 
  $F$ is proper. By convergence of $S_{f^n}$ to $S_f$ and minimality,
  we have
  \[
  F(u) \geq \limsup_{n \to \infty} \ F_n(u) \geq \limsup_{n \to \infty} 
  \ F_n(u^n) \geq F(u)
  \]
  so the whole sequence of functional values converges. 

  Finally, in case $F_n(u^n) \to \infty$ as $n \to \infty$, $F$ cannot
  be proper: Otherwise, we obtain analogously to the above that
  $\infty > F(u) \geq \limsup_{n \to \infty} F_n(u) \geq \liminf_{n
    \to \infty} F_n(u^n)$
  for some $u \in \BV(\Omega)$ which is a contradiction.
\end{proof}

\begin{remark} The convergence of discrepancies as in \eqref{eq:discrepancy_convergence} is related to Gamma convergence. Indeed, the difference is that, for the latter, on the right hand side of the $\limsup$ inequality, an arbitrary sequence converging to $v$ is allowed (instead of the constant sequence). In this context, as can be seen in the proof of the stability result above, one could still weaken the $\limsup$-assumption in \eqref{eq:discrepancy_convergence} by allowing not only the constant recovery sequence but any sequence for which the regularisation functional converges. However, in order to maintain an assumption on the discrepancy term that is independent of the choice of regularisation, we chose the slightly stronger condition.
\end{remark}

\begin{example}
  \mbox{}
  \begin{itemize}
  \item A typical discrepancy is some power of the norm-distance in $Y$, i.e.,
    $S_f(v) = \frac1q \norm[Y]{v - f}^q$ for some $q \geq 1$.  It is
    easy to show that whenever $f^n \to f$ in $Y$, $S_{f^n}$ converges
    to $S_f$ in the above sense. Also, the equi-coercivity of
    $\seq{S_{f^n}}$ is immediate.
  \item
    For the Kullback--Leibler divergence, let
    $Y = \LPspace{1}{\Omega'}$ for some $\Omega'$
    and assume that $\seq{f^n}, f$ in 
    $\LPspace{1}{\Omega'}$ are such that
    $f_n \leq Cf$ a.e. in $\Omega'$ for some $C>0$ and
    $\KL(f,f^n) \to 0$ as $n \to \infty$.
 
    Then, it follows from Lemma \ref{lem:kl_lim_inf_sup_estimates} 
    in the appendix that $S_{f^n} = \KL(\placeholder, f^n)$ converges to
    $S_f = \KL(\placeholder, f)$, and also that $\|f_n-f\|_1\rightarrow 0$.
    The latter in particular implies boundedness of $\seq{f_n}$ in $L^1(\Omega')$
    which, together with the coercivity estimate of Lemma \ref{lem:kl_basic_properties}
   shows that $\seq{S_{f^n}}$ is equi-coercive.
  \end{itemize}
\end{example}

In addition to well-posedness of the Tikhonov-functional minimisation,
one is of course interested in regularisation results, i.e., the
convergence of solutions to a minimum-$\TV$-solution provided that the
data converges and $\alpha \to 0$ in some sense. For this purpose, let
$u^\dagger \in \BV(\Omega)$ be a minimum-$\TV$-solution of
$Ku^\dagger = f^\dagger$ for some data $f^\dagger$ in $Y$, i.e.,
$\TV(u^\dagger) \leq \TV(u)$ for each $Ku = f^\dagger$, suppose
that for each $\delta > 0$ one has given a $f^\delta \in Y$ such that
$S_{f^\delta}(f^\dagger) \leq \delta$, and denote by $u^{\alpha,\delta}$ a
solution of~\eqref{eq:general_tv_min} for parameter $\alpha > 0$ and
data $f^\delta$.

\begin{theorem}
  \label{thm:tv_reg_convergence}
   In the situation of Theorem \ref{thm:tv_reg_existence}, let the discrepancy functionals $\seq{S_{f^\delta}}$ be 
  equi-coercive and converge to
   $S_{f^\dagger}$ in the sense of \eqref{eq:discrepancy_convergence} for some
   data $f^\dagger \in Y$ with $S_{f^\dagger}(v) = 0$ if and only if $v = f^\dagger$.
  Choose for each $\delta > 0$ the parameter $\alpha > 0$ such that
  \[
  \alpha \to 0, \quad
  \frac{\delta}{\alpha} \to 0 \qquad \text{as} \qquad \delta \to 0.
  \]
  Then, again up to constant shifts, $\seq{u^{\alpha,\delta}}$ has at least one weak accumulation
  point in $L^{d/(d-1)}(\Omega)$ (weak* in case $d=1$). Each such accumulation point is a 
  minimum-$\TV$-solution of $Ku = f^\dagger$ and $\lim_{\delta \to 0}
  \TV(u^{\alpha,\delta}) = \TV(u^\dagger)$.
\end{theorem}

\begin{proof}
  Again we assume that $\int_\Omega u^{\alpha,\delta} \dd{x} = 0$ for all $(\alpha,\delta)$ if
  $K\ones = 0$.
  Using the optimality of $u^{\alpha,\delta}$ for~\eqref{eq:general_tv_min}
  compared to $u^\dagger$ gives
  \[
  S_{f^\delta}(Ku^{\alpha,\delta}) + \alpha \TV(u^{\alpha,\delta})
  \leq \delta + \alpha \TV(u^\dagger).
  \]
  Since $\alpha \to 0$ as $\delta \to 0$, we have that
  $S_{f^\delta}(Ku^{\alpha,\delta}) \to 0$ as $\delta \to 0$. 
  Moreover, as also $\delta/\alpha \to 0$, it follows that
  $\limsup_{\delta \to 0} \TV(u^{\alpha,\delta}) \leq \TV(u^\dagger)$.
  This allows to conclude that $\seq{u^{\alpha,\delta}}$ is bounded
  in $\BV(\Omega)$ and, by embedding, admits a weak accumulation point in $L^{d/(d-1)}(\Omega)$ (weak* in case $d=1$).
  
  Next, let $u^*$ be such an accumulation point associated with
  $\seq{\delta_n}$, $\delta_n \to 0$ as well as the corresponding
  parameters $\seq{\alpha_n}$. Then,
  $S_{f^\dagger}(Ku^*) \leq \liminf_{n \to \infty}
  S_{f^{\delta_n}}(Ku^{\alpha_n,\delta_n}) = 0$,
  so
  $Ku^* = f^\dagger$.  Moreover,
  $\TV(u^*) \leq \liminf_{n \to \infty} \TV(u^{\alpha_n,\delta_n})
  \leq \TV(u^\dagger)$, hence $u^*$ is a minimum-$\TV$-solution.
  In particular, $\TV(u^*) = \TV(u^\dagger)$, so $\lim_{n \to \infty}
  \TV(u^{\alpha_n,\delta_n}) = \TV(u^\dagger)$.

  Finally, each sequence of $\seq{\delta_n}$, $\delta_n \to 0$
  contains another subsequence $\seq{u^{\delta_n}}$ for which
  $\TV(u^{\alpha_n,\delta_n}) \to \TV(u^\dagger)$ as $n \to \infty$,
  so $\TV(u^{\alpha,\delta}) \to \TV(u^\dagger)$ as $\delta \to 0$.
\end{proof}

Finally, if a respective source condition is satisfied, we can, under
some circumstances, give rates for some \emph{Bregman distance} with respect
to $\TV$ associated with respect to a particular subgradient element \cite{Burger04_convergence_rates_convex_variational_mh}.
Recall that the Bregman distance $D^F_{x^*}(y,x)$ of $x,y \in X$ for a convex
functional $F: X \to {]{-\infty,\infty}]}$ and subgradient element $x^*\in \subgrad F(x)$ is
given by
\[
D_{x^*}^F(y,x) = F(y) - F(x) - \scp{x^*}{y-x}.
\]
The convergence rate results are then a consequence of the following
proposition.

\begin{proposition}
  \label{prop:tv_reg_convergence_rate}
   In the situation of Theorem \ref{thm:tv_reg_convergence}, 
   let $K^*w^\dagger \in \subgrad \TV(u^\dagger)$ for some
  $w^\dagger \in Y^*$. Then,
  \begin{equation}
    \label{eq:bregman_dist_est}
    \breg{\TV}{K^*w^\dagger}(u^{\alpha,\delta},u^\dagger) \leq \frac1\alpha
    \bigl( S_{f^\delta}^*(\alpha w^\dagger) + S_{f^\delta}^*(-\alpha
    w^\dagger) + 2 \delta \bigr).   
  \end{equation}
\end{proposition}

\begin{proof}
  Using the minimality of $u^{\alpha,\delta}$ yields
  $S_{f^\delta}(Ku^{\alpha,\delta}) + \alpha \TV(u^{\alpha,\delta}) \leq
  \alpha \TV(u^\dagger) + \delta$. Rearranging, adding
  $\scp{K^*w^\dagger}{u^\dagger - u^{\alpha,\delta}}$ on both sides as well as
  using Fenchel's inequality twice yields
  \begin{eqnarray*}
    \fl
    S_{f^\delta}(Ku^{\alpha,\delta}) + \alpha \breg{\TV}{K^*w^\dagger}(u^{\delta,\alpha},
    u^\dagger) 
    &\leq \alpha \scp{K^*w^\dagger}{u^\dagger - u^{\alpha,\delta}} + \delta \\
    &= \scp{\alpha w^\dagger}{f^\dagger} - \scp{\alpha w^\dagger}{Ku^{\alpha,\delta}} + \delta
      \\
    &\leq S_{f^\delta}^*(\alpha w^\dagger)
      - \scp{\alpha w^\dagger}{Ku^{\alpha,\delta}}
      + 2\delta 
    \\
    & \leq  
      S_{f^\delta}^*(\alpha w^\dagger) + S_{f^\delta}^*(-\alpha w^\dagger)
      + S_{f^\delta}(Ku^{\alpha,\delta}) 
      + 2\delta.
  \end{eqnarray*}
  Subtracting $S_{f^\delta}(Ku^{\alpha,\delta})$ and dividing by $\alpha$ gives
  the result.
\end{proof}

For well-known discrepancy terms, one easily gets parameter choice rules
that lead to rates for $\breg{\TV}{K^*w^\dagger}(u^{\alpha,\delta})$.

\begin{example} \label{ex:discrepancies_rates}
  \mbox{}
  \begin{itemize}
  \item For $S_{f^\delta}(v) = \frac1q \norm[Y]{v - f^\delta}^q$ with
    $q > 1$,
    $S_{f^\delta}^*(w) = \frac1{q^*}\norm[Y^*]{w}^{q^*} +
    \scp{f^\delta}{w}$
    where $1/q + 1/q^* = 1$, hence~\eqref{eq:bregman_dist_est} reads
    as
    \[
    \breg{\TV}{K^*w^\dagger}(u^{\alpha,\delta}, u^\dagger) \leq
    \frac{2\alpha^{q^* - 1}}{q^*} \norm[Y^*]{w^\dagger}^{q^*} +
    \frac{2\delta}{\alpha}.
    \]
    In the non-trivial case of $w^\dagger \neq 0$, the right-hand side
    becomes minimal for
    $\alpha = \norm[Y^*]{w^\dagger}^{-1} (\frac{q^*}{q^*-1})^{1/q^*}
    \delta^{1/q^*}$
    giving the well-known rate of
    $\kronO(\delta^{1/q}) = \kronO(\norm[Y]{f^\delta - f^\dagger})$
    for the Bregman distance.
  \item For the Kullback--Leibler discrepancy, i.e.,
    $S_{f^\delta}(v) = \KL(v, f^\delta)$ on $\LPspace{1}{\Omega'}$,
    a direct, pointwise computation shows that the dual functional obeys
    $S_{f^\delta}^*(w) + S_{f^\delta}^*(-w) = \int_{\Omega'} -f^\delta
    \log (1 - w^2) \dd{x}$ if $\abs{w} \leq 1$ almost everywhere,
    setting $-t \log(0) = \infty$ for $t > 0$ and
    $-0\log(0)=0$, and
    $S_{f^\delta}^*(w) + S_{f^\delta}^*(-w) = \infty$ else.  As
    $w^\dagger \in \LPspace{\infty}{\Omega'}$, we may choose
    $\alpha > 0$ such that
    $\alpha \norm[\infty]{w^\dagger} \leq \frac1{\sqrt{2}}$.  Then, the
    equivalence
    \[
    \fl
    \alpha^2 \int_{\Omega'} f^\delta (w^\dagger)^2 \dd{x} \leq - \int_{\Omega'}
    f^\delta \log(1 - \alpha^2 (w^\dagger) ^2) \dd{x} \leq
    \alpha^2 2 \log(2) \int_{\Omega'} f^{\delta} (w^\dagger) ^2 \dd{x}
    \]
    holds. Assuming $\int_{\Omega'} f^\dagger (w^\dagger)^2 \dd{x} > 0$, the weak convergence $f^\delta \wrightarrow f$ in $\LPspace{1}{\Omega'}$
    (see Lemma~\ref{lem:kl_lim_inf_sup_estimates}) implies
    $S_{f^\delta}^*(\alpha w^\dagger) + S_{f^\delta}^*(-\alpha
    w^\dagger) \sim \alpha^2$
    independent from $\delta$. Hence, choosing
    $\alpha \sim \sqrt{\delta}$ yields the rate
    $\kronO(\sqrt{\delta})$ for the Bregman distance as
    $\delta \to 0$.
  \end{itemize}
\end{example}

\subsection{Further first-order approaches}
\label{subsec:further_first_order}

Besides these functional-analytic properties, functions of bounded
variation admit interesting structural and fine properties.  Let us
briefly discuss the structure of the gradient $\grad u$ for a
$u \in \BV(\Omega)$. By Lebesgue's decomposition theorem, $\grad u$
can be split into an absolutely continuous part $\grad^{a} u$ with
respect to the Lebesgue measure and a singular part $\grad^s u$. We
tacitly identify $\grad^a u$ with the Radon--Nikod\'ym derivative, i.e.,
$\grad^a u \in \LPspace{1}{\Omega,\RR^d}$ via the measure
$\grad^a u \lebesgue{d}$.

The singular part $\grad^s u$ therefore has to capture the jump
discontinuities of $u$. Indeed, introducing the jump set,
it can further be decomposed.
Recall that a $u \in \LPspace{1}{\Omega}$ is almost everywhere
approximately continuous, i.e., for almost every $x \in \Omega$ there
exists a $z \in \RR$ such that
\[
\lim_{r \to 0} \dashint_{\ball{r}{x}} \abs{u(y) - z} \dd{y} = 0.
\]
The collections of all points $S_u$ for which $u$ is not approximately
continuous is called the \emph{discontinuity set} of $u$.

\begin{definition}
  \label{def:approximate_diff_jump}
  Let $u \in \LPlocspace{1}{\Omega}$ and $x \in \Omega$.
  \begin{enumerate}
  \item The function $u$ is called \emph{approximately differentiable}
    in $x$ if there exists a $v \in \RR^d$ such that
    \[
    \lim_{r \to 0} \frac1r \dashint_{\Omega} \abs{u(y) - u(x) - v \inprod
    (y-x)} \dd{y} = 0.
    \]
    The vector $\grad^{\approx} u(x) = v$ is called
    the \emph{approximate gradient} of $u$ at $x$.
  \item
    The point $x$ is an \emph{approximate jump point} of $u$ if there
    exist $u^+(x) > u^-(x)$ and a $\nu \in \RR^d$,
    $\abs{\nu} = 1$ such that
    \[
    \fl
    \lim_{r \to 0} \dashint_{\ball[+]{r}{x,\nu}} \abs{u(y) - u^+(x)} \dd{y} = 0,
    \qquad 
    \lim_{r \to 0} \dashint_{\ball[-]{r}{x,\nu}} \abs{u(y) - u^{-}(x)} \dd{y} = 0
    \]
    where $\ball[+]{r}{x,\nu}$ and $\ball[-]{r}{x,\nu}$ are balls cut
    by the hyperplane perpendicular to $\nu$ and containing $x$, i.e.,
    \[
    \begin{array}{rl}
      \ball[+]{r}{x,\nu} &= \set{y \in \RR^d}{ \abs{y - x} < r,
                           \ (y - x) \inprod \nu > 0}, \\[\smallskipamount]
      \ball[-]{r}{x,\nu} &= \set{y \in \RR^d}{ \abs{y - x} < r,
                           \ (y - x) \inprod \nu < 0}.
    \end{array}
    \]
    The set $J_u$ of all approximate jump points is called the
    the \emph{jump set} of $u$. 
  \end{enumerate}
\end{definition}

\begin{theorem}[\cite{Ambrosio_mh}]
  \label{thm:bv_grad_decomp}
  Let $u \in \BV(\Omega)$. Then,
  \begin{enumerate}
  \item $u$ is almost everywhere approximately differentiable with
    $\grad^a u = \grad^{\approx} u$ in $\LPspace{1}{\Omega,\RR^d}$, 
  \item
    the jump set satisfies $\hausdorff{d-1}(S_u \without J_u) = 0$
    and we have
    $\grad u \restricted J_u = (u^+ - u^-) \nu_u \hausdorff{d-1}$,
  \item
    the restriction $\grad u \restricted (\Omega \without S_u)$
    is absolutely continuous with respect to $\hausdorff{d-1}$.
  \end{enumerate}
  In particular, the involved sets and functions are Borel sets and
  functions, respectively.
\end{theorem}

Denoting by
\[
\grad^j u = \grad^s u \restricted J_u, \qquad
\grad^c u = \grad^s u \restricted (\Omega \without S_u)
\]
where $\grad^j u$ and $\grad^c u$ is the jump and Cantor part of $\grad u$,
respectively, the gradient of a $u \in \BV(\Omega)$ 
can be decomposed into
\begin{equation}
  \label{eq:bv_grad_decomp}
  \grad u = \grad^a u \lebesgue{d} + (u^+ - u^-) \nu_u \hausdorff{d-1} 
  \restricted J_u +
  \grad^c u
\end{equation}
with $\grad^c u$ being singular with respect to $\lebesgue{d}$ and
absolutely continuous with respect to $\hausdorff{d-1}$.

This construction allows in particular to define penalties beyond the
total variation seminorm (see, for instance \cite[Section 5.5]{Ambrosio_mh}). Letting $g: \RR^d \to [0,\infty]$ a
proper, convex and lower semi-continuous function and 
$g_\infty$ be given according to 
\[
g_\infty(x) = \lim_{t \to \infty} \frac{g(tx)}{t}
\]
with $\infty$ allowed, then the functional
\begin{equation}
  \label{eq:general_bv_penalty}
  \fl
  \mR_g(u) = \int_\Omega g(\grad^a u) \dd{x} + 
  \int_{J_u} (u^+ - u^-) g_\infty(\nu_u) \dd{\hausdorff{d-1}}
  + \int_\Omega
  g_\infty (\sigma_{\grad^c u})
  \dd{\abs{\grad^c u}}
\end{equation}
where $\sigma_{\grad^c u}$ is the sign of $\grad^c u$, i.e.,
$\grad^c u = \sigma_{\grad^c u} \abs{\grad^c u}$,
is proper, convex and lower semi-continuous on $\BV(\Omega)$.
With the Fenchel-dual functional, i.e., $g^*(y) = \sup_{x \in \RR^d} \ x \inprod y - g(x)$,
it can also be expressed in (pre-)dual form as
\[
  \mR_{g}(u) = 
  \sup \ \Bigset{\int_\Omega u \divergence \varphi - g^*(\varphi)
    \dd{x}}{\varphi \in \Ccspace{\infty}{\Omega,\RR^d}}.
\]
Obviously, the usual $\TV$-case corresponds to $g$ being the Euclidean
norm on $\RR^d$. Also, $g_\infty(x) = \infty$ for some $\abs{x} = 1$
does not allow jumps in the direction of $x$, so one usually assumes that
$g_\infty(x) < \infty$ for each $\abs{x} = 1$ in order to obtain a genuine
penalty in $\BV(\Omega)$.
In addition, if there are $c_0 > 0$ and $R > 0$ such that $g(x) \geq c_0 \abs{x}$
for each $\abs{x} \geq R$, 
then there is a constant $C > 0$ such that
\[
\mR_g(u) \geq c_0 \TV(u) - C
\]
for all $u \in \BV(\Omega)$, i.e., $\mR_g$ is as coercive as $\TV$.
Consequently, the well-posedness and convergence statements in
Theorems~\ref{thm:tv_reg_existence},~\ref{thm:tv_reg_stability} 
and~\ref{thm:tv_reg_convergence}
as well as in Proposition~\ref{prop:tv_reg_convergence_rate} can be
adapted to $\mR_g$ in a straightforward manner with the proofs following
the same line of argumentation.

\begin{example}
  \label{ex:huber_smooth_tv}
  There are several possibilities for replacing the non-differentiable
  norm function $\abs{\placeholder}$ in the $\TV$-functional
  by a smooth approximation in $0$.

  Choosing a $\varepsilon > 0$, consider
  \[
  g_\varepsilon^1(x) =
  \cases{\frac{1}{2\varepsilon} \abs{x}^2 & 
    for $\abs{x} \leq \varepsilon$,\\
    \abs{x} - \frac{\varepsilon}2 & else,}
  \qquad
  g_\varepsilon^2(x) = \sqrt{\abs{x}^2 + \varepsilon^2} - \varepsilon,
  \]  
  both being continuously differentiable in $\RR^d$ and approximating
  $\abs{\placeholder}$ for $\varepsilon \to 0$.

  The associated penalties $\mR_{g_\varepsilon^1}$ and $\mR_{g_\varepsilon^2}$
  are often referred to as \emph{Huber-$\TV$} and \emph{smooth $\TV$}, 
  respectively.
\end{example}

\begin{example}
  Taking $g$ as a non-Euclidean norm on $\RR^d$ yields functionals of
  \emph{anisotropic total-variation} type. The %
  common choice
  is $g = \abs[1]{\placeholder}$ which is also often referred to as
  anisotropic $\TV$. 
\end{example}

\begin{remark}
It is worth noting that $g$ as above can also be made spatially dependent, 
which has applications for instance the context of regularisation
for inverse problems involving multiple modalities or multiple spectra.
Under some assumptions, functionals $\mR_g$ as in \eqref{eq:general_bv_penalty}
with spatially dependent $g$ are again lower semi-continuous on $\BV$ \cite{Fusco08} and
well-posedness results for TV apply \cite{holler18structural_mh}.
\end{remark}

\subsection{Colour and multichannel images}

Colour and multichannel images are usually represented by functions
mapping into a vector-space.  Total-variation functionals and
regularisation approaches can easily be extended to such vector-valued
functions; Definition~\ref{def:tv} already contains an isotropic
variant for functions with values in a finite-dimensional space
$H$, where we used the Hilbert-space norm $|x| = (\sum_{i=1}^d {x_i \cdot x_i})^{1/2}$
as pointwise norm on $H^d$ for the test functions $\varphi \in \Ccspace{\infty}{\Omega,H^d}$.

 However, in contrast to the scalar case, this is not 
the only choice yielding $\TV$-functionals that are invariant under
distance-preserving transformations. The essential property for a norm
$\abs[\circ]{\placeholder}$ on $H^d$ needed for the latter is
\[
\abs[\circ]{Ox} = \abs[\circ]{x} \quad
\text{with} 
\qquad
\text{for all} \quad
x \in H^d \quad
\text{and} \quad O \in \RR^{d \times d}, \quad
O^*O = \id
\]
where $(Ox)_i = \sum_{j=1}^d o_{ij} x_j$. We call such norms
\emph{unitarily left invariant}. Denoting by
$\abs[\ast]{\placeholder}$ the dual norm, the associated total
variation for a $u \in \LPlocspace{1}{\Omega,H}$ is given by
\[
\TV(u) = \sup \ \Bigset{\int_\Omega u \inprod \divergence \varphi
  \dd{x}}{\varphi \in \Ccspace{\infty}{\Omega,H^d}, \
  \abs[\ast]{\varphi(x)} \leq 1 \ \forall x \in \Omega}.
\]
and invariant to distance-preserving transformations. If the norm
$\abs[\circ]{\placeholder}$ is moreover \emph{unitarily right invariant},
i.e.,
\[
\abs[\circ]{xO} = \abs[\circ]{x} \qquad
\text{for all} \quad x \in H^d \quad
O: H \to H \quad \text{unitary}
\]
where $(xO)_i = \bigl( O(x) \bigr)_i$, then it can be written as a
unitarily invariant matrix norm and hence $\abs[\circ]{x}$ only
depends on the singular values of the mapping associated with $x$ in a
permutation- and sign-invariant manner. More precisely, there exists a
norm $\abs[\Sigma]{\placeholder}$ on $\RR^{d}$ with
$\abs[\Sigma]{P \sigma} = \abs[\Sigma]{\sigma}$ for all
$\sigma \in \RR^{d}$ and $P \in \RR^{d \times d}$ with $\abs{P}$ being
a permutation matrix, such that
$\abs[\circ]{x} = \abs[\Sigma]{\sigma}$ for all $x \in H^d$, where
$\sigma$ are the singular values of the mapping $H^d \to \RR^d$ given
by $y \mapsto (x_i \inprod y)_i$. Conversely, any such norm on $\RR^d$
induces a unitarily invariant matrix norm. A common choice are the
norms generated by the $p$-vector norm, the \emph{Schatten-$p$-norms}.
For $p = 1$, $p = 2$ and $p=\infty$, they correspond to the
\emph{nuclear norm}, the \emph{Frobenius norm} and the usual spectral
norm, respectively, all of which have been proposed in the existing
literature to use in conjunction with $\TV$, see, e.g., \cite{sapiro1996anisotropic,Duran2016color_tv}.
Among those possibilities, the nuclear norm appears particularly
attractive as it provides a relaxation of the rank functional \cite{Recht2010guaranteed}.
Hence, solutions with low-rank gradients and more pronounced edges can
be expected from nuclear-norm-$\TV$ regularisation. 

Also here, the well-posedness and convergence results in
Theorems~\ref{thm:tv_reg_existence},~\ref{thm:tv_reg_stability}
and~\ref{thm:tv_reg_convergence} as well as in
Proposition~\ref{prop:tv_reg_convergence_rate} are transferable to the
vector-valued case, as can be seen from equivalence of norms.

Moreover, functionals of the type~\eqref{eq:general_bv_penalty} are
possible with $g: H^d \to [0,\infty]$ proper, convex and lower
semi-continuous such that $g_\infty$ exists. However, $u$ takes values
in $H$ which calls for some adaptations which we briefly describe in
the following.  First, concerning
Definition~\ref{def:approximate_diff_jump} (i), we are able to
generalise in a straightforward way by considering $v \in H^d$, the
norm in $H$ and the scalar product in $H^d$ such that the approximate
gradient of $u$ at $x$ is $\grad^{\approx} u(x) \in H^d$. For jump
points according to (ii), we are no longer able to require
$u^+(x) > u^-(x)$ such that we have to replace this by
$u^+(x) \neq u^{-}(x)$ and arrive at a meaningful definition replacing
the absolute value by the norm in $H$. However, $u^+$, $u^-$ and $\nu$
are then only unique up to a sign. Nevertheless,
$(u^+ - u^-) \tensor \nu$ according to
$\bigl((u^+ - u^-) \tensor \nu\bigr)_i = (u^+ - u^-)\nu_i$ is still
unique. The analogue of Theorem~\ref{thm:bv_grad_decomp}
and~\eqref{eq:bv_grad_decomp} holds with these notions, with the
following adaptation:
\[
\grad u = \grad^a u \lebesgue{d} + (u^+ - u^-) \tensor \nu_u
\hausdorff{d-1} \restricted J_u + \grad^c u
\]
with the Cantor part being of rank one, i.e.,
$\grad^c u = \sigma_{\grad^c u} \abs{\grad^c u}$ where
$\sigma_{\grad^c u}$ is rank one $\abs{\grad^c u}$-almost everywhere \cite[Theorem 3.94]{Ambrosio_mh}.
The functional $\mR_g$ according to
\[
\fl \mR_g(u) = \int_\Omega g(\grad^a u) \dd{x} + \int_{J_u}
g_\infty\bigl((u^+ - u^-) \tensor \nu_u \bigr) \dd{\hausdorff{d-1}} +
\int_\Omega g_\infty (\sigma_{\grad^c u}) \dd{\abs{\grad^c u}}
\]
then realises a regulariser with the same regularisation properties as
its counterpart for scalar functions.

\section{Higher-order TV regularisation}
\label{sec:higher_order_tv}

First-order regularisation for imaging problems might not always lead
to results of sufficient quality. Recall that taking the total 
variation as regularisation functional has the advantage that the
solution space $\BV(\Omega)$ naturally allows for discontinuities
along hypersurfaces (``jumps'') which correspond, for imaging applications,
to object boundaries. Indeed, $\TV$ has a good performance in edge 
preservation which can also be observed numerically. 

However, for noisy data, the solutions suffer from non-flat regions
appearing flat in conjunction with the introduction of undesired
edges.  This effect is called the \emph{staircasing effect}, see
Figure~\ref{fig:first-order-reg}, in particular panel (c).
Thinking of $\TV$ as a $1$-norm
type penalty for the gradient, this is, on the one hand, due to the
``linear growth'' of the Euclidean norm $\abs{\placeholder}$ at
infinity (which implies $\BV(\Omega)$ as solution space).  On the
other hand, $\abs{\placeholder}$ is non-differentiable in $0$ which can
be seen to be responsible for the flat regions in the solutions.

As we have seen in Subsection~\ref{subsec:further_first_order}, the
latter can be remedied by considering convex functions of the measure
$\grad u$ instead of $\TV$ which are smooth in the origin and have
linear growth at $\infty$, also see Example~\ref{ex:huber_smooth_tv}.
Then, $\mR_g$ can be taken as a first-order
regulariser under the same conditions as
for $\TV$ regularisation leading to solutions which are still in 
$\BV(\Omega)$ and may, in particular, admit jumps. Additionally,
less flat regions tend to appear in solutions for noisy 
data as we no longer have a singularity at $0$. 
However, this feature comes with two drawbacks: First, compared
to $\TV$, noise removal seems not to be so strong in numerical
solutions. Second, in addition to the regularisation parameter
for the inverse problem,
one has to choose the parameter $\varepsilon$ appropriately.
A too small choice might again lead to staircasing to appear
while choosing $\varepsilon$ too big may lead to edges being lost,
see Figure~\ref{fig:first-order-reg} (d).
The question remains whether we can improve on this.

Here, we like to discuss and 
study the use of \emph{higher-order derivatives}
for regularisation in imaging. This can be motivated by 
modelling images as piecewise smooth functions, i.e., 
assuming that an image is several times differentiable (in some
sense) while still allowing for
object boundaries where the function may jump. With this model
in mind,
higher-order variational approaches arise quite naturally and we refer for instance
to \cite{demengel1984bounded_hessian,hinterberger2006boundedhessian_mh,bergounioux2010secondordermodel_mh}
for spaces and regularisation approaches related to second-order variational approaches.

\begin{figure}
  \centering
  \begin{tabular}{c@{\ }c@{\ }c@{\ }c}
    \includegraphics[width=0.22\linewidth]{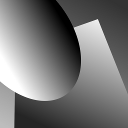}
    & 
    \includegraphics[width=0.22\linewidth]{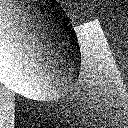}
    &
    \includegraphics[width=0.22\linewidth]{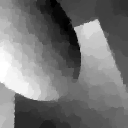}
    & 
    \includegraphics[width=0.22\linewidth]{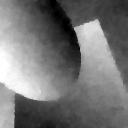}
    \\
    (a) & (b) &
    (c) & (d)
  \end{tabular}
  \caption{First-order denoising example. (a) Ground truth, (b) 
    noisy image with additive Gaussian noise, (c)
    $\TV$-regularised solution (best PSNR), 
    (d) regularisation with smooth $\TV$-like penalty $\varphi(x) 
    = \sqrt{x^2 + \varepsilon^2} - \varepsilon$ (best PSNR).}
  \label{fig:first-order-reg}
\end{figure}

\subsection{Symmetric tensor calculus}

For smooth functions, higher-order derivatives can be represented as
tensor fields, i.e., the derivative represents a tensor in each
point. As the order of partial differentiation might be interchanged,
these tensors turn out to be symmetric. Symmetric tensors are
therefore a suitable tool for representing these objects independent
from indices.  There are several ways to introduce and motivate
tensors and vector spaces of tensors. For our purposes, the following
definition will be sufficient. Note that there and throughout this chapter, 
$l\geq 0$ will always be a tensor order.

\begin{definition}
We define
  \[
  \begin{array}{cc}
    \tensorspace[l]{\RR^d}
    = \set{\xi: \underbrace{\RR^d \times \cdots \times \RR^d}
    _{l \ \mathop{\rm times}} \rightarrow \RR}{\xi \ \text{$l$-linear}},
    \\[\medskipamount]
    \Sym^l(\RR^d) = \set{\xi: \underbrace{\RR^d \times \cdots \times \RR^d}
    _{l \ \mathop{\rm times}} \rightarrow \RR}
    {\xi \ \text{$l$-linear and symmetric}},
  \end{array}
  \]
  as the vector space of \emph{$l$-tensors} and \emph{symmetric
    $l$-tensors}, respectively.
  
  Here, $\xi\in \tensorspace[l]{\RR^d}$ is called symmetric, if
  $\xi(a_1,\dots,a_l)= \xi(a_{\pi(1)},\dots,a_{\pi(l)})$ for all
  $a_1,\ldots a_l \in \RR^d$ and $\pi \in S_l$, where $S_l$ denotes
  the permutation group of $\sett{1,\ldots,l}$.

  For $\xi \in \tensorspace[k]{\RR^d}$, $k \geq 0$ and
  $\eta \in \tensorspace[l]{\RR^d}$ the \emph{tensor product} is
  defined as the element
  $\xi \tensor \eta \in \tensorspace[k+l]{\RR^d}$ obeying
  \[
  (\xi \tensor \eta)(a_1,\ldots,a_{k+l}) = \xi(a_1,\ldots,a_k)
  \eta(a_{k+1}, \ldots, a_{k+l})
  \]
  for all $a_1,\ldots,a_{k+l} \in \RR^d$.
\end{definition}
Note that the space of $l$-tensors is actually the space of
$(0,l)$-covariant tensors, however, we will not need to distinguish
between co- and contravariant tensors. We have
\[
\tensorspace[0]{\RR^d} \equiv \RR, \quad 
\tensorspace[1]{\RR^d} \equiv \RR^d,
\quad \ldots, \quad 
\tensorspace[l]{\RR^d} \equiv \RR^{d \times \cdots \times d},
\]
while for low orders, the symmetric tensor spaces
coincide with well-known spaces $\Sym^0(\RR^d) \equiv \RR$,
$\Sym^1(\RR^d) \equiv \RR^d$ and $\Sym^2(\RR^d) \equiv S^{d \times d}$, the
space of symmetric $d \times d$ matrices.

In the following, we give a brief overview of the tensor operations
that are the most relevant to define regularisation functionals on
higher-order derivatives.

\begin{remark}
  The space $\tensorspace[l]{\RR^d}$ can be associated with a unit basis.
  Indexed by $p \in \sett{1,\ldots,d}^l$, its elements are given by
  $e_p(a_1,\ldots, a_l) = \prod_{i=1}^l a_{i,p_i}$ while the
  respective coefficient for a $\xi \in \tensorspace[l]{\RR^d}$ is
  given by $\xi_p = \xi(e_{p_1}, \ldots, e_{p_l})$.
  Each $\xi \in \tensorspace[l]{\RR^d}$ thus has the representation
  \[
  \xi(a_1,\ldots,a_l) = \sum_{p \in \sett{1,\ldots,d}^l} \xi_p
  e_p(a_1,\ldots,a_l). %
  \]
  The identity of vector spaces
  $\tensorspace[l]{\RR^d} = \RR^{d \times \cdots \times d}$ is evident
  from that.

  The space $\Sym^l(\RR^d)$ is obviously a (generally proper) subspace
  of $\tensorspace[l]{\RR^d}$. A (non-symmetric) tensor 
  $\xi \in \tensorspace[l]{\RR^d}$ can be symmetrised by averaging 
  over all permuted arguments, i.e.,
  \[
  (\interleave \xi)(a_1,\ldots,a_l) = \frac1{l!} \sum_{\pi \in S_l}
  \xi(a_{\pi(1)}, \ldots, a_{\pi(l)}).
  \]
  The symmetrisation operator
  $\interleave: \tensorspace[l]{\RR^d} \to \Sym^l(\RR^d)$ obviously
  defines a projection.  A basis for $\Sym^l(\RR^d)$ is given by
  $e^{\Sym}_p = \interleave e_p$ for $p$ ranging over all
  tuples in $\sett{1,\ldots,d}^l$ with non-decreasing entries.
   The
  coefficients $\xi_p$ 
  can still be obtained by $\xi_p = \xi(e_{p_1}, \ldots, e_{p_l})$.
\end{remark}

We would like to equip the spaces with a Hilbert space structure.

\begin{definition}
  For $\xi, \eta \in \tensorspace[l]{\RR^d}$, the \emph{scalar
    product} and \emph{Frobenius norm} are defined as
  \[
  \xi \inprod \eta %
  = \sum_{p \in \sett{1,\ldots,d}^k}
  \xi(e_{p_1},\ldots,e_{p_l}) \eta(e_{p_1}, \ldots, e_{p_l}),
  \qquad
  \abs{\xi} = \sqrt{\xi \inprod \xi}.
  \]
\end{definition}

\begin{example}
  For $\xi \in \Sym^l(\RR^d)$, the 
  norm corresponds to
  the absolute value for $l=0$, the Euclidean norm in $\RR^d$
  for $l=1$ and in case $l=2$, we
  can identify $\xi \in \Sym^2(\RR^d)$ with
  \[
  \xi =
  \left(\begin{array}{ccc}
    \xi_{11} & \cdots & \xi_{1d} \\
    \vdots & \ddots & \vdots \\
    \xi_{1d} & \cdots & \xi_{dd}
  \end{array}\right)
  \quad , \quad \abs{\xi} = \Bigl( \sum_{i=1}^d \xi_{ii}^2 + 2 \sum_{i
    < j} \xi_{ij}^2 \Bigr)^{1/2}.
  \]
\end{example}

With the Frobenius norm, tensor spaces become Hilbert spaces of finite dimension and the symmetrisation becomes an orthogonal projection, see, e.g., \cite{Hackbusch_tensor_spaces_mh}.
\begin{proposition}\mbox{}
  \begin{enumerate}
  \item 
    With the above scalar-product and norm,
    the spaces $\tensorspace[l]{\RR^d}$, $\Sym^l(\RR^d)$ are
    finite-dimensional Hilbert spaces
    with $\dim \tensorspace[l]{\RR^d} = d^l$ and
    $\dim \Sym^l(\RR^d) = {d + l -1 \choose l}$.
  \item The symmetrisation $\interleave$ is the orthogonal projection
    in $\tensorspace[l]{\RR^d}$ onto $\Sym^l(\RR^d)$.
  \end{enumerate}
\end{proposition}

Tensor-valued mappings $\Omega \to \tensorspace[l]{\RR^d}$ on the
domain $\Omega \subset \RR^d$ are called \emph{tensor fields}.
The tensor-field spaces
$\Cspace{}{\closure{\Omega}, \tensorspace[l]{\RR^d}}$,
$\Ccspace{}{\closure{\Omega}, \tensorspace[l]{\RR^d}}$ and
$\Cspace[0]{}{\closure{\Omega}, \tensorspace[l]{\RR^d}}$ as well as
the Lebesgue spaces $\LPspace{p}{\Omega, \tensorspace[l]{\RR^d}}$ are
then given in the usual manner.  Also, measures can be tensor-valued,
giving $\radonspace{\Omega,\tensorspace[l]{\RR^d}}$, the space of
$l$-tensor-valued Radon measures.  Duality according to
Proposition~\ref{prop:duality_radon_measures} holds, i.e.,
$\radonspace{\Omega,\tensorspace[l]{\RR^d}} =
\Cspace[0]{}{\Omega,\tensorspace[l]{\RR^d}}^*$.
Note that for all spaces, the Frobenius norm is used as pointwise norm
in the respective definitions of the tensor-field norm.  Furthermore,
all the above applies analogously to symmetric tensor fields, i.e.,
mappings between $\Omega \to \Sym^l(\RR^d)$.

Turning to differentiation, the $k$-th Fr\'echet derivative of a
sufficiently smooth $l$-tensor field, where from now on $k\geq 1$ will always denote 
an order of differentiation, is naturally a $(k+l)$-tensor
field which we denote by
$\grad^k \tensor u: \Omega \to \tensorspace[k+l]{\RR^d}$ according to
\[
(\grad^k \tensor u)(x)(a_1,\ldots,a_{k+l}) = \bigl( \deriv^k
u(x) (a_1,\ldots, a_k) \bigr) (a_{k+1},\ldots, a_{k+l}).
\]
The fact that gradient tensor-fields are not symmetric in general
gives rise to consider the \emph{$k$-th symmetrised derivative} given
by $\symgrad^ku = \interleave \grad^k \tensor u$. This definition is
consistent as $\symgrad^{k_2}\symgrad^{k_1} = \symgrad^{k_1 + k_2}$
for $k_1,k_2 \geq 0$. Divergence operators are then, up to the sign,
formal adjoints of these differentiation operators. They are given as follows. Introducing the
\emph{trace} of a tensor $\xi \in \tensorspace[l+2]{\RR^d}$ according
to
\[
\trace(\xi) (a_1,\ldots,a_l) = \sum_{i=1}^d
\xi(e_i,a_1,\ldots,a_l,e_i)
\]
gives an $l$-tensor. It can be interpreted as the tensor contraction of
the first and the last component of the tensor.  As for the
vector-field case, the divergence is now the trace of the derivative. For
$k$-times differentiable $v: \Omega \to \tensorspace[k+l]{\RR^d}$, the
$k$-th \emph{divergence} is thus given by
\[
\divergence^k v = \trace^k(\grad^k \tensor v).
\]
Again, this is consistent with repeated application, i.e.,
$\divergence^{k_1+k_2} = \divergence^{k_2}\divergence^{k_1}$. Note
that there might be other choices of the divergence, such as
contracting the derivative with any other than the last components of
the tensor. This affects, however, only non-symmetric tensor
fields. For symmetric tensor fields, the result is independent from
the choice of the contraction components and always a symmetric tensor
field.

\begin{example}
  The symmetrised gradient of scalar functions
  $\Omega \to \Sym^0(\RR^d)$ coincides with the usual gradient while
  the divergence for mappings $\Omega \rightarrow \Sym^1(\RR^d)$
  coincides with the usual divergence.
  
  The cases $\symgrad^2 u^0$ and $\symgrad u^1$ for
  $u^0: \Omega \to \Sym^0(\RR^d)$ and $u^1: \Omega \to \Sym^1(\RR^d)$
  can be handled with the identification of $\Sym^2(\RR^d)$ and
  symmetric matrices $S^{d\times d}$:
  \[
  (\symgrad^2 u^0)_{ij} = \frac{\partial^2 u^0}{\partial x_i \partial
    x_j} \quad, \quad (\symgrad u^1)_{ij} = \frac12 \Bigl(
  \frac{\partial u_i^1}{\partial x_j} + \frac{\partial u_j^1}{\partial
    x_i} \Bigr).
  \]
  Analogously, for the divergence of a $v: \Omega \to \Sym^2(\RR^d)$,
  we have that
  \[
  (\divergence v)_i = \sum_{j=1}^d \frac{\partial v_{ij}}{\partial
    x_j} \quad , \quad \divergence^2 v = \sum_{i=1}^d \frac{\partial^2
    v_{ii}} {\partial x_i^2} + \sum_{i < j} 2 \frac{\partial^2
    v_{ij}}{\partial x_i
    \partial x_j}.
  \]
\end{example}

In particular, for $k \geq 1$, there are the usual spaces of
continuously differentiable tensor fields which are denoted by
$\Cspace{k}{\closure{\Omega},\tensorspace[l]{\RR^d}}$ and equipped
with the usual norm
$\norm[k,\infty]{u} = \max_{0 \leq m \leq k} \
\norm[\infty]{\grad^{m} \tensor u}$.
Likewise, we consider $k$-times continuously differentiable tensor
fields with compact support
$\Ccspace{k}{\Omega, \tensorspace[l]{\RR^d}}$ where $k = \infty$ leads
to the space of test tensor fields. Also, for finite $k$, the space
$\Cspace[0]{k}{\Omega, \tensorspace[l]{\RR^d}}$ is given as the
closure of $\Ccspace{k}{\Omega, \tensorspace[l]{\RR^d}}$ in
$\Cspace{k}{\closure{\Omega},\tensorspace[l]{\RR^d}}$. Of course, the
analogous constructions apply to symmetric tensor fields, leading to
the spaces $\Cspace{k}{\closure{\Omega},\Sym^l(\RR^d)}$,
$\Ccspace{k}{\Omega, \Sym^l(\RR^d)}$ and
$\Cspace[0]{k}{\Omega, \Sym^l(\RR^d)}$ as well as the space of test
symmetric tensor fields $\Ccspace{\infty}{\Omega, \Sym^l(\RR^d)}$.

As $\Omega$ is assumed to be a connected set, we are able to describe
the kernels of $\grad^k$ and $\symgrad^k$ for (symmetric) tensor
fields in terms of finite-dimensional spaces of polynomials.

\begin{proposition}
  \label{prop:grad_kernel_smooth}
  Let $u \in \Cspace{k}{\closure{\Omega},\tensorspace[l]{\RR^d}}$ such
  that $\grad^k \tensor u = 0$. Then, $u$ is a
  $\tensorspace[l]{\RR^d}$-valued polynomial of maximal order $k-1$,
  i.e., there are $\xi^m \in \tensorspace[l+m]{\RR^d}$,
  $m = 0,\ldots, k-1$ such that
  \begin{equation}
    \label{eq:grad_kernel_poly}
    u(x) = \sum_{m = 0}^{k-1} \tr^m\bigl(\xi^m \tensor (\underbrace{x
      \tensor \ldots \tensor x}_{m \,\mathop{times}}) \bigr) \qquad
    \text{for each} \ x \in \Omega.
  \end{equation}
  If $\symgrad^k u = 0$ for
  $u \in \Cspace{k+l}{\closure{\Omega}, \Sym^l(\RR^d)}$, then $u$ is a
  $\Sym^l(\RR^d)$-valued polynomial of maximal order $k + l - 1$, i.e., %
  the above representation holds for $\xi^m \in \Sym^{l+m}(\RR^d)$,
  $m = 0,\ldots, k+l-1$ with the sum ranging from $0$ to $k+l-1$. %
\begin{proof}
At first we note that any $\tensorspace[l]{\RR^d}$- and $\Sym^l(\RR^d)$-valued polynomial of maximal order $k-1$ and $k+l-1$, respectively, admits a representation as claimed. In case $\nabla^ k \otimes u = 0$ for $u \in \Cspace{k}{\closure{\Omega},\tensorspace[l]{\RR^d}} $ it follows directly from a basis representation of $u(x)$ that $u$ is a $\tensorspace[l]{\RR^d}$-valued polynomial of maximal order $k-1$. 

Now in case $\symgrad ^{k}u=0 $ for $u \in \Cspace{k+l}{\closure{\Omega}, \Sym^l(\RR^d)}$,
we get that $\grad^{k+l} \tensor u = 0$, see Lemma~\ref{lem:kernel_symgrad}.
This implies that $u$ is a $\Sym^{l}(\RR^d)$-valued polynomial of maximal degree $k+l-1$ as claimed. 
\end{proof}
\end{proposition}

Next, we would like to introduce and discuss weak forms of
differentiation for (symmetric) tensor fields. Starting point for this
is a version of the well-known Gauss--Green theorem for smooth
(symmetric) tensor fields \cite{bredies2013boundeddeformation_mh}.

\begin{proposition}
  Let $\Omega \subset \RR^d$ be a bounded Lipschitz domain,
  $u \in \Cspace{}{\closure{\Omega},\tensorspace[l]{\RR^d}}$,
  $v \in \Cspace{1}{\closure{\Omega}, \tensorspace[l+1]{\RR^d}}$.
  Then, a \emph{Gauss--Green theorem} holds in the following form:
  \[
  \int_\Omega u \inprod \divergence v \dd{x} 
  = \int_{\bdry\Omega} (u \tensor \nu) \inprod v
  \dd{\hausdorff{d-1}}
  - \int_\Omega (\grad \tensor u) \inprod v \dd{x}
  \]
  with $\nu$ being the outward unit normal on $\bdry\Omega$.
  
  If  $u \in \Cspace{}{\closure{\Omega}, \Sym^l(\RR^d)}$,
  $v \in \Cspace{1}{\closure{\Omega}, \Sym^{l+1}(\RR^d)}$, 
  the identity reads as
  \[
  \int_\Omega u \inprod \divergence v \dd{x} 
  = \int_{\bdry\Omega} \interleave (u \tensor \nu) \inprod v
  \dd{\hausdorff{d-1}}
  - \int_\Omega \symgrad u \inprod v \dd{x}.
  \]
  If one of the tensor fields $u$ or $v$ have compact support in
  $\Omega$ the boundary term does not appear and the identities
  are valid for arbitrary domains $\Omega$.
\end{proposition}

As usual, being able to express integrals of the form
$\int_\Omega (\grad \tensor u) \inprod v \dd{x}$ and
$\int_\Omega \symgrad u \inprod v \dd{x}$ for test tensor fields
without the derivative of $u$ allows to introduce a weak notion of
$\grad \tensor u$ and $\symgrad u$, respectively, as well as
associated Sobolev spaces.

\begin{definition}
  For $u \in \LPlocspace{1}{\Omega,\tensorspace[l]{\RR^d}}$,
  $w \in \LPlocspace{1}{\Omega,\tensorspace[l+1]{\RR^d}}$ is the
  \emph{weak derivative} of $u$, denoted $w = \grad \tensor u$, if for
  all $\varphi \in \Ccspace{\infty}{\Omega,\tensorspace[l+1]{\RR^d}}$, it
  holds that
  \[
  \int_\Omega u \inprod \divergence \varphi \dd{x} = - \int_\Omega w \inprod 
  \varphi \dd{x}.
  \]
  Likewise, for $u \in \LPlocspace{1}{\Omega,\Sym^l(\RR^d)}$,
  $w \in \LPspace{1}{\Omega,\Sym^{l+1}(\RR^d)}$ is the \emph{weak
    symmetrised derivative} of $u$, denoted $w = \symgrad u$, if the
  above identity holds for all
  $\varphi \in \Ccspace{\infty}{\Omega, \Sym^{l+1}(\RR^d)}$.
\end{definition}

Like the scalar versions, $\grad$ and $\symgrad$ are well-defined and
constitute closed operators between the respective Lebesgue spaces with
dense domain.

\begin{definition}
  The \emph{Sobolev space of tensor fields of order $l$} of
  differentiation order $k$ and exponent $p \in [1,\infty]$ is
  defined as
  \[
  \begin{array}{rl}
    \Hspace[p]{k}{\Omega,\tensorspace[l]{\RR^d}} 
    &=
      \set{u \in \LPspace{p}{\Omega,\tensorspace[l]{\RR^d}}}
      {\norm[k,p]{u} < \infty}, \\
    \norm[k,p]{u} 
    &= \Bigl(\sum_{m = 0}^k
    \norm[p]{\grad^m \tensor u}^p\Bigr)^{1/p} \quad \text{if} \ p < \infty, \\
    \norm[k,\infty]{u} 
    &=
      \max_{m = 0,\ldots, k} \ \norm[\infty]{\grad^m \tensor u},
  \end{array}
  \]
  while $\Hcspace[p]{k}{\Omega,\tensorspace[l]{\RR^d}}$ 
  is the closure of the subspace
  $\Ccspace{\infty}{\Omega,\tensorspace[l]{\RR^d}}$
  with respect to the $\norm[k,p]{\placeholder}$-norm.
  
  Replacing $\tensorspace[l]{\RR^d}$ by $\Sym^l(\RR^d)$ and
  letting
  \[
  \norm[k,p]{u} = \Bigl(\sum_{m = 0}^k \norm[p]{\symgrad^m
    u}^p\Bigr)^{1/p} \quad \text{if} \ p < \infty, \quad
  \norm[\infty,k]{u} = \max_{m = 0,\ldots,k} \
  \norm[\infty]{\symgrad^m u},
  \]
  defines the \emph{Sobolev space of symmetric tensor fields}, denoted
  by $\Hspace{k,p}{\Omega,\Sym^l(\RR^d)}$. The space
  $\Hcspace{k,p}{\Omega,\Sym^l(\RR^d)}$ is again the closure
  $\Ccspace{\infty}{\Omega, \Sym^l(\RR^d)}$ with respect to the
  corresponding norm.
\end{definition}

By closedness of the differential operators, the Sobolev spaces are
Banach spaces. Also, since weak derivatives are symmetric, we have
that
$\Hspace{k,p}{\Omega,\tensorspace[0]{\RR^d}} =
\Hspace{k,p}{\Omega,\Sym^0(\RR^d)}$
in the sense of Banach space isometry, as well as coincidence with the
usual Sobolev spaces.  For $l \geq 1$, the space
$\Hspace{k,p}{\Omega,\tensorspace[l]{\RR^d}}$ corresponds to the space
where all components of $u$ are in $\Hspace{k,p}{\Omega}$.  However,
generally, for $l \geq 1$, the norm of
$\Hspace{k,p}{\Omega,\Sym^l(\RR^d)}$ is weaker than the norm in
$\Hspace{k,p}{\Omega,\tensorspace[l]{\RR^d}}$, such that only
$\Hspace{k,p}{\Omega,\tensorspace[l]{\RR^d}} \embed
\Hspace{k,p}{\Omega,\Sym^l(\RR^d)}$
in the sense of continuous embedding and the latter is a strictly
larger space.

Nevertheless, equality holds if some kind of \emph{Korn's inequality}
can be established which is, for instance, the case for the spaces
$\Hcspace{1,p}{\Omega,\Sym^1(\RR^d)}$ for $1 < p < \infty$ \cite[Section 5.6]{Kikuchi1988elasticitybook} as
well as and the spaces $\Hcspace{1,2}{\Omega,\Sym^l(\RR^d)}$ for
$l \geq 1$ (which follows from \cite[Proposition 3.6]{bredies2013boundeddeformation_mh} via smooth approximation).

Finally, let us briefly discuss (symmetric) tensor-valued
distributions and the distributional forms of $\grad^k$ and $\symgrad^k$.

\begin{definition}
  A \emph{$\tensorspace[l]{\RR^d}$-valued distribution on $\Omega$} is
  a linear mapping
  $u: \Ccspace{\infty}{\Omega, \tensorspace[l]{\RR^d}} \to \RR$ that
  satisfies the following continuity estimate: For each
  $K \compactin \Omega$, there is an $m \in \NN$ and a $C > 0$ such
  that
  \[
  \abs{u(\varphi)} \leq C \norm[m,\infty]{\varphi} \qquad \text{for all} \
  \varphi \in \Ccspace{\infty}{K, \tensorspace[l]{\RR^d}}.
  \]
  The distribution $u$ is \emph{regular} if there is a
  $\bar u \in \LPlocspace{1}{\Omega,\tensorspace[l]{\RR^d}}$ such that
  \[
  u(\varphi) = \int_\Omega \bar u \inprod \varphi \dd{x} \qquad \text{for all}
  \ \varphi \in \Ccspace{\infty}{\Omega, \tensorspace[l]{\RR^d}}.
  \]
  A \emph{$\Sym^l(\RR^d)$-valued distribution on $\Omega$} and its
  regularity is analogously defined by replacing
  $\tensorspace[l]{\RR^d}$ by $\Sym^l(\RR^d)$ in the above definition.
\end{definition}

Then, the distributional (symmetrised) derivatives are given by
$(\grad^k \tensor u)(\varphi) = (-1)^k u(\divergence^k
\varphi)$,
$\varphi \in \Ccspace{\infty}{\Omega,\tensorspace[l]{\RR^d}}$ and
$(\symgrad^k u)(\varphi) = (-1)^k u(\divergence^k \varphi)$,
$\varphi \in \Ccspace{\infty}{\Omega,\Sym^l(\RR^d)}$ which makes them a
$\tensorspace[k+l]{\RR^d}$- and $\Sym^{k+l}(\RR^d)$-valued
distribution, respectively. We then have the following generalisation
of Proposition~\ref{prop:grad_kernel_smooth} which will be useful for
analysing functionals that depend on (symmetrised) distributional
derivatives.

\begin{proposition}
  \label{prop:grad_kernel_nonsmooth}
  If $\grad^k \tensor u = 0$ for a
  $\tensorspace[l]{\RR^d}$-valued distribution, then $u$ is
  regular and a $\tensorspace[l]{\RR^d}$-valued polynomial of maximal
  degree $k-1$. 
  
  If $\symgrad^k u = 0$ for a $\Sym^l(\RR^d)$-valued
  distribution, then $u$ is regular and a
  $\Sym^l(\RR^d)$-valued polynomial of maximal degree $k+l-1$.
\end{proposition}
\begin{proof}
This can be deduced from Proposition \ref{prop:grad_kernel_smooth} via mollification arguments similar as in \cite[Proposition 3.3]{bredies2013boundeddeformation_mh}.
\end{proof}

\subsection{Functions of higher-order bounded variation}
\label{subsec:higher_order_tv}

In the following, we discuss functions whose derivative is a Radon
measure for a fixed order of differentiation. As higher-order
derivatives of scalar functions are always symmetric, it suffices to
consider only the symmetrised higher-order derivative $\symgrad^k$ in
this case as well as symmetric tensor fields. However, as we are also
interested in intermediate differentiation orders, we moreover discuss
spaces of symmetric tensors for which the symmetrised derivative of
some order is a Radon measure.

In the following, recall that $k \geq 1$ denotes a differentiation order and
$l \geq 0$ denotes a tensor order.

\begin{definition}
  Let $\Omega \subset \RR^d$ be a domain.
  \begin{enumerate}
  \item 
    In the case $l = 0$, for $u \in \LPlocspace{1}{\Omega}$, the
    \emph{total variation of order $k$} is defined as
    \[
    \fl
    \TV^k(u) = \sup\ \Bigset{
      \int_\Omega u \divergence^k \varphi \dd{x}}{\varphi \in \Ccspace{k}{\Omega,\Sym^k(\RR^d)},
      \ \norm[\infty]{\varphi} \leq 1}.
    \]
    For general $l \geq 0$ and
    $u \in \LPlocspace{1}{\Omega,\Sym^l(\RR^d)}$, the \emph{total
      deformation of order $k$} is  %
    \[
    \fl
    \TD^k(u) = \sup\ \Bigset{
      \int_\Omega u \inprod \divergence^k \varphi \dd{x}}
    {\varphi \in \Ccspace{k}{\Omega,\Sym^{k+l}(\RR^d)},
      \ \norm[\infty]{\varphi} \leq 1}.
    \]
  \item
    The normed space according to
    \[
    \fl
    \begin{array}{rl}
      \BD^k(\Omega,\Sym^l(\RR^d)) &= 
      \set{u \in \LPspace{1}{\Omega,\Sym^l(\RR^d)}}{\TD^k(u) < \infty},
      \\
      \norm[\BD^k]{u} &= \norm[1]{u} + \TD^k(u)
    \end{array}
    \]
    is called the space of \emph{symmetric tensor fields of bounded
      deformation of order $k$}. The scalar case, i.e., $l=0$, is
    referred to as the space of \emph{functions of bounded variation of
      order $k$}. The latter spaces are denoted by $\BV^k(\Omega)$.
  \end{enumerate}
\end{definition}
We note that the Hilbert-space norm on the tensor space for the definition of $\TV^k$ leads to a corresponding pointwise norm on the derivatives. While this choice is rather natural, and does not require to distinguish primal and dual norms, also other choices are possible for which we refer to \cite{lefkimmiatis13hessian_shatten_mh} in the second-order case.

Let us analyse some of the basic properties of these spaces.
\begin{proposition}
  \label{prop:tdk_semi_norm}
  Let $\Omega \subset \RR^d$ be a domain, $p \in {[{1,\infty}]}$.
  Then:
  \begin{enumerate}
  \item
    $\TD^k$ is proper, convex and a lower semi-continuous seminorm
    on $\LPspace{p}{\Omega,\Sym^l(\RR^d)}$.
  \item $\TD^k(u) = 0$ if and only if $\symgrad^k u = 0$. In
    particular, $\TD^k(u) = 0$ implies that $u$ is a
    $\Sym^l(\RR^d)$-valued polynomial of maximal degree $k+l-1$.
  \end{enumerate}
\end{proposition}

\begin{proof}
  With $p^*$ being the dual exponent to $p$, each test tensor field
  obeys $\divergence^k \varphi \in \LPspace{p^*}{\Omega,\Sym^l(\RR^d)}$ for
  $\varphi \in \Ccspace{k}{\Omega,\Sym^{k+l}(\RR^d)}$. The functional
  $\TD^k$ is thus a pointwise supremum over a set of continuous linear
  functionals and, consequently, convex and lower semi-continuous.  By
  definition, it is obviously proper and positively homogeneous since
  if $\divergence^k \varphi$ is a test vector field, then also
  $-\divergence^k \varphi$ is.

  By definition of $\TD^k$ we see that $\TD^k(u) = 0$ if and only if
  $\int_\Omega u \inprod \divergence^k \varphi \dd{x} = 0$ for each
  $\varphi \in \Ccspace{k}{\Omega,\Sym^{k+1}(\RR^d)}$. But this is
  equivalent to $\symgrad^k u = 0$ in the distributional sense such
  that in
  particular, the polynomial representation follows from
  Proposition~\ref{prop:grad_kernel_nonsmooth}.
\end{proof}

In order to show more properties, for instance, that
$\BD^k(\Omega,\Sym^l(\RR^d))$ is a Banach space, let us adopt a more
abstract viewpoint.  We say that a function
$\abs{\placeholder}: X \to [0,\infty]$ for $X$ a Banach space is a
\emph{lower semi-continuous seminorm on $X$} if $\abs{\placeholder}$
is positive homogeneous, satisfies the triangle inequality and is
lower semi-continuous. The \emph{kernel} of $\abs{\placeholder}$,
denoted $\kernel{\abs{\placeholder}}$, is the set
$\set{x \in X}{\abs{x} = 0}$ which is a closed linear subspace of $X$.

\begin{lemma}
  \label{lem:banach_space_plus_lsc_seminorm}
  Let $\abs{\placeholder}$ be a lower semi-continuous seminorm on the
  Banach space $X$ with norm $\norm[X]{\placeholder}$. Then,
  \[
  Y = \set{x \in X}{\abs{x} < \infty},
  \qquad
  \norm[Y]{x} = \norm[X]{x} + \abs{x}
  \]
  is a Banach space. The seminorm $\abs{\placeholder}$ is
  continuous in $Y$.
\end{lemma}

\begin{proof}
  It is immediate that $Y$ is a normed space.  Let $\seq{x^n}$ be a
  Cauchy sequence in $Y$ which is obviously a Cauchy sequence in
  $X$. Hence, a limit $x \in X$ exists for which the lower
  semi-continuity yields
  $\abs{x} \leq \liminf_{n \to \infty} \abs{x^n} < \infty$,
  the latter since $\seq{\abs{x^n}}$ is a real Cauchy sequence.
  In particular, $x \in Y$.
  
  To obtain convergence with respect to $\abs{\placeholder}$,
  choose, for $\varepsilon > 0$, an $n$ such that for all 
  $m \geq n$, $\abs{x^n - x^m} \leq \varepsilon$. Letting $m \to \infty$
  gives, as $x^n - x^m \to x^n - x$ in $X$,
  \[
  \abs{x^n - x} \leq \liminf_{n \to \infty} \ \abs{x^n - x^m} \leq \varepsilon.
  \]
  This implies $x^n \to x$ in $Y$ which is what we intended to show.

  Finally, the continuity of $\abs{\placeholder}$ follows from the
  standard estimate
  $\bigabs{\abs{x^1} - \abs{x^2}} \leq \abs{x^1 - x^2} \leq
  \norm[Y]{x^1 - x^2}$ for $x^1, x^2 \in Y$.
\end{proof}

It is then obvious from Proposition~\ref{prop:tdk_semi_norm} and
Lemma~\ref{lem:banach_space_plus_lsc_seminorm} that
$\BD^k(\Omega,\Sym^l(\RR^d))$ is a Banach space. In order to examine
the structure of these spaces, it is crucial to understand the case
$k=1$, i.e.,
$\BD(\Omega,\Sym^l(\RR^d)) = \BD^1(\Omega,\Sym^l(\RR^d))$, where the
symmetrised derivative is only a measure. For $l \geq 1$, these spaces
are strictly greater that $\BV(\Omega,\Sym^l(\Omega))$ as a
consequence of the failure of Korn's inequality. Important
properties of these spaces are summarised as follows.

\begin{theorem}[{\cite[Theorem 2.6]{holler14inversetgv_mh}}] 
  \label{thm:symgrad_measure_dist}
  If $u$ is a $\Sym^l(\RR^d)$-valued distribution on $\Omega$ a bounded
  Lipschitz domain with
  $\symgrad u \in \radonspace{\Omega,\Sym^{l+1}(\RR^d)}$, then
  $u \in \BD(\Omega,\Sym^l(\RR^d))$.
\end{theorem}
\begin{theorem}[{\cite[Theorems 4.16 and 4.17]{bredies2013boundeddeformation_mh}}] \label{thm:bd_embedding}
For $\Omega$ a bounded Lipschitz domain and, $1 \leq p \leq d/(d-1)$, the space $\BD(\Omega,\Sym^l(\RR^d))$ is continuously embedded in $\LPspace{p}{\Omega,\Sym^l(\RR^d)}$. Moreover, for $p < d/(d-1)$, the embedding is compact.
\end{theorem}

\begin{theorem}[{Sobolev--Korn inequality \cite[Corollary 4.20]{bredies2013boundeddeformation_mh}}]
  \label{thm:sobolev_korn}
  For $\Omega$ a bounded Lipschitz domain and
  $R_l: \LPspace{d/(d-1)}{\Omega,\Sym^l(\RR^d)} \to \kernel{\symgrad}$ a linear and
  continuous projection onto the kernel of $\symgrad$, there exists a
  constant $C > 0$ such that for each $u \in \BD(\Omega,\Sym^l(\RR^d))$
  it follows that
  \begin{equation}
    \label{eq:sobolev-korn}
    \norm[d/(d-1)]{u - R_lu} \leq C \norm[\radon]{\symgrad u}.
  \end{equation}
\end{theorem}
Note that the projection $R_l$ as stated always exists as
$\ker(\symgrad)$ is finite-dimensional (see
Proposition~\ref{prop:grad_kernel_nonsmooth}).

Now, for general $k$ and $u \in \BD^k(\Omega,\Sym^l(\RR^d))$ fixed,
$w = \symgrad^{k-1}u$ is a $\Sym^{l+k-1}(\RR^d)$-valued distribution with the property
\[
\fl
(\symgrad w)(\varphi) = -w(\divergence \varphi) = (-1)^k u(\divergence^k \varphi) =
(-1)^k \int_\Omega u \inprod \divergence^k \varphi \dd{x} = \int_\Omega \varphi
\inprod \dd{\symgrad^k u}
\]
for $\varphi \in \Ccspace{\infty}{\Omega,\Sym^{k+l}(\RR^d)}$. In other words,
$\symgrad w = \symgrad^k u \in \radonspace{\Omega,\Sym^{k+l}(\RR^d)}$,
thus Theorem~\ref{thm:symgrad_measure_dist} implies that
$\symgrad^{k-1}u = w \in \BD(\Omega,\Sym^{k+l-1}(\RR^d))$ and, in particular,
we have $u \in \BD^{k-1}(\Omega,\Sym^l(\RR^d))$. Hence, the 
spaces are nested:
\[
\BD^k(\Omega,\Sym^l(\RR^d)) \subset \BD^{k-1}(\Omega,\Sym^l(\RR^d))
\subset \ldots \subset \BD(\Omega,\Sym^l(\RR^d)).
\]
Let us look at the norms: By the Sobolev--Korn
inequality~\eqref{eq:sobolev-korn}, for some linear projection
$R_{k+l-1}: \BD(\Omega,\Sym^{k+l-1}(\RR^d)) \to \kernel{\symgrad}$, we
see
\[
  \norm[1]{\symgrad^{k-1}u - R_{k+l-1}\symgrad^{k-1}u} \leq
  C \norm[\radon]{\symgrad^k u}
\]
which implies
\[
  \norm[\radon]{\symgrad^{k-1}u } 
  \leq 
  C \bigl( \norm[\radon]{\symgrad^k u} + \norm[1]{R_{k+l-1}\symgrad^{k-1}u} 
  \bigr).
\]
 Now, $u \mapsto R_{k+l-1}\symgrad^{k-1} u$ is well-defined on 
 $\BD^k(\Omega,\Sym^l(\RR^d))$, linear, has finite-dimensional image
 and is hence continuous. We may therefore estimate
\[
\norm[\radon]{\symgrad^{k-1} u} \leq C \bigl( \norm[1]{u} 
+ \norm[\radon]{\symgrad^k u} \bigr).
\]
Proceeding inductively, we arrive at the estimate
\begin{equation}
  \label{eq:higer_order_bd_norm_equiv1}
  \sum_{m = 0}^{k} \norm[\radon]{\symgrad^m u} \leq 
  C \bigl( \norm[1]{u} + \norm[\radon]{\symgrad^k u} \bigr)
\end{equation}
for some $C > 0$ independent of $u$.
Therefore, we obtain the following theorem.

\begin{theorem}
  \label{thm:l1-tvx-equiv-norm}
  If $\Omega \subset \RR^d$ is a bounded Lipschitz domain, then the
  norm equivalence
  \begin{equation}
    \label{eq:higer_order_bd_norm_equiv}
    \norm[1]{u} + \norm[\radon]{\symgrad^k u} \sim
    \sum_{m = 0}^k \norm[\radon]{\symgrad^m u}
  \end{equation}
  holds on $\BD^k(\Omega,\Sym^l(\RR^d))$.
  The embeddings
  \[
  \fl
  \BD^k(\Omega,\Sym^l(\RR^d)) \embed \BD^{k-1}(\Omega,\Sym^l(\RR^d))
  \embed \ldots \embed \BD(\Omega,\Sym^l(\RR^d))
  \]
  are continuous.
\end{theorem}

\begin{proof}
  The nontrivial estimate to establish norm equivalence has just been
  shown in~\eqref{eq:higer_order_bd_norm_equiv1}. The continuity of
  the embedding follows from the fact that the norm on the right-hand 
  side in~\eqref{eq:higer_order_bd_norm_equiv} is 
  increasing with respect to $k$.
\end{proof}

In the scalar case, we can furthermore establish Sobolev embeddings.

\begin{theorem}
  \label{thm:tvk-sobolev-embedding}
  Let $\Omega$ be a bounded Lipschitz domain and
  $0 \leq m < k$.
  \begin{description}

  \item 
    [For $k - m \leq d$:]
    The space $\BV^k(\Omega)$ is continuously embedded in
    $\Hspace[p]{m}{\Omega}$ for $1 \leq p \leq \frac{d}{d - (k-m)}$, where we set $\frac{d}{d - (k-m)} = \infty$ for $k-m = d$.

    If $p < \frac{d}{d - (k - m)}$, then the embedding is
    compact.

  \item
    [For $k - m > d$:] The space $\BV^k(\Omega)$
    is compactly embedded in $\Cspace{m,\alpha}
    {\closure{\Omega}}$ for each
    $\alpha \in {]{0,1}[}$.

  \end{description}
\end{theorem}

\begin{proof}
  In the scalar case, $\norm[1]{u} + \sum_{\abs{\beta} \leq k-1} 
  \norm[\radon]{\grad \partial^\beta u}$ for $\beta \in \NN^d$ a multiindex
  and $u \in \BV^k(\Omega)$
  constitutes an equivalent norm on $\BV^k(\Omega)$, as a consequence of
  Theorem~\ref{thm:l1-tvx-equiv-norm}.
  By the Poincar\'e inequality in $\BV(\Omega)$,
  \[
  \norm[d/(d-1)]{\partial^\beta u} 
  \leq C \bigl( \norm[\radon]{\grad \partial^\beta u} + \norm[1]{u} \bigr)
  \]
  for each $\abs{\beta} \leq k-1$. This establishes the continuous embedding
  $\BV^k(\Omega) \to W^{k-1,d/(d-1)}(\Omega)$. Application of the
  well-known embedding 
  theorems for Sobolev spaces (see \cite[Theorems 5.4 and 6.2]{Adams2003sobolev_mh}) then give the results for the cases $k-m < d$ and $k - m >d$ as well as for the case $k-m=d$ and $p< \infty$.
  
  For the case $k-m=d$ and $p=\infty$ we note that again by Sobolev embeddings \cite[Theorem 5.4]{Adams2003sobolev_mh} we get for a constant $C>0$ and all $u \in \Hspace[1]{k}{\Omega}$ that
  \[ \sum_{i=0}^m \norm[\infty]{\nabla ^i u} \leq C \sum_{i=0}^k \norm[1]{\nabla ^i u} .\]
  Approximating $u \in \BV^k(\Omega)$ with a sequence $\seq{u^n}$ in $C^\infty (\Omega) \cap \BV^k(\Omega)$ strictly converging to $u$ in $\BV^k(\Omega)$ as in Lemma \ref{lem:bvk_strict_convergence}, the result follows from applying this estimate to each $u^n$ and using lower semi-continuity of the $L^\infty$-norm with respect to convergence in $L^1$.
\end{proof}

We would like to employ $\TD^k$ as a regulariser and first characterise
its kernel. 
For that purpose,
we note that $\TD^k(u) = 0$ for some $u \in \BD^k(\Omega,\Sym^l(\RR^d))$
implies that $\symgrad   u = 0$ in the distributional sense, 
hence Proposition \ref{prop:grad_kernel_nonsmooth} implies that 
$u$ is a $\Sym^ l(\RR^ d)$-valued polynomial of maximal degree $k+l-1$.
This yields the following result.

\begin{proposition}
  \label{prop:td^k-kernel}
  The space $\kernel{\TD^k}$ is a subspace of polynomials
  of degree less than $k+l$. If $l=0$, then $\kernel{\TV^k} = \poly^{k-1}
  = \set{u: \Omega \to \RR}{u \ \text{polynomial of degree} \leq k-1}$.
\end{proposition}

Next, we like to discuss coercivity of the higher-order
total variation functionals.

\begin{proposition}
  \label{prop:tv^k-coercive}
  Let $k \geq 1$, $l\geq 0$ and $\Omega$ be a bounded Lipschitz domain.
  Then, $\TD^k$ is coercive in the following sense:
  For each linear and continuous projection
  $R: \LPspace{d/(d-1)}{\Omega,\Sym^l(\RR^d)} 
  \to \kernel{\TD^k}$, there is a $C > 0$ such that
  \[
  \norm[d/(d-1)]{u - Ru} \leq C \TD^k(u) \qquad
  \text{for all} \quad u \in \BD^k(\Omega,\Sym^l(\RR^d)).
  \]
\end{proposition}

\begin{proof}
  At first note that by the embeddings $\BD^k(\Omega,\Sym^l(\RR^d)) \embed \BD(\Omega,\Sym^l(\RR^d)) \embed \LPspace{d/(d-1)}{\Omega,\Sym^l(\RR^d)}$ the left hand side of the claimed inequality is well defined and finite.

  We use a contradiction argument in conjunction with 
  compactness.
  Suppose for $R$ as stated above there is a sequence
  $\seq{u^n}$ such that $\norm[d/(d-1)]{u^n - Ru^n} = 1$ and
  $\TD^k(u^n) \to 0$ as $n \to \infty$.
  This implies $\seq{\norm[1]{u^n - Ru^n}}$ being bounded,
  $\TD^k(u^n - Ru^n) \to 0$ and by 
  Theorems~\ref{thm:l1-tvx-equiv-norm} and \ref{thm:bd_embedding}
  $\seq{u^n - Ru^n}$ has to be precompact in 
  $\LPspace{1}{\Omega,\Sym^l(\RR^d)}$, 
  i.e., without loss of generality, we may assume
  that $u^n - Ru^n \to u$ in $\LPspace{1}{\Omega,\Sym^l(\RR^d)}$. 
  By lower semi-continuity,
  \[
  \TD^k(u) \leq \liminf_{n \to \infty} \ \TD^k(u^n) = 0,
  \]
  hence $u \in \kernel{\TD^k} = \range{R}$. 
  On the other hand, $R(u^n - Ru^n) = 0$
  for each $n$ as $R$ is a projection, thus, $Ru = 0$ and, consequently,
  $u = 0$. In total, we have $\lim_{n \to \infty} (u^n - Ru^n) = 0$
  in $\BD^k(\Omega)$, and again by continuous embedding, also in 
  $\LPspace{d/(d-1)}{\Omega}$ which is a contradiction to $\norm[d/(d-1)]{u^n - Ru^n} 
  = 1$ for all $n$. Consequently, coercivity has to hold.
\end{proof}

\begin{corollary}
  \label{cor:tv^k-coercive-scalar}
  In the scalar case, for
  $p \in [1,\infty]$ with $p \leq \frac{d}{d-k}$ if $k < d$,
  we also have
  \[
  \norm[p]{u - Ru} \leq C \TV^k(u).
  \]
\end{corollary}

\begin{proof}
  This follows with the embedding Theorem~\ref{thm:tvk-sobolev-embedding}:
  \[
  \norm[p]{u - Ru} \leq
  C \bigl(\norm[1]{u - Ru} + \TV^k(u) \bigr) \leq
  C \TV^k(u). \qedhere
  \]
\end{proof}

  \begin{remark} 
  The above coercivity estimate also implies that
  the Fenchel conjugate of $\TV^k$ is the
  indicator functional of closed convex set
  in $\LPspace{p^*}{\Omega} \cap \kernel{\TV^k}^\perp$ with
  non-empty interior. Indeed, for $\xi \in 
  \LPspace{p^*}{\Omega} \cap \kernel{\TV^k}^\perp$ such that
  $\norm[p^*]{\xi} \leq C^{-1}$ it follows for any
  $u \in \LPspace{p}{\Omega}$ that
  \[
  \scp{\xi}{u} = \scp{\xi}{u - Ru} \leq 
  \norm[p^*]{\xi}\norm[p]{u - Ru} \leq \TV^k(u)
  \]
  which means that $(\TV^k)^*(\xi) = 0$.
  On the other hand, if $\xi \in \LPspace{p^*}{\Omega} 
  \setminus \kernel{\TV^k}^\perp$, then $\scp{\xi}{u} > 0$
  for some $u \in \kernel{\TV^k}$. Thus,
  $\scp{\xi}{u} > \TV^k(u)$ so $(\TV^k)^*(\xi) = \infty$.
\end{remark}

It is interesting to note that a coercivity estimate similar to the one of Corollary \ref{cor:tv^k-coercive-scalar} also holds between two higher-order TV functionals of different order.

\begin{lemma} \label{lem:tv-k1_tv-k2_coercivity_estimate} Let $\Omega$ be a bounded Lipschitz domain, $1\leq k_1 < k_2$ be two orders of differentiation, $p \in [1,\infty[$ with $p \leq d/(d-k_2)$ if $k_2 < d$ and $R:L^p(\Omega) \rightarrow \ker(\TV^{k_2})$ be a continuous, linear projection. Then there exists a constant $C>0$ such that

  \begin{equation}
    \TV^{k_1}(u - Ru)
    \leq
    C \TV^{k_2}(u)
  \end{equation}
  holds for each $u \in \BV^{k_2}(\Omega)$.

\begin{proof}
  Assume the opposite, i.e., the existence of $\seq{u^n}$ such that
  $\TV^{k_1}(u^n - Ru^n) = 1$ and $\TV^{k_2}(u^n) \to 0$ as
  $n \to \infty$.
  Then, by compact embedding
  $\BD^{k_2-k_1}(\Omega,\Sym^{k_1}(\RR^d)) \to
  \LPspace{1}{\Omega,\Sym^{k_1}(\RR^d)}$, we have
  $\grad^{k_1}(u^n - Ru^n) \to v$ as $n \to \infty$ for some
  $v \in \LPspace{1}{\Omega, \Sym^{k_1}(\RR^d)}$ for a subsequence
  (not relabelled). On the other hand, the
  Poincar\'e estimate gives
  $\norm[p]{u^n - Ru^n} \leq C\TV^{k_2}(u^n)$, so
  $u^n - Ru^n \to 0$ as $n \to \infty$ in $\LPspace{1}{\Omega}$.
  By closedness of $\grad^{k_1}$ this yields $v = 0$. By convergence in
  $\LPspace{1}{\Omega,\Sym^{k_1}(\RR^d)}$, this gives the
  contradiction $\TV^{k_1}(u^n - Ru^n) \to 0$ as $n \to \infty$.
\end{proof}

\end{lemma}

\subsection{Tikhonov regularisation}

The coercivity which has just been established can be regarded as the
most important step towards existence for variational problems with
$\TV^k$-regularisation.  Here, we first prove an existence result for
linear inverse problems in a general abstract version.

\begin{theorem}
  \label{thm:general_reg_existence_linear}
  Let $X$ be a reflexive Banach space, $Y$ be a Banach space, $K: X\to Y$ be linear and continuous,
  $S_f: Y \to {[{0,\infty}]}$ a proper, convex, lower
  semi-continuous and coercive discrepancy functional associated with
  some data $f$, $\abs{\placeholder}:X \rightarrow [0,\infty]$ a lower
  semi-continuous seminorm and $\alpha > 0$. Assume that there exists a linear and continuous projection $R:X \rightarrow \ker(\abs{\placeholder})$
  and a $C>0$ such that
    \[
      \norm[X]{u - Ru} \leq C \abs{u}
      \qquad \text{for all}  \quad u \in X,
    \]
  and either 
  \begin{enumerate}
  \item $\ker(\abs{\placeholder})$ is finite-dimensional or, more generally,
  \item $\ker(K) \cap \ker(\abs{\placeholder})$ admits a complement $Z$ in $\ker(\abs{\placeholder})$ and
  $\norm[X]{u} \leq C \norm[Y]{Ku}$ for some $C > 0$ and all $u \in Z$.
  \end{enumerate}
  Then, the Tikhonov minimisation problem %
  \begin{equation}
    \label{eq:general_seminorm_min}
    \min_{u \in X} \ S_f(Ku) + \alpha \abs{u}.
  \end{equation}
  is well-posed, i.e., there exists a solution and the solution mapping is stable in sense
  that, if $S_{f^n}$ converges to $S_f$ as in \eqref{eq:discrepancy_convergence} and $\seq{S_{f^n}}$ is
  equi-coercive, then for each sequence of minimizers $\seq{u^n}$
  of~\eqref{eq:general_seminorm_min} with discrepancy $S_{f^n}$,

  \begin{itemize}
  \item
    either $S_{f^n}(Ku^n) + \alpha \abs{u^n} \to \infty$ as
    $n \to \infty$ and~\eqref{eq:general_seminorm_min} with discrepancy
    $S_f$ does not admit a finite solution,
  \item or
    $S_{f^n}(Ku^n) + \alpha \abs{u^n} \to \min_{u \in
      X} S_f(u) + \alpha \abs{u}$
    as $n \to \infty$ and there is, possibly up to shifts by functions in $\ker(K) \cap \ker(\abs{\placeholder})$, a
    weak accumulation point $u \in X$ that
    minimises~\eqref{eq:general_seminorm_min} with discrepancy $S_f$.
  \end{itemize}  
  Further, in case \eqref{eq:general_seminorm_min} with discrepancy $S_f$ admits a 
  finite solution, for each subsequence $\seq{u^{n_k}}$ weakly converging to some $u \in X$,
  it holds that $\abs{u^{n_k}} \to \abs{u}$ as $k \to \infty$.
  Also, if $S_f$ is strictly convex and $K$ is injective, finite solutions $u$ of \eqref{eq:general_seminorm_min}
  are unique and $u^n \wrightarrow u$ in $X$.
\end{theorem}
The same result is true if, for instance, instead of being reflexive, $X$ is the dual of a separable space,
and we replace weak convergence by weak* convergence in the (lower semi-) continuity assumptions on  $K$, $\abs{\placeholder}$,  $S_f$ and in \eqref{eq:discrepancy_convergence}.
\begin{proof}
  At first note that $\ker(\abs{\placeholder})$ being finite-dimensional implies condition (ii) above, hence we can assume that (ii) holds.
  We start with existence. Assume that the objective functional in~\eqref{eq:general_seminorm_min} is
  proper as otherwise, there is nothing to show. For a minimising
  sequence $\seq{u^n}$, by the coercivity assumption,
  $\seq{u^n - R u^n }$ is bounded in $X$.
  Now, (ii) implies the existence of a linear and continuous projection $P_Z: \ker(\abs{\placeholder}) \to Z$ such that $\id - P_Z$ projects $\ker(\abs{\placeholder})$ onto $\ker(K) \cap \ker(\abs{\placeholder})$. With $v^n = P_Z Ru^n$, we see that
  also
  $\seq{u^n - R u^n  + v^n}$ is a minimising sequence and it suffices
  to show boundedness of $\seq{v^n}$ to obtain a convergent subsequence.
  But the latter holds true since by assumption $\norm[p]{v^n} \leq C \norm[Y]{Kv^n}$, such that $\norm[Y]{Kv^n} \leq \norm[Y]{K(u^n- Ru^n + v^n)}
  + \norm{K}\norm[X]{u^n - Ru^n}$, with
  the right-hand side being bounded as a consequence of the coercivity of $S_f$ and the boundedness of $\seq{u^n - R u^n}$. 
  Hence, as $X$ is reflexive, %
  a subsequence of  $\seq{u^n - R u^n  + v^n}$ converges weakly
  to a limit $u \in X$.
  By continuity of $K$ and
  lower semi-continuity of both $S_f$ and $\abs{\placeholder}$ it follows that
  $u$ is a solution to~\eqref{eq:general_seminorm_min}.
  In case $S_f$ is strictly convex and $K$ is injective,
  $S_f \compose K$ is already strictly convex, so finite minimizers 
  of \eqref{eq:general_seminorm_min} have to be unique.

  Now let $\seq{u^n}$ be a sequence of minimizers of~\eqref{eq:general_seminorm_min}
  with discrepancy $S_{f^n}$. We denote by $F = S_f \compose K + \alpha \abs{\placeholder}$ as well
  as $F_n = S_{f^n} \compose K + \alpha \abs{\placeholder}$ and first suppose
  that $\seq{F_n(u^n)}$ is bounded. 
  We can then add $v^n - Ru^n \in \ker(K) \cap \ker(\abs{\placeholder})$ to $u^n$,
  with $v^n = P_ZRu^n$,
  and 
  from equi-coercivity of $\seq{S_{f^n}}$ obtain 
  boundedness of %
  $\seq{u^n - Ru^n + v^n}$ as before.  %
  This shows that by shifting the minimizers %
  within $\ker(K) \cap \ker(\abs{\placeholder})$ always leads to a bounded
  sequence, 
  i.e., we may assume without loss of generality that
  $\seq{u^n}$ is bounded such that a weak accumulation point exists.
  Suppose that $u^{n_k} \wrightarrow u$ as $k \to \infty$. Then,
  estimating as in the proof of Theorem \ref{thm:tv_reg_stability},
  we can obtain that $u$ is a minimizer for $F$ and that
  $\lim_{k \to \infty} F_{n_k}(u^{n_k}) = F(u)$ as well as
  $\lim_{k \to \infty} \abs{u^{n_k}} = \abs{u}$.
  Also, if $u$ is the unique minimizer for~\eqref{eq:general_seminorm_min} with
  discrepancy $S_f$, $u^n \wrightarrow u$ as $n \to \infty$ follows since
  any subsequence has to contain another subsequence that converges
  weakly to $u$.
  
  The result for the two remaining cases $\liminf_{n \to \infty} F_n(u^n) < \infty$ and
  $F_n(u^n) \to \infty$, respectively, finally follows
  analogously to Theorem \ref{thm:tv_reg_stability}. \qedhere

\end{proof}

Given that $\ker(\TV^k)$ is finite dimensional, the above result immediately implies well-posedness for $\abs{\placeholder} = \TV^k$ with $X=\LPspace{p}{\Omega}$, as stated in the following corollary. The crucial ingredient here is the estimate $\norm[p]{u - R u} \leq C \TV^k(u)$, which restricts the exponent of the underlying $L^p$-space to $p \leq d/(d-k)$ if $k < d$. This shows that, the higher the order of differentiation used in the regularisation, the weaker are the requirements on the underlying spaces and, consequently, on the continuity of the operator $K$.

\begin{corollary} \label{cor:existence_tvk} With $X=L^p(\Omega)$, $\Omega$ being a bounded Lipschitz domain, and $S_f$ and $K$ as in Theorem \ref{thm:general_reg_existence_linear}, 
  \begin{equation}
    \label{eq:general_tvk_min}
    \min_{u \in \LPspace{p}{\Omega}} \ S_f(Ku) + \alpha \TV^ k(u).
  \end{equation}
  is well-posed in the sense of Theorem \ref{thm:general_reg_existence_linear} whenever $p \in {]{1,\infty}[}$ with $p \leq d/(d-k)$ if $k < d$. %
\end{corollary}

As can be easily seen from the respective proofs, also the convergence result of Theorem \ref{thm:tv_reg_convergence} and the result on convergence rates as in Proposition \ref{prop:tv_reg_convergence_rate} transfer to $\TV^k$ regularisation.

\begin{theorem}
  \label{thm:tvk_reg_convergence}
  With the assumptions of Corollary \ref{cor:existence_tvk}, let
$u^\dagger \in \BV(\Omega)$ be a minimum-$\TV^k$-solution of
$Ku^\dagger = f^\dagger$ for some data $f^\dagger$ in $Y$ and for each 
$\delta > 0$ let $f^\delta $ be such that
$S_{f^\delta}(f^\dagger) \leq \delta$ and denote by $u^{\alpha,\delta}$ a
finite solution of~\eqref{eq:general_tvk_min} for parameter $\alpha > 0$ and
data $f^\delta$.
  Let the discrepancy functionals $\seq{S_{f^\delta}}$ be equi-coercive and converge to
  $S_{f^\dagger}$ in the sense of \eqref{eq:discrepancy_convergence} and $S_{f^\dagger}(v) = 0$ if and only if $v = f^\dagger$.
  Choose for each $\delta > 0$ the parameter $\alpha > 0$ such that
  \[
  \alpha \to 0, \quad
  \frac{\delta}{\alpha} \to 0 \qquad \text{as} \qquad \delta \to 0.
  \]
  Then, up to shifts by functions in $\ker(K) \cap \poly^{k-1}$, $\seq{u^{\alpha,\delta}}$ has at least one $L^p$-weak accumulation
  point. Each $L^p$-weak accumulation point is a 
  minimum-$\TV^ k$-solution of $Ku = f^\dagger$ and $\lim_{\delta \to 0}
  \TV^ k(u^{\alpha,\delta}) = \TV^ k(u^\dagger)$.
\end{theorem}

\begin{proposition}
  \label{prop:tvk_reg_convergence_rate} In the situation of Theorem \ref{thm:tvk_reg_convergence},
  let $K^*w^\dagger \in \subgrad \TV^ k(u^\dagger)$ for some
  $w^\dagger \in Y^*$. Then,
  \begin{equation}
    \label{eq:bregman_dist_est_tvk}
    \breg{\TV^k}{K^*w^\dagger}(u^{\alpha,\delta},u^\dagger) \leq \frac1\alpha
    \bigl( S_{f^\delta}^*(\alpha w^\dagger) + S_{f^\delta}^*(-\alpha
    w^\dagger) + 2 \delta \bigr).   
  \end{equation}
\end{proposition}
The last result in particular guarantees convergence rates for the settings of Example \ref{ex:discrepancies_rates}. Note also that the above results remain true in case $p=1$ or in case $p=d/(d-k)= \infty$ and $K$ is weak*-to-weak continuous.

Let us finally note some first-order optimality conditions. For this
purpose, recall that for $X$ a Banach space, the \emph{normal cone}
$\mN_K(u)$ of a set $K \subset X$ at $u \in K$ is given by the
collection of all $w \in X^*$ for which
$\scp[X^* \times X]{w}{v - u} \leq 0$ for all $v \in K$. If we set $\mN_K(u) = \emptyset$ for $u \notin K$, we have that $\mN_K = \subgrad \mI_K$ where $\mI_K$ is the indicator function of $K$, i.e., $\mI_K(u) = 0$ if $u \in K$ and $\infty$ otherwise.
\begin{proposition}
  \label{prop:tv^k-tikh-optim}
  In the situation of Corollary \ref{cor:existence_tvk}, if
  $S_f(v) = \frac{1}{2}\|v - f\|_Y^2$ and $Y$ is a Hilbert space,
  $u^* \in \LPspace{p}{\Omega}$ is a solution
  of
  \begin{equation}
    \label{eq:tv^k-tikh-hilbert}
    \min_{u \in \LPspace{p}{\Omega}} \ \frac12\norm[Y]{Ku - f}^2 + \alpha \TV^k(u)
  \end{equation}
  if and only if
  \[
  u^* \in \mN_{\TV^k}\Bigl( \frac{K^*(f - Ku^*)}{\alpha} \Bigr)
  \]
  where $\mN_{\TV^k}$ is the normal cone associated with the set
  $\closure{\mB_{\TV^k}}$ where
  \[
  \mB_{\TV^k} = \set{w \in \LPspace{p^*}{\Omega}}
  {w = \divergence^k \varphi, \ \varphi \in \Ccspace{k}{\Omega,\Sym^k(\RR^d)},
    \ \norm[\infty]{\varphi} \leq 1}.
  \]
\end{proposition}

\begin{proof}
  As $u \mapsto \frac12 \norm[Y]{Ku - f}^2$ is G\^ateaux differentiable, 
  it is continuous with unique subgradient, so, by 
  subdifferential calculus, optimality of $u^*$ is equivalent to
  $K^*(f - Ku^*) \in \alpha \subgrad \TV^k(u^*)$ which can also
  be expressed as
  \[
  u^* \in \subgrad (\TV^k)^* \Bigl( \frac{K^*(f - Ku^*)}{\alpha} \Bigr).
  \]
  Now since $\TV^k = \mI_{\mB_{\TV^k}}^*$, it follows that
  $(\TV^k)^* = \mI_{\mB_{\TV^k}}^{**} = \mI_{\closure{\mB_{\TV^k}}} $
  , so $\subgrad (\TV^k)^* = \mN_{\TV^k}$.
\end{proof}

\begin{remark}
  \label{rem:tv^k-tikh-hilbert-bound}
  In the situation of Proposition~\ref{prop:tv^k-tikh-optim}, it is
  also possible to give an a-priori estimate for the solutions
  of~\eqref{eq:tv^k-tikh-hilbert} in case $K$ is injective on
  $\poly^{k-1}$. Indeed, with $R: \LPspace{p}{\Omega} \to \poly^{k-1}$
  the continuous projection operator on the kernel of $\TV^k$ and
  $C > 0$ the coercivity constant, i.e.,
  $\norm[p]{u - Ru} \leq C \TV^k(u)$ for all
  $u \in \LPspace{p}{\Omega}$, by optimality, a solution $u^*$
  satisfies $\alpha \TV^k(u^*) \leq \frac12 \norm[Y]{f}^2$ and
  consequently,
  $\norm[p]{u - Ru} \leq \frac1{2\alpha} C \norm[Y]{f}^2$. Likewise,
  comparing with $u^* - Ru^*$, optimality also gives
  $\norm[Y]{Ku^* - f}^2 \leq \norm[Y]{K(u^* - Ru^*) - f}^2$, which is
  equivalent to
  $\norm[Y]{KRu^*}^2 \leq 2 \scp{KRu^*}{f - K(u^* - Ru^*)}$. Using
  $ab \leq \frac14 a^2 + b^2$, the latter leads to
  $
  \norm[Y]{KRu^*}^2 \leq 4  \norm[Y]{f - K(u^* - Ru^*)}^2
  $,
  where the right-hand side can further be estimated, using
  $(a+b)^2 \leq (1 + \varepsilon) ( a^2 + \frac1{\varepsilon} b^2)$ with
  $\varepsilon = \frac1{4\alpha^2} C^2 \norm{K}^2$ to give
  \[
  \norm[Y]{KRu^*}^2
  \leq 4 \Bigl(1 + \frac{C^2 \norm{K}^2}{4\alpha^2} \Bigr)
    (1 + \norm[Y]{f}^2)\norm[Y]{f}^2.
  \]
  Now, as $K$ is injective on $\poly^{k-1} = \range{R}$, there is a
  $c > 0$ such that $c \norm[p]{Ru} \leq \norm[Y]{KRu}$ for all
  $u \in \LPspace{p}{\Omega}$. Consequently, employing the triangle
  inequality and estimating yields
  \begin{equation}
    \label{eq:tv^k-tikh-hilbert-bound}
    \norm[p]{u^*}
    \leq \frac1{2\alpha } \Bigl( \frac2c \sqrt{4 \alpha^2 + C^2 \norm{K}^2} + C \Bigr) \sqrt{1 + \norm[Y]{f}^2} \norm[Y]{f},
  \end{equation}
  which is an a-priori bound that only requires the knowledge of the
  Poincaré--Wirtinger-type constant $C$, the constant $c$ in the
  inverse estimate for $K$ on $\poly^{k+1}$, as well as an estimate on
  $\norm{K}$. Beyond being of theoretical interest, such a bound can
  for instance be used in numerical algorithms, see
  Section~\ref{sec:numerical_algorithms},
  Example~\ref{ex:discrete_l2_tv2}.

  If the Kullback--Leibler divergence is used instead of the quadratic
  Hilbert space discrepancy, i.e., $S_f(v) = \KL(v,f)$,
  $Y = \LPspace{1}{\Omega'}$, and data $f \geq 0$ a.e., then one has
  to choose a $u^0 \in \BV^k(\Omega)$ such that
  $\KL(Ku^0,f) < \infty$.  Set
  $C_f = \KL(Ku^0,f) + \alpha\TV^k(u^0)$. Then, an optimal solution
  $u^*$ will satisfy $\TV^k(u^*) \leq \frac{C_f}{\alpha}$. Further,
  we have $\norm[1]{v} \leq 2 \KL(v,f) + 2\norm[1]{f}$ for
  $v \in \LPspace{1}{\Omega'}$ with $v \geq 0$ a.e., see Lemma~\ref{lem:kl_basic_properties}, such that, if $c > 0$ is a constant
  with $c \norm[p]{Ru} \leq \norm[1]{KRu}$ for all
  $u \in \LPspace{p}{\Omega}$, we get
  \[
    \fl
    \norm[p]{Ru^*} \leq \frac1c \bigl(\norm[1]{Ku^*} +
    \norm{K}\norm[p]{u^* - Ru^*} \bigr) \leq \frac1c \Bigl(
    \frac{2 \alpha +  C\norm{K}}{\alpha} C_f + 2 \norm[1]{f} \Bigr),
  \]
  and finally arrive at
  \begin{equation}
    \label{eq:tv^k-tikh-kl-bound}
    \norm[p]{u^*} \leq \frac1c \Bigl( \frac{2\alpha + C \norm{K} +
      cC}{\alpha} C_f + 2 \norm[1]{f} \Bigr).
  \end{equation}
  This constitutes an a-priori estimate similar
  to~\eqref{eq:tv^k-tikh-hilbert-bound} for the Kullback--Leibler
  discrepancy, however, with the difference that also a suitable constant
  $C_f$ has to determined.
\end{remark}

\begin{figure}
  \centering
  \begin{tabular}{c@{\ }c@{\ }c@{\ }c}
    \includegraphics[width=0.22\linewidth]{pics_affine_denoising_noise.png}
    & 
    \includegraphics[width=0.22\linewidth]{pics_affine_denoising_tv1.png}
    &
    \includegraphics[width=0.22\linewidth]{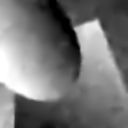}
    &
    \includegraphics[width=0.22\linewidth]{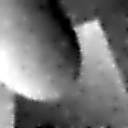}
    \\
    (a) & (b) & (c) & (d)
  \end{tabular}

  \caption{Second-order total-variation
    denoising example. (a) Noisy image, (b) 
    regularisation with $\TV$, 
    (c) regularisation with $\TV^2$, (d) 
    regularisation with $\norm[\radon]{\laplace\placeholder}$.
    All parameters are tuned to give highest PSNR with respect to
    the ground truth (Figure~\ref{fig:first-order-reg} (a)).
  }
  \label{fig:second-order-reg}
\end{figure}

\begin{remark}
  \label{rem:tv2-denoising}
  In order to show the effect of $\TV^2$ regularisation in contrast to
  $\TV$ regularisation, we performed a numerical denoising experiment
  for $f$ shown in Figure~\ref{fig:second-order-reg} (a), i.e., solved
  $\min_{u \in \LPspace{2}{\Omega}} \frac12 \norm[2]{u - f}^2 +
  \mR_\alpha(u)$ where $\mR_\alpha = \alpha \TV$ or
  $\mR_\alpha = \alpha
  \TV^2$. %
  One clearly sees that $\TV^2$ regularisation
  (Figure~\ref{fig:second-order-reg} (c)) reduces the staircasing
  effect of $\TV$ regularisation (Figure~\ref{fig:second-order-reg}
  (b)) and piecewise linear structures are well recovered. However,
  $\TV^2$ regularisation also blurs the object boundaries which appear
  less sharp in contrast to $\TV$ regularisation. 

  This is due to the fact that $\TV^k$ regularisation for $k \geq 2$
  is not able to produce solutions with jump discontinuities.  Indeed,
  $\TV^k$ regularisation implies that a solution $u$ has be to
  contained in $\BV^k(\Omega)$ which embeds into the Sobolev space
  $\Hspace[1]{k-1}{\Omega} \embed \Hspace[1]{1}{\Omega}$. As we have
  seen, for instance, in Example~\ref{ex:sobolev-no-jump}, this means
  that characteristic functions cannot be solutions. More generally,
  for $u \in \Hspace[1]{1}{\Omega} \subset \BV(\Omega)$, the
  derivative $\grad u$ interpreted as a measure is absolutely
  continuous with respect to the Lebesgue measure such that the
  singular part satisfies $\grad^s u =
  0$. Theorem~\ref{thm:bv_grad_decomp} then implies that
  the jump set $J_u$ is a $\hausdorff{d-1}$-negligible set, i.e., $u$ cannot jump on
  $(d-1)$-dimensional hypersurfaces.
\end{remark}

\begin{remark}
  \label{rem:laplace-tv-reg}
  Instead of taking higher-order TV which bases on the full gradient,
  one could also try to regularise with other differential operators,
  for instance with the (weak) Laplacian:
  \[
  \mR_\alpha(u) = \alpha \norm[\radon]{\laplace u}.
  \]
  However, the kernel of this seminorm is the space of 
  $p$-integrable harmonic functions
  on $\Omega$, the Bergman spaces, which are infinite-dimensional.
  Therefore, in view of Theorem \ref{thm:general_reg_existence_linear},
  to use $\mR_\alpha$ for the regularisation of ill-posed
  linear inverse problems, the forward operator $K$ must be
  continuously invertible on a complement of $\kernel{\mR_\alpha}\cap \ker(K)$, i.e., well-posed.
  This limits the applicability of this regulariser.
  Nevertheless, denoising problems can, for instance, still be solved,
  see Figure~\ref{fig:second-order-reg} (d), leading to
  ``speckle'' artefacts in the solutions.
  Another possibility would be to add more regularising
  functionals, which is discussed
  in the next section.
\end{remark}

\paragraph{Higher order TV for multichannel images.} 
In analogy to TV, also higher order TV can be extended to colour and multichannel images represented by functions mapping into a vector space, say $\R^m$. 
This is achieved by testing with $\Sym^k(\RR^d)^m$-valued tensor fields, where
\[ \Sym^k(\RR^d)^m = \set{ \xi = (\xi_1,\ldots,\xi_m) }{ \xi_i \in \Sym^k(\RR^d), \ i=1,\ldots,m} \]
and requires to choose a norm for this space. While, also in view of the Frobenius norm used in $\Sym^k(\RR^d)$, the most natural choice seems to pick the norm that is induced by the inner product
\[ \xi \cdot \eta = \sum_{i=1}^m \xi_i \cdot \eta_i \quad \text{for }\xi,\eta \in \Sym^k(\RR^d)^m,\]
as with $\TV$, this is not the only possible choice and different norms imply different types of coupling of the multiple channels. Generally, we can take $\abs[\circ]{\placeholder}$ to be any norm on $\Sym^{k}(\RR^d)^m$, set $\abs[*]{\placeholder}$ to be the corresponding dual norm and extend $\TV^k$ to functions $u \in \LPlocspace{1}{\Omega,\R^m}$ as
  \begin{equation}
    \label{eq:higher_order_tv_color_def}
    \fl
    \TV^k(u) = \sup\ \Bigset{\int_\Omega u \inprod \divergence^k \varphi \dd{x}
    }{ \varphi \in \Ccspace{k}{\Omega, \Sym^k(\RR^d)^ m}, \ 
    \norm[\infty,*]{\varphi} \leq 1
    }
  \end{equation}
where $\norm[\infty,*]{\varphi}$ is the pointwise supremum of the scalar function $x \mapsto \abs[*]{\varphi(x)}$. By equivalence of norms in finite dimensions, the functional-analytic properties of $\TV^k$ and the results on regularisation for inverse problems transfer one-to-one to its multichannel extension. Further, $\TV^k$ is invariant under rotations whenever the tensor norm $\abs[*]{\placeholder}$ is unitarily invariant in the sense that for any orthonormal matrix $O \in \RR^{d \times d}$ and $(\xi_1,\ldots,\xi_m) \in \Sym^k(\RR^d)^m$ it holds that 
\[\abs[*]{(\xi_1O ,\ldots,\xi_m O)} = \abs[*]{(\xi_1 ,\ldots,\xi_m )},\] where we define $(\xi_i O) (a_1,\ldots,a_k) = \xi_i(Oa_1,\ldots,Oa_k)$ for $i=1,\ldots,m$.

\paragraph{Fractional-order TV.} 
Recently, ideas from fractional calculus started to be transferred to
construct new classes of higher-order TV, namely fractional-order
total variation. The latter bases on fractional partial
differentiation with respect to the coordinate axes. The partial
fractional derivative of a non-integral order $\alpha > 0$ of a
function $u$ compactly supported on the interval ${]{a,b}[} \to \RR$
can, for instance, be defined as
\[
  \frac{\partial^\alpha_{[a,b]} u}{\partial x^\alpha} (x) =
  \frac12 \Bigl(\frac{\partial^\alpha_{[a,x]} u}{\partial x^\alpha} + (-1)^k
  \frac{\partial^\alpha_{[x,b]} u}{\partial x^\alpha} \Bigr),
\]
where $k \in \NN$ is such that $k-1 < \alpha < k$ and, denoting by
$\Gamma$ the Gamma-function, i.e., $\Gamma(t) = \int_0^\infty s^{t-1} \expE^{-t}
\dd{s}$,
\[
  \frac{\partial^\alpha_{[a,x]} u}{\partial x^\alpha} = \frac1{\Gamma(k - \alpha)}
  \frac{\partial^k}{\partial x^k} \int_a^x \frac{u(t)}{(x -
    t)^{\alpha-k+1}} \dd{t}
\]
as well as
\[
  \frac{\partial^\alpha_{[x,b]} u}{\partial x^\alpha} =
  \frac{(-1)^k}{\Gamma(k - \alpha)} \frac{\partial^k}{\partial x^k}
  \int_x^b \frac{u(t)}{(t - x)^{\alpha - k + 1}} \dd{t}.
\]
This fractional-order derivative corresponds to a central version of
the Riemann--Liouville definition
\cite{zhang2015fractionaltv,oldham1974fractional}. However, one has to
mention that there are also other possibilities to define
fractional-order derivatives \cite{podlubny1998fractional}. On a
rectangular domain
$\Omega = {]{a_1,b_1}[} \times \ldots \times {]{a_d,b_d}[} \subset
\RR^d$ and for test vector fields $\varphi \in \Ccspace{k}{\Omega,\RR^d}$,
the fractional divergence of order $\alpha$ can then be defined as
$\divergence^\alpha \varphi = \sum_{i=1}^d \frac{\partial^\alpha_{[a_i,b_i]}
  \varphi_i}{\partial x_i^\alpha}$ which is still a bounded
function. Consequently, the fractional total variation of order
$\alpha$ for $u \in \LPspace{1}{\Omega}$ is given as
\[
  \TV^\alpha(u) = \sup \ \Bigset{\int_\Omega u \inprod
    \divergence^\alpha \varphi \dd{x}}{\varphi \in
    \Ccspace{k}{\Omega,\RR^d}, \ \norm[\infty]{\varphi} \leq 1}.
\]
It is easy to see that this defines a proper, convex and lower
semi-continuous functional on each $\LPspace{p}{\Omega}$ which makes
the functional suitable as a regulariser for denoising
\cite{zhang2015fractionaltv,williams2016fractionaltv}, typically for
$1 < \alpha < 2$. The use of $\TV^\alpha$ for the regularisation of
linear inverse problems, however, seems to be unexplored so far, and
not many properties of the solutions appear to be known.

\section{Combined approaches}
\label{sec:combined_approaches}

We have seen that employing higher-order total variation for regularisation yields well-posedness results for general linear inverse problems that are comparable to first-order $\TV$ regularisation, where the use of higher-order differentiation even weakens the continuity requirements on the forward operator. On the other hand, $\TV^k$ regularisation, for $k>1$, does not allow to recover jump discontinuities, as we have shown analytically and observed numerically (see Remark~\ref{rem:tv2-denoising}). An interesting question in this context is how combinations of TV functionals with different orders behave with respect to these properties. As we will see, this crucially depends on how such functionals are combined.

\subsection{Additive multi-order regularisation}

In this section, we consider the additive combination of total variation functionals with different orders. That is, we are interested in the following Tikhonov approach:
  \begin{equation}
    \label{eq:tvk-multiorder-tihk}
    \min_{u \in \LPspace{p}{\Omega}} S_f(Ku)
    +  \alpha_{1} \TV^{k_1}(u) +  \alpha_{2} \TV^{k_2}(u)
  \end{equation}
  with $\alpha_{i} > 0 $ for $i=1,2$ and $1 \leq k_1 < k_2$. With $k_1=1, k_2=2$, such an approach has for instance been considered in \cite{papafitsoros2014firstandsecondorder_mh} for the regularisation of linear inverse problems.
  
  The following proposition summarises, in the general setting of seminorms, 
  basic properties of the function spaces arising
  from the additive combination of two different regularisers. Its proof is
  straightforward.

\begin{proposition}
  \label{prop:abstract-sum-space}
  Let $\abs[1]{\placeholder}$ and $\abs[2]{\placeholder}$
  be two lower semi-continuous seminorms on the Banach space $X$.
  Then,
  \begin{enumerate}
  \item 
    The functional
    $\abs{\placeholder} = 
    \abs[1]{\placeholder} + \abs[2]{\placeholder}$ is 
    a seminorm on $X$.
  \item
    We have
    \[
      \kernel{\abs{\placeholder}} = 
    \kernel{\abs[1]{\placeholder}} \cap \kernel{\abs[2]{\placeholder}}.
    \]
  \item
    The seminorm $\abs{\placeholder}$ is lower semi-continuous
    and
    \[
    Y = \set{x \in X}{\abs{\placeholder} < \infty}, 
    \qquad \norm[Y]{x} = \norm[X]{x} + \abs{x}
    \]
    constitutes a Banach space.
  \item
  	With $Y_i$ the Banach spaces arising from the norms
    $\norm[X]{\placeholder} + \abs[i]{\placeholder}$, $i=1,2$ (see 
    Lemma~\ref{lem:banach_space_plus_lsc_seminorm}), 
    \[ Y \hookrightarrow Y_i \quad \text{for } i=1,2. \]
  \end{enumerate}
\end{proposition}
Setting $\abs[i]{\placeholder} = \alpha_i \TV^{k_i}$ for $i=1,2$ shows in particular that the function space associated with the additive combination of the $\TV^{k_i}$ is embedded in $\BV^{k_2}(\Omega)$, i.e., the BV space corresponding to the highest order. Hence non-trivial combinations of different $\TV^{k_i}$ again do not allow to recover jumps and, as the following proposition shows, in fact even yield the same space as the single $\TV$ term with the highest order.

\begin{theorem} \label{thm:tvk_sum_poincare}
  Let $1 \leq k_1 < k_2$, $\alpha_1 > 0$, $\alpha_2 > 0$ and $\Omega$ be a bounded Lipschitz domain.
  For $X = \LPspace{1}{\Omega}$ and the seminorm $\abs{\placeholder} = \alpha_1 \TV^{k_1} + \alpha_2 \TV^{k_2}$, let
  $Y$ be the associated Banach space 
  according to Lemma \ref{lem:banach_space_plus_lsc_seminorm}.  Then,
  \[ Y = \BV^{k_2}(\Omega) \]
  in the sense of Banach space equivalence, and 
  for $p \in [1,\infty]$, $p \leq d/(d-k_2)$ if $k_2 < d$,
  $R: \LPspace{p}{\Omega} \to \kernel{\TV^{k_1}}$ a continuous, linear projection, there is a $C > 0$ independent of $u$ such that
  \[
  \norm[p]{u - Ru} \leq C\min\{ \alpha_1,\alpha_2\}^{-1} (\alpha_1 \TV^{k_1} + \alpha_2 \TV^{k_2})(u) 
  \]
  for all $u \in \LPspace{p}{\Omega}$.
  \begin{proof}
  For the claimed norm equivalence, one estimate is immediate, while the other one follows from Theorem \ref{thm:l1-tvx-equiv-norm}.
  Denoting by $R_2: \LPspace{p}{\Omega} \to \kernel{\TV^{k_2}}$ a continuous, linear projection,
  the estimate on $\norm[p]{u - Ru}$ follows from Corollary \ref{cor:tv^k-coercive-scalar} and norm
  equivalence in finite-dimensional spaces as
   \begin{eqnarray*}
   \norm[p]{u - Ru}  
   &\leq \norm[p]{u - R_2u} + \norm[p]{Ru -  R_2u} \\
   & \leq C  \left( \TV^ {k_2}(u) + \norm[1]{Ru -  R_2u}\right) \\
   &  \leq   C  \left( \TV^ {k_2}(u) + \norm[1]{ u - Ru} + \norm[1]{ u -  R_2u}\right) \\
   & \leq C  \left( \TV^ {k_1}(u) + \TV^ {k_2}(u) \right),\\
      & \leq C\min\{ \alpha_1,\alpha_2\}^{-1}   \left( \alpha_1 \TV^ {k_1}(u) + \alpha_2 \TV^ {k_2}(u) \right),
   \end{eqnarray*}
   with $C>0$ a generic constant.
  \end{proof}
\end{theorem}

\paragraph*{Tikhonov regularisation.}

For employing $\alpha_1 \TV^{k_1} + \alpha_2 \TV^{k_2}$ as regularisation in a Tikhonov setting, the coercivity estimate in Theorem \ref{thm:tvk_sum_poincare} is crucial since it allows to transfer the well-posedness result of Theorem \ref{thm:general_reg_existence_linear}. Observe in particular that $\ker(\alpha_1 \TV^{k_1} + \alpha_2 \TV^{k_2})$ is finite-dimensional, such that assumption (i) in Theorem~\ref{thm:general_reg_existence_linear} is satisfied.

\begin{proposition} \label{prop:well_posed_sum_of_tvk} With $X=L^p(\Omega)$, $p \in {]{1,\infty}[}$, $\Omega$ a bounded Lipschitz domain, $Y$ a Banach space, $K: X \to Y$ linear and continuous, $S_f: Y \to {[{0,\infty}]}$ proper, convex, lower semi-continuous and coercive, %
  $1 \leq k_1 < k_2$, $\alpha_1 > 0$, $\alpha_2 > 0$ the Tikhonov minimisation problem
  \begin{equation}
    \label{eq:general_sum_of_tvk_min}
    \min_{u \in \LPspace{p}{\Omega}} \ S_f(Ku) +  \alpha_1 \TV^{k_1}(u) + \alpha_2 \TV^{k_2}(u).
  \end{equation}
  is well-posed in the sense of Theorem \ref{thm:general_reg_existence_linear} whenever $p \leq d/(d-k_2)$ if $k_2 < d$. %
\end{proposition}
It is interesting to note that the necessary coercivity estimate on $\alpha_1 \TV^{k_1} + \alpha_2 \TV^{k_2}$ uses a projection to the smaller kernel of $\TV^{k_1}$ and an $L^p$ norm with a larger exponent corresponding to $\TV^{k_2}$. Hence, in view of the assumptions in Theorem \ref{thm:general_reg_existence_linear}, the additive combination of $\TV^{k_1}$ and $\TV^{k_2}$ inherits the best properties of the two summands, i.e., the ones that are the least restrictive for applications in an inverse problems context.

Regarding the convergence result of Theorem \ref{thm:tvk_reg_convergence} and the rates of Proposition \ref{prop:tvk_reg_convergence_rate}, a direct extension to regularisation with $\alpha_1 \TV^{k_1} + \alpha_2 \TV^{k_2}$ can be obtained by regarding the weights $\alpha_1,\alpha_2$ to be fixed and introducing an additional factor $\alpha>0$ for both terms, which then acts as the regularisation parameter. A more natural approach, however, would be to regard both $\alpha_1,\alpha_2$ as regularisation parameters and study the limiting behaviour of the method as as $\alpha_1,\alpha_2$  converge to zero in some sense. 
This is covered by the following theorem.

\begin{theorem} 
  \label{thm:sum_tvk_reg_convergence}
  In the situation of Proposition~\ref{prop:well_posed_sum_of_tvk},
  let for each 
$\delta > 0$ the data $f^\delta$ be given such that
$S_{f^\delta}(f^\dagger) \leq \delta$, let
$\seq{S_{f^\delta}}$ be equi-coercive and converge to
  $S_{f^\dagger}$ for some data $f^\dagger \in Y$ in the sense of \eqref{eq:discrepancy_convergence} with $S_{f^\dagger}(v) = 0$ if and only if $v = f^\dagger$.
  
  Choose the positive parameters $\alpha = (\alpha_1,\alpha_2)$
  in dependence of $\delta$ such that
  \[
  \max\{\alpha_1,\alpha_2\} \to 0, \quad
  \frac{\delta}{\max\{\alpha_1,\alpha_2\}} \to 0, \qquad \text{as} \qquad \delta \to 0,
  \]
  and $(\tilde{\alpha}_1,\tilde{\alpha}_2) = (\alpha_1,\alpha_2)/\max\{\alpha_1,\alpha_2\} \rightarrow (\alpha_1^\dagger,\alpha_2^\dagger)$ as $\delta \to 0$.
 Set 
 \[ %
   k = \left\{
     \begin{array}{rl}
       k_1 & \text{if} \ \alpha_2^\dagger = 0, \\
             k_2 & \text{else},
     \end{array}
   \right.
 \]
 and assume $p \leq d/(d-k)$ in case of $k < d$,
 and
 that there exists $u_0 \in \BV^k(\Omega)$ such that
 $Ku_0 = f^\dagger$.
  Then, up to shifts in $\ker(K) \cap \poly^{k_1-1}$, any sequence $\seq{u^{\alpha,\delta}}$, with each $u^{\alpha,\delta}$ being a solution to \eqref{eq:tvk-multiorder-tihk} for parameters $(\alpha_1,\alpha_2)$ and data $f^\delta$, has at least one $L^p$-weak accumulation
  point. Each $L^p$-weak accumulation point is a 
  minimum-$(\alpha_1^\dagger \TV^{k_1}+\alpha_2^\dagger \TV^{k_2})$-solution of $Ku = f^\dagger$ and $\lim_{\delta \to 0}
  (\tilde{\alpha}_1\TV^ {k_1} + \tilde{\alpha}_2 \TV^{k_2} )(u^{\alpha,\delta}) = ( \alpha_1^\dagger \TV^ {k_1}+\alpha_2^\dagger \TV^ {k_2} ) (u^\dagger)$.

\end{theorem}

\begin{proof} First note that, as consequence of Theorem \ref{thm:general_reg_existence_linear} and the fact that $u_0 \in \BV^k(\Omega)$ with $Ku_0 = f^\dagger$, there exists a minimum-$(\alpha_1^\dagger \TV^ {k_1}+\alpha_2^\dagger \TV^ {k_2})$-solution to $Ku=f^\dagger$, that we denote by $u^\dagger$, such that $\TV^k(u^\dagger) < \infty$.
  Using optimality of $u^{\alpha,\delta}$
  compared to $u^\dagger$ gives
  \[
    \fl
  S_{f^\delta}(Ku^{\alpha,\delta}) + \left( \alpha_1 \TV^{k_1} + \alpha_2 \TV^{k_2} \right)(u^{\alpha,\delta})
  \leq \delta + \left( \alpha_1 \TV^{k_1} + \alpha_2 \TV^{k_2} \right)(u^\dagger).
  \]
  Since $\max\{\alpha_1,\alpha_2\} \to 0$ 
  as $\delta \rightarrow 0 $, we have that
  $S_{f^\delta}(Ku^{\alpha,\delta}) \to 0$ as $\delta \to 0$. 
  Moreover, as also $\delta/\max\{\alpha_1,\alpha_2\} \to 0$, it follows that
  \begin{eqnarray*}
    \fl
    \limsup_{\delta \rightarrow 0} \ \left( \alpha^\dagger_{1} \TV^{k_1} 
    + \alpha^\dagger_2 \TV^{k_2}\right) (u^{\alpha,\delta})
  & \leq \limsup_{\delta \rightarrow 0} \
  \left( \tilde{\alpha}_1 \TV^{k_1} + \tilde{\alpha}_2 \TV^{k_2} \right)(u^{\alpha,\delta})  \\
  & \leq  \left( \alpha_1 ^\dagger \TV^{k_1} + \alpha_2^\dagger \TV^{k_2} \right)(u^\dagger)
  \end{eqnarray*}
  The choice of $k$ allows to conclude that $\seq{\TV^k(u^{\alpha,\delta})}$
  is bounded, which, in case $k=k_1$, means that $\seq{\TV^{k_1}(u^{\alpha,\delta})}$ is bounded.
  Now we show that also in the other case when $k=k_2$, $\seq{\TV^{k_1}(u^{\alpha,\delta})}$ is bounded.
  To this aim, denote by $R:L^p(\Omega) \rightarrow \ker(\TV^{k_2})$ and $P_Z: \ker(\TV^{k_2}) \rightarrow Z$
  linear, continuous projections, where $Z$ is a complement of $\ker(K)\cap \ker(\TV^{k_2})$ in $\ker(\TV^{k_2})$,
  i.e., $\id-P_Z$ projects $\kernel{\TV^{k_2}}$ to $\ker(K)\cap \ker(\TV^{k_2})$. Then, by optimality and invariance of $K$ and $\TV^{k_2}$ on  $\ker(K)\cap \ker(\TV^{k_2})$ we estimate
  \begin{eqnarray*}
    \fl
    S_{f^\delta}(Ku^{\alpha,\delta}) + \left( \alpha_1 \TV^{k_1} + \alpha_2 \TV^{k_2} \right)(u^{\alpha,\delta}) 
    & \leq
    S_{f^\delta}(Ku^{\alpha,\delta})+ \alpha_2 \TV^{k_2} (u^{\alpha,\delta}) \\
    & \quad +  \alpha_1 \TV^{k_1} \left(u^{\alpha,\delta} - (\id - P_Z)R u^{\alpha,\delta}\right),
  \end{eqnarray*}
  which, together with Lemma \ref{lem:tv-k1_tv-k2_coercivity_estimate}, norm equivalence on finite-dimensional spaces and injectivity of $K$ on the finite-dimensional space $Z$, yields
  \[
    \fl      
    \begin{array}{rl}
      \displaystyle
      \TV^{k_1} (u^{\alpha,\delta}) \!\!\!\!\!\!
    & \displaystyle \leq \TV^{k_1}\left(u^{\alpha,\delta} - (\id - P_Z)R u^{\alpha,\delta}\right) 
      \leq \TV^{k_1} \left(u^{\alpha,\delta} - R u^{\alpha,\delta}\right) + \TV^{k_1}\left(P_ZR u^{\alpha,\delta} \right) \\
    & \displaystyle \leq C \left( \TV^{k_2} (u^{\alpha,\delta}) + \| P_ZR u^{\alpha,\delta}\|_p \right)
      \leq C \left( \TV^{k_2} (u^{\alpha,\delta}) + \| K P_ZR u^{\alpha,\delta}\|_Y  \right) \\
    & \displaystyle  \leq C \left( \TV^{k_2} (u^{\alpha,\delta}) +  \| K ( u^{\alpha,\delta} - R u^{\alpha,\delta} + P_ZR u^{\alpha,\delta})\|_Y + \|K\|\|(u^{\alpha,\delta}-R u^{\alpha,\delta})\|_p \right) \\
    & \displaystyle  \leq C \left( \TV^{k_2} (u^{\alpha,\delta}) + \| K u^{\alpha,\delta}\|_Y \right).
    \end{array}
  \]
  Now, the last expression is bounded due to
  boundedness of $\TV^{k_2}(u^{\alpha,\delta})$ and equi-coercivity of $\seq{S_{f^\delta}}$.
  Hence, %
  $\seq{\TV^{k_1}(u^{\alpha,\delta})}$ is always bounded and, again using the equi-coercivity of $\seq{S_{f^\delta}}$ and the techniques
  in the proof of Theorem~\ref{thm:general_reg_existence_linear}, one sees
  that with possible shifts in $\ker(K) \cap \poly^{k_1-1}$, one can achieve that
  $\seq{u^{\alpha,\delta}}$ is bounded
  in $\BV^{k}(\Omega)$. Therefore, by continuous embedding and reflexivity, it admits a weak accumulation point in $L^p(\Omega)$.
  
  Next, let $u^*$ be a $L^p$-weak accumulation point associated with
  $\seq{\delta_n}$, $\delta_n \to 0$ as well as the corresponding
  parameters $\seq{\alpha_n}=\seq{(\alpha_{1,n},\alpha_{2,n})}$. Then,
  $S_{f^\dagger}(Ku^*) \leq \liminf_{n \to \infty}
  S_{f^{\delta_n}}(Ku^{\alpha_n,\delta_n}) = 0$
  by convergence of $S_{f^\delta}$ to $S_{f^\dagger}$, so
  $Ku^* = f^\dagger$.  Moreover,
  \begin{eqnarray*}
    \fl
  \left( \alpha_1 ^\dagger \TV^{k_1} + \alpha_2^\dagger \TV^{k_2} \right)(u^*) 
  & \leq \liminf_{n \to \infty} \ \tilde{\alpha}_{1,n} \TV^{k_1}(u^{\alpha_n,\delta_n})  + \liminf_{n \to \infty} \  \tilde{\alpha}_{2,n} \TV^{k_2}(u^{\alpha_n,\delta_n}) \\
   & \leq \liminf_{n \rightarrow \infty}
  \left(  \tilde{\alpha}_{1,n} \TV^{k_1} +  \tilde{\alpha}_{2,n} \TV^{k_2} \right)(u^{\alpha_n,\delta_n}) \\
  & \leq  \left( \alpha_1 ^\dagger \TV^{k_1} + \alpha_2^\dagger \TV^{k_2} \right)(u^\dagger),
  \end{eqnarray*}
  hence, $u^*$ is a minimum-$(\alpha_1 ^\dagger \TV^{k_1} + \alpha_2^\dagger \TV^{k_2})$-solution.
  In particular,
  \[\left( \alpha_1 ^\dagger \TV^{k_1} + \alpha_2^\dagger \TV^{k_2} \right)(u^*) =  \left( \alpha_1 ^\dagger \TV^{k_1} + \alpha_2^\dagger \TV^{k_2} \right)(u^\dagger), \] so
  \[ \lim_{n \to \infty}
    \left( \tilde{\alpha}_{1,n}  \TV^{k_1} +  \tilde{\alpha}_{2,n} \TV^{k_2} \right)(u^{\alpha_n,\delta_n})=  \left( \alpha_1 ^\dagger \TV^{k_1} + \alpha_2^\dagger \TV^{k_2} \right)(u^\dagger). \]
  Finally, each sequence of $\seq{\delta_n}$, $\delta_n \to 0$
  contains another subsequence (not relabelled) %
  for which
  $( \tilde{\alpha}_{1,n}  \TV^{k_1} +  \tilde{\alpha}_{2,n} \TV^{k_2} )(u^{\alpha_n,\delta_n}) \to ( \alpha_1 ^\dagger \TV^{k_1} + \alpha_2^\dagger \TV^{k_2} )(u^\dagger)$ as $n \to \infty$,
  so we have $( \tilde{\alpha}_1  \TV^{k_1} +  \tilde{\alpha}_2 \TV^{k_2} ) (u^{\alpha,\delta}) \to ( \alpha_1 ^\dagger \TV^{k_1} + \alpha_2^\dagger \TV^{k_2} )(u^\dagger)$ as $\delta \to 0$.
\end{proof}

\begin{remark}
  \mbox{}
  \begin{itemize}
   	 \item Theorem \ref{thm:sum_tvk_reg_convergence} shows that, with the additive combination
   	 higher-order $\TV$ functionals, the maximum of the parameters plays the role of the
   	 regularisation parameter. The regularity assumption on $u_0$ such that $Ku_0 = f^\dagger$
   	 on the other hand depends on whether some parameters converge to zero faster than 
   	 the maximum or not. Assuming for instance that $\alpha_2 /\max\{ \alpha_1,\alpha_2\} \rightarrow 0$
   	 leads to the weaker $\BV^{k_1}$-regularity requirement for $u_0$.
      \item
    Although~\eqref{eq:tvk-multiorder-tihk} incorporates multiple orders,
    a solution is always contained in $\BV^{k_2}(\Omega)$. Since
    $k_2 \geq 2$, this space is always contained in 
    $\Hspace[1]{1}{\Omega}$, so jump discontinuities cannot appear.
    One can observe that for numerical solutions, this is reflected in
    blurry reconstructions of edges while higher-order smoothness
    is usually captured quite well, see Figure~\ref{fig:second-order-reg}
    (c).
  \item Naturally, it is also possible to consider the weighted sum
  	of more than two $\TV$-type functionals for regularisation, i.e., 
    \begin{equation}
      \label{eq:tikh-multi-order-sum}
      \min_{u \in \LPspace{p}{\Omega}} S_f(Ku) +
      \Bigl(\sum_{i=1}^m \alpha_i \TV^{k_i}\Bigr)(u).
    \end{equation}
    with orders $k_1,\ldots,k_m \geq 1$ and weights $\alpha_1,\ldots,\alpha_m>0$. 
    Solutions then exist, for appropriate $p$, in the space
    $\BV^{k}(\Omega)$ for $k = \max\sett{k_1,\ldots,k_m}$.
  \end{itemize} 
  \end{remark}
 \paragraph*{Optimality conditions.} As for $\TV^k$, one can also consider optimality conditions for variational problems with $\alpha_1\TV^{k_1} + \alpha_2\TV^{k_2}$ as regularisation. Again, in the case that $Y$ is a Hilbert space, $q=2$ and $S_f(v) = \frac{1}{2}\norm[Y]{v-f}^2$,
  one can argue according to Proposition~\ref{prop:tv^k-tikh-optim} and obtain that
  $u^*$ is optimal for~\eqref{eq:tvk-multiorder-tihk} if and only if
  \[
  K^*(f - Ku^*) \in \subgrad\bigl( \alpha_1\TV^{k_1} + \alpha_2\TV^{k_2}
  \bigr)\bigl(u^* \bigr)
  \]
or, equivalently, 
  \[
  u^* \in \subgrad\bigl( (\alpha_1\TV^{k_1} + \alpha_2\TV^{k_2})^* 
  \bigr)\bigl(K^*(f - Ku^*) \bigr).
  \]
  A difficulty with a further specification of these statements, however, is that it is not
  immediate that either the subdifferential is additive in this situation or that the dual of the 
  sum of the $\TV^{k_i}$ equals the infimal convolution of the duals (see Definition~\ref{def:infconv} in the next subsection for a definition of the infimal convolution). A possible remedy is to consider the original minimisation problem in the space $\BV^{k_1}(\Omega)$ instead, such that $\TV^{k_1}$ becomes continuous. This, however, yields subgradients in the dual of $\BV^{k_1}(\Omega)$ instead of $L^{p*}(\Omega)$ making the optimality conditions again difficult to interpret.

\paragraph*{A-priori estimates.} In order to obtain a bound on a
solution $u^*$ for a quadratic Hilbert-norm discrepancy, i.e.,
$S_f(v) = \frac12\norm[Y]{v - f}^2$, $Y$ Hilbert space, on can proceed
analogously to Remark~\ref{rem:tv^k-tikh-hilbert-bound}, provided that
$K$ is injective on the space $\poly^{k_1-1}$. We then also arrive
at~\eqref{eq:tv^k-tikh-hilbert-bound}, with $\alpha$ replaced by
$\alpha_1$, $C$ being the coercivity constant for $\TV^{k_1}$ and $c$
the inverse bound for $K$ on $\poly^{k_1-1}$.  Of course, in case $K$
is still injective on the larger space $\poly^{k_2-1}$, the analogous
bound can be obtained with $\alpha_2$ instead of $\alpha_1$ and
respective constants $C, c$. In case of the Kullback--Leibler
discrepancy, i.e., $S_f(v) = \KL(v,f)$, the analogous statements apply
to the estimate~\eqref{eq:tv^k-tikh-kl-bound}.

\paragraph*{Denoising performance.}
Figure~\ref{fig:additive-reg} shows the effect of $\alpha_1\TV +  \alpha_2\TV^2$
regularisation compared to 
pure $\TV$-regularisation.  While staircase artefacts are slightly reduced,
the overall image is more blurry than the one obtained with $\TV$, see Figure~\ref{fig:additive-reg} (b) and (c).
This is expected as additive regularisation inherits the analytical properties of the stronger
regularisation term, hence $\alpha_1\TV +  \alpha_2\TV^2$ is not able to recover jumps. 
The result is not much different when $\norm[\radon]{\laplace\placeholder}$
is used instead of $\TV^2$, see Figure~\ref{fig:additive-reg} (d). Nevertheless,
although not discussed in this paper, the
issue of limited applicability of $\norm[\radon]{\laplace\placeholder}$
for regularisation of general inverse problems, as mentioned in Remark \ref{rem:laplace-tv-reg}, is overcome in an 
additive combination with $\TV$ since the properties of $\TV$ are sufficient to guarantee well-posedness results.

\begin{figure}
  \centering
  \begin{tabular}{c@{\ }c@{\ }c@{\ }c}
    \includegraphics[width=0.22\linewidth]{pics_affine_denoising_noise.png}
    & 
    \includegraphics[width=0.22\linewidth]{pics_affine_denoising_tv1.png}
    &
    \includegraphics[width=0.22\linewidth]{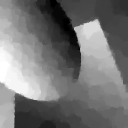}
    & 
    \includegraphics[width=0.22\linewidth]{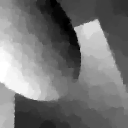}
    \\
    (a) & (b) &
    (c) & (d)
  \end{tabular}
  \caption{Additive multi-order denoising example. (a) Noisy image, (b) 
    regularisation with $\TV$, (c) regularisation with 
    $\alpha_1 \TV+\alpha_2 \TV^2$, 
    (d) regularisation with $\alpha_1 \TV +
    \alpha_2 \norm[\radon]{\laplace \placeholder}$. 
    All parameters are tuned to give highest PSNR
    with respect to
    the ground truth (Figure~\ref{fig:first-order-reg} (a)).}
  \label{fig:additive-reg}
\end{figure}

\subsection{Multi-order infimal convolution}
\label{subsec:tvk-infconv}

In order to overcome the smoothing effect of 
total variation of order two and higher, and additive combinations thereof, another idea
 would be to model an
image $u$ as the sum of a first-order part and a second order part, i.e.,
\[
u =u_1 + u_2
\qquad
\text{with} \qquad
u_1 \in \BV(\Omega), u_2 \in \BV^2(\Omega).
\]
This has originally been proposed in \cite{chambolle1997inftv_mh}, and different variants have subsequently been analysed in \cite{Bergounioux15infcon_analysis} and considered in \cite{Setzer08_mh,setzer2011infimal} in a discrete setting.

Obviously, such a decomposition exists for each $u \in \BV(\Omega)$ but
is, of course, not unique.
The parts $u_1$ and $u_2$ are now regularised 
with first and second-order total 
variation associated with some weights $\alpha_1 > 0$,
$\alpha_2 > 0$. The %
associated %
Tikhonov minimisation problem
reads as
\[
\min_{u_1 \in \BV(\Omega), \atop u_2 \in \BV^2(\Omega)}
\ S_f(K(u_1 + u_2)) + \alpha_1 \TV(u_1) + \alpha_2 \TV^2(u_2).
\]
As we are only interested in $u$, we rewrite this problem as
\begin{equation} \label{eq:infcon_tv_regularization_general}
\min_{u \in \LPspace{p}{\Omega}} \ S_f(Ku)
+ \inf_{u_1 \in \BV(\Omega), \atop
    {u_2 \in \BV^2(\Omega), \atop
    u= u_1 + u_2}} \ \alpha_1 \TV(u_1) + \alpha_2 \TV^2(u_2).
\end{equation}
This regularisation functional is called \emph{infimal convolution}
of $\alpha_1 \TV$ and $\alpha_2 \TV^2$.
\begin{definition}
  \label{def:infconv}
  Let $F_1, F_2: X \to {]{-\infty,\infty}]}$. Then,
  \[
  (F_1 \infconv F_2)(u) = \inf_{u_1 + u_2 = u} \ F_1(u_1) + F_2(u_2)
  \]
  is the \emph{infimal convolution} of $F_1$ and $F_2$.

  An infimal convolution is called \emph{exact}, if for each 
  $u \in X$ there
  is a pair $u_1,u_2 \in X$ with
  \[
  u_1 + u_2 = u \qquad
  \text{and} \qquad
  F_1(u_1) + F_2(u_2) = (F_1 \infconv F_2)(u).
  \]
\end{definition}
The infimal convolution may or may not be exact and may or may not be lower semi-continuous, even if both $F_1$, $F_2$ are lower semi-continuous.
The next proposition, which should be compared to Proposition \ref{prop:abstract-sum-space} above, provides basic properties and the function spaces associated with infimal convolutions. 
\begin{proposition}
  \label{prop:abstract-inf-conv-space}
  Let $\abs[1]{\placeholder}$ and $\abs[2]{\placeholder}$
  be two lower semi-continuous seminorms on the Banach space $X$.
  Then,
  \begin{enumerate}
  \item 
    The functional
    $\abs{\placeholder} = 
    \abs[1]{\placeholder} \infconv \abs[2]{\placeholder}$ is 
    a seminorm on $X$.
  \item
    We have
    \[
    \kernel{\abs[1]{\placeholder}} + \kernel{\abs[2]{\placeholder}} \subset
    \kernel{\abs{\placeholder}}
    \]
    with equality if $\abs[1]{\placeholder} \infconv
    \abs[2]{\placeholder}$ is exact.
  \item
    If $\abs[1]{\placeholder} \infconv \abs[2]{\placeholder}$ is lower semi-continuous, 
	then
    \[
    Y = \set{u \in X}{\abs{u} < \infty}, 
    \qquad \norm[Y]{u} = \norm[X]{u} + \abs{u}
    \]
    constitutes a Banach space.
  \item
    With $Y_i$ the Banach spaces arising from the norms
    $\norm[X]{\placeholder} + \abs[i]{\placeholder}$, $i=1,2$ (see 
    Lemma~\ref{lem:banach_space_plus_lsc_seminorm}), 
    \[ Y_i \hookrightarrow Y \quad \text{for } i=1,2. \]
  \item It holds that  $(\abs[1]{\placeholder} \infconv
    \abs[2]{\placeholder})^ * = \abs[1]{\placeholder}^ * + 
    \abs[2]{\placeholder}^ * $ and if $\abs[1]{\placeholder} \infconv
    \abs[2]{\placeholder}$ is exact, then
    \[ \subgrad (\abs[1]{\placeholder}^ * + \abs[2]{\placeholder}^ * )= \subgrad \abs[1]{\placeholder}^ * + \subgrad \abs[2]{\placeholder}^ * \]
  \end{enumerate}
  
\end{proposition}
\begin{proof}
  The seminorm axioms 
  can easily be verified for 
  $\abs{\placeholder}$. If $u = u_1 + u_2$ for
  $u_i \in \kernel{\abs[i]{\placeholder}}$, $i=1,2$,
  then 
  \[
  \abs{u} \leq \abs[1]{u_1} + \abs[2]{u_2} = 0.
  \]
  The converse inclusion follows directly from the exactness.
  The third statement is a direct consequence of 
  Lemma~\ref{lem:banach_space_plus_lsc_seminorm} while the forth
  immediately follows from 
  $\norm[X]{u} + \abs{u} \leq \norm[X]{u} + \abs[i]{u}$ for 
  $i=1,2$.

  For the fifth statement, the assertion on the Fenchel dual follows by
  direct computation. Regarding equality of the subdifferentials, let $w \in X^*$ and
  $u \in X$ such that $u \in \subgrad ( \abs[1]{\placeholder}^ * + \abs[2]{\placeholder}^ *)(w)$. 
  Then, the Fenchel
  identity yields
  \begin{equation}
    \label{eq:inf_conv_fenchel}%
    \fl
    \scp{w}{u} = %
    (\abs[1]{\placeholder} \infconv \abs[2]{\placeholder})(u) +
    \abs[1]{\placeholder}^*(w) + \abs[2]{\placeholder}^*(w) +
    =
    \abs[1]{u_1} + \abs[1]{\placeholder}^*(w)
    + \abs[1]{u_2} + \abs[2]{\placeholder}^*(w)
  \end{equation}
  for the minimising $u_1,u_2 \in X$ with $u_1 + u_2 = u$.
  As by %
  the Fenchel inequality
  \begin{equation}
    \label{eq:inf_conv_fenchel2}%
    \fl
    \scp{w}{u_1} \leq %
    \abs[1]{u_1} + \abs[1]{\placeholder}^*(w)
    \qquad
    \text{and} \qquad
    \scp{w}{u_2} \leq %
    \abs[2]{u_2} + \abs[2]{\placeholder}^*(w),
  \end{equation}
  the equation~\eqref{eq:inf_conv_fenchel} can only be true when
  there is equality in~\eqref{eq:inf_conv_fenchel2}.
  But this means, in turn, that
  $u_1 \in \subgrad \abs[1]{\placeholder}^*(w)$ and $u_2 \in \subgrad
  \abs[2]{\placeholder}^*(w)$.
  Hence, 
  $\subgrad(\abs[1]{\placeholder}^* + \abs[2]{\placeholder}^*)
  \subset \subgrad \abs[1]{\placeholder}^* + \subgrad \abs[2]{\placeholder}^*$.
  The other inclusion holds trivially.
\end{proof}
The statement (v) will be relevant for obtaining optimality conditions and we note that, as can be seen from the proof, it holds true for arbitrary convex functionals, not necessarily seminorms.

The previous proposition shows in particular that lower semi-continuity and exactness of the infimal convolution are important for obtaining an appropriate function space setting. Regarding the infimal convolution of TV functionals, this holds true on $L^p$-spaces as follows.

\begin{proposition}
  Let $\Omega$ be a bounded Lipschitz domain, $1 \leq k_1 < k_2$ and $p \in {[{1,\infty}]}$ with
  $p \leq d/(d-k_1)$ if $k_1 < d$. Then, for $\alpha = (\alpha_1,\alpha_2)$, $\alpha_1 > 0$,
  $\alpha_2 > 0$, the infimal convolution
  \begin{equation}
    \label{eq:tv-inf-conv-reg}
    \mR_\alpha = \alpha_1 \TV^{k_1} \infconv \alpha_2 \TV^{k_2},
  \end{equation}
  is exact and lower semi-continuous in $\LPspace{p}{\Omega}$.
\end{proposition}
\begin{proof}
  By continuous embedding, we may assume without loss of generality that $p < \infty$.
  Take a sequence $\seq{u^n}$ converging to some $u$ in $\LPspace{p}{\Omega}$ for which
  $\liminf_{n \to \infty} \ \mR_\alpha(u^n)
  < \infty$. %
For each $n$, we can select $u_1^n, u_2^n \in \BV^{k_1}(\Omega)$ such that $u^n = u_1^n + u_2^n$,
\[ \alpha_1 \TV^{k_1}(u_1^n) + \alpha_2 \TV^{k_2} (u_2^n) \leq \left( \alpha_1 \TV^{k_1} \infconv \alpha_2 \TV^{k_2} \right) (u^n)  + \frac1n, \]
and $u_1^n$ is in the complement of $\ker(\TV^{k_1})$ in the sense that $Ru_1^n = 0$
for $R: \LPspace{p}{\Omega} \to \ker(\TV^{k_1})$ a linear and continuous projection.
The latter condition can always be satisfied since
both $\TV^{k_1}$ and $\TV^{k_2}$ are invariant on $\ker(\TV^{k_1})$. %
Now, by coercivity of $\TV^{k_1}$ as in Corollary \ref{cor:tv^k-coercive-scalar}, we get that $\seq{u^n_1}$ is bounded in $\BV^{k_1}(\Omega)$. Hence, by the embedding of $\BV^{k_1}(\Omega)$ in either $\LPspace{\infty}{\Omega}$ or $\LPspace{d/(d-k_1)}{\Omega}$ in case of $k_1 < d$ as in Theorem \ref{thm:tvk-sobolev-embedding}, the choice of $p$ and convergence of $\seq{u^n}$ in $\LPspace{p}{\Omega}$, we can extract (non-relabelled) subsequences of $\seq{u_1^n}$ and $\seq{u_2^n}$ converging weakly to some $u_1$ and $u_2$ in $\LPspace{p}{\Omega}$, respectively, such that $u = u_1 + u_2$. Thus, lower semi-continuity of both $\TV^{k_1}$ and $\TV^{k_2}$ implies
\begin{eqnarray*}
\left( \alpha_1 \TV^{k_1} \infconv \alpha_2 \TV^{k_2} \right) (u) 
&\leq \alpha_1 \TV^{k_1}(u_1) + \alpha_2 \TV^{k_2} (u_2) \\
&\leq \liminf_{n \to \infty} \ \left( \alpha_1 \TV^{k_1}(u^n_1) + \alpha_2 \TV^{k_2} (u^n_2) \right) \\
&= \liminf_{n \to \infty} \ \left( \alpha_1 \TV^{k_1} \infconv \alpha_2 \TV^{k_2} \right) (u^n)
\end{eqnarray*}
such that lower semi-continuity holds. Finally, exactness for $u \in \LPspace{p}{\Omega}$ with $\mR_\alpha(u)<\infty $ follows from choosing $\seq{u^n}$
as the sequence that is constant $u$.
\end{proof}
Given this, the special case $\abs[i]{\placeholder} = \alpha_i \TV^{k_i}$ and $X = \LPspace{1}{\Omega}$ of Proposition \ref{prop:abstract-inf-conv-space} shows that both $\BV^{k_1}(\Omega)$ and $\BV^{k_2}(\Omega)$ are embedded in the Banach space $Y$. Hence, in contrast to the sum of different TV terms, their infimal convolution allows to recover jumps whenever $k_1=1$, independent of $k_2$. In fact, as the following proposition shows, the space $Y$ is even equivalent to the BV space corresponding to the lowest order, in particular to $\BV(\Omega)$ for $k_1 = 1$. Again, the result should be compared to Theorem \ref{thm:tvk_sum_poincare} above.
\begin{theorem} \label{thm:tvk_infcon_poincare}
  Let $1 \leq k_1 < k_2$, $\alpha_1 > 0$, $\alpha_2 > 0$,
  $\Omega $ be a bounded Lipschitz domain, and $Y$ be the Banach space associated with $X = \LPspace{1}{\Omega}$ and
  total-variation infimal
  convolution according to~\eqref{eq:tv-inf-conv-reg}. %
  Then,
  \[ Y = \BV^{k_1}(\Omega) \]
  in the sense of Banach space equivalence, and 
  for $p \in [1,\infty]$, $p \leq d/(d-k_1)$ if $k_1 < d$,
  and for $R: \LPspace{p}{\Omega} \to \kernel{\TV^{k_2}}$ a linear, continuous projection there exists a $C > 0$ such that
  \begin{equation}
    \label{eq:tvk-inf-conv-coercivity}
    \norm[p]{u - Ru} \leq C \min \sett{\alpha_1, \alpha_2}^{-1} (\alpha_1 \TV^{k_1} \infconv \alpha_2 \TV^{k_2})(u)
  \end{equation}
  for all $u \in \LPspace{p}{\Omega}$.
\end{theorem}

\begin{proof}
  We first show the claimed norm equivalence. For this purpose, note that one estimate corresponds to 
  the fourth statement in Proposition~\ref{prop:abstract-inf-conv-space}.
  For the converse estimate, let 
  $u \in \BV^{k_1}(\Omega)$ and $R: \LPspace{1}{\Omega} \to 
  \kernel{\TV^{k_2}}$ be a projection. Then,
  \begin{equation}
    \label{eq:inf-conv-norm-eq}%
    \TV^{k_1}(u) \leq C \bigl(\norm[1]{u} + \TV^{k_1}(u - Rw) \bigr)
  \end{equation}
  for $C$ independent of $u,w \in \BV^{k_1}(\Omega)$.
  Indeed, if for $\seq{u^n}$ and $\seq{w^n}$ we have
  $\TV^{k_1}(u^n) = 1$ and $\norm[1]{u^n} \to 0$ as well as
  $\TV^{k_1}(u^n - Rw^n) \to 0$, meaning $u^n \to 0$ in $\LPspace{1}{\Omega}$
  and $\grad^{k_1}(u^n - Rw^n) \to 0$ in 
  $\radonspace{\Omega,\Sym^{k_1}(\RR^d)}$. The latter implies that 
  $\seq{\grad^{k_1}Rw^n}$ is bounded in a finite-dimensional space,
  hence there is a convergent subsequence (not relabelled) with limit
  $v \in \radonspace{\Omega,\Sym^{k_1}(\RR^d)}$. 
  Then, $\grad^{k_1}u^n \to v$ in $\radonspace{\Omega,\Sym^{k_1}(\RR^d)}$
  and the closedness of $\grad^{k_1}$ yields $v = 0$ which is 
  a contradiction to $\TV^{k_1}(u^n) = 1$ for all $n$.
  
  Using this, together with the estimate
  \begin{equation}
    \label{eq:inf-conv-norm-eq2}%
    \TV^{k_1}(w - Rw)
    \leq
    C \TV^{k_2}(w)
  \end{equation}
  from Lemma \ref{lem:tv-k1_tv-k2_coercivity_estimate},
  it holds for
  $u \in \BV^{k_1}(\Omega)$ and $w \in \BV^{k_2}(\Omega)$ that  
  \begin{eqnarray*}
    \TV^{k_1}(u) &\leq C \bigl(
    \norm[1]{u} + \TV^{k_1}(u - Rw) \bigr) \\
    & \leq C \bigl( \norm[1]{u} + \TV^{k_1}(u - w) + \TV^{k_1}(w - Rw)
    \bigr) \\
    &\leq C \bigl(\norm[1]{u} + 
    \TV^{k_1}(u - w) + \TV^{k_2}(w) \bigr).
  \end{eqnarray*}
  Taking the infimum over all $w \in \BV^{k_2}(\Omega)$, adding
  $\norm[1]{u}$ on both sides as well as observing that
  $\TV^{k_1} \infconv \TV^{k_2} \leq \min \sett{\alpha_1,
    \alpha_2}^{-1} (\alpha_1 \TV^{k_1} \infconv \alpha_2 \TV^{k_2})$
  yields
  \[
    \norm[1]{u} + \TV^{k_1}(u) \leq C \bigl(
    \norm[1]{u} + \min \sett{\alpha_1,
      \alpha_2}^{-1} (\alpha_1 \TV^{k_1} \infconv \alpha_2 \TV^{k_2})(u) \bigr),
  \]
  and, consequently, the desired norm estimate.
  Likewise, %
  the estimate
  $\|u-Ru\|_p \leq C (\TV^{k_1} \infconv \TV^{k_2})(u)$ follows in
  analogy to Proposition~\ref{prop:tv^k-coercive} and
  Corollary~\ref{cor:tv^k-coercive-scalar}, which immediately
  gives the claimed
  estimate for arbitrary $\alpha_1 > 0$, $\alpha_2 > 0$.
\end{proof}

\paragraph*{Tikhonov regularisation.}
Again, the second estimate in Theorem \ref{thm:tvk_infcon_poincare} is crucial as it allows to apply the well-posedness result of Theorem \ref{thm:general_reg_existence_linear}. 

\begin{proposition} \label{prop:well_posed_infconv_of_tvk} With $X=L^p(\Omega)$,
  $p \in {]{1,\infty}[}$, $\Omega$ being a bounded Lipschitz domain, $Y$ a Banach space, $K: X \to Y$ linear and continuous, $S_f: Y \to [0,\infty]$ proper, convex lower semi-continuous and coercive,  $1 \leq k_1 < k_2$, $\alpha_1 > 0$, $\alpha_2 > 0$, the Tikhonov minimisation problem
  \begin{equation}
    \label{eq:general_infcon_of_tvk_min}
    \min_{u \in \LPspace{p}{\Omega}} \ S_f(Ku) + \left(\alpha_1 \TV^{k_1} \infconv \alpha_2 \TV^{k_2}\right) (u).
  \end{equation}
  is well-posed in the sense of Theorem \ref{thm:general_reg_existence_linear} whenever $p \leq d/(d-k_1)$ if $k_1 < d$.
\end{proposition}

Compared to the sum of different TV terms, we see that now the necessary coercivity estimate incorporates a projection to the larger kernel of $\TV^{k_2}$ and an $L^p$ norm with a smaller exponent corresponding to $\TV^{k_1}$. Hence, in view of the assumptions of Theorem \ref{thm:general_reg_existence_linear}, the infimal convolution of $\TV^{k_1}$ and $\TV^{k_2}$ inherits the worst properties of the two summands, i.e., the ones that are more restrictive for applications in an inverse problems context.
Nevertheless, such a slightly more restrictive assumption on the continuity of the forward operator is compensated by the fact that the infimal convolution with $k_1 =1$ allows to reconstruct jumps. In addition, each solution $u^*$ of a Tikhonov functional admits an optimal decomposition 
$u^* = u_1^* + u_2^*$ with $u_i^* \in\BV^{k_i}(\Omega)$, $i=1,2$, which follows from the
  exactness of the infimal convolution.

Regarding the convergence result of Theorem \ref{thm:tvk_reg_convergence} and the rates of Proposition \ref{prop:tvk_reg_convergence_rate}, again a direct extension to regularisation with $\alpha_1 \TV^{k_1} \infconv \alpha_2 \TV^{k_2}$ can be obtained by regarding the weights $\alpha_1,\alpha_2$ to be fixed and introducing a additional factor $\alpha>0$ for both terms, which then acts as the regularisation parameter. Considering the limiting behaviour for both weights converging to zero, a counterpart of Theorem  \ref{thm:sum_tvk_reg_convergence} can be obtained as follows. There, we allow also for infinite weights $\alpha_i$, i.e., for $\alpha_i=\infty$, we set $\alpha_i \TV^{k_i}(u) = 0$ if $ u \in {\ker(\TV^{k_i})}$ and $\alpha_i \TV^{k_i}(u) = \infty$ else. We first need a lower semi-continuity result.

\begin{lemma}
  \label{lem:inf_conv_lim_inf}
  Let $\Omega$ be a bounded Lipschitz domain,
  $p \in {[{1,\infty}[}$ with $1 \leq p \leq d/(d-k_1)$ if $k_1 < d$,
  $\seq{(\alpha_{1,n},\alpha_{2,n})}$ be a sequence of positive parameters converging to some $(\alpha^\dagger_1,\alpha_2^ \dagger)\in {]{0,\infty}]}^2 $ and $\seq{u^n}$ be a sequence in $L^p(\Omega)$  weakly converging to $u^* \in L^p(\Omega)$. Then,
  \[
  \left( \alpha_1 ^\dagger \TV^{k_1} \infconv \alpha_2^\dagger \TV^{k_2} \right)(u^*) 
    \leq  \liminf_{n \rightarrow \infty} \ 
  \left(  \alpha_{1,n} \TV^{k_1} \infconv  \alpha_{2,n} \TV^{k_2} \right)(u^n) .
  \]
\end{lemma}
\begin{proof}
 By moving to a subsequence, we can assume that $ \left( \alpha_{1,n} \TV^{k_1} \infconv  \alpha_{2,n} \TV^{k_2} \right)(u^n) $ converges to the limes inferior on the right-hand side of the claimed assertion and that the latter is finite.
  Choose $\seq{u_1^n}$, $\seq{u_2^n}$ sequences such that for each $n$, we have $u_1^n + u_2^n = u^n$, %
  $\left( \alpha_{1,n} \TV^{k_1} \infconv  \alpha_{2,n} \TV^{k_2} \right)(u^n) = \alpha_{1,n} \TV^{k_1} (u_1^n) +  \alpha_{2,n} \TV^{k_2} (u_2^n)$ 
  and $u_1^n$ being
  in a complement of $\ker(\TV^{k_1})$ in the sense that $u_1^n \in \kernel{R}$
  for a linear, continuous projection $R: \LPspace{p}{\Omega} \to \kernel{\TV^{k_1}}$.
  Setting $\hat{\alpha}_i = \inf \ \{ \alpha_{i,n} \}>0$, we obtain
  \[ \hat{\alpha}_1 \TV^{k_1}(u_1^n) + \hat{\alpha}_2 \TV^{k_2}(u_2^n) \leq  \alpha_{1,n} \TV^{k_1}(u_1^n) + \alpha_{2,n} \TV^{k_2}(u_2^n). \]
  In particular, for a constant $C>0$ it holds that
  \[ \norm[p]{u_1^n} \leq C \hat{\alpha}_1 \TV^{k_1} (u_1^n) \leq C (  \alpha_{1,n} \TV^{k_1}(u_1^n) + \alpha_{2,n} \TV^{k_2}(u_2^n) ) ,\]
  which implies that $\seq{u_1^n}$ is bounded in $\BV^{k_1}(\Omega)$. By the embedding of Theorem \ref{thm:tvk-sobolev-embedding} and since $\seq{u^n}$ is convergent, both $\seq{u_1^n}$ and $\seq{u_2^n}$ admit
  subsequences (not relabelled) weakly converging to some $u_1^*$ and $u_2^*$ in
  $L^p(\Omega)$, respectively. 
  Now, in case $\alpha_i^\dagger < \infty$, we can conclude
  \[ \alpha_i^\dagger \TV^{k_i} (u_i^*) \leq \liminf_{n \to \infty} \ \alpha_{i,n} \TV^{k_i} (u_i^n) .\]
  Otherwise, we get by boundedness of $\alpha_{i,n} \TV^{k_i} (u_i^n)$ that $\TV^{k_i} (u_i^n) \rightarrow 0$ and by lower semi-continuity that $\TV^{k_i}(u^*_i) = 0$. Together, this implies
  \begin{eqnarray*}
   \left( \alpha_1 ^\dagger \TV^{k_1} \infconv \alpha_2 ^\dagger \TV^{k_2} \right) (u^*)  \leq 
   \alpha_1 ^\dagger \TV^{k_1} (u_1^*) + \alpha_2 ^\dagger \TV^{k_2}(u_2^*) \\ 
    \leq \liminf_{n \to \infty} \ \left( \alpha_{1,n} \TV^{k_1} (u_1^n) + \alpha_{2,n} \TV^{k_2}(u_2^n) \right) \\
    = \liminf_{n \to \infty} \ \left( \alpha_{1,n}\TV^{k_1} \infconv \alpha_{2,n} \TV^{k_2}\right) (u^n). %
  \end{eqnarray*}
  This implies the desired statement.
\end{proof}
 
\begin{theorem}
  \label{thm:infconv_tvk_reg_convergence}
  In the situation of Proposition~\ref{prop:well_posed_infconv_of_tvk} and
  for $p \in {]{1,\infty}[}$ with $p \leq d/(d - k_1)$ in case of $k_1 < d$, 
  let for each 
$\delta > 0$ the data $f^\delta$ be given such that
$S_{f^\delta}(f^\dagger) \leq \delta$, let %
$\seq{S_{f^\delta}}$ be equi-coercive and converge to
  $S_{f^\dagger}$ for some data $f^\dagger$ in $Y$ in the sense of \eqref{eq:discrepancy_convergence} and $S_{f^\dagger}(v) = 0$ if and only if $v = f^\dagger$.
  
  Choose the parameters $\alpha = (\alpha_1,\alpha_2)$ in dependence of $\delta$ such that
  \[
  \min\{\alpha_1,\alpha_2\} \to 0, \quad
  \frac{\delta}{\min\{\alpha_1,\alpha_2\}} \to 0, \qquad \text{as} \qquad \delta \to 0,
  \]
  and assume that $(\tilde{\alpha}_1,\tilde{\alpha}_2) = (\alpha_1,\alpha_2)/\min\{\alpha_1,\alpha_2\} \rightarrow (\alpha_1^\dagger,\alpha_2^\dagger) \in {]{0,\infty}]}^2$ as $\delta \to 0$.
 Set 
 \[ %
   k = \left\{
     \begin{array}{rl}
       k_1 & \text{if} \ \alpha_1^\dagger < \infty, \\
       k_2 & \text{else},
     \end{array}
   \right.
 \]
  and assume that there exists $u_0 \in \BV^ k(\Omega)$ such that
 $Ku_0 = f^\dagger$.

 Then, up to shifts in $\kernel{K} \cap \poly^{k_2-1}$, any sequence $\seq{u^{\alpha,\delta}}$, with each $u^{\alpha,\delta}$ being a solution to \eqref{eq:general_infcon_of_tvk_min} %
 for parameters $(\alpha_1,\alpha_2)$ and data $f^\delta$, has at least one $L^p$-weak accumulation
  point. Each $L^p$-weak accumulation point is a 
  minimum-$(\alpha_1^\dagger \TV^ {k_1} \infconv \alpha_2^\dagger \TV^ {k_2})$-solution of $Ku = f^\dagger$ and $\lim_{\delta \to 0}
  (\tilde{\alpha}_1\TV^ {k_1} \infconv \tilde{\alpha}_2 \TV^{k_2} )(u^{\alpha,\delta}) = ( \alpha_1^\dagger \TV^ {k_1} \infconv \alpha_2^\dagger \TV^ {k_2} ) (u^\dagger)$.
\end{theorem}

\begin{proof}
  First note that, with $R:L^p(\Omega) \rightarrow \ker(\TV^{k_2})$ a linear, continuous projection, for any $u \in L^p(\Omega)$, we have
\[ \| u - Ru\|_p \leq C( \TV^{k_1}\infconv \TV^{k_2})(u) \leq C(\alpha^\dagger_1 \TV^{k_1}\infconv \alpha^\dagger_2\TV^{k_2})(u), \]
and by the choice of $k$ as well as $u_0 \in \BV^k(\Omega)$, that
$(\alpha_1^\dagger \TV^{k_1} \infconv \alpha_2^\dagger\TV^{k_2})(u_0) < \infty$.
Hence, as a consequence of Theorem \ref{thm:general_reg_existence_linear}, there exists a minimum-$(\alpha_1^\dagger \TV^ {k_1} \infconv \alpha_2^\dagger \TV^ {k_2})$-solution $u^\dagger \in \BV^k(\Omega)$ to $Ku=f^\dagger$.
  Using optimality of $u^{\alpha,\delta}$
  compared to $u^\dagger$ gives
  \[
    \fl
  S_{f^\delta}(Ku^{\alpha,\delta}) + \left( \alpha_1 \TV^{k_1} \infconv \alpha_2 \TV^{k_2} \right)(u^{\alpha,\delta})
  \leq \delta + \left( \alpha_1 \TV^{k_1} \infconv \alpha_2 \TV^{k_2} \right)(u^\dagger).
  \]
  Now since $\left( \alpha_1 \TV^{k_1} \infconv \alpha_2 \TV^{k_2} \right)(u^\dagger) \leq \min_{i=1,2} \{ \alpha_{i} \TV^{k_i}(u^\dagger)\}$ and  $\min\{\alpha_1,\alpha_2\} \to 0$ 
  as $\delta \rightarrow 0 $, we have that
  $S_{f^\delta}(Ku^{\alpha,\delta}) \to 0$ as $\delta \to 0$. 
  Moreover, as also $\delta/\min\{\alpha_1,\alpha_2\} \to 0$, it follows that
  \begin{eqnarray*}
     \limsup_{\delta \rightarrow 0} \
  \left( \tilde{\alpha}_1 \TV^{k_1} \infconv \tilde{\alpha}_2 \TV^{k_2} \right)(u^{\alpha,\delta}) %
  & \leq  \limsup_{\delta \rightarrow 0} \ \left( \tilde{\alpha}_1 \TV^{k_1} \infconv \tilde{\alpha}_2 \TV^{k_2} \right)(u^\dagger)  \\
  & \leq   (1 + \varepsilon) \left(\alpha_1^ \dagger \TV^{k_1} \infconv \alpha_2^ \dagger \TV^{k_2} \right)(u^\dagger)
  \end{eqnarray*}
  for $\varepsilon > 0$ independent of $u^{\alpha,\delta}$ and letting $\epsilon \rightarrow 0$, we obtain
  \[ \limsup_{\delta \rightarrow 0}
  \left( \tilde{\alpha}_1 \TV^{k_1} \infconv \tilde{\alpha}_2 \TV^{k_2} \right)(u^{\alpha,\delta}) \leq \left( \alpha_1^ \dagger \TV^{k_1} \infconv  \alpha_2^ \dagger \TV^{k_2} \right)(u^\dagger) .\]
  In particular, using~\eqref{eq:tvk-inf-conv-coercivity}, we can conclude that $\seq{u^{\alpha,\delta} - Ru^{\alpha,\delta}}$
  is bounded in $\LPspace{p}{\Omega}$. By introducing appropriate
  shifts in $\kernel{K} \cap \poly^{k_2 - 1}$ as done in Theorem~\ref{thm:general_reg_existence_linear} and using the equi-coercivity of $\seq{S_{f^\delta}}$, one can then achieve that $\seq{u^{\alpha,\delta}}$ is bounded in $\LPspace{p}{\Omega}$ such that by reflexivity, it admits a $L^p$-weak accumulation point.

  Next, let $u^*$ be a $L^p$-weak accumulation point associated with
  $\seq{\delta_n}$, $\delta_n \to 0$ as well as the corresponding
  parameters $\seq{\alpha_n}=\seq{(\alpha_{1,n},\alpha_{2,n})}$. Then,
  $S_{f^\dagger}(Ku^*) \leq \liminf_{n \to \infty} \ 
  S_{f^{\delta_n}}(Ku^{\alpha_n,\delta_n}) = 0$
  by convergence of $S_{f^\delta}$ to $S_{f^\dagger}$, so
  $Ku^* = f^\dagger$.  
  Moreover, employing Lemma~\ref{lem:inf_conv_lim_inf}, we get
  \begin{eqnarray*}
  \left( \alpha_1 ^\dagger \TV^{k_1} \infconv \alpha_2^\dagger \TV^{k_2} \right)(u^*) 
   & \leq  \liminf_{n \to \infty} \ 
  \left(  \tilde{\alpha}_{1,n} \TV^{k_1} \infconv  \tilde{\alpha}_{2,n} \TV^{k_2} \right)(u^{\alpha_n,\delta_n}) \\
  & \leq  \left( \alpha_1 ^\dagger \TV^{k_1} \infconv \alpha_2^\dagger \TV^{k_2} \right)(u^\dagger),
  \end{eqnarray*}
  hence, $u^*$ is a minimum-$ \left( \alpha_1 ^\dagger \TV^{k_1} \infconv \alpha_2^\dagger \TV^{k_2} \right)(u^\dagger)$-solution.
  The remaining assertions follow as in the proof of Theorem \ref{thm:sum_tvk_reg_convergence} by replacing the sum with the infimal convolution.  
\end{proof}

\begin{remark}
  \mbox{}
  \begin{itemize}
  \item It is also possible to construct
    infimal convolutions of more than two $\TV$-type functionals
    and, of course, other functionals than $\TV^k$.
  \item
    Introducing orders $k_1,\ldots,k_m \geq 1$ and
    weights $\alpha_1,\ldots,\alpha_m > 0$, one can consider 
    \begin{equation}
      \label{eq:tikh-multi-order-inf-conv}
      \min_{u \in \LPspace{p}{\Omega}} S_f(Ku)+
      \Bigl(\bigtriangleup_{i=1}^m \alpha_i \TV^{k_i}\Bigr)(u).
    \end{equation}
    Solutions then exist, for appropriate $p$, in the space
   $\BV^{k}(\Omega)$ for $k = \min\sett{k_1,\ldots,k_m}$.
  \item
    The latter is in contrast to the multi-order $\TV$ regularisation 
    \eqref{eq:tvk-multiorder-tihk} where the solution space is
    determined by the highest effective order of differentiation.
    Letting $k_i=1$ for some $i$, the solution space is then
    $\BV(\Omega)$ which allows for discontinuities; a desirable 
    property for image restoration.
  \end{itemize}
\end{remark}

\paragraph*{Optimality conditions.}
Again, in the situation that $Y$ is a Hilbert space, $q=2$ and $S_f(v) = \frac{1}{2}\norm[Y]{v-f}^2$,
we obtain some first-order optimality conditions. Noting that the dual of the infimal convolution of two functions is the
sum of the respective duals, and arguing according to Proposition~\ref{prop:tv^k-tikh-optim}, an
$u^*$ is optimal for~\eqref{eq:general_infcon_of_tvk_min} if and only if
\[
  u^* \in \subgrad\bigl( (\alpha_1\TV^{k_1})^* + (\alpha_2\TV^{k_2})^* 
  \bigr)\bigl(K^*(f - Ku^*) \bigr).
\]
By Proposition~\ref{prop:abstract-inf-conv-space},
the subgradients are additive, so in terms of the normal cones
introduced in Proposition~\ref{prop:tv^k-tikh-optim}, the optimality
condition reads as
\begin{equation}
  \label{eq:tv-inf-conv-optim}
  u^* \in \mN_{\TV^{k_1}} \Bigl( \frac{K^*(f - Ku^*)}{\alpha_1} \Bigr)
  + \mN_{\TV^{k_2}} \Bigl( \frac{K^*(f - Ku^*)}{\alpha_2} \Bigr).
\end{equation}

\paragraph*{A-priori estimates.} Also here, in the above Hilbert space
situation, i.e., $S_f(v) = \frac12 \norm[Y]{v - f}^2$ and $Y$ Hilbert
space, an a-priori bound of solutions $u^*$ can be derived thanks to
the coercivity estimate~\eqref{eq:tvk-inf-conv-coercivity}. One indeed
has
$\norm[p]{u - Ru} \leq \frac1{2\min\sett{\alpha_1,\alpha_2}} C
\norm[Y]{f}^2$ with $R$ and $C$ coming
from~\eqref{eq:tvk-inf-conv-coercivity}. Hence, assuming that $K$ is
injective on $\poly^{k_2 - 1}$, which leads to
$c \norm[p]{Ru} \leq \norm[Y]{KRu}$ for all $u$ and some $c > 0$, one
proceeds analogously to Remark~\ref{rem:tv^k-tikh-hilbert-bound} to
obtain the bound~\eqref{eq:tv^k-tikh-hilbert-bound} with $\alpha$
replaced by $\min\sett{\alpha_1,\alpha_2}$. By analogy, for
$S_f(v) = \KL(v,f)$ being the Kullback--Leibler discrepancy, an
a-priori estimate of the type~\eqref{eq:tv^k-tikh-kl-bound} follows.

Moreover, it is possible to control $w^*$ up to $\poly^{k_1 - 1}$
whenever
$(\alpha_1 \TV^{k_1} \infconv \alpha_2 \TV^{k_2})(u^*) = \alpha_1
\TV^{k_1}(u^* - w^*) + \alpha_2 \TV^{k_2}(w^*)$.
Let, in the following, $C_f \geq 0$ be an
a-priori estimate for the optimal functional value, for instance,
$C_f = \frac12 \norm[Y]{f}^2$ in case of
$S_f(v) = \frac12 \norm[Y]{v-f}^2$, and
$C_f = \KL(Ku^0,f) + (\alpha_1 \TV^{k_1} \infconv \alpha_2\TV^{k_2}) (u^0)$ for a $u^0 \in \BV^{k_1}(\Omega)$
with $\KL(Ku^0,f) < \infty$ in case of $S_f(v) = \KL(v,f)$.
Further, denoting by
$\tilde C > 0$ a constant such that
$\TV^{k_1}(u) \leq \tilde C \bigl(\norm[1]{u} +
\min\sett{\alpha_1,\alpha_2}^{-1} (\alpha_1 \TV^{k_1} \infconv
\alpha_2 \TV^{k_2})(u) \bigr)$ for all $u \in \BV^{k_1}(u)$ (which
exists by virtue of the norm equivalence in
Theorem~\ref{thm:tvk_infcon_poincare}), we see that
$\TV^{k_1}(u^*) \leq \tilde C \bigl( \abs{\Omega}^{1/p} \norm[p]{u^*}
+ \min\sett{\alpha_1,\alpha_2}^{-1} C_f \bigr)$, hence
\[
  \fl
  \alpha_1 \TV^{k_1}(w^*) \leq \alpha_1 \TV^{k_1}(u^* - w^*) +
  \alpha_1 \TV^{k_1}(u^*) + \alpha_2 \TV^{k_2}(w^*) \leq \alpha_1
  \TV^{k_1}(u^*) + C_f.
\]
Consequently, we obtain the bound
  \begin{equation}
  \TV^{k_1}(w^*) \leq \tilde C \abs{\Omega}^{1/p}
  \norm[p]{u^*} + \frac{(\tilde C + 1)C_f}{\min\sett{\alpha_1,\alpha_2}},
\label{eq:tvk-inf-conv-hilbert-bound}  
\end{equation}
which gives an a-priori estimate when plugging in the already-obtained
bound on $\norm[p]{u^*}$.  Moreover, this estimate implies a bound on
$\norm[p]{w^* - Rw^*}$ by the Poincaré--Wirtinger inequality. However,
the norm of $w^*$ can not fully be controlled since adding an element
in $\poly^{k_1 - 1} = \ker(\TV^{k_1})$ to $w^*$ would still realise
the infimum in the infimal convolution. Thus, an estimate of the
type~\eqref{eq:tvk-inf-conv-hilbert-bound} is the best one could
except in the considered setting.

\paragraph*{Denoising performance.}
Figure~\ref{fig:inf-conv-reg} shows that it is indeed beneficial 
for denoising
to regularise with $\alpha_1\TV \infconv \alpha_2\TV^2$ compared to
pure $\TV$-regularisation: Higher-order features as well as edges
are recognised by this image model. Nevertheless, staircase artefacts 
are still
present, see Figure~\ref{fig:inf-conv-reg} (c). Essentially,
this does not
change when the second-order component of the infimal convolution is
replaced, for instance by $\norm[\radon]{\laplace\placeholder}$ as
in Remark~\ref{rem:laplace-tv-reg}, see Figure~\ref{fig:inf-conv-reg} (d).
(For the latter penalty functional, 
basically the same problems as the ones mentioned in 
Remark~\ref{rem:laplace-tv-reg}
appear.)

\begin{figure}
  \centering
  \begin{tabular}{c@{\ }c@{\ }c@{\ }c}
    \includegraphics[width=0.22\linewidth]{pics_affine_denoising_noise.png}
    & 
    \includegraphics[width=0.22\linewidth]{pics_affine_denoising_tv1.png}
    &
    \includegraphics[width=0.22\linewidth]{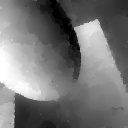}
    & 
    \includegraphics[width=0.22\linewidth]{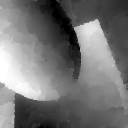}
    \\
    (a) & (b) &
    (c) & (d)
  \end{tabular}
  \caption{Infimal-convolution denoising example. (a) Noisy image, (b) 
    regularisation with $\TV$, (c) regularisation with 
    $\alpha_1 \TV \infconv \alpha_2 \TV^2$, 
    (d) regularisation with $\alpha_1 \TV \infconv 
    \alpha_2 \norm[\radon]{\laplace \placeholder}$. 
    All parameters are tuned to give highest PSNR
    with respect to
    the ground truth (Figure~\ref{fig:first-order-reg} (a)).}
  \label{fig:inf-conv-reg}
\end{figure}

\section{Total generalised variation (TGV)}
\label{sec:tgv}

\subsection{Basic concepts}

As a motivation for TGV, consider the formal predual ball associated
with the infimal convolution $\mR_\alpha = \alpha_1 \TV \infconv \alpha_0 \TV^2$ for
$\alpha = (\alpha_0,\alpha_1)$, $\alpha_0, \alpha_1 > 0$. Then 
\[
\fl
\mR^*_\alpha = (\alpha_1 \TV)^* + (\alpha_0 \TV^2)^*
= \mI_{\closure{\mB}}
\qquad
\text{with}
\qquad
\closure{\mB} = \alpha_1 \closure{\mB_{\TV}} \cap \alpha_0 \closure{\mB_{\TV^2}}
\]
and
\begin{eqnarray*}
  \alpha_1 \mB_{\TV} &= \set{\divergence \varphi_1}{\varphi_1 \in \Ccspace{1}{\Omega,\RR^d}, 
    \ \norm[\infty]{\varphi_1} \leq \alpha_1}, \\
  \alpha_0 \mB_{\TV^2} &= \set{\divergence^2 \varphi_2}{\varphi_2 
    \in \Ccspace{2}{\Omega,\Sym^2(\RR^d)}, 
    \ \norm[\infty]{\varphi_2} \leq \alpha_0}.  
\end{eqnarray*}
Neglecting the closure for a moment, this leads to the predual ball
according to
\begin{equation} \label{eq:infcon_predual_ball}
\fl
\mB = \set{\varphi}{\varphi = \divergence \varphi_1
  = \divergence^2 \varphi_2, \ \varphi_i \in \Ccspace{i}{\Omega,\Sym^i(\RR^d)},
  \ \norm[\infty]{\varphi_i} \leq \alpha_{2-i}, \ i=1,2}.
\end{equation}
Each $\varphi \in \mB$ possesses a representation as an $\infty$-bounded
first and second-order divergence of some $\varphi_1$ and $\varphi_2$. 
However, as the kernel of the
divergence is non-trivial (and even infinite-dimensional for $d \geq 2$),
we can only conclude that $\varphi_1 = \divergence \varphi_2 + \eta$ for some 
$\eta$ with $\divergence \eta = 0$. Enforcing $\eta = 0$ thus gives the
set
\[
\mB_{\TGV^2_\alpha} = \set{\divergence^2 \varphi}{\varphi \in 
  \Ccspace{2}{\Omega,\Sym^2(\RR^d)}, \ \norm[\infty]{\divergence^m \varphi} 
  \leq \alpha_m, \ m=0,1}
\]
which leads, interpreted as a predual ball, to a seminorm which also
incorporates first- and second-order derivatives but is different
from infimal convolution: the total generalised variation 
\cite{bredies2010tgv}.

There is also a primal version of this 
motivation via the $(\TV$-$\TV^2)$-infimal convolution which reads as follows:
Writing
 \[ (\alpha_1 \TV \infconv \alpha_0 \TV^2) (u) = \inf_{v \in \BV^2(\Omega)} \alpha_1 \|\nabla u - \nabla v \| _\M + \alpha_0 \| \nabla^2 v \|_\M \]
 we see that the infimal convolution allows to subtract the a vector field $w = \nabla v$ from the derivative of $u$ at the cost of 
 penalising its derivative $\nabla w= \symgrad w$, where the equality is due to symmetry of the weak Hessian $\nabla ^2v$. While, by the  embedding $\BD(\Omega,\RR^d) \embed L^{d/(d-1)}(\Omega,\RR^d)$, necessarily 
 $w \in \BD(\Omega,\RR^d)$, it is not arbitrary among such functions but still restricted to be the gradient of $v \in \BV^2(\Omega)$.
 Omitting this additional constraint (in the predual version above, this corresponds to enforcing $\eta=0$), we arrive at
   \begin{equation}
   \mR_\alpha(u) = \inf_{w \in \BD(\Omega,\RR^d)} \alpha_1 \|\nabla u - w \| _\M + \alpha_0 \| \symgrad w \|_\M,\label{eq:tgv2_primal}
 \end{equation}
 which is, as will be shown in this section, is an equivalent formulation of the TGV functional.

 \begin{definition}
   \label{def:tgv}
  Let $\Omega \subset \RR^d$ be a domain, $k \geq 1$ and
  $\alpha_0, \ldots,
  \alpha_{k-1} > 0$.
  Then, the \emph{total generalised variation} of order $k$
  with weight $\alpha$ for
  $u \in \LPlocspace{1}{\Omega}$ is defined as the value of the
   functional
  \begin{equation}
    \label{eq:tgv_def}
    \fl
    \TGV^k_\alpha(u) = \sup\ \Bigl\{\int_\Omega u \divergence^k \varphi \dd{x}
    \ \Bigl| \ \varphi \in \Ccspace{k}{\Omega, \Sym^k(\RR^d)}, \ 
    \underbrace{\norm[\infty]{\divergence^m \varphi} \leq 
      \alpha_m}_{m = 0,\ldots, k-1}
    \Bigr\}
  \end{equation}
  which takes the value $\infty$ in case the respective set is unbounded from
  above.

  For symmetric tensors $u \in \LPlocspace{1}{\Omega,\Sym^l(\RR^d)}$
  of order $l \geq 0$, the total generalised variation is given by
  \begin{equation}
    \label{eq:tensor-tgv-def}
    \fl
    \TGV^{k,l}_\alpha(u) = \sup\ \Bigl\{\int_\Omega u \inprod
    \divergence^k \varphi \dd{x}
    \ \Bigl| \ \varphi \in \Ccspace{k}{\Omega, \Sym^{k+l}(\RR^d)}, \ 
    \underbrace{\norm[\infty]{\divergence^m \varphi} \leq 
      \alpha_m}_{m = 0,\ldots, k-1}
    \Bigr\}.
  \end{equation}
  The space
  \begin{eqnarray*}
    \BGV^{k}_\alpha(\Omega, \Sym^l(\RR^d)) = 
    \set{u \in \LPspace{1}{\Omega,\Sym^l(\RR^d)}}{\TGV_\alpha^{k,l}(u) < \infty}, 
    \\
    \norm[\BGV_\alpha^{k,l}]{u} = \norm[1]{u} + \TGV_\alpha^{k,l}(u)
  \end{eqnarray*}
  is called the space of \emph{symmetric tensor fields
    of bounded generalised variation}
  of order $k$ with weight $\alpha$.
  The special case $l=0$ is denoted by
  $\BGV^k_\alpha(\Omega)$.
\end{definition}

\begin{remark}
  \label{rem:tgv_eq_tv}
  For $k=1$ and $\alpha > 0$, the definition coincides, up to a factor,
  with the total deformation of symmetric tensor
  fields of order $l$, i.e., $\TGV_\alpha^{1,l} = \alpha \TD$, in particular
  $\TGV_\alpha^{1,0} = \TGV_\alpha^1 = \alpha \TV$.
  Hence, we can identify the spaces $\BGV_\alpha^{1}(\Omega,\Sym^l(\RR^d)) 
  = \BD(\Omega,\Sym^l(\RR^d))$. In particular,
  $\BGV^{1}_\alpha(\Omega) = \BV(\Omega)$.
\end{remark}

In the following, we will derive some basic properties of the total
generalised variation.

\begin{proposition}
  \label{prop:tgv_basic_prop}
  The following basic properties hold:
  \begin{enumerate}
  \item 
    $\TGV_\alpha^{k,l}$ is a lower semi-continuous seminorm on
    $\LPspace{p}{\Omega,\Sym^l(\RR^d)}$ for each $p \in [1,\infty]$.
  \item
    The kernel satisfies
    $\kernel{\TGV_\alpha^{k,l}} = \kernel{\TD^k}$ 
    for the $k$-th order
    total deformation for symmetric tensor fields of order $l$.
    In particular, $\kernel{\TGV_\alpha^{k,l}}$ is a finite-dimensional
    subspace of polynomials of order less than $k+l$. For $l=0$,
    we have $\kernel{\TGV^k_\alpha} = \poly^{k-1}$.
  \item
    $\BGV^k_\alpha(\Omega,\Sym^l(\RR^d))$ is a Banach space independent
    of $\alpha$.
  \end{enumerate}
\end{proposition}

\begin{proof}
  Observe that $\TGV_\alpha^{k,l}$ is the seminorm associated with the
  predual ball
  \begin{equation}
    \label{eq:tgv_predual_unit_ball}
    \fl
    \mB_{\TGV_\alpha^{k,l}}
    = \{\divergence^k \varphi \,\bigl|\, \varphi \in \Ccspace{k}{\Omega,
      \Sym^{k+l}(\RR^d)}, \ \norm[\infty]{\divergence^m \varphi} \leq \alpha_m,
    \ m = 0,\ldots, k-1\}.
  \end{equation}
  By definition, each element of $\mB_{\TGV_\alpha^{k,l}}$ can be associated with an
  element of the dual space of $\LPspace{p}{\Omega,\Sym^l(\RR^d)}$,
  so $\TGV_\alpha^{k,l}$ is convex and lower semi-continuous as pointwise
  supremum over a set of linear and continuous functionals.
  The positive homogeneity finally follows from 
  $\lambda \mB_{\TGV_\alpha^{k,l}} \subset \mB_{\TGV_\alpha^{k,l}}$ 
  for each $\abs{\lambda} \leq 1$.
  
  The statement about the kernel of $\TGV_\alpha^{k,l}$ 
  is a consequence of $\TGV^{k,l}_\alpha(u) = 0$
  if and only if $\scp{u}{\divergence^k \varphi} = 0$ for each
  $\varphi \in \Ccspace{k}{\Omega,\Sym^{k+l}(\RR^d)}$ (compare with 
  Proposition~\ref{prop:td^k-kernel}). 

  Finally, $\BGV^k_\alpha(\Omega,\Sym^l(\RR^d))$ is a Banach space by 
  Lemma~\ref{lem:banach_space_plus_lsc_seminorm}. The equivalence for
  parameter sets $\alpha_0,\ldots,\alpha_{k-1} > 0$ and
  $\tilde\alpha_0, \ldots, \tilde\alpha_{k-1} > 0$ can be seen as follows.
  Choosing $C > 0$ large enough, we can achieve that
  \[
  \mB_{\TGV_\alpha^{k,l}} \subset \mB_{\TGV_{C\tilde\alpha}^{k,l}}
  = C \mB_{\TGV_{\tilde\alpha}^{k,l}}.
  \]
  This implies
  \[
  \TGV_\alpha^{k,l} \leq C \TGV_{\tilde\alpha}^{k,l}.
  \]
  Interchanging roles we get $\TGV_{\tilde\alpha}^{k,l} \leq C 
  \TGV_{\alpha}^{k,l}$, so the spaces $\BGV_\alpha^k(\Omega,\Sym^l(\RR^d))$
  and $\BGV_{\tilde\alpha}^k(\Omega,\Sym^l(\RR^d))$ have equivalent
  norms.
\end{proof}

\begin{remark}
  As $\BGV_\alpha^k(\Omega,\Sym^l(\RR^d))$ are all equivalent
  for different $\alpha$, we will drop, in the following, the subscript
  $\alpha$.
\end{remark}

\begin{proposition}
  The scalar total generalised variation, i.e., $\TGV_\alpha^k$ possesses
  the following invariance and scaling properties:
  \begin{enumerate}
  \item
    $\TGV_\alpha^k$ is translation invariant,
    i.e.~for $x_0 \in \RR^d$ and $u \in \BGV^k(\Omega)$
    we have that $\tilde u$ given by $\tilde u(x) = u(x+x_0)$
    is in $\BGV^k(\Omega - x_0)$ and 
    $\TGV_\alpha^k(\tilde u) = \TGV_\alpha^k(u)$,
  \item $\TGV_\alpha^k$ is rotationally invariant,
    i.e.~for each orthonormal
    matrix $O \in \RR^{d \times d}$ and $u \in \BGV^k(\Omega)$
    we have, defining $\tilde u(x) = u(Ox)$, that
    $\tilde u \in \BGV^k(O^\transp\Omega)$ with
    $\TGV_\alpha^k(\tilde u)
    = \TGV_\alpha^k(u)$,
  \item
    for $r > 0$ and $u \in \BGV^k(\Omega)$, we have, defining
    $\tilde u(x) = u(rx)$, that $\tilde u \in \BGV^k(r^{-1}\Omega)$
    with \[
      \TGV_{\alpha}^k(\tilde u) = r^{-d} \TGV_{\tilde \alpha}^k(u),
      \qquad \tilde \alpha_m = \alpha_m r^{k-m} \quad \text{for} \ m=0,\ldots,k-1.
    \]
  \end{enumerate}
\end{proposition}

\begin{proof}
  See \cite{bredies2010tgv}.
\end{proof}

\paragraph*{The derivative versus the symmetrised derivative.}
In both ways to motivate the second-order TGV functional as presented at the beginning of this section, we see that symmetric tensor fields and a symmetrised derivative appear naturally in the penalisation of higher-order derivatives. Indeed, in the motivation via the predual ball $\mB$ of the infimal convolution, symmetric tensor fields (resulting in a symmetrised derivative in the primal version) appear as the most economic way to write the predual ball, since for $\varphi \in  \Ccspace{2}{\Omega,\tensorspace[2]{\RR^d}}$, $\divergence^2 \varphi = \divergence ^2 (\interleave \varphi)$. In the primal version, the symmetrised derivative results from writing $\nabla^2 v = \symgrad \nabla v$ and then relaxing $\nabla v$ to be an arbitrary vector field $w$. Nevertheless, also %
non-symmetric tensor fields and the equality $\nabla^2 v = \nabla (\nabla v)$ could have been used in these motivations.
For the TGV functional, this would have resulted in a primal version of second-order TGV according to 
\[ \mR_\alpha (u) = \inf_{w \in \BV(\Omega,\RR^d)} \ \ \alpha_1 \|\nabla u - w \| _\M + \alpha_0 \| \nabla w \|_\M, \]
which is genuinely different from the definition in~\eqref{eq:tgv2_primal}.
The following example %
provides some insight 
on the differences between using the derivative and the symmetrised derivative of vector fields of bounded variation in a Radon-norm penalty.
\begin{example} \label{ex:symgrad_example}
On $\Omega = \set{(x_1,x_2) \in \R^2}{x_1^2 + x_2^2 < \frac14}$ define, for given $\nu,n \in \mathcal{S}^1$ (the unit sphere in $\RR^2$),
\[ w(x)= \cases{
w_1 = \nu_1 n + \nu_2 n^\perp  &if  $x \cdot n >0$, \\
w_2 = -(\nu_1 n + \nu_2 n^\perp)  &if $x \cdot n <0 $,
}
\]
where $n^\perp = (n_2,-n_1)$. 
Then, $w \in \BV(\Omega,\RR^ 2)$ and, with $L= \set{ \lambda n^\perp }{\lambda \in {]{-\frac12,\frac12}[}}$,
\[ \nabla w = (w_1 -w_2) \otimes n \,\hausdorff{1} \restricted  L.
\]
A direct computation shows that
\[ \|\nabla w\|_\M =2, \qquad \|\symgrad w\|_\M =2\sqrt{\nu_1^2 + \frac{\nu_2^2}{2}}, \]
thus, the symmetrised derivative depends on the angle of the vector field relative to the jump set, while the derivative does not. In particular, whenever $\nu_2=0$ such that the vector field can be written as the gradient of a function in $\BV^2(\Omega)$, the two notions coincide. See Figure \ref{fig:symgrad_example} for a visualisation of $w$ for different values of $\nu$.

\end{example}

\begin{figure}
\centering
  \begin{tikzpicture}[scale=1]

  \begin{scope}
  \def\st{0.012} %
  \def\pxa{-1} %
  \def\pya{1}
  \def\pxb{1} %
  \def\pyb{-1}
  \def\vx{1} %
  \def\vy{1}

  \begin{scope}
    \clip(-1.42,1.42) -- (1.42,-1.42) -- (1.42,1.42) -- cycle;
    \clip (0,0) circle (1.41421356);
   \shade[bottom color = blue, top color = white,shading angle=-45] (-1.42,-1.42) rectangle (1.42,1.42);    
   \end{scope}

     \draw[->,thick] (0.5*\pxa + 0.5*\pxb + 0.2*\vx,0.5*\pya + 0.5*\pyb + 0.2*\vy) -- (0.5*\pxa + 0.5*\pxb + 0.4*\vx,0.5*\pya + 0.5*\pyb + 0.4*\vy) node[right=0.5]{\scriptsize $w_1$};

   \begin{scope}
    \clip(-1.42,1.42) -- (1.42,-1.42) -- (-1.42,-1.42) -- cycle;
    \clip (0,0) circle (1.41421356);
	\foreach \x in {0,5,...,90}
	{	
   \shade[bottom color = blue, top color = white,shading angle=135] (-1.42,-1.42) rectangle (1.42,1.42);    
	}
   \end{scope}
      \draw[->,thick] (-0.5*\pxa - 0.5*\pxb - 0.2*\vx,-0.5*\pya - 0.5*\pyb - 0.2*\vy) -- (-0.5*\pxa - 0.5*\pxb - 0.4*\vx,-0.5*\pya - 0.5*\pyb - 0.4*\vy) node[left=0.5]{\scriptsize $w_2$};
   
   \draw[thick] (0,0) circle (1.41421356);
  \draw[thick] (-1,1) -- (1,-1);
   \draw (0,-1.8) node{\scriptsize $n=\frac1{\sqrt{2}}(1,1)$, $\nu = (1,0)$};
   \draw (0,-2.2) node{\scriptsize $\|\nabla w\| = 2$, $\|\symgrad w\| = 2$};
  \end{scope}

  \begin{scope}[shift={(4,0)}]
  \def\st{0.03075} %
  \def\pxa{-1} %
  \def\pya{1.42}
  \def\pxb{-1} %
  \def\pyb{-1}
  \def\vx{1} %
  \def\vy{0}

  \begin{scope}
    \clip(-1.42,1.42) -- (1.42,-1.42) -- (1.42,1.42) -- cycle;
    \clip (0,0) circle (1.41421356);
	\foreach \x in {0,2.5,...,90}
	{	
        \shade[bottom color = blue, top color = white,shading angle=-90] (-1.42,-1.42) rectangle (1.42,1.42);    
	}
   \end{scope}
   \draw[->,thick] (0.5*\pxa + 0.5*\pxb + 1.2*\vx,0.5*\pya + 0.5*\pyb + 1.2*\vy) -- (0.5*\pxa + 0.5*\pxb + 1.4*\vx,0.5*\pya + 0.5*\pyb + 1.4*\vy) node[right=0.5]{\scriptsize $w_1$};
   \begin{scope}
    \clip(-1.42,1.42) -- (1.42,-1.42) -- (-1.42,-1.42) -- cycle;
    \clip (0,0) circle (1.41421356);
	\foreach \x in {0,2.5,...,90}
	{	
		\shade[bottom color = blue, top color = white,shading angle=90] (-1.42,-1.42) rectangle (1.42,1.42);    
	}
   \end{scope}
   \draw[->,thick] (-0.5*\pxa - 0.5*\pxb - 1.2*\vx,-0.5*\pya - 0.5*\pyb - 1.2*\vy) -- (-0.5*\pxa - 0.5*\pxb - 1.4*\vx,-0.5*\pya - 0.5*\pyb - 1.4*\vy) node[left=0.5]{\scriptsize $w_2$};

     \draw[thick] (0,0) circle (1.41421356);
  \draw[thick] (-1,1) -- (1,-1);
   \draw (0,-1.8) node{\scriptsize $n=\frac1{\sqrt{2}}(1,1)$, $\nu = \frac1{\sqrt{2}}(1,1)$};
   \draw (0,-2.2) node{\scriptsize $\|\nabla w\| = 2$, $\|\symgrad w\| = \sqrt{3}$};
  \end{scope}

  \begin{scope}[shift={(8,0)}]
  \def\st{0.025} %
  \def\pxa{0} %
  \def\pya{2}
  \def\pxb{-1} %
  \def\pyb{1}
  \def\vx{1} %
  \def\vy{-1}

  \begin{scope}
    \clip(-1.42,1.42) -- (1.42,-1.42) -- (1.42,1.42) -- cycle;
    \clip (0,0) circle (1.41421356);
	\foreach \x in {0,2.5,...,90}
	{	
		\shade[bottom color = blue, top color = white,shading angle=-135] (-1.42,-1.42) rectangle (1.42,1.42);    
	}
   \end{scope}
   \draw[->,thick] (0.5*\pxa + 0.5*\pxb + 0.8*\vx,0.5*\pya + 0.5*\pyb + 0.8*\vy) -- (0.5*\pxa + 0.5*\pxb + 1*\vx,0.5*\pya + 0.5*\pyb + 1*\vy) node[right=0.5]{\scriptsize $w_1$};
   \begin{scope}
    \clip(-1.42,1.42) -- (1.42,-1.42) -- (-1.42,-1.42) -- cycle;
    \clip (0,0) circle (1.41421356);
	\foreach \x in {0,2.5,...,90}
	{	
		\shade[bottom color = blue, top color = white,shading angle=45] (-1.42,-1.42) rectangle (1.42,1.42);    
	}
   \end{scope}
   \draw[->,thick] (-0.5*\pxa - 0.5*\pxb - 0.8*\vx,-0.5*\pya - 0.5*\pyb - 0.8*\vy) -- (-0.5*\pxa - 0.5*\pxb - 1*\vx,-0.5*\pya - 0.5*\pyb - 1*\vy) node[left=0.5]{\scriptsize $w_2$};

     \draw[thick] (0,0) circle (1.41421356);
  \draw[thick] (-1,1) -- (1,-1);
   \draw (0,-1.8) node{\scriptsize $n=\frac1{\sqrt{2}}(1,1)$, $\nu = (0,1)$};
      \draw (0,-2.2) node{\scriptsize $\|\nabla w\| = 2$, $\|\symgrad w\| = \sqrt{2}$};
  \end{scope}
  \end{tikzpicture}
  \caption{\label{fig:symgrad_example}Visualisation of the function $w$ of Example \ref{ex:symgrad_example} and values of $\|\nabla w\|_\M$ and $\|\symgrad w\|_\M$ for different choices of $\nu$. The blue lines show the level lines of a function $v$ such that $w = \nabla v$. } 
\end{figure}
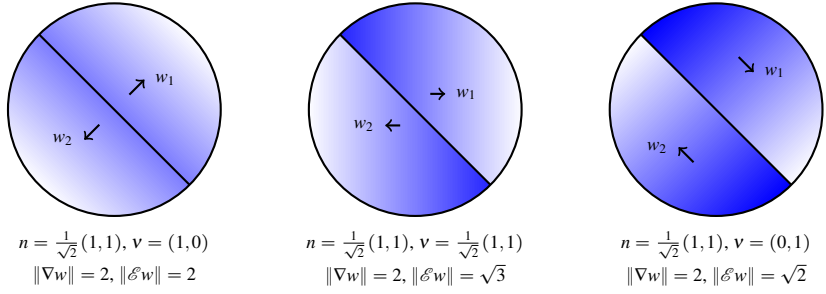
\subsection{Functional analytic and regularisation properties}

We would like to characterise $\TGV_\alpha^{k,l}$ in terms of a minimisation
problem. This characterisation will base on 
\emph{Fenchel-Rockafellar duality}. Here, the following theorem
by \cite{Attouch86duality_mh} is employed. Recall that the \emph{domain} of a function
$F: X \to {]{-\infty,\infty}]}$ is defined as $\dom(F) = \set{x \in X}{F(x) < \infty}$.

\begin{theorem}
  \label{thm:fenchel-rockafellar-duality}
  Let $X,Y$ be Banach spaces and $\Lambda: X \to Y$ linear and continuous.
  Let $F: X \to {]{-\infty,\infty}]}$ and $G: Y \to {]{-\infty,\infty}]}$
  be proper, convex and lower semi-continuous. Assume that
  \begin{equation}
    \label{eq:fr-duality-qual}
    \bigcup_{\lambda \geq 0} \lambda \bigl( \dom(G) - \Lambda\dom(F) \bigr) = Y.
  \end{equation}
  Then,
  \begin{equation}
    \label{eq:fr-duality}
    \inf_{x \in X} \ F(x) + G(\Lambda x) = \max_{y^* \in Y^*} 
    \ -F^*(-\Lambda^* y^*) - G^*(y^*).
  \end{equation}
  In particular, the maximum on the right-hand side is attained.
\end{theorem}

As a preparation for employing this duality result, we note:

\begin{lemma} \label{lem:cascade_element_finite}
  Let $l \geq 0$, $i \geq 1$ and $w_{i-1} \in \Cspace[0]{i-1}{\Omega, \Sym^{l+i-1}(\RR^d)}^*$,
  $w_i \in \Cspace[0]{i}{\Omega, \Sym^{l+i}(\RR^d)}^*$ be distributions
  of order $i-1$ and $i$, respectively.
  Then,
  \begin{equation}
    \label{eq:symgrad_diff_predual}
    \fl
    \norm[\radon]{\symgrad w_{i-1} - w_i} 
    = \sup \ \set{ \scp{w_{i-1}}{\divergence \varphi} + \scp{w_i}{\varphi}}
    {\varphi \in 
      \Ccspace{i}{\Omega, \Sym^{l+i}(\RR^d)}, \ \norm[\infty]{\varphi} \leq 1}
  \end{equation}
  with the right-hand side being finite if and only if $\symgrad w_{i-1} 
  - w_i \in \radonspace{\Omega,\Sym^{l+i}(\RR^d)}$ in the distributional
  sense.
\end{lemma}

\begin{proof}
  Note that in the distributional sense,
  $\scp{w_i - \symgrad w_{i-1}}{\varphi} = \scp{w_{i-1}}{\divergence \varphi} 
  + \scp{w_i}{\varphi}$ 
  for all $\varphi \in \Ccspace{\infty}{\Omega,\Sym^{l+i}(\RR^d)}$.
  Since $\Ccspace{\infty}{\Omega,\Sym^{l+i}(\RR^d)}$ is dense in 
  $\Cspace[0]{}{\Omega,\Sym^{l+i}(\RR^d)}$, the distribution $w_i - 
  \symgrad w_{i-1}$ can be extended to an element in 
  $\Cspace[0]{}{\Omega,\Sym^{l+i}(\RR^d)}^* = 
  \radonspace{\Omega,\Sym^{l+i}(\RR^d)}$ 
  if and only if the supremum in~\eqref{eq:symgrad_diff_predual} is
  finite. In case of finiteness, it coincides with the Radon norm
  by definition.
\end{proof}

This enables us to derive the problem which is dual to the 
maximisation problem in~\eqref{eq:tensor-tgv-def}. We will refer to the resulting problem
as the \emph{minimum representation} of $\TGV_\alpha^{k,l}$.

\begin{theorem} \label{thm:tgv_primal}  
  For $k \geq 1$, $l \geq 0$, $\Omega$ a bounded Lipschitz domain and 
  $\TGV_\alpha^{k,l}$ according to~\eqref{eq:tensor-tgv-def}, 
  we have for
  each $u \in \LPspace{1}{\Omega,\Sym^l(\RR^d)}$:
  \begin{equation} \label{eq:tgv_primal}
    \TGV^{k,l}_\alpha(u) = \min_{w_i \in \BD(\Omega, \Sym^{l+i}(\RR^d)), \atop
      {i=0,\ldots,k, \atop
      w_0 = u, \ w_k = 0}}
    \ \sum_{i=1}^k \alpha_{k-i} \norm[\radon]{\symgrad{w_{i-1}} - w_i} %
  \end{equation}
  with the minimum being finite if and only if
  $u \in \BD(\Omega,\Sym^l(\RR^d))$ and attained for some 
  $w_0,\ldots,w_k$ where $w_i \in \BD(\Omega, \Sym^{l+i}(\RR^d))$ for 
  $i=0,\ldots,k$
  and $w_0 = u$ as well as
  $w_k = 0$ in case of $u \in \BD(\Omega,\Sym^l(\RR^d))$.
\end{theorem}

\begin{proof}
  First, take $u \in  \LPlocspace{1}{\Omega,\Sym^l(\RR^d)}$ such
  that $\TGV_\alpha^{k,l}(u) < \infty$.
  We will employ Fenchel--Rockafellar duality. For this purpose, 
  introduce the Banach spaces
  \begin{eqnarray*}
    X &= \Cspace[0]{1}{\Omega, \Sym^{1+l}(\RR^d)}
    \times \ldots \times \Cspace[0]{k}{\Omega, \Sym^{k+l}(\RR^d)}, \\
    Y &= \Cspace[0]{1}{\Omega, \Sym^{1+l}(\RR^d)}
    \times \ldots \times \Cspace[0]{k-1}{\Omega, \Sym^{k-1+l}(\RR^d)},
  \end{eqnarray*}
  the linear operator
  \[
  \Lambda \in \linspace{X}{Y}, \qquad \Lambda \varphi =
  \left(\begin{array}{c}
    -\varphi_1 - \divergence \varphi_2 \\
    \cdots \\
    -\varphi_{k-1} - \divergence \varphi_k
  \end{array}\right),
  \]
  and the proper, convex and lower semi-continuous functionals
  \[
  \begin{array}{rlrl}
    F&: X \rightarrow {]{-\infty,\infty}]}, & F(\varphi) &= 
    -\scp{u}{\divergence \varphi_1} + 
    \sum_{i=1}^k \mI_{\sett{\norm[\infty]{\placeholder} \leq
        \alpha_{k-i}}}(\varphi_i) , \\
    G&: Y \rightarrow {]{-\infty,\infty}]}, & G(\psi) &= \mI_{\sett{(0,\ldots,0)}}(\psi).
  \end{array}
  \]
  With these choices, the identity
  \[
  \TGV_\alpha^{k,l}(u) = \sup_{\varphi \in X} \ -F(\varphi) - G(\Lambda \varphi)
  \]   
  follows from the definition in~\eqref{eq:tensor-tgv-def}.
  In order to show the representation of $ \TGV^{k,l}_\alpha (u) $ as in \eqref{eq:tgv_primal}, we would like to obtain
  \begin{equation} \label{eq:tgv_from_attouch}
  \TGV_\alpha^{k,l}(u) = \min _{w \in Y^*} \ F^*(-\Lambda^* w) + G^* (w).
  \end{equation}
  This follows as soon as~\eqref{eq:fr-duality-qual} is verified.
  For the purpose of showing~\eqref{eq:fr-duality-qual}, 
  let $\psi \in Y$ and define backwards 
  recursively: $\varphi_k = 0 \in 
  \Cspace[0]{k}{\Omega,\Sym^{k+l}(\RR^d)}$, $\varphi_i = \psi_i - \divergence \varphi_{i+1}
  \in \Cspace[0]{i}{\Omega,\Sym^{i+l}(\RR^d)}$ for $i=k-1,\ldots,1$. Hence,
  $\varphi \in X$ and $-\Lambda \varphi = \psi$. Moreover, choosing $\lambda > 0$ large 
  enough, one can achieve that $\norm[\infty]{\lambda^{-1} \varphi} \leq 
  \alpha_{k-i}$ for all $i=1,\ldots,k$, so $\lambda^{-1} \varphi \in \dom(F)$
  and since $0 \in \dom(G)$, we get the representation
  $\psi = \lambda(0 - \Lambda \lambda^{-1} \varphi)$.
  Thus, the identity \eqref{eq:tgv_from_attouch} holds and the 
  minimum is attained in
  $ Y^* $. Now, $ Y^* $ can be written as 
  \[
  Y^* = \Cspace[0]{1}{\Omega, \Sym^{1+l}(\RR^d)}^* \times \ldots \times \Cspace[0]{k-1}{\Omega, \Sym^{k-1+l}(\RR^d)}^*,
  \]
  with elements $ w=(w_1,\ldots,w_{k-1}) $, $ w_i \in  \Cspace[0]{i}{\Omega, \Sym^{i+l}(\RR^d)}^* $, 
  for $ 1\leq i \leq k-1 $. Therefore,
  with $w_0 = u$ and $w_k = 0$ we get, as $G^*=0$, that
  \[
  \begin{array}{rl}
    \fl
    F^*(-\Lambda^* w) + G^*(w) &= \displaystyle \sup_{\varphi \in X} \ 
    \Bigl( \langle -\Lambda ^*w, \varphi \rangle  
    + \scp{u}{\divergence \varphi_1} 
    - \sum_{i=1}^k \mI_{\sett{\norm[\infty]{\placeholder} \leq \alpha_{k-i}}}(\varphi_i) 
    \Bigr) \\
    &= \displaystyle  \sup_{\varphi \in X, \atop
      {\|\varphi_i\|_\infty \leq \alpha_{k-i}, \atop
      i=1,\ldots,k}}
    \ \Bigl( \scp{u}{\divergence \varphi_1} + 
    \sum_{i=1}^{k-1} \scp{w_i}{\divergence \varphi_{i+1}} + \scp{w_{i}}{\varphi_i}
    \Bigr) \\
    &= \displaystyle  \sum_{i=1}^{k} \alpha _{k-i} \Bigl( 
    \sup_{\varphi_i \in \Cspace[0]{i}{\Omega, \Sym^{i+l}(\RR^d)}, \atop
      \|\varphi_i\|_\infty \leq 1} 
    \scp{w_{i-1}}{\divergence \varphi_i} + \scp{w_i}{\varphi_i} 
    \Bigr).
  \end{array}
  \]
  From Lemma~\ref{lem:cascade_element_finite} we obtain that
  each supremum is finite and coincides with 
  $\norm[\radon]{\symgrad w_{i-1} - w_i}$ if and only if
  $\symgrad w_{i-1} - w_i \in \radonspace{\Omega,\Sym^{k+i}(\RR^d)}$
  for $i=1,\ldots,k$. 
  Then, as $w_k = 0$, 
  according to Theorem~\ref{thm:symgrad_measure_dist}, 
  this already yields $w_{k-1} 
  \in \BD(\Omega,\Sym^{k+l-1}(\RR^d))$, in particular $w_{k-1} \in 
  \radonspace{\Omega, \Sym^{k+l-1}(\RR^d)}$. Proceeding inductively, we
  see that $w_i \in \BD(\Omega,\Sym^{k+i}(\RR^d))$ for each $i=0,\ldots,k$.
  Hence, it suffices to take the minimum in~\eqref{eq:tgv_from_attouch} 
  over all $\BD$-tensor fields which gives~\eqref{eq:tgv_primal}.
  
  In addition, the minimum in \eqref{eq:tgv_primal} is finite if 
  $ u\in \BD(\Omega,\Sym^l(\RR^d)) $.
  Conversely, if $\TD(u) = \infty$, also $\norm[\radon]{\symgrad w_0 
   - w_1} = \infty$ for all $w_1 \in \BD(\Omega,\Sym^{l+1}(\RR^d))$. Hence,
  the minimum in~\eqref{eq:tgv_primal} has to be $\infty$.
\end{proof}

\begin{remark} \label{rem:tgv_minimum_representation_scalar}
  In the scalar case, i.e., $l=0$ it holds that
  \begin{equation}
    \TGV^{k}_\alpha(u) = \min_{w_i \in \BD(\Omega, \Sym^{i}(\RR^d)), \atop
      i=0,\ldots,k, \ w_0 = u, \ w_k = 0}
    \ \sum_{i=1}^k \alpha_{k-i} \norm[\radon]{\symgrad{w_{i-1}} - w_i}.
  \end{equation}
\end{remark}

\begin{remark}
  The minimum representation also allows to define $\TGV_\alpha^{k,l}$ 
  recursively:
  \begin{equation}
    \label{eq:tgv_recursive_def}
    \fl
    \left\{
      \begin{array}{rl}
        \TGV^{1,l}_{\alpha_0}(u) &= \alpha_0 \norm[\radon]{\symgrad u} \\[\smallskipamount]
        \TGV^{k+1,l}_{\alpha}(u) &= \displaystyle
        \min_{w \in \BD(\Omega,\Sym^{l+1}(\RR^d))} \alpha_k 
        \norm[\radon]{\symgrad u - w} + \TGV_{\alpha'}^{k,l+1}(w)
      \end{array}
    \right.
  \end{equation}
  where $\alpha' = (\alpha_0,\ldots,\alpha_{k-1})$ if $\alpha = 
  (\alpha_0,\ldots,\alpha_k)$.
\end{remark}

\begin{remark}
  For the scalar 
  $\TGV_\alpha^2$, the minimum representation reads as
  \[
  \TGV_\alpha^2(u) = \min_{w \in \BD(\Omega,\Sym^1(\RR^d))} \ \alpha_1
  \norm[\radon]{\grad u - w} + \alpha_0 \norm[\radon]{\symgrad w}.
  \]
  This can be interpreted as follows. For $u \in \BV(\Omega)$,
  $\grad u$ is a measure which can be decomposed into a regular
  and singular component with respect to the Lebesgue measure.
  The singular part is always penalised with the Radon norm where
  from the regular part, an optimal bounded deformation vector field $w$
  is extracted. This vector field is penalised by $\TD$ which, like
  $\TV$, implies certain regularity but also allows for jumps. 
  Thus, $\symgrad w$ essentially contains the second-order derivative
  information of $u$. 

  Provided that $w$ is optimal, 
  the total generalised variation of second order then
  penalises the first-order remainder $\grad u - w$ which essentially
  contains the jumps of $u$ as well as
  the second-order information $\symgrad w$.
\end{remark}

The next step is to examine the spaces $\BGV^k(\Omega,\Sym^l(\RR^d))$.
Our aim is to prove that these space coincide with 
$\BD(\Omega,\Sym^l(\RR^d))$ for fixed $l \geq 0$ and 
all $k \geq 1$. We will proceed 
inductively with respect to $k$ and hence vary $k,l$ but leave 
$\Omega$ fixed and assume that $\Omega$ is a bounded Lipschitz domain.
For what follows, we choose a family of
projection operators onto the kernel of $\TGV^{k,l}_\alpha = 
\kernel{\symgrad^k}$ (see Proposition~\ref{prop:tgv_basic_prop}). 
\begin{definition}
  For each $k \geq 1$ and $l \geq 0$,
  denote by $R_{k,l}: \LPspace{d/(d-1)}
  {\Omega,\Sym^l(\RR^d)} \rightarrow \kernel{\symgrad^k}$ a linear and
  continuous projection. 
\end{definition}
As $\kernel{\symgrad^k}$ (on $\Omega$ and for 
symmetric tensor fields of order $l$) is finite-dimensional, such 
a $R_{k,l}$ always exists but is not necessarily unique. 
A coercivity estimate in $\LPspace{d/(d-1)}{\Omega,\Sym^l(\RR^d)}$
for $\TGV^{k,l}_\alpha$
will next be formulated and proven in terms of
these projections.
As we will see, the induction step in the proof requires an
intermediate estimate as follows.
\begin{lemma} 
  \label{lem:preparatory_tgvk_topological_equivalence}
  For each $k \geq 1$ and $l \geq 0$ there
  exists a constant $C > 0$, only depending on $\Omega, k$ and $l$
  such that for each $u \in \BD(\Omega,\Sym^{l}(\RR^d))$ and 
  $w \in \LPspace{d/(d-1)}{\Omega,\Sym^{l+1}(\RR^d)}$,
  \[
  \|\symgrad u \|_{\radon} \leq C \bigl(
  \|\symgrad u - R_{k,l+1}w\|_{\radon} + \|u\|_1 \bigr).
  \]
\end{lemma}

\begin{proof}
  If this is not true for some $k$ and $l$, then there exist 
  $\seq{u^n}$ in $\BD(\Omega,\Sym^{l}(\RR^d))$ 
  and $\seq{w^n}$ in $\LPspace{d/(d-1)}{\Omega,\Sym^{l+1}(\RR^d)}$
  such that
  \[ 
  \| \symgrad u^n \| _\mathcal{M} = 1 \quad 
  \mbox{and} \quad \frac{1}{n}\geq \| u^n \|_1 
  + \| \symgrad u^n - R_{k,l+1} w^n \| _\mathcal{M}. 
  \] 
  This implies that $\seq{R_{k,l+1}w^n}$ is bounded in terms of 
  $\norm[\radon]{\placeholder}$ in the finite-dimensional space 
  $\kernel{\TGV_\alpha^{k,l+1}} = \kernel{\symgrad^k}$, see 
  Proposition~\ref{prop:tgv_basic_prop}. 
  Consequently, there exists a subsequence, again denoted 
  by $\seq{w^n}$, %
  such that 
  $R_{k,l+1}w^n \rightarrow w $ as $n \to \infty$ with respect to $ \|\cdot \|_{1} $. 
  Hence, $ \symgrad u^n \rightarrow w$ as $n \to \infty$. Further, we have that 
  $ u^n \rightarrow 0 $ as $n \to \infty$ and thus, by closedness of the weak 
  symmetrised gradient, $ \symgrad u^n \rightarrow 0 $ as $n \to \infty$ in $\radonspace{\Omega, \Sym^{l+1}(\RR^d)}$, which contradicts 
  $ \| \symgrad u^n \|_\mathcal{M} = 1 $ for all $n$.
\end{proof}

\begin{proposition}
  \label{prop:tgvk_basic_topological_equivalence}
For each $k\geq 1$ and $l\geq 0$, there exists 
a constant $C > 0$
such that
\begin{eqnarray}
  \label{eq:tgvk_basis_topological_equivalence}
  \fl
  \norm[\radon]{\symgrad u} &\leq C \bigl(\norm[1]{u} 
  + \min\sett{\alpha_0, \ldots, \alpha_{k-1}}^{-1} \TGV_\alpha^{k,l}(u) \bigr)  
  \qquad
                              \text{as well as} \\
  \label{eq:tgvk_coercivity}
  \fl
  \norm[d/(d-1)]{u - R_{k,l}u} &\leq C
                                 \min\sett{\alpha_0, \ldots, \alpha_{k-1}}^{-1}                                 \TGV_\alpha^{k,l}(u)
\end{eqnarray}
for all $u \in \BD(\Omega,\Sym^l(\RR^d))$.
\end{proposition}

\begin{proof}
We prove the result by induction with respect to $k$.
In the case $ k=1 $ and $l \geq 0$ arbitrary, 
the first inequality is immediate while the
second is equivalent to the 
Sobolev--Korn inequality in $\BD(\Omega, \Sym^l(\RR^d))$,
see Theorem~\ref{thm:sobolev_korn}.

Now assume that both inequalities hold for a fixed $k$ and each $l\geq 0$
and perform an induction step with respect to $k$, i.e., we fix
$l\in \NN$, $ \alpha = (\alpha _0 , \ldots ,\alpha _k) $ with $ \alpha _i >0 $
for $i=0,\ldots,k$.
We assume 
that assertion \eqref{eq:tgvk_basis_topological_equivalence} holds for 
$\alpha' = (\alpha _0 , \ldots ,\alpha _{k-1}) $ %
and any
$l'\in \NN$. %

We will first show the uniform estimate for $\|\symgrad u\|_{\radon}$
for which it suffices to consider $u \in \BD(\Omega, \Sym^{l}(\RR ^d)) $, 
as otherwise, according to Theorem~\ref{thm:tgv_primal}, 
$\TGV^{k+1,l}_\alpha(u) = \infty$. 
Then, with the projection $R_{k,l+1}$, the help of Lemma 
\ref{lem:preparatory_tgvk_topological_equivalence},
the continuous embeddings \[
\fl \BD(\Omega,\Sym^{l+1}(\RR^d)) \embed
L^{d/(d-1)}(\Omega,\Sym^{l+1}(\RR^d)) 
\embed L^{1}(\Omega,\Sym^{l+1}(\RR^d))
\] 
and the induction
hypothesis, we can estimate for arbitrary $w \in \BD(\Omega,\Sym^{l+1}(\RR^d))$,
\begin{eqnarray*}
\|\symgrad u \|_{\radon} &\leq  C( \|\symgrad u - R_{k,l+1} w\|_{\radon} + \|u\|_1  )\\
&\leq C( \|\symgrad u - w\|_{\radon} + \|w - R_{k,l+1} w\|_{d/(d-1)} 
+ \|u\|_1 )\\
&\leq C(\|\symgrad u - w\|_{\radon} +  \min\sett{\alpha_0,\ldots,\alpha_{k-1}}^{-1} \TGV^{k,l+1}_{\alpha'} (w) + \|u\|_1 ) \\
&\leq C\bigl( \min\sett{\alpha_0,\ldots,\alpha_k}^{-1} (\alpha_k  \|\symgrad u - w\|_{\radon} + \TGV^{k,l+1}_{\alpha'} (w)) + \|u\|_1 \bigr)
\end{eqnarray*}
for  $ C >0 $ suitable generic constants. Taking the minimum over all such $w \in 
\BD(\Omega,\Sym^{l+1} (\RR^d))$ then yields
\[
\|\symgrad u\|_{\radon} \leq C \bigl(\|u\|_1
+ \min\sett{\alpha_0,\ldots,\alpha_k}^{-1} \TGV_\alpha^{k+1,l}(u) \bigr)
\]
by virtue of the recursive minimum 
representation~\eqref{eq:tgv_recursive_def}.

The coercivity estimate can be shown analogously to 
Proposition~\ref{prop:tv^k-coercive} and 
Corollary~\ref{cor:tv^k-coercive-scalar}.
First, assume that the inequality does not hold true for $\alpha=(1,\ldots,1)$.
Then, there is a sequence $\seq{u^n}$ in 
$L^{d/(d-1)}(\Omega,\Sym^l(\RR^d))$ such that
\[
\|u^n - R_{k+1,l}u^n\|_{d/(d-1)} = 1 \quad \text{and} \quad
 \frac{1}{n}\geq \TGV^{k+1,l}_\alpha (u^n).
\]
By $\kernel{\TGV_\alpha^{k+1,l}} %
= \range{R_{k+1,l}}$, we have
$\TGV_{\alpha}^{k+1,l}(u^n - R_{k+1,l}u^n) = \TGV_\alpha^{k+1,l}(u^n)$ for each $n$. 
Thus, since we already know the first estimate in \eqref{eq:tgvk_basis_topological_equivalence} to hold,
\begin{equation}
  \label{eq:coercitiy_intermediate}
  \fl
  \|\symgrad (u^n - R_{k+1,l}u^n)\|_{\radon} \leq 
  C \bigl(  \TGV_\alpha^{k+1,l}(u^n) + \|u^n - R_{k+1,l}u^n\|_1 \bigr),
\end{equation}
implying, by continuous embedding, that
$\seq{u^n - R_{k+1,l}u^n}$ is bounded in $\BD(\Omega,\Sym^l(\RR^d))$. By
compact embedding (see Proposition \ref{thm:bd_embedding}),
we may therefore conclude that $u^n - R_{k+1,l}u^n \rightarrow 
u$ in $L^{1}(\Omega,\Sym^l(\RR^d))$ for some 
subsequence (not relabelled). Moreover, 
as $R_{k+1,l}(u^n - R_{k+1,l}u^n) = 0$ for all $n$, the limit has to
satisfy $R_{k+1,l}u = 0$. On the other hand, by lower semi-continuity 
(see Proposition~\ref{prop:tgv_basic_prop}),
\[
0 \leq \TGV_\alpha^{k+1,l}(u) \leq \liminf_{n \rightarrow \infty} 
\ \TGV_\alpha^{k+1,l}(u^n) = 0,
\]
hence $u \in \kernel{\symgrad^{k+1}} = \range{R_{k+1,l}}$. Consequently,
$\lim_{n \rightarrow \infty} u^n - R_{k+1,l}u^n =  u = R_{k+1,l}u = 0$. 
From~\eqref{eq:coercitiy_intermediate} it follows that also
$\symgrad(u^n - R_{k+1,l}u^n) \rightarrow 0$ in 
$\radon(\Omega,\Sym^{l+1}(\RR^d)) $, so $u^n - R_{k+1,l}u^n \rightarrow 0$
in $\BD(\Omega,\Sym^l(\RR^d))$ and by continuous embedding also in
$L^{d/(d-1)}(\Omega,\Sym^l(\RR^d))$.
However, this contradicts 
$\|u^n - R_{k+1,l}u^n\|_{d/(d-1)} = 1$ for all $n$, and thus, the claimed
coercivity for the particular choice $\alpha=(1,\ldots,1)$ holds.
The result for general $\alpha$ then follows from monotonicity of $\TGV^{k+1,l}_\alpha$
with respect to each component of $\alpha$.
\end{proof}

  \begin{corollary} \label{cor:bv_bgv_equivalence}
    For $k\geq 1$ and $l \geq 0$ 
    there exist $ C,c>0 $ such that for all $ u\in\BD(\Omega, \Sym^l(\RR^d)) $
    we have
    \begin{equation}
      \label{eq:bv_bgv_equivalence}
      c \bigl( \|u\|_1 + \TGV^{k,l}_\alpha (u) \bigr) \leq  \|u\|_1 + \TD(u) 
      \leq C \bigl( \|u\|_1 + \TGV^{k,l}_\alpha (u) \bigr).
    \end{equation}
    In particular, $\BGV^{k}(\Omega,\Sym^l(\RR^d)) = \BD(\Omega,\Sym^l(\RR^d))$
    in the sense of Banach space isomorphy.
  \end{corollary}

\begin{proof}
  The estimate on the right is a consequence 
  of~\eqref{eq:tgvk_basis_topological_equivalence} while 
  the estimate on the left follows by the minimum 
  representation~\eqref{eq:tgv_primal} which gives $\TGV_\alpha^{k,l} \leq 
  \alpha_{k-1} \TD$.
\end{proof}

\paragraph*{Tikhonov regularisation.}
Once again, the second estimate in Proposition \ref{prop:tgvk_basic_topological_equivalence} is crucial to transfer the well-posedness result of Theorem \ref{thm:general_reg_existence_linear} as follows.

\begin{proposition} \label{prop:well_posed_tgvk} With $X=L^p(\Omega)$, $p \in {]{1,\infty}[}$, $\Omega$ being a bounded Lipschitz domain, $Y$ a Banach space, $K: X \to Y$ linear and continuous, $S_f: Y \to [0,\infty]$ proper, convex, lower semi-continuous and coercive, $k \geq 1$, $\alpha = (\alpha_0,\ldots,\alpha_{k-1})$ with $\alpha_i > 0$ for $i=0,\ldots, k-1$, the Tikhonov minimisation problem
  \begin{equation}
    \label{eq:general_tgvk_min}
    \min_{u \in \LPspace{p}{\Omega}} \ S_f(Ku) +\TGV^{k}_\alpha (u) .
  \end{equation}
  is well-posed in the sense of Theorem \ref{thm:general_reg_existence_linear} whenever $p \leq d/(d-1)$ if $d>1$.
\end{proposition}
Regarding the assumptions of Theorem \ref{thm:general_reg_existence_linear} on the kernel of the seminorm and the constraint on the exponent $p$ in the underlying $L^p$-space, we see that, as one would expect, $\TGV^k$ resembles the situation of the infimal convolution of TV-type functionals rather than their sum, in particular the constraint $p\leq d/(d-1)$ is the same as with first-order $\TV$ regularisation.

This is also true for the following convergence result, which should be compared to the results of Theorems \ref{thm:sum_tvk_reg_convergence} and \ref{thm:infconv_tvk_reg_convergence} for the sum and the infimal convolution of higher-order TV functionals, respectively. Here, similar as with the infimal convolution, we extend $\TGV_\alpha^k$ to weights in $]0,\infty]$ by using the minimum representation and defining $\alpha_i \norm[\radon]{\placeholder} = \mathcal{I}_{\{ 0\}}$ for $\alpha_i = \infty$.

\begin{theorem}
  \label{thm:tgvk_reg_convergence}
  In the situation of Proposition~\ref{prop:well_posed_tgvk} and $p \in {]{1,\infty}[}$ with $p \leq d/(d-1)$ if $d > 1$,
  let for each 
$\delta > 0$ the data $f^\delta$ be such that
$S_{f^\delta}(f^\dagger) \leq \delta$, and let the discrepancy functionals $\seq{S_{f^\delta}}$ be equi-coercive and converge to
  $S_{f^\dagger}$ for some data $f^\dagger$ in $Y$ in the sense of \eqref{eq:discrepancy_convergence} and $S_{f^\dagger}(v) = 0$ if and only if $v = f^\dagger$.
  
  Choose the parameters $\alpha = (\alpha_0,\ldots,\alpha_{k-1})$
  in dependence of $\delta$ such that
  \[
  \min\{\alpha_0,\ldots,\alpha_{k-1}\} \to 0, \quad
  \frac{\delta}{\min\{\alpha_0,\ldots,\alpha_{k-1}\}} \to 0, \qquad \text{as} \qquad \delta \to 0,
  \]
  and assume that $(\tilde{\alpha}_0,\ldots,\tilde{\alpha}_{k-1}) = (\alpha_0,\ldots,\alpha_{k-1})/\min\{\alpha_0,\ldots,\alpha_{k-1}\} \rightarrow (\alpha_0^\dagger,\ldots,\alpha_{k-1}^\dagger) \in ]0,\infty]^k$ as $\delta \to 0$.
 Set 
 \[ m = \min\set{ m'\in \{1,\ldots,k\}}{\alpha^\dagger _{k-m'} \neq \infty } 
 \]
  and assume that there exists $u_0 \in \BV^ m(\Omega)$ such that
 $Ku_0 = f^\dagger$.

 Then, up to shifts in $\kernel{K} \cap \poly^{k-1}$, any sequence $\seq{u^{\alpha,\delta}}$, with each $u^{\alpha,\delta}$ being a solution to~\eqref{eq:general_tgvk_min} with
 parameters $(\alpha_0,\ldots,\alpha_{k-1})$ and data $f^\delta$,
 has at least one $L^p$-weak accumulation point. Each $L^p$-weak accumulation point is a 
  minimum-$\TGV_{\alpha^\dagger}^k$-solution of $Ku = f^\dagger$ and $\lim_{\delta \to 0}
  \TGV^k_{\tilde{\alpha}}(u^{\alpha,\delta}) = \TGV_{\alpha^\dagger}^k (u^\dagger)$.
 \begin{proof}
   The proof is analogous to the one of \cite[Theorem 4.8]{holler14inversetgv_mh}, which considers the case $S_f(v) = \frac1q\|v-f\|_Y^q$ for $q \in [1,\infty[$.
   Alternatively, one can proceed along the lines of the proof of Theorem~\ref{thm:infconv_tvk_reg_convergence} with the infimal convolution replaced by $\TGV$ to obtain the result.
 \end{proof}
\end{theorem}

\paragraph*{A-priori estimates.} In case of Hilbert-space data and
quadratic norm discrepancy, i.e., $S_f(v) = \frac12\norm[Y]{v - f}^2$
for $Y$ Hilbert space, one can, in the situation of
Proposition~\ref{prop:well_posed_tgvk} once again find an a-priori
bound thanks to the coercivity estimate~\eqref{eq:tgvk_coercivity}.
Let $C> 0$ be a constant such that
$\norm[p]{u - Ru} \leq C \min\sett{\alpha_0,\ldots,\alpha_{k-1}}^{-1}
\TGV_\alpha^k(u)$ for a linear and continuous projection operator $R$
onto $\poly^{k-1}$ for all $u \in \BV(\Omega)$.  Further, assume that
$K$ is injective on $\poly^{k-1}$ and $c > 0$ is chosen such that
$c \norm[p]{Ru} \leq \norm[Y]{KRu}$ for all
$u \in \LPspace{p}{\Omega}$.  Then, for a solution $u^*$ of the minimisation
problem
\[
  \min_{u \in \LPspace{p}{\Omega}} \ \frac12 \norm[Y]{Ku - f}^2 +
  \TGV_\alpha^k(u),
\]
the norm $\norm[p]{u^*}$ obeys the a-priori
estimate~\eqref{eq:tv^k-tikh-hilbert-bound} with $\alpha$ replaced by
$\min \sett{\alpha_0,\ldots,\alpha_{k-1}}$.  Also here, if the
discrepancy is replaced by the Kullback--Leibler discrepancy
$S_f(v) = \KL(v,f)$, then $\norm[p]{u^*}$ can be estimated analogously
in terms of~\eqref{eq:tv^k-tikh-kl-bound}.  Let, again, $C_f \geq 0$
be an a-priori estimate for the optimal functional value, analogous to
the $C_f$ that leads to~\eqref{eq:tvk-inf-conv-hilbert-bound}.
Moreover, analogous to the multi-order infimal-convolution case in
Subsection~\ref{subsec:tvk-infconv}, it is
possible to estimate each tuple $(w_1^*, \ldots, w_{k-1}^*)$ that
realises the minimum in the primal representation~\eqref{eq:tgv_primal}
of $\TGV_\alpha^k(u^*)$.  Now, in order to estimate, for instance,
$\norm[1]{w_1^*}$, set $w_0^* = u^*$ and note that we already have the
bound $\norm[1]{w_0^*} \leq \abs{\Omega}^{1/p} \norm[p]{u^*}$
where~\eqref{eq:tv^k-tikh-hilbert-bound} or~\eqref{eq:tv^k-tikh-kl-bound} provides an a-priori estimate
of the right-hand side.  Choosing a $C_1 > 0$ such that
$\norm[\radon]{\symgrad w_0} \leq C_1 (\norm[1]{w_0} +
\min\sett{\alpha_0,\ldots,\alpha_{k-1}}^{-1} \TGV_{\alpha}^k(w_0))$
for all $w_0 \in \BD(\Omega,\Sym^0(\RR^d)) = \BV(\Omega)$, we obtain
$\norm[\radon]{\symgrad w_0^*} \leq C_1 (\norm[1]{w_0^*} +
\min\sett{\alpha_0,\ldots,\alpha_{k-1}}^{-1} C_f)$ and, consequently,
\[
  \fl
  \begin{array}{rl}
    \alpha_{k-1} \norm[1]{w_1^*}
    & \leq
      \alpha_{k-1} \norm[\radon]{\symgrad w_0^*} +
      \alpha_{k-1} \norm[\radon]{\symgrad w_0^* - w_1^*}
      \leq \alpha_{k-1} \norm[\radon]{\symgrad w_0^*} + \TGV_\alpha^k(u^*)
    \\[\smallskipamount]
    & \leq \alpha_{k-1} \norm[\radon]{\symgrad w_0^*} + C_f.
  \end{array}
\]
We thus obtain the bound
\begin{equation}
  \norm[1]{w_1^*} \leq C_1 \norm[1]{w_0^*} + \frac{(C_1 + 1)C_f}{\min\sett{\alpha_0,\ldots,\alpha_{k-1}}},
  \label{eq:tgv_tikh_hilbert_bound1}
\end{equation}
which is similar to~\eqref{eq:tvk-inf-conv-hilbert-bound}, but
involves a norm and not a seminorm due to the structure of
$\TGV_\alpha^k$. Using this line of argumentation, one can now
inductively obtain bounds on $w_2,\ldots,w_{k-1}$ according to
\begin{equation}
  \norm[1]{w_i^*} \leq C_i \norm[1]{w_{i-1}^*} + \frac{(C_i + 1)C_f}{
    \min\sett{\alpha_0, \ldots, \alpha_{k-1}}}
  \label{eq:tgv_tikh_hilbert_bound}
\end{equation}
for $i = 1,\ldots,k-1$, where each $C_i > 0$ is a constant such that
$\norm[\radon]{\symgrad w_{i-1}} \leq C_i (\norm[1]{w_{i-1}} +
\min\sett{\alpha_0,\ldots,\alpha_{k-i}}^{-1}
\TGV_{(\alpha_0,\ldots,\alpha_{k-i})}^{k-i+1}(w_{i-1}))$ for all
$w_{i-1} \in \BD(\Omega, \Sym^{i-1}(\RR^d))$, whose existence is
guaranteed by
Proposition~\ref{prop:tgvk_basic_topological_equivalence}. This
provides an a-priori estimate for $u^*$ and $w_1^*,\ldots, w_{k-1}^*$.

\paragraph*{Denoising performance.}
In Figure~\ref{fig:tgv2-reg}, one can see how second-order TGV regularisation (Figure~\ref{fig:tgv2-reg} (d))
performs in comparison to first-order TV (Figure~\ref{fig:tgv2-reg} (b)) and $\alpha_1\TV \infconv \alpha_2\TV^2$ (Figure~\ref{fig:tgv2-reg} (c)) as
regulariser for image denoising.
It is apparent that TGV covers higher-order features more accurately than 
the associated infimal-convolution regulariser with the staircase effect being absent, while at the same time, jump discontinuities are preserved as for first-order TV. This is in particular reflected
in the underlying function space for TGV being $\BV(\Omega)$, see Proposition~\ref{cor:bv_bgv_equivalence}. In conclusion, the total generalised variation can be
seen as an adequate model for piecewise smooth images and will, in the following,
be the preferred regulariser for this class of functions.

\begin{figure}
  \centering
  \begin{tabular}{c@{\ }c@{\ }c@{\ }c}
    \includegraphics[width=0.22\linewidth]{pics_affine_denoising_noise.png}
    & 
    \includegraphics[width=0.22\linewidth]{pics_affine_denoising_tv1.png}
    &
    \includegraphics[width=0.22\linewidth]{pics_affine_denoising_inf_tv_tv2.png}
    & 
    \includegraphics[width=0.22\linewidth]{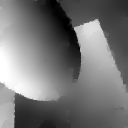}
    \\
    (a) & (b) &
    (c) & (d)
  \end{tabular}
  \caption{Total generalised variation denoising
    example. (a) Noisy image, (b) 
    regularisation with $\TV$, (c) regularisation with 
    $\alpha_1 \TV \infconv \alpha_2 \TV^2$, 
    (d) regularisation with $\TGV_\alpha^2$. 
    All parameters are tuned to give highest PSNR
    with respect to
    the ground truth (Figure~\ref{fig:first-order-reg} (a)).}
  \label{fig:tgv2-reg}
\end{figure}

\subsection{Extensions} \label{sec:tgv_extensions}
\paragraph{TGV for multichannel images.} 
Again, in analogy to TV and higher-order TV, TGV can also be extended to colour and multichannel images represented by functions mapping into the vector space $\R^m$ by testing with $\Sym^k(\RR^d)^m$-valued tensor fields. This requires to define pointwise norms on $\Sym^l(\RR^d)^m$ for $l=1,\ldots,k$ where, apart from the standard Frobenius norm, one can take any norm $\abs[\circ_l]{\placeholder}$ on $\Sym^{l}(\RR^d)^m$, noting that different norms imply different types of coupling of the multiple channels.
With each $\abs[*_l]{\placeholder}$ denoting the dual norm of $\abs[\circ_l]{\placeholder}$, $\TGV^k_\alpha$ can be extended to functions $u \in \LPlocspace{1}{\Omega,\R^m}$ as
  \begin{equation}
    \label{eq:tgv_color_def}
    \fl
    \TGV^k_\alpha(u) = \sup\ \Bigl\{\int_\Omega u \inprod \divergence^k \varphi \dd{x}
    \ \Bigl| \ \varphi \in \Ccspace{k}{\Omega, \Sym^k(\RR^d)^ m}, \ 
    \underbrace{\norm[\infty,*_l]{\divergence^l \varphi} \leq 
      \alpha_l}_{l = 0,\ldots, k-1},
    \Bigr\}
  \end{equation}
  where $\norm[\infty,*_l]{\psi}$ %
  is the pointwise supremum of the scalar function $x \mapsto \abs[*_l]{\psi(x)}$ on $\Omega$ for $\psi \in \Ccspace{}{\Omega,\Sym^l(\RR^d)}$. As before, by equivalence of norms in finite dimensions, the functional-analytic and regularisation properties of TGV transfer to its multichannel extension, see e.g. \cite{holler15tgvrec_p1_mh,Bredies14_multichannel_mh}. Rotationally invariance holds whenever all tensor norms $\abs[*_l]{\placeholder}$ are unitarily invariant. For $k=2$, particular instances that are unitarily invariant can be constructed by choosing $\abs{\placeholder}_{*_1}$ as a unitarily invariant matrix norm and $\abs{\placeholder}_{*_2}$ as either the Frobenius tensor norm or $\abs{\xi}_{*_2} = \sum_{i=1}^m \abs{\xi_i}_{*_1}$, i.e., a decoupled norm. This allows, for instance, to penalise the nuclear norm of first-order derivatives and the Frobenius tensor norm of the second order component, as it was done, e.g., in \cite{knoll15mr_pet_tgv_mh}.

\paragraph{Infimal-convolution TGV.}
Beyond the realisation of different couplings of multiple colour channels, the extension to arbitrary pointwise tensor norms in the definition of TGV can also be beneficial in the context of scalar-valued functions.
In \cite{holler14ictv}, the infimal convolution of different TGV functionals with different, anisotropic norms was considered in the context of dynamic data as well as anisotropic regularisation for still images. With $\TGV_{\beta_i}^{k_i}$ for $i=1,\ldots,n$ denoting TGV functionals according to~\eqref{eq:tgv_color_def} for $m=1$ of order $k_i$ and each $\beta_i$ denoting a tuple of pointwise norms, the functional $\ICTGV^n_\beta$ can be defined for $u \in \LPlocspace{1}{\Omega}$ as
\[ \ICTGV^n_\beta (u) = \inf_{ u_i \in \LPlocspace{1}{\Omega}, \atop  u_0 = u, \ u_n = 0 } \ \sum_{i=1}^n \TGV^{k_i}_{\beta_i}(u_{i-1} - u_i). \]
As shown in \cite{holler14ictv}, this functional is equivalent to $\TGV^k_\alpha$ for $k = \max\{ k_i\}$ and $\alpha$ any parameter vector,  and, in case $\ICTGV^n_\beta(u)< \infty$, the minimum is attained for $u_i \in \LPspace{d/(d-1)}{\Omega} $ for $i=1,\ldots,n-1$. Hence, the coercivity estimate on $\TGV^k_\alpha$ transfers to $\ICTGV^n_\beta$ and again, all results in the context of Tikhonov regularisation apply. 

For applications in the context of dynamic data, the norms for the different $\TGV^{k_i}_{\beta_i}$ can be chosen to realise different weightings of spatial and temporal derivatives. This allows, in a convex setting, for an adaptive regularisation of video data via a motion-dependent separation into different components, see, for instance, Figure \ref{fig:ictgv_juggler_example}.

Similarly, for still image regularisation, one can choose $\TGV^{k_1}_{\beta_1}$ to employ isotropic norms and correspond to the usual total generalised variation, and each $\TGV^{k_i}_{\beta_i}$ for $i=2,\ldots,n$ to employ different anisotropic norms that favour one particular direction. This yields again an adaptive regularisation of image data via a decomposition into an isotropic and several anisotropic parts and can be employed, for instance, to recover certain line structures for denoising \cite{holler14ictv} or applications in CT imaging \cite{kongskov2017tomographic}.

\begin{figure} 
\centering
  \begin{tikzpicture}
    \scriptsize
    \node at (0,0) {\includegraphics[width=5.01cm]{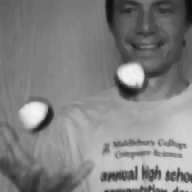}};

    \node at (3.76,1.26) {\includegraphics[width=2.5cm]{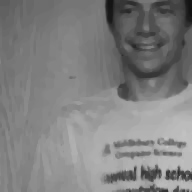}};
    \node at (3.76+2.51,1.26) {\includegraphics[width=2.5cm]{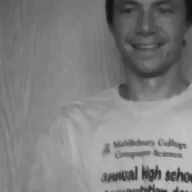}};
    \node at (3.76+2*2.51,1.26) {\includegraphics[width=2.5cm]{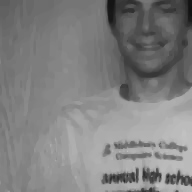}};

    \node at (3.76,-1.26) {\includegraphics[width=2.5cm]{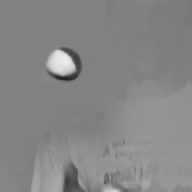}};
    \node at (3.76+2.51,-1.26) {\includegraphics[width=2.5cm]{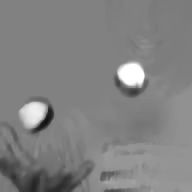}};
    \node at (3.76+2*2.51,-1.26) {\includegraphics[width=2.5cm]{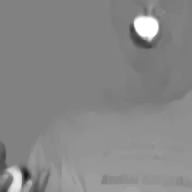}};
	\end{tikzpicture}
	\caption{\label{fig:ictgv_juggler_example} Frame of an image sequence showing a juggler (left), and three frames of a decomposition into components capturing slow (top right images) and fast (bottom right images) motion that was achieved with $\ICTGV$ regularisation.}

\end{figure}

\paragraph{Oscillation TGV and its infimal convolution.}
The total generalised variation model can also be extended to account
for functions with piecewise oscillatory behaviour, which is, for
instance, useful to model texture in images
\cite{Bredies2018oscillation_mh}. The basic idea to include
oscillations is to fix a direction $\omega \in \RR^d$, $\omega \neq 0$
and to modify the definition of second-order TGV such that its kernel
corresponds to oscillatory functions in the
$\omega/\abs{\omega}$-direction with frequency $\abs{\omega}$:
\[
  \fl
  \begin{array}{rl}
    \displaystyle
    \TGV^{\osci}_{\alpha,\omega}(u) = \sup \ \Bigl\{\int_\Omega u
    (\divergence^2 \varphi + \varphi \inprod \omega \tensor \omega)
    \dd{x} 
    & \Bigl| \  \varphi \in \Ccspace{2}{\Omega, \Sym^2(\RR^d)}, \\
    & \ \ \norm[\infty]{\varphi} \leq \alpha_0, \ \norm[\infty]{\divergence
      \varphi} \leq \alpha_1 \Bigr\},
  \end{array}
\]
where, as before, $\alpha = (\alpha_0,\alpha_1)$, $\alpha_0, \alpha_1 > 0$.
Indeed, the kernel of $\TGV_{\alpha,\omega}^{\osci}$ is spanned by the
functions $x \mapsto \sin(x \inprod \omega)$ and
$x \mapsto \cos(x \inprod \omega)$. Further, the functional is proper,
convex and lower semi-continuous in each $\LPspace{p}{\Omega}$, and
admits the minimum representation
\[
  \TGV_{\alpha,\omega}^{\osci}(u) = \min_{w \in \BD(\Omega)} \
  \alpha_1 \norm[\radon]{\grad u - w} + \alpha_0
  \norm[\radon]{\symgrad w + (\omega \tensor \omega) u}.
\]
With
$R_\omega: \LPspace{d/(d-1)}{\Omega} \to
\ker(\TGV_{\alpha,\omega}^{\osci})$ a linear and continuous
projection, a coercivity estimate holds as follows:
\[
  \norm[d/(d-1)]{u - R_\omega u} \leq C \TGV_{\alpha,\omega}^{\osci}(u)
\]
for all $u \in \BV(\Omega)$, see \cite{Bredies2018oscillation_mh}.
The functional can therefore be used as a regulariser in all cases
where $\TV$ is applicable. 

In order to obtain a texture-aware image model, one can now take the
infimal convolution of $\TGV_{\alpha_0}^2$ with parameter vector
$\alpha_0 \in {]{0,\infty}[}^2$ and $\TGV_{\alpha_i,\omega_i}^{\osci}$ for parameter
vectors $\alpha_1,\ldots, \alpha_n \in {]{0,\infty}[}^2$ and directions
$\omega_1,\ldots,\omega_n \in \RR^d$ with $\omega_i \neq 0$ for
$i=1,\ldots,n$, i.e.,
\[
  \ICTGV_{\alpha,\omega}^{\osci}%
  =
  \TGV^2_{\alpha_0} \infconv \TGV_{\alpha_1,\omega_1}^{\osci} \infconv
  \cdots \infconv \TGV_{\alpha_n,\omega_n}^{\osci},
\]
which again yields a proper, convex and lower semi-continuous
regulariser on each $\LPspace{p}{\Omega}$ which is coercive in the
sense that
$\norm[d/(d-1)]{u - Ru} \leq C \ICTGV_{\alpha,\omega}^{\osci}(u)$ for
a linear and continuous projection
$R: \LPspace{d/(d-1)}{\Omega} \to \ker(\ICTGV_{\alpha,\omega}^{\osci})
= \ker(\TGV_{\alpha_0}^2) + \ker(\TGV_{\alpha_1,\omega_1}^{\osci}) +
\ldots + \ker(\TGV_{\alpha_n,\omega_n}^{\osci})$, see again
\cite{Bredies2018oscillation_mh}. It is therefore again applicable as
a regulariser for inverse problems whenever $\TV$ is applicable. See
Figure~\ref{fig:ictgv_osci} for an example of $\ICTGV^{\osci}$-based
denoising and its benefits for capturing and reconstructing textured
regions.

\begin{figure}[t]
  \centering
  \begin{tabular}{c@{\ \ }c@{\ \ }c}
    \includegraphics[width=0.315\linewidth]{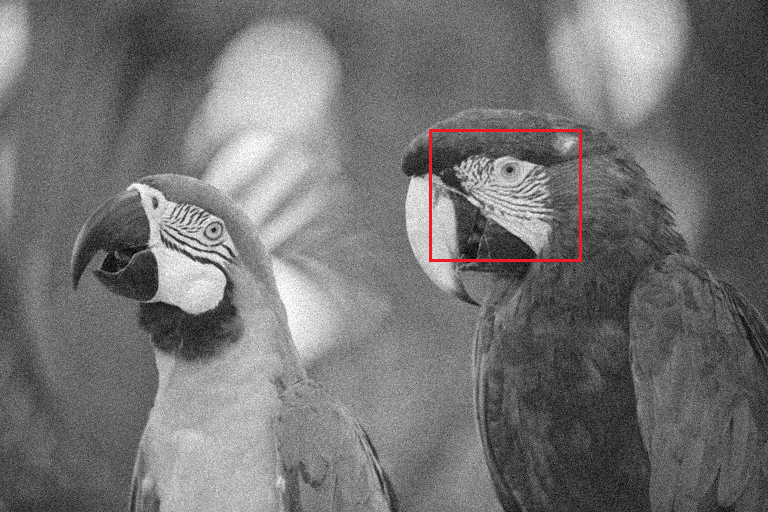}
    &
      \includegraphics[width=0.315\linewidth]{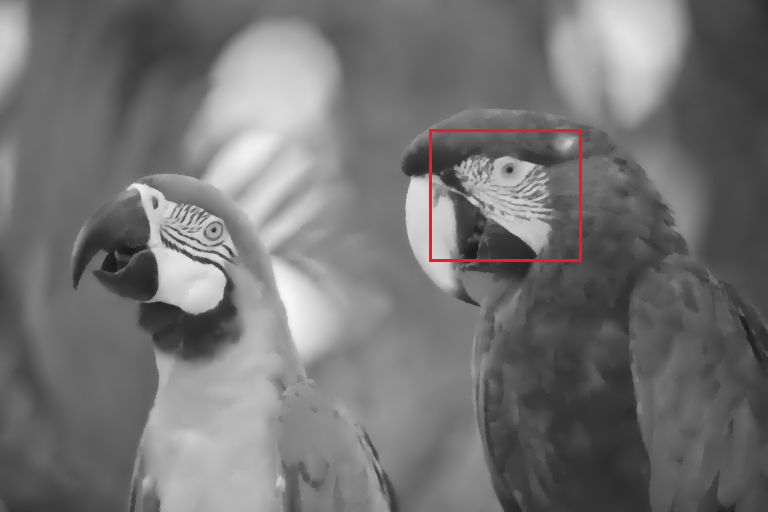}
    &
      \includegraphics[width=0.315\linewidth]{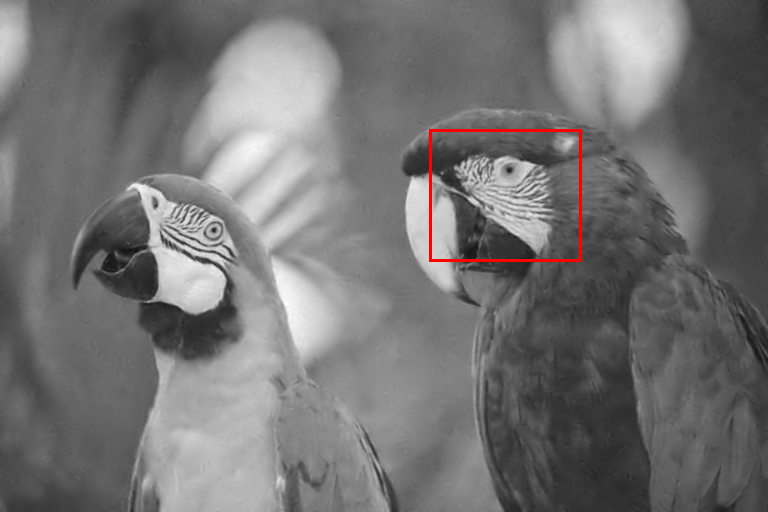}
    \\[0.5\smallskipamount]
    \includegraphics[width=0.315\linewidth]{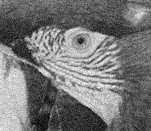}
    &
      \includegraphics[width=0.315\linewidth]{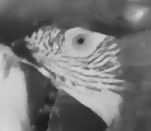}
    &
      \includegraphics[width=0.315\linewidth]{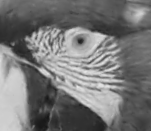} \\
    (a) & (b) & (c)
  \end{tabular}
  \caption{Example of $\ICTGV^{\osci}$ denoising.  In the top row, the
    whole image is depicted, while a closeup of the respective marked
    region is shown in the bottom row.  (a) A noisy image. (b) Results
    of $\TGV^2$-denoising. (c) Results of $\ICTGV^{\osci}$ denoising.
    Parameters were optimised towards best peak signal-to-noise
    ratio with respect to the ground truth (not shown).}
  \label{fig:ictgv_osci}
\end{figure}

\paragraph{TGV for manifold-valued data.}

In different applications in inverse problems and imaging, the data of interest takes values not in a vector space but rather a non-linear space such as a manifold. Examples are sphere-valued data in synthetic aperture radar (SAR) imaging or data with values in the space of positive matrices, equipped with the Fisher--Rao metric, which is used diffusion tensor imaging. Motivated by such applications, TV regularisation has been extended to cope with manifold-valued data, using different approaches and numerical algorithms \cite{cremers13tv_manifold_mh,Lellmann13_manifold_tv_mh,weinmann2014total_mh,grohs2016tv_manifold}.
A rather simple extension of TV for discrete and finite, univariate signals $(u_i)_i $ living in a complete Riemannian manifold $\Mc \subset \R^d$ with metric $d_{\Mc}$ is given as
\[ \TV(u) = \sum_i d_\Mc(u_{i+1},u_i). \]
For this setting, and an extension to bivariate signals, the work \cite{weinmann2014total_mh} provides simple numerical algorithms which yield, in case $\Mc$ is a Hadamard space, globally optimal solutions of variational TV denoising for manifold-valued data. While this allows in particular to extend edge-preserving regularisation to non-linear geometric data, it can again be observed that TV regularisation has a tendency towards piecewise constant solutions with artificial jump discontinuities. To overcome this, different works have proposed extensions of this approach to higher-order TV \cite{Bavcak2016tv2_manifold_mh}, the ($\TV$-$\TV^2$)-infimal convolution \cite{Steidl17_infcon_manifold_mh,bergmann2018ictv_tgv_mh} and second-order TGV \cite{holler18tgvm_mh,bergmann2018ictv_tgv_mh}. Here we briefly sketch the main underlying ideas, presented in \cite{holler18tgvm_mh}, for an extension of TGV to manifold-valued data. For simplicity, we consider only the case of univariate signals $(u_i)_i$ and assume that length-minimising geodesics are unique (see \cite{holler18tgvm_mh} for the general case and details on the involved differential-geometric concepts).

From the continuous perspective, a natural approach to extend $\TGV$ for manifold-valued data, at least in a smooth setting, would be to use tangent spaces for first-order derivatives and, for the second-order term, invoke a connection on the manifold for the differentiation of vector fields. In contrast to that, the motivation for the definition of TGV as in \cite{holler18tgvm_mh} was to exploit a discrete setting in order to avoid high-level differential-geometric concepts but rather to come up with a definition of TGV that can be written only in terms of the distance function on the manifold. To this aim, we identify tangential vectors $v \in T_{a}\M$, with $T_{a}\M$ denoting the tangent space at a point $a \in \M$, with point tuples $[a,b]$ via the exponential map $b = \exp_a (v)$. A discrete gradient operator then maps a signal $(u_i)_i$ to a sequence of point-tuples $([u_{i},u_{i+1}])_i$, where we regard $(\nabla u)_i = [u_{i},u_{i+1}]$, which generalises first-order differences in vector spaces, since in this case, $\exp_{u_i}(u_{i+1} - u_i) = u_{i+1}$. Vector fields whose base points are $(u_i)_i$ can then be identified with a sequence $([u_i,y_i])_i$ with each $y_i \in \Mc$ and, assuming $D:\M^2 \times \M^2 \rightarrow [0,\infty)$  to be an appropriate distance-type function for such tuples, an extension of second-order $\TGV$ can be given as
\[ \fl \text{M-}\TGV( (u_i)_i) = \min_{(y_i)_i} \ \sum_i  \alpha_1 D( [u_{i},u_{i+1}],[u_i,y_i] ) + \alpha_0 D( [u_i,y_i],[u_{i-1},y_{i-1}]) .\]
The difficulty here is in particular how to define $D$ for two point tuples with different base points, as those represent vectors in different tangent spaces. To overcome this, a variant for $D$ as proposed in \cite{holler18tgvm_mh} uses the Schild's ladder \cite{kheyfets2000schild_mh} construction as a discrete approximation of the parallel transport of vector fields between different tangent spaces.
In order to describe this construction, denote by $[u,v]_t$ for $u,v \in \Mc$ and $t \in \RR$ the point reached at time $t$ after travelling on the geodesics from $u$ to $v$, i.e., $[u,v]_t = \exp_{u}(t \log_u (v))$, where $\log$ is the inverse exponential map. Then, the parallel transport of $[u,v]$ (which represents $\log_u(v) \in T_u\Mc$) to the base point $x \in \Mc$ is approximated by $[x,y']$ where $y' = [u,[x,v]_{\frac{1}{2}}]_2$ (which represents $\log_x(y') \in T_x \Mc$).
Using this, a distance on point tuples, denoted by $D_S$, can be given as
\[ D_S([x,y],[u,v]) = d_{\Mc}( y',y) \quad \text{ with } \quad y' = [u,[x,v]_{\frac{1}{2}}]_2.\]
Exploiting the fact that $D_S ([u,v],[u,w]) = d_\Mc(v,w)$ for tuples having the same base point, a concrete realisation of discrete second order $\TGV$ for manifold-valued data is then given as
\[ \text{S-}\TGV( (u_i)_i) = \min_{(y_i)_i} \ \sum_i  \alpha_1 d_\Mc( u_{i+1},y_i) + \alpha_0 D_S( [u_i,y_i],[u_{i-1},y_{i-1}]) .\]
The $\text{S-}\TGV$ denoising problem for $(f_i)_i$ some given data with $f_i \in \Mc$ then reads as
\[ \min_{(u_i)_i} \ \text{S-}\TGV((u_i)_i) + \lambda \sum_i  d_\Mc (u_i, f_i)^2,
\]
and a numerical solution (which can only be guaranteed to deliver stationary points due to non-convexity) can be obtained, for instance, using the cyclic proximal point algorithm \cite{bavcak2013means_meadians_mhadamard_mh,holler18tgvm_mh}. Figure \ref{fig:tgvm_example_s2} shows the results for this setting using both $\TV$ and second-order $\TGV$ regularisation for the denoising of $\mathcal{S}^2$ valued image data, which is composed of different blocks of smooth data with sharp interfaces. It can be seen that both TV and TGV are able to recover the sharp interfaces, but TV suffers from piecewise-constancy artefacts which are not present with TGV.

\begin{figure}
\begin{center}
\includegraphics[width=0.3\linewidth]{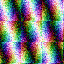}
\includegraphics[width=0.3\linewidth]{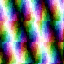}
\includegraphics[width=0.3\linewidth]{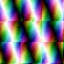}
\end{center}
\caption{ \label{fig:tgvm_example_s2}
  Example of variational denoising for manifold-valued data.
  The images show noisy $\mS^2$-valued data (left) which is denoised with TV-regulariser (middle) and TGV-regulariser (right). The sphere $\mS^2$ is colour-coded with hue and value representing the longitude and latitude, respectively. All parameters are selected optimally. }
\end{figure}

\paragraph{Image-driven TGV.}
In case of denoising problems, i.e., $f \in \LPspace{p}{\Omega}$, second-order
$\TGV$ can be modified to incorporate directional information obtained from
$f$, resulting in \emph{image-driven TGV} (ITGV) \cite{ranftl2012stereo}.
The latter is defined by introducing a diffusion tensor field
into the functional:
\[
\ITGV_\alpha^2(u) = \min_{w \in \BD(\Omega)}
\ \alpha_1 \int_\Omega \dd{\abs{D^{1/2}\grad u - w}}
+ \alpha_0 \int_\Omega \dd{\abs{\symgrad w}}
\]
where $D: \closure{\Omega} \to \Sym^2(\RR^d)$ is assumed to be 
continuous and positive semi-definite in each point. Denoting by $f_\sigma = f \conv G_\sigma$ a smoothed version
of the data $f$ obtained by convolution with a Gaussian kernel $G_\sigma$ 
of variance $\sigma > 0$ and suitable extension outside of $\Omega$, the 
diffusion tensor field $D$ may be chosen according to
\[
D = \id - (1 - \expE^{-\gamma \abs{\grad f_\sigma}^\beta}) 
\frac{\grad f_\sigma}{\abs{\grad f_\sigma}}
\tensor \frac{\grad f_\sigma}{\abs{\grad f_\sigma}}
\]
with parameters $\gamma > 0$ and $\beta > 0$. If the smallest eigenvalue of $D$
is uniformly bounded away from $0$ in $\closure{\Omega}$, then $\ITGV$ admits
the same functional-analytic and regularisation properties as second-order
TGV. We refer to \cite{ranftl2012stereo} for an application and numerical results regarding
this regularisation approach in stereo estimation.

\paragraph{Non-local TGV.}
The concept of \emph{non-local total variation} (NLTV) \cite{gilboa2008nonlocaloperators_mh} can also be transferred
to the total generalised variation. Recall that instead of taking the derivative,
non-local total variation penalises the differences of the function
values of $u$ for each pair of points by virtue of a weight function:
\[
\NLTV(u) = \int_\Omega \int_\Omega a(x,y) \abs{u(y) - u(x)} \dd{y} \dd{x},
\]
where the weight function 
$a: \Omega \times \Omega \to [0,\infty]$ is measurable and
a.e.~bounded from below by a positive constant. We note that, alternatively, the weight function
$a$ may also be chosen as $a(x,y) = |x-y|^{-(\theta + d)}$ with $\theta \in (0,1)$ such that low-order Sobolev--Slobodeckij seminorms can be realised \cite{Nezza12_fractional_order_sobolev_mh}.
In the context of non-local total variation, $a$ allows to incorporate a-priori information 
for the image to reconstruct. For instance, if one already knows disjoint
segments $\Omega_1,\ldots,\Omega_n$ where the solution is piecewise constant, 
one can set
\[
a(x,y) =
\left\{
  \begin{array}{rl}
    c_1 & \text{if} \ x,y \in \Omega_i \ \text{for some} \ i, \\
    c_0 & \text{else},
  \end{array}
\right.
\]
where $c_1 \gg c_0 > 0$. This way, the difference between two function
values of $u$ in $\Omega_i$ is forced to $0$, meaning $u$ is constant
in $\Omega_i$.

Non-local total generalised variation now gives the possibility 
to enforce piecewise linearity of $u$ in the segments
by incorporating the vector field
$w$ corresponding to the slope of the linear part in a non-local cascade.
This results in
\[
\fl
\begin{array}{rl}
  \displaystyle \NLTGV^2(u) = \inf_{w \in \LPspace{1}{\Omega,\RR^d}} 
  & \displaystyle\int_\Omega \int_\Omega a_1(x,y) \abs{u(y) - u(x) - w(x) \inprod (y-x)} 
  \dd{y} \dd{x} \\[\medskipamount]
  &
  \displaystyle+ \int_\Omega \int_\Omega a_0(x,y) \abs{w(y) - w(x)} \dd{y} \dd{x}.
\end{array}
\]
with two weight functions $a_0, a_1: \Omega \times \Omega \to [0,\infty]$,
again measurable and bounded a.e.~away from zero \cite{ranftl2014opticalflow}.
In analogy to $\NLTV$, a-priori information on, for instance, disjoint segments
where the sought solution is piecewise linear, allows to choose weight functions
such that the associated $\NLTGV^2$ regulariser properly reflect this information.
See Figure~\ref{fig:nltgv-reg} for a denoising example where non-local TGV
turns out to be beneficial, in particular in the regions near the jump discontinuities of sought solution.

\begin{figure}
  \centering%
  \begin{tabular}{c@{\ \ }c@{\ \ }c@{\ \ }c}
    \includegraphics[width=0.20\linewidth]{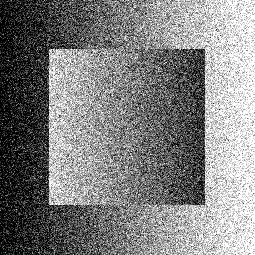}
    & 
    \includegraphics[width=0.20\linewidth]{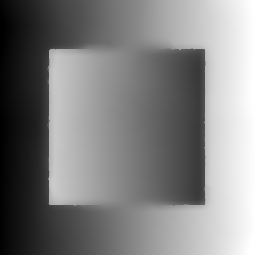}
    \includegraphics[width=0.17\linewidth]{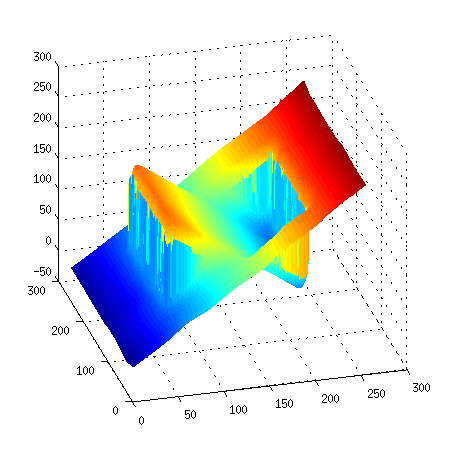}
    &
    \includegraphics[width=0.20\linewidth]{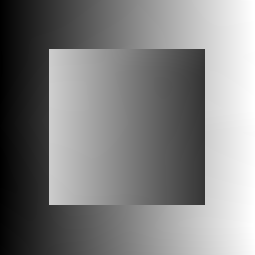}
    \includegraphics[width=0.17\linewidth]{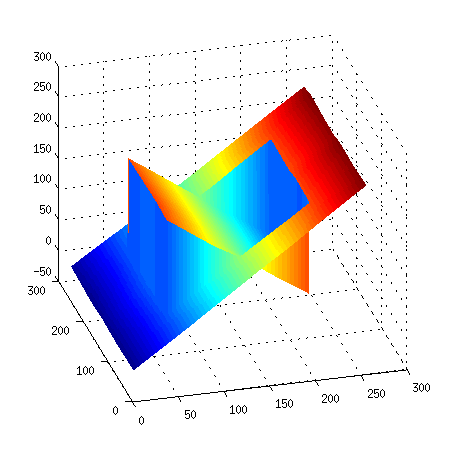}
    \\
    (a) & (b) &
    (c) 
  \end{tabular}
  \caption{Example for non-local total generalised variation
    denoising.  (a) An noisy piecewise linear image. (b) Results of
    $\TGV_\alpha^2$-denoising together with a surface plot of its
    graph.  (c) Results of non-local TGV-denoising together with a
    surface plot of its graph. Images taken from \cite{ranftl2014opticalflow}.
    Reprinted by permission from Springer Nature.}
  \label{fig:nltgv-reg}
\end{figure}

\section{Numerical algorithms}
\label{sec:numerical_algorithms}

Tikhonov regularisation with higher-order total variation, its
combination via addition or infimal convolution, as well as total
generalised variation poses a non-smooth optimisation problem in an
appropriate Lebesgue function space. In practice, these minimisation
problems are discretized and solved by optimisation algorithms that
exploit the structure of the discrete problem. While there are many
possibilities for a discretization of the considered regularisation
functionals as well as for numerical optimisation, most of the algorithms
that can be found in the literature base on finite-difference
discretization and first-order proximal optimisation methods.  In the
following, we provide an overview of the building blocks necessary to
solve the considered Tikhonov functional minimisation problems
numerically.  We will exemplarily discuss the derivation of respective
algorithms on the basis of the popular primal-dual algorithm with
extragradient \cite{pock2011primaldual_mh} and, as alternative, 
briefly address implicit and preconditioned optimisation methods.

\subsection{Discretization of higher-order TV functionals} \label{sec:discretization}

We discretize the discussed
functionals in 2D, higher dimensions follow by analogy.
Moreover, for the sake of simplicity, we assume a rectangular domain,
i.e., $\Omega = {]{0,N_1}[} \times {]{0,N_2}[} \subset \RR^2$ for some
positive $N_1,N_2 \in \NN$. A generalisation to non-rectangular
domains will be straightforward.

Following essentially the presentation in \cite{bredies2010tgv}
we first replace $\Omega$ by the discretized grid
\[
\Omega_h = \set{(i,j)}{i,j \in \NN, \ 1 \leq i \leq N_1, \ 1 \leq j \leq N_2}.
\]
One consistent way of discretizing higher-order derivatives is to define
partial derivatives as follows: A discrete partial derivative takes
the difference between two neighbouring elements in the grid with
respect to a specified axis. This difference is associated with the
midpoint between the two grid elements, resulting in staggered grids.
For a finite sequence of directions $p \in \bigcup_{k \geq 0} \sett{1,2}^k$,
this results, on the one hand, in the recursively defined grids
\begin{eqnarray}
  \label{eq:discrete_grid_recursion}
  \begin{array}{rl}
    \displaystyle \Omega_h^{()}
    &= \Omega_h, \\
    \displaystyle \Omega_h^{(1, p_k,\ldots,p_{1})}
    &= \set{(i+\frac12, j)}{(i,j), (i+1,j) \in \Omega_h^{(p_k,\ldots,p_1)}} \\
    \displaystyle \Omega_h^{(2, p_k,\ldots,p_{1})}
    &= \set{(i, j+\frac12)}{(i,j), (i,j+1) \in \Omega_h^{(p_k,\ldots,p_1)}}
  \end{array}
\end{eqnarray}
Note that $\Omega_h^p$ does not depend on the order of the $p_i$ and
one could use multiindices in $\NN^2$ instead. Likewise, the discrete
partial derivatives recursively given by
\begin{eqnarray}
  \label{eq:discrete_derivative_recursion}
  \begin{array}{rl}
    \partial^{()} u &= u, \\
    (\partial^{(1,p_k,\ldots,p_1)}_h u)_{i, j} &= (\partial^{(p_k,\ldots,p_1)}_h u)_{i+\frac12, j} - (\partial^{(p_k,\ldots,p_1)}_h u)_{i-\frac12, j}, \\
    (\partial^{(2,p_k,\ldots,p_1)}_h u)_{i, j} &= (\partial^{(p_k,\ldots,p_1)}_h u)_{i,j+\frac12} - (\partial^{(p_k,\ldots,p_1)}_h u)_{i,j-\frac12},
  \end{array}
\end{eqnarray}
yield well-defined functions $\partial_h^p u: \Omega_h^p \to \RR$ for
$u: \Omega_h \to \RR$ which do not depend on the order of the entries
in $p$.  The discrete gradient $\grad_h^k u$ of order
$k \geq 1$ for $u: \Omega_h \to \RR$ is then the tuple that collects
all the partial derivatives of order $k$:
\[
  \grad_h^k u = (\partial_h^p u)_{p \in \sett{1,2}^k}
  \in  \bigtimes_{p \in \sett{1,2}^k}\{ u_p: \Omega_h^p \rightarrow \RR\}.
\]
Note that due to this construction, the partial derivatives
$\partial_h^p u$ are generally defined on different grids.
However, in order to define the Frobenius norm of $\grad^k_h$, and,
consequently, an $\ell^1$-type norm, a common grid is needed. There are
several possibilities for this task (such as interpolation) which has
been studied mainly for the first-order total variation. Here, we
discuss a %
strategy
that results in a simple definition of a
discrete higher-order total variation. It bases on collecting,
for $(i,j) \in \ZZ^2$, all the nearby points in the different
$\Omega_h^p$. This can, for instance, be done by moving
half-steps forward and backward in the directions indicated by
$p \in \sett{1,2}^k$:
\begin{equation}
  \label{eq:discrete_point_shift}
  (i_p, j_{p}) = (i,j)
  + \frac12 \sum_{m=1}^k (-1)^{m+1} e_{p_m}
\end{equation}
where $e_1, e_2$ are the unit vectors in $\RR^2$. The Frobenius norm
in a point $(i,j) \in \ZZ^2$ is then given by
\begin{equation}
  \label{eq:discrete_deriv_frob}
  \abs{\grad_h^k u}_{i,j} = \Bigl( \sum_{p \in \sett{1,2}^k}
  \abs{(\partial_h^p u)_{i_p,j_p}}^2 \Bigr)^{1/2},
\end{equation}
where $\partial_h^p u$ is extended by zero outside of
$\Omega_h^p$. Note that here, although $\partial^p_h$ does not depend
on the order of discrete differentiation, the point $(i_p, j_p)$
does. Thus, a different Frobenius norm for the $k$-th discrete
derivative would be constituted by symmetrisation, which means
symmetrising $\grad_h^k u$ and taking the Frobenius norm
afterwards. In this context, is makes sense %
to average over the grid points as follows.  Denoting by
$\alpha(p) \in \NN^2$ the multiindex associated with
$p \in \sett{1,2}^k$, i.e., $\alpha(p)_i = \#\set{m}{p_m = i}$, we
define
\begin{equation}
  \label{eq:discrete_point_average}
  (i_\alpha, j_\alpha) = \binom{\abs{\alpha}}{\alpha}^{-1}
  \sum_{p \in \sett{1,2}^{\abs{\alpha}}, \alpha(p) = \alpha} (i_p, j_p)
\end{equation}
for $(i,j) \in \ZZ^2$ and $\alpha \in \NN^2$, where 
$\binom{\abs{\alpha}}{\alpha}=
\frac{(\alpha_1+\alpha_2)!}{\alpha_1!\alpha_2!}$.
Then, the grid
associated with an $\alpha \in \NN^2$ reads as
\[
  \fl
  \quad
  \Omega_{h}^\alpha = \set{(i_\alpha,j_\alpha)}{(i,j) \in
    \ZZ^2, \ (i_p,j_p) \in \Omega_h^p \ \text{for some}
    \ p \in \{1,2\}^k \ \text{with} \ \alpha(p) = \alpha},
\]
while the $\alpha$-component of the symmetrised gradient is given by
\[
  (\interleave \partial^\alpha_h u)_{i_\alpha,j_\alpha} =
  \binom{\abs{\alpha}}{\alpha}^{-1} \sum_{p \in
    \sett{1,2}^{\abs{\alpha}}, \alpha(p) = \alpha} (\partial^p_h u)_{i_p, j_p}
\]
where $(i,j) \in \ZZ^2$ is chosen such that
$(i_\alpha,j_\alpha) \in \Omega_h^\alpha$ and
$(\partial^p_h u)_{i_p, j_p}$ is zero for points outside of
$\Omega_h^p$. This results in the symmetrised derivative as follows:
\[
  \symgrad_h^k u = (\interleave \partial^\alpha_h u)_{\alpha \in
    \NN^2, \abs{\alpha}=k} \in \bigtimes_{\alpha \in \NN^2,
    \abs{\alpha} = k} \sett{u_\alpha: \Omega_h^\alpha \to \RR}.
\]
The Frobenius norm of $\symgrad_h^k u$ in a point $(i,j) \in \ZZ^2$
can finally be obtained by
\begin{equation}
  \label{eq:discrete_symgrad_frob}
  \abs{\symgrad^k_h u}_{i,j} = \Bigl( \sum_{\alpha \in \NN^2, \abs{\alpha} =
    k} \binom{\abs{\alpha}}{\alpha} \abs{(\interleave
    \partial^\alpha_h u)_{i_\alpha,j_\alpha}}^2  \Bigr)^{1/2}.
\end{equation}

\begin{remark}
  For $\alpha \in \NN^2$ with $\abs{\alpha}$ even, we have
  $(i_\alpha,j_\alpha) = (i,j)$ for each $(i,j) \in \ZZ$. Indeed, for
  $p \in \sett{1,2}^{\abs{\alpha}}$ and the reversed tuple
  $\bar p = (p_{\abs{\alpha}},\ldots,p_1)$ it holds
  $\alpha(p) = \alpha(\bar p)$.  Further, either $p = \bar p$
  leading to $(i_p,j_p) = (i,j)$ or $p \neq \bar p$ leading to
  $(i_p,j_p) + (i_{\bar p}, j_{\bar p}) = 2(i,j)$. Consequently,
  $(i_\alpha,j_\alpha) = (i,j)$ according to the definition. In other
  words, the symmetrisation of the discrete gradient is a natural way
  of aligning the different grids $\Omega_h^p$ to a common grid in
  this case.

  For $\abs{\alpha}$ odd, the grid points still do not align. However,
  we can say that for $(i,j)$, the point $(i_\alpha,j_\alpha)$ lies on
  the line connecting $(i+\frac12, j)$ and $(i, j+\frac12)$. Indeed,
  for $p \in \sett{1,2}^{\abs{\alpha}}$ with $\alpha(p) = \alpha$ we
  can consider
  $\bar p = (p_{\abs{\alpha}-1}, \ldots, p_1, p_{\abs{\alpha}})$. If
  $p = \bar p$, then $(i_p,j_p)$ is either $(i+\frac12, j)$ or
  $(i, j+\frac12)$. In the case $p \neq \bar p$, the point
  $\frac12(i_p,j_p) + \frac12(i_{\bar{p}}, j_{\bar{p}})$ is either
  $(i+\frac12,j)$ or $(i, j+\frac12)$. As $(i_\alpha, j_{\alpha})$ is
  a convex combination of such points, it lies on the line connecting
  $(i+\frac12, j)$ and $(i, j+\frac12)$.  Hence, the symmetrisation of
  the gradient leads to more localised grid points.
\end{remark}

We now have everything at hand to define two versions of a discrete
total variation of arbitrary order.

\begin{definition}
  Let $k \in \NN$, $k \geq 1$ a differentiation order. Then,
  for $u: \Omega_h \to \RR$, the discrete total variation is defined as
  \[
    \TV_h^k(u) = \norm[1]{\grad_h^k u} = \sum_{(i,j) \in \ZZ^2}
    \abs{\grad_h^k u}_{i,j}
  \]
  with $\abs{\grad_h^k u}_{i,j}$ according
  to~\eqref{eq:discrete_deriv_frob}, and discrete total variation for
  the symmetrised gradient is defined as
  \[
    \TV_{h,\sym}^k(u) =
    \norm[1]{\symgrad_h^k u} = \sum_{(i,j) \in \ZZ^2}
    \abs{\symgrad_h^k u}_{i,j},
  \]
  with $\abs{\symgrad_h^k u}_{i,j}$ according
  to~\eqref{eq:discrete_symgrad_frob}.
\end{definition}
In order to define a discrete version of the total generalised
variation, we still need to discuss the discretization of the total
deformation for discrete symmetric tensor fields. For this purpose, we
say that the components of a discrete symmetric tensor field of order
$l \in \NN$, live on the grids $\Omega_h^\alpha$, resulting in
\[
  u = (u_\alpha)_{\alpha \in \NN^2, \abs{\alpha} = l} \in
  \bigtimes_{\alpha \in \NN^2, \abs{\alpha} = l}  \sett{u_\alpha: \Omega_h^\alpha \to \RR},
\]
realising a discrete symmetric tensor field of order $l$. Its Frobenius norm
is given in the points $(i,j) \in \ZZ^2$ according to
\begin{equation}
  \label{eq:discrete_sym_tensor_frob}
  \abs{u}_{i,j} = \Bigl( \sum_{\alpha \in \NN^2, \abs{\alpha} = l}
  \binom{\abs{\alpha}}{\alpha} \abs{(u_\alpha)_{i_\alpha,
      j_\alpha}}^2 \Bigr)^{1/2},
\end{equation}
which is compatible with~\eqref{eq:discrete_symgrad_frob} if one plugs
in $\symgrad_h^l u$ for some $u: \Omega_h \to \RR$.  The partial
derivative of order $k$ described by $p \in \sett{1,2}^k$ applied to
$u_\alpha$ is then also given
by~\eqref{eq:discrete_derivative_recursion}, but acts on the grid
$\Omega_h^\alpha$ and results in a discrete function on the grid
$\Omega_h^{\alpha,p}$ which given 
in analogy
to~\eqref{eq:discrete_grid_recursion} by replacing $\Omega_h$ with
$\Omega_h^\alpha$. The symmetrised derivative $\symgrad_h^k u$, whose
components are indexed by $\beta \in \NN^2$, $\abs{\beta} = k+l$, is
then defined in a point $(i_{\beta},j_{\beta}) \in \Omega_h^\beta$
where $(i,j)\in\ZZ^2$ by
\begin{equation}
  \label{eq:discrete_symgrad_tensor}
  \fl
  \bigl( (\symgrad_h^k u)_\beta \bigr)_{i_\beta,j_\beta} =
  \binom{\abs{\beta}}{\beta}^{-1}
  \sum_{\alpha \in \NN^2, p \in \sett{1,2}^k, \alpha + \alpha(p) =
    \beta} \binom{\abs{\alpha}}{\alpha} (\partial^p_h u_\alpha)_{i_{\alpha,p}, j_{\alpha,p}}
\end{equation}
where
\[
  (i_{\alpha,p},j_{\alpha,p}) = (i_\alpha, j_\alpha) +
  \frac12 \sum_{m=1}^k (-1)^{l+m+1} e_{p_m}.
\]
This is sufficient to define a discrete
total deformation.

\begin{definition}
  Let $k,l \in \NN$, $k \geq 1$ and $l \geq 0$. Then, for
  $u = (u_\alpha)_{\alpha \in \NN^2, \abs{\alpha} = l}$, the discrete total
  deformation of order $k$ is defined as
  \[
    \TD_h^k(u) = \norm[1]{\symgrad_h^k u} = \sum_{(i,j) \in \ZZ^2}
    \abs{\symgrad_h^k u}_{i,j},
  \]
  with $\symgrad_h^k u$ according to~\eqref{eq:discrete_symgrad_tensor}
  and $\abs{\placeholder}_{i,j}$ according
  to~\eqref{eq:discrete_sym_tensor_frob}.
\end{definition}
For the sake of completeness, the respective definitions for
non-symmetric tensor fields of order $l$ read as
  \begin{equation}
  \fl
  u = (u_p)_{p \in \sett{1,2}^l} \in \bigtimes_{p \in \sett{1,2}^l}
  \sett{u_p: \Omega_h^p \to \RR}, \qquad \abs{u}_{i,j} = \Bigl( \sum_{p \in \sett{1,2}^l}
  \abs{(u_p)_{i_p,j_p}}^2 \Bigr)^{1/2},\label{eq:discrete_tensor_abs}
\end{equation}
and the $p$-th component, $p \in \sett{1,2}^{k+l}$, of the discrete
gradient of order $k$ is given as
\begin{equation}
  \label{eq:discrete_grad_tensor}
  (\grad_h^k u)_p = \partial_h^{(p_{k+l}, \ldots, p_{l+1})}
  u_{(p_l,\ldots,p_1)}.
\end{equation}

For numerical algorithms, it is necessary to write the discrete total
variation and total deformation as the $1$-norm of a (symmetric)
tensor with respect to the respective discrete differentiation operator.
We therefore introduce the underlying spaces.

\begin{definition}
  \label{def:discrete_spaces}
  Let $l \in \NN$, and $q \in [1,\infty]$. The $\ell^q$-space of
  discrete $l$-tensors on $\Omega_h$ is given by
  \[
    \ell^q(\Omega_h, \tensorspace[l]{\RR^2})
    = \set{u}{u = (u_p)_{p \in \sett{1,2}^l}, \ u_p: \Omega_h^p \to \RR \ \text{for all} \ p\in\sett{1,2}^l}
  \]
  with $\Omega_h^p$ according to~\eqref{eq:discrete_grid_recursion}, and norm
  \[
    \begin{array}{rlrrl}
      \norm[q]{u}
      &= \displaystyle \Bigl( \sum_{(i,j) \in \ZZ^2} \abs{u}_{i,j}^q\Bigr)^{1/q}
      & \text{if} \ q < \infty,
      &
        \quad
        \norm[\infty]{u} &= \displaystyle \max_{(i,j) \in \ZZ^2} \abs{u}_{i,j},
    \end{array}
  \]
  with pointwise norm according to~\eqref{eq:discrete_tensor_abs}.
  The space $\ell^2(\Omega_h, \tensorspace[l]{\RR^2})$ is equipped with
  the scalar product
  \[
    \scp[\ell^2(\Omega_h, {\tensorspace[l]{\RR^2})}]{u}{v} = \sum_{(i,j)
      \in \ZZ^2} \sum_{p \in \sett{1,2}^l} (u_p)_{i_p,j_p}
    (v_p)_{i_p,j_p}.
  \]
  Analogously, the $\ell^q$-space of discrete symmetric $l$-tensors on
  $\Omega_h$ is defined as
  \[
    \fl
    \ell^q(\Omega_h, \Sym^l(\RR^2))
    = \set{u}
    {u= (u_\alpha)_{\alpha \in \NN^2, \abs{\alpha} = l}, \ 
      u_\alpha: \Omega_h^\alpha \to \RR \ \text{for all} \ \alpha \in \NN^2, 
    \abs{\alpha} = l}
  \]
  with an analogous norm using~\eqref{eq:discrete_sym_tensor_frob} as
  pointwise norm. The scalar product on
  $\ell^2(\Omega_h, \Sym^l(\RR^2))$ is given by
  \[
    \scp[\ell^2(\Omega_h, {\Sym^l(\RR^2)})]{u}{v} = \sum_{(i,j) \in
      \ZZ^2} \sum_{\alpha \in \NN^2, \abs{\alpha} = l}
    \binom{\abs{\alpha}}{\alpha}
    (u_\alpha)_{i_\alpha,j_\alpha} (v_\alpha)_{i_\alpha,j_\alpha}.
  \]
\end{definition}
For $k \in \NN$, Equation~\eqref{eq:discrete_symgrad_tensor} then
defines a linear operator mapping
\[
  \symgrad_h^k: \ell^2(\Omega_h, \Sym^l(\RR^2)) \to \ell^2(\Omega_h,
  \Sym^{k+l}(\RR^2))
\]
and~\eqref{eq:discrete_grad_tensor} induces a linear operator mapping
\[
  \grad_h^k: \ell^2(\Omega_h, \tensorspace[l]{\RR^2}) \to
  \ell^2(\Omega_h, \tensorspace[k+l]{\RR^2}).
\]
The norm of these operators can easily be estimated:
\begin{lemma} \label{lem:norm_estimate_grad_symgrad}
  We have $\norm{\grad_h^k} \leq 8^{k/2}$ and
  $\norm{\symgrad_h^k} \leq 8^{k/2}$ independent of $l$.
\end{lemma}

\begin{proof}
  As $\grad^k_h = \grad_h \cdots \grad_h$ on the respective discrete
  tensor fields, it is sufficient to prove the statement for $k=1$ and
  $l$ arbitrary.  For this purpose, observe that for
  $u: \Omega_h^p \to \RR$, $p \in \sett{1,2}^l$, we have
  \begin{eqnarray*}
    \norm[2]{\partial_h^{(1,p_l,\ldots,p_1)}u}^2 = 
    \sum_{(i,j) \in \ZZ^2}
    \abs{\partial_h^{(1,p_l,\ldots,p_1)}u}_{i_p+1/2,j_p}^2
    &\leq 2
      \sum_{(i,j) \in \ZZ^2} \abs{u}^2_{i+1,j} + \abs{u}^2_{i,j}
      \\ &\leq 4
    \sum_{(i,j) \in \ZZ^2} \abs{u}_{i,j}^2 =4 \norm[2]{u}^2,
  \end{eqnarray*}
  and an analogous estimate for
  $\norm[2]{\partial_h^{(2,p_l,\ldots,p_1)}u}^2$.  Consequently,
  $\norm[2]{\grad_h u}^2 \leq 8\norm[2]{u}^2$ for such $u$. If
  $u \in \ell^2(\Omega_h,\tensorspace[l]{\RR^2})$, then
  \[
    \norm[2]{\grad_h u}^2 = \sum_{p \in \sett{1,2}^l} \norm[2]{\grad_h
      u_p}^2 \leq 8 \sum_{p \in \sett{1,2}^l} \norm[2]{u_p}^2 = 8
    \norm[2]{u}^2,
  \]
  so the claim follows.

  For the symmetrised gradient, it is possible to pursue the same
  strategy since $\symgrad_h^k = \symgrad_h \cdots \symgrad_h$ on the
  respective discrete symmetric tensor fields. Indeed, with the
  Cauchy--Schwarz inequality and Vandermonde's identity (which reduces to the standard recurrence relation for binomial coefficients in most cases), one obtains
  \begin{eqnarray*}
    \abs{\symgrad_h u}_{i,j}^2
    &= \sum_{\beta
      \in \NN^2, \abs{\beta} = l + 1} \binom{\abs{\beta}}{\beta}
      \bigabs{\bigl( (\symgrad_h u)_\beta \bigr)_{i_\beta,j_\beta}}^2 \\
    &\leq
      \sum_{\beta
      \in \NN^2, \abs{\beta} = l + 1}
      \ \ \sum_{\alpha \in \NN^2,
      \abs{\alpha} = l, p \in \{1,2\}, \alpha + \alpha(p) = \beta}
      \binom{\abs{\alpha}}{\alpha}
      \abs{(\partial^p_h u_\alpha)_{i_{\alpha,p},j_{\alpha,p}}}^2 \\
    &= \sum_{p=1}^2 \sum_{\abs{\alpha} = l} \binom{\abs{\alpha}}{\alpha}
      \abs{(\partial^p_h u_\alpha)_{i_{\alpha,p},j_{\alpha,p}}}^2 =
      \sum_{p=1}^2 \abs{\partial^p_h u}^2_{i_p, j_p}.
  \end{eqnarray*}
  It is then easy to see that
  $\norm[2]{\symgrad_h u}^2 \leq \sum_{p=1}^2 \norm[2]{\partial_h^p
    u}^2 \leq 8 \norm[2]{u}^2$.
\end{proof}

\begin{remark}
  For $p \in \sett{1,2}^l$ and $p_0 \in \sett{1,2}$, consider the
  discrete partial derivative
  $\partial_h^{p_0}: \Omega_h^p \to \Omega_h^{(p_0,p)}$ and its
  negative adjoint $\partial_{h,0}^{p_0}$, i.e.,
  $\scp{\partial_h^{p_0} u}{v} = - \scp{u}{\partial_{h,0}^{p_0}
    v}$ for $u: \Omega_h^p \to \RR$, $v: \Omega_h^{(p_0,p)} \to
  \RR$. For $u: \Omega_h^{(p_0,p)} \to \RR$, this results in
  \[
    (\partial_{h,0}^{p_0} u)_{i_p,j_p} =
    \left\{
      \begin{array}{cc}
        u_{i_p+\frac12,j_p} - u_{i_p-\frac12,j_p} & \text{if} \ p_0 = 1,\\
        u_{i_p,j_p+\frac12} - u_{i_p,j_p-\frac12} & \text{if} \ p_0 = 2,
      \end{array}
    \right.
  \]
  for $(i_p,j_p) \in \Omega_h^{p}$, where $u$ is extended by $0$
  outside of $\Omega_h^{(p_0,p)}$. (In contrast, $\partial_h^{p_0}u$
  is only defined for $(i_p,j_p) \in \Omega_h^{(p_0,p_0,p)}$.)

  Consequently, the negative adjoint of the discrete gradient induces
  a divergence for discrete tensor fields
  $u \in \ell^2(\Omega_h, \tensorspace[l+1]{\RR^2})$ such that
  $\divergence_h u \in \ell^2(\Omega_h, \tensorspace[l]{\RR^2})$ and
  \[
    (\divergence_h u)_p = \partial^{1}_{h,0} u_{(1,p)} +
    \partial^{2}_{h,0} u_{(2,p)}
  \]
  for $p \in \sett{1,2}^l$. For the discrete divergence that arises as
  the negative adjoint of the symmetrised gradient
  on $\ell^2(\Omega_h, \Sym^l(\RR^d))$, %
  one has to take the symmetrisation into
  account: For $u \in \ell^2(\Omega_h,
  \Sym^{l+1}(\RR^2))$, we have
  \[
    (\divergence_h u)_{\alpha} = \partial_{h,0}^1 u_{\alpha + e_1} +
    \partial_{h,0}^2 u_{\alpha + e_2}.
  \]
  Here, the operators
  $\partial_{h,0}^{p_0}$ act on functions on the grid
  $\Omega_{h}^{\alpha +
    \alpha(p_0)}$ and yield functions on the grid
  $\Omega_h^\alpha$.  Note that in the grid point
  $(i_\alpha,j_\alpha)$, these partial derivatives have to be
  evaluated in the grid points $(i_{\alpha + e_1} + \frac12 (-1)^l,
  j_{\alpha + e_1})$ and $(i_{\alpha + e_2}, j_{\alpha + e_2} + \frac12
  (-1)^l)$, respectively.  This way, the discrete divergence operator
  is consistently defined.
\end{remark}

\subsection{A general saddle-point framework}
\label{subsec:general_saddle_point}

Having appropriately discretized versions of higher-order regularisation functionals available, we now deal with the numerical solution of corresponding Tikhonov approaches for inverse problems. To this aim, we first consider a general framework and then derive concrete realisations for different regularisation approaches.

Let $\Omega_h$ be the discretized grid of Subsection \ref{sec:discretization} and define $U_h = \ell^2(\Omega_h)$. We assume a discrete linear forward operator $K_h: U_h \rightarrow Y_h$, with $(Y_h,\|\cdot \|_{Y_h})$ a finite-dimensional Hilbert space, and a proper, convex, lower semi-continuous and coercive discrepancy term $S_{f_h}:Y_h \rightarrow [0,\infty]$ with corresponding discrete data $f_{h}$ to be given.
Further, we define $\mR_\alpha :U_h \rightarrow [0,\infty]$ to be a regularisation functional given in a general form as $\mR_\alpha(u) = \min _{w \in W_h} \| D_h (u,w)\|_{1,\alpha}$, with $D_h :U_h \times W_h \rightarrow V_h$, $(u,w) \mapsto D_h^1 u + D_h^2 w$ a discrete differential operator and $W_h, V_h$ finite-dimensional Hilbert spaces. The expression $\|\cdot \|_{1,\alpha}$ here denotes an appropriate $\ell^1$-type norm weighted using the parameters $\alpha$ and will be specified later for concrete examples. Its dual norm is denoted by $\norm[\infty,\alpha^{-1}]{\placeholder}$.
We consider the general minimisation problem
\begin{equation} \label{eq:discrete_min_prob_single_variable}
\min _{ u \in U_h } \,   S_{f_h} (K_hu) + \mR_\alpha(u) ,
\end{equation}
for which we will numerically solve the equivalent reformulation
\begin{equation} \label{eq:discrete_min_prob_general}
\min _{ (u,w) \in U_h \times W_h} \,   S_{f_h} (K_hu) + \| D_h (u,w)\|_{1,\alpha}.
\end{equation}

\begin{remark}
  Note that here, the auxiliary variable $w$ and the space $W_h$ are used to include balancing-type regularisation approaches such as the infimal convolution of two functionals. Setting, for example, $W_h = U_h$, $V_h = \ell^2(\Omega, \tensorspace[1]{\RR^2}) \times \ell^2(\Omega,\tensorspace[2]{\RR^2})$, $D_h(u,w) = (\nabla _h u - \nabla _h w,\nabla _h ^2 w )$ and $\| (v_1,v_2) \|_{1,\alpha} = \alpha_1 \|v_1\|_1 + \alpha_2 \|v_2\|_1$ for positive $\alpha = (\alpha_1, \alpha_2)$ yields
\[ \mR_\alpha (u) = \min_{w \in W_h} \ \alpha_1 \| \nabla _h u - \nabla _h w\|_1 + \alpha_2 \|\nabla_h^2 w\|_1 = \left( \alpha _1 \TV\infconv \alpha_2 \TV^2 \right) (u).\]
Total-variation regularisation can, on the other hand, be obtained by choosing $W_h = \{0\}$,  $V_h = \ell^2(\Omega_h, \tensorspace[1]{\RR^2})$, $D_h(u,0) = \nabla _h u$ and $\| v \|_{1,\alpha} = \alpha \|v\|_1$ for $\alpha > 0$.
\end{remark}

\begin{remark}
  \label{rem:norm_fenchel_dual}
  The dual norm $\norm[\infty,\alpha^{-1}]{\placeholder}$ will become
  relevant in the context of primal-dual algorithms via its Fenchel dual. Indeed, we have
  the identity
  $(\norm[X]{\placeholder})^* =
  \mI_{\sett{\norm[X^*]{\placeholder} \leq
      1}}$
  for
  $\norm[X]{\placeholder}$ the norm of a general
  Banach space $X$ and $\norm[X^*]{\placeholder}$ its dual norm on
  $X^*$.
  This is a consequence of
  $\scp{w}{u} - \norm[X^*]{w} \norm[X]{u} \leq 0$ for all $u \in X$ and $w \in X^*$, so
  $(\norm[X]{\placeholder})^*(w) = 0$ for $\norm[X^*]{w} \leq 1$. For
  $\norm[X^*]{w} > 1$ one can find a $u \in X$, $\norm[X]{u} \leq 1$
  such that, for a $c > 0$, $\scp{w}{u} \geq 1+c \geq \norm[X]{u} + c$.
  For each $t > 0$, we get
  $\scp{w}{tu} - \norm[X]{tu} \geq tc$, hence
  $(\norm[X]{\placeholder})^*(w) = \infty$.
\end{remark}

\begin{remark}
  While the setting of \eqref{eq:discrete_min_prob_single_variable} allows to include rather general forward operators $K_h$ and discrepancy terms $S_{f_h}$, it will still not capture all applications of higher-order regularisation that we consider later in Subsections \ref{sec:applications_imaging_computer_vision} and \ref{sec:applications_medical_imaging_reconstruction}. It rather comprises a balance between general applicability and uniform presentation, and we will comment on possible extensions later on such that the interested reader should be able to adapt the setting presented here to the concrete problem setting at hand.
\end{remark}

From a more general perspective, the reformulation~\eqref{eq:discrete_min_prob_general} of~\eqref{eq:discrete_min_prob_single_variable} constitutes a non-smooth, convex optimisation problem of the form
\begin{equation} \label{eq:min_problem_vectorized_general}
\min _{x \in \mX}\, \mF(\mK x) + \mG(x),
\end{equation}
with $ \mX,\mY$ Hilbert spaces, $\mF:\mY \rightarrow [0,\infty]$, $\mG:\mX \rightarrow [0,\infty]$ proper, convex and lower semi-continuous functionals and $\mK:\mX \rightarrow \mY$ linear and continuous. For this class of problems, duality-based first-order optimisation algorithms of ascent/descent-type have become very popular in the past years as they are rather generally applicable and yield algorithms for the solution of \eqref{eq:min_problem_vectorized_general} that provably converge to a global optimum, while allowing
a simple implementation and practical stepsize choices, such as constant stepsizes.
The algorithm of \cite{chambolle2004algorithm_mh}, for instance, constitutes a relatively early step in this direction, as it solves the TV-denoising problem with constant stepsizes %
in terms of a dual problem.

For problems of the type \eqref{eq:min_problem_vectorized_general},
it is often beneficial to consider a primal-dual saddle-point reformulation instead of the dual problem alone, in particular in view of general applicability.
This is given as 
\begin{equation} \label{eq:saddle_point_problem_vectorized}
\min _{x \in \dom \mG} \  \max_{y \in \dom \mF^*} \,\langle \mK x ,y \rangle _{\mY}+ \mG(x) - \mF^*(y),
\end{equation}
with $\langle \cdot ,\cdot \rangle _{\mY}$ denoting the inner product in $\mY$.
By interchanging minimum and maximum and minimising with respect to $x$,
one
further arrives at the %
dual problem which reads as
\begin{equation}
  \label{eq:dual_problem_vectorize_general}
  \max_{y \in \mY} \ -\mF^*(y) - \mG^*(-\mK^*y).
\end{equation}
Under certain conditions, the minimum in~\eqref{eq:min_problem_vectorized_general} and maximum in~\eqref{eq:dual_problem_vectorize_general} admit the same value and primal/dual solution pairs for~\eqref{eq:min_problem_vectorized_general} an~\eqref{eq:dual_problem_vectorize_general} correspond to solutions of the saddle-point problem~\eqref{eq:saddle_point_problem_vectorized}, see below. 

Now, indeed, many different algorithmic approaches for solving \eqref{eq:saddle_point_problem_vectorized} are nowadays available (see for instance \cite{Pesquet15primal_dual_review_mh,pock2011primaldual_mh,chambolle2015ergodic_primal_dual_mh,
bredies2015tvtgvpdr_mh,Bredies16accelerated_dr}) and which one of them delivers the best performance typically depends on the concrete problem instance. Here, as exemplary algorithmic framework, we consider the popular primal-dual algorithm of \cite{pock2011primaldual_mh} (see also \cite{Zhu08primal_dual_hypbrid_mh,Pock09pd_algorithm_early}), which has the advantage of being simple and yet rather generally applicable.

Conceptually, the algorithm of \cite{pock2011primaldual_mh} solves the saddle-point problem \eqref{eq:saddle_point_problem_vectorized} via implicit gradient descent and ascent steps with respect to the primal and dual variables, respectively. 
With $\mL(x,y) =  \langle \mK x ,y \rangle _{\mY}+ \mG(x) - \mF^*(y)$, carrying out these implicit steps simultaneously in both variables would correspond to computing the iterates $\seq{(x^n,y^n)}$ via 
\begin{equation} 
\left\{
\begin{array}{cl}
y^{n+1} & = \quad y^n  + \sigma \subgrad_y \mL(x^{n+1},y^{n+1}),\\
x^{n+1} & = \quad x^n  - \tau \subgrad_x \mL(x^{n+1},y^{n+1}),
\end{array}
\right.
\end{equation}
where $\subgrad _x$ and $\subgrad_y$ denotes the subgradient with respect to the first and second variable, respectively, and $\sigma$, $\tau$ are positive constants. 
To obtain computationally feasible iterations, the implicit step for $y^{n+1}$ in the primal-dual algorithm uses an extrapolation $\overline{x}^n = 2x^n - x^{n-1}$ of the previous iterate instead of $x^{n+1}$, such that the descent and ascent steps decouple and can be re-written as
\begin{equation} \label{eq:primal_dual_iterates}
\left\{
\begin{array}{cl}
y^{n+1} & = \quad (\id + \sigma \subgrad \mF^*)^{-1}(y^n + \sigma \mK\overline{x}^n),\\
x^{n+1} & = \quad (\id + \tau \subgrad \mG)^{-1}(x^n - \tau \mK^*y^{n+1}),\\
\overline{x}^{n+1} & = \quad  2 x^{n+1} - x^n.
\end{array}
\right.
\end{equation}
The mappings $(\id + \sigma \subgrad \mF^*)^{-1}$ and $(\id + \tau \subgrad \mG)^{-1}$ used here are so-called proximal mappings of $\mF^*$ and $\mG$, respectively, which, as noted in Proposition \ref{def:prox-mapping} below, are well-defined and single valued whenever $\mG,\mF^*$ are proper, convex and lower semi-continuous. 
The resulting algorithm can then be interpreted as proximal-point algorithm (see \cite{He12overrelaxedpd_mh,Rockafellar76proximal_point_mh}) and weak convergence
in the sense that $(x^n,y^n) \rightharpoonup (x^*,y^*)$ for $(x^*, y^*)$ being a solution to the saddle-point problem~\eqref{eq:saddle_point_problem_vectorized} can be ensured for positive stepsize choices $\sigma,\tau$ such that $\sigma \tau \|\mK\| ^2< 1$, see for instance
\cite{chambolle2015ergodic_primal_dual_mh,pock2011diagonal_mh}, or \cite{pock2011primaldual_mh} for the finite-dimensional case.
In %
contrast, explicit methods for non-smooth optimisation problems such as
subgradient descent, for instance, usually require stepsizes that converge to zero \cite{nesterov2004convexoptimization_mh} and could stagnate numerically.

Overall, the efficiency of the iteration steps in \eqref{eq:primal_dual_iterates} crucially depends on the ability to evaluate $\mK$ and $\mK^*$ and to compute the proximal mappings efficiently. Regarding the latter, this is possible for a large class of functionals, in particular for many functionals that are defined pointwise, which is one of the reasons for the high popularity of these kind of algorithms. We now consider proximal mappings in more detail and provide concrete examples later on. 
\begin{proposition}
  \label{def:prox-mapping}
  Let $H$ be a Hilbert space, $F: H \to {]{-\infty, \infty}]}$
  proper, convex and lower semi-continuous, and $\sigma > 0$.
  
  \begin{enumerate}
  \item 
    Then, the mapping 
    \begin{equation}
      \label{eq:prox-mapping}
      \prox_{\sigma F}: H \to H, 
      \qquad
      u \mapsto \argmin_{\bar u \in H} \ \frac{\norm[H]{\bar u - u}^2}{2}
      + \sigma F(\bar u)
    \end{equation}
    is well-defined.
  \item
     For $u \in H$, $u^* = \prox_{\sigma F}(u)$ solves
     the inclusion relation
     \[
     u \in u^* + \sigma \subgrad F(u^*),
     \]
     i.e., $\prox_{\sigma F} = (\id + \sigma \subgrad F)^{-1}$.
   \item
     The mapping $\prox_{\sigma F}$ is 
     Lipschitz-continuous with constant not exceeding
     $1$.
  \end{enumerate}
\end{proposition}

\begin{proof} See, for instance,
  \cite[Proposition IV.1.5, Corollary IV.1.3]{Showalter2013monotone_mh},
  or \cite[Proposition 12.15, Example 23.3, Corollary 23.10]{Combettes11_convex_analysis_book_mh}.
\end{proof}

In general, the computation of proximal mappings can be as difficult as solving the original optimisation problem itself. However, if, for instance, the corresponding functional can be
``well separated'' into simple building blocks, then proximal mappings can be reduced to 
some basic ones which are simple and easy to compute. 
\begin{lemma}
  \label{lem:prox-mapping-separable}
  Let $H = H_1 \perp \cdots \perp H_n$ with closed
  subspaces $H_1,\ldots, H_n$, the mappings $P_1,\ldots,P_n$ their
  orthogonal projectors, 
  \[
  F(u) = \sum_{i=1}^n F_i(P_iu)
  \]
  with each $F_i: H_i \to {]{-\infty,\infty}]}$ proper, convex
  and lower semi-continuous. Then,
  \[
  \prox_{\sigma F}(u) = \sum_{i=1}^n \prox_{\sigma F_i}(P_iu).
  \]
\begin{proof}
This is immediate since the corresponding minimisation problem decouples. 
\end{proof}
\end{lemma}
Furthermore, \emph{Moreau's identity} (see \cite{Rockafellar_convex_mh}, for instance) provides a relation between the proximal mapping of a function $F$ and the proximal mapping of its dual $F^*$ according to
\begin{equation} \label{eq:moreaus_identity}
 u = (\id + \sigma \subgrad F)^{-1}(u) + \sigma \Bigl(\id + \frac{1}{\sigma}\subgrad F^* \Bigr)^ {-1} \Bigl(\frac{u}{\sigma} \Bigr).
 \end{equation}
 This immediately implies that for general $\sigma > 0$, the computation of 
$(\id + \sigma \subgrad F)^{-1}$ is essentially as difficult as the computation of 
$(\id + \sigma \subgrad F^*)^{-1}$, in particular the latter can be obtained from the former as follows.
\begin{lemma} \label{lem:moreau_proximal}
  Let $H $ be a Hilbert space and $F: H \to {]{-\infty,\infty}]}$ be proper, convex
  and lower semi-continuous. Then, for $\sigma >0$, 
  \[
  \prox_{\sigma F}(u) = u - \sigma \prox_{\frac{1}{\sigma}F^*} \Bigl(\frac{u}{\sigma}\Bigr),
  \quad \ \ \prox_{\sigma F^*}(u) = u - \sigma \prox_{\frac{1}{\sigma}F} \Bigl(\frac{u}{\sigma}\Bigr).
  \]
\end{lemma}
In some situations, the computation of the proximal mappings of the sum of two functions decouples into the composition of two mappings.
\begin{lemma}  \label{lem:proximal_mapping_decoupling}
 Let $H $ be a Hilbert space, $F: H \to {]{-\infty,\infty}]}$ be proper, convex
  and lower semi-continuous and $\sigma >0$. 
  \begin{enumerate}
  \item If $F(u) = G(u) + \frac{\alpha}{2} \|u - u_0\|_H^2$ with $G:H\rightarrow {]{-\infty,\infty}]}$ proper, convex and lower semi-continuous, $u_0 \in H$ and $\alpha > 0$, then 
    \[
      \fl
      \prox_{\sigma F} (u) = \prox_{\frac{\sigma}{1+\sigma\alpha} G} \circ \prox_{\sigma\frac{ \alpha}{2}\|\cdot - u_0\|_H^2} (u) =  \prox_{\frac{\sigma}{1+\sigma\alpha} G} \left(\frac{u + \sigma \alpha u_0}{1+\sigma \alpha} \right) \]
  \item If $H=\R^M$ equipped with the Euclidean norm and
    $F(u) =  \sum_{m=1}^M \mI_{[a_m,b_m]}(u_m) + F_m(u_m)$ with
    $\dom(F_m) \cap [a_m,b_m] \neq \emptyset$ for each $m=1,\ldots, M$, then 
    \[ \prox_{\sigma F} (u)_m = \proj_{[a_m,b_m]}\circ  \prox_{\sigma F_m} (u_m) \]
    for $m = 1,\ldots,M$,
  where 
  \[ \proj_{[a_m,b_m]} (t) = \prox_{\sigma \mI_{[a_m,b_m]}}(t) = 
  \left\{\begin{array}{cl}
  t & \text{if} \ t \in [a_m,b_m], \\
  a_m & \text{if} \ t < a_m, \\
  b_m & \text{if} \ t >b_m,
  \end{array}\right.
  \]
  is the projection to $[a_m,b_m]$ in $\RR$. 
  \end{enumerate}
\begin{proof}
Regarding (i), we note that by first-order optimality conditions (additivity of the subdifferential follows from \cite[Proposition I.5.6]{Ekeland_book_mh}), we have the following equivalences:
\begin{eqnarray*}
u^* = \prox_{\sigma F} (u) 
&\quad  \Leftrightarrow \quad 0 \in  u^*-u + \sigma \alpha (u^*-u_0) + \sigma \subgrad G(u) \\
&\quad  \Leftrightarrow \quad 0 \in  u^*- \frac{u + \sigma \alpha u_0}{1 + \sigma \alpha } + \frac{\sigma }{1 + \sigma \alpha} \subgrad G(u) \\
&\quad  \Leftrightarrow \quad u^* = \prox_{\frac{\sigma }{1 + \sigma \alpha}  G} \left( \frac{u + \sigma \alpha u_0}{1 + \sigma \alpha } \right),
\end{eqnarray*}
which proves the explicit form of $\prox_{\sigma F}$. The intermediate equality follows from $\prox_{\sigma\frac{ \alpha}{2}\|\cdot - u_0\|_H^2} (u) =  \frac{u + \sigma \alpha u_0}{1+\sigma \alpha}$, which can again be seen from the optimality conditions.

In order to show (ii), first note that, using Lemma \ref{lem:prox-mapping-separable}, it suffices to show the assertion for $F(u) = \mI_{[a,b]}(u) + f(u)$ with $u \in \R$, $a \leq b$ and $f: \RR \to {]{-\infty,\infty}]}$ proper, convex and lower semi-continuous such that $\dom(f) \cap [a,b] \neq \emptyset$. Also, the identity $\prox_{\sigma \mI_{[a,b]}} = \proj_{[a,b]}$ as well as the explicit form of the projection is immediate. Now, denote by $u^* = \proj_{[a,b]}\circ  \prox_{\sigma f} (u) $ and write it as $u^* = \nu u_F + (1-\nu)u_f$ with $u_F = \prox_{\sigma F}(u)$ and $u_f = \prox_{\sigma f}(u)$ and $\nu \in [0,1]$. To see that this is possible, note that in case of $u_f \in [a,b]$, $u^* = u_f$, in case $u_f < a$ we have that $u_f < a = u^* \leq u_F$ and similarly in case of $u_f > b$. But, with $E(\bar u)= \frac{|\bar u-u|^2}{2} + \sigma f(\bar u)$, convexity and minimality of $u_f$ implies that
\[ E(u^*)  \leq \nu E(u_F) + (1-\nu)E(u_f) \leq \nu E(u_F) + (1-\nu)E(u_F) = E(u_F).
\]
Since both $u^*$ and $u_F$ are in $[a,b]$, the result follows from uniqueness of minimizers.
\end{proof}
\end{lemma}

In the following we provide some examples of explicit proximal mappings for some particular choices of $F$ that are relevant for applications. For additional examples and further, general results on proximal mappings, we refer to \cite{Combettes11_convex_analysis_book_mh,Pesquet11_splitting_review}.

\begin{lemma} \label{lem:proximal_mapping_examples} Let $H$ be a Hilbert space, $\sigma >0$ and $F:H \rightarrow {]{-\infty,\infty}]}$. Then, the following identities hold.
  \mbox{}
  \begin{enumerate}
  \item
    For $F(u) = \frac\alpha 2\norm[H]{u - f}^2$ with $f \in H$, $\alpha >0$, 
    \[
      \prox_{\sigma F}(u) = \frac{ u + \alpha \sigma f}{1 + \alpha \sigma},
      \qquad
    \prox_{\sigma F^*}(u) = \frac{ u - \sigma f }{1 + \sigma/\alpha}.
    \]
  \item 
    For $F = \mI_C$ for some non-empty, convex and closed set $C \subset H$,
    \[ \prox_{\sigma F} = \proj_C,\]
    with $\proj_C$ denoting the orthogonal projection onto $C$.
  \item For $F(u) = G(u-u_0)$ with $G:H \rightarrow {]{-\infty,\infty}]}$ and $u_0 \in H$,
   \[ \prox_{\sigma F}(u) = \prox_{\sigma G}(u-u_0) + u_0.\]
 \item For
   $H = \mH_1 \times \cdots \times \mH_M$ product of the Hilbert spaces $\mH_1,\ldots,\mH_M$,
   $\norm[H]{u}^2 = \sum_{m=1}^M |u_m|_{\mH_m}^2$ with each $\abs[\mH_m]{\placeholder}$
   denoting the norm on $\mH_m$, $\alpha \in {]{0,\infty}[}^M$, and
  $F = \mI_{ \{\|\cdot \|_{\infty,\alpha^{-1}} \leq 1 \}}$ where $\|u\|_{\infty,\alpha^{-1}} = \max_{m=1,\ldots,M} \alpha_m^{-1} |u_m|_\mH$,
  \[
    \prox_{\sigma F}(u)_m =  \proj_{ \{|\placeholder |_{\mH_m} \leq \alpha_m \}}(u_m)= \frac{u_m}{\max\{1, \alpha_m^{-1}\abs{u_m}_{\mH_m}\}}.
    \]
  \item %
    In the situation of (iv) and $F(u) = \sum_{m=1}^M \alpha_m  |u_m|_{\mH_m}$,
  \[ \prox_{\sigma F} (u)_m = 
  \left\{\begin{array}{cl}
  (1 - \sigma \alpha_m |u_m|_{\mH_m}^{-1}) u_m  & \text{if } |u_m|_{\mH_m} > \sigma \alpha_m, \\
  0 & \text{else, }
  \end{array}\right. \]
  and 
  \[ \prox_{\sigma F^*} (u)_m = \proj_{ \{|\placeholder |_{\mH_m} \leq \alpha_m \}}(u_m)\]
  with $\proj_{ \{|\placeholder |_{\mH_m} \leq \alpha_m \}}$ as in (iv).
  \end{enumerate}
  
\begin{proof}
  The assertion on $\prox_{\sigma F}$ in (i) follows from first-order optimality conditions as in Lemma \ref{lem:proximal_mapping_decoupling}, the assertion on $\prox_{\sigma F^*}$ is a consequence of Lemma \ref{lem:moreau_proximal}.  Assertion (ii) is immediate from the definition of the orthogonal projection in Hilbert spaces and (iii) follows from a simple change of variables for the minimisation problem in the definition of the proximal mapping. Regarding (iv), using Lemma \ref{lem:prox-mapping-separable} and noting that %
$\norm[\infty,\alpha^{-1}]{u} \leq 1$ if and only if $\abs[\mH_m]{u_m} \leq \alpha_m$ for $m=1,\ldots,M$,
  it suffices to show that for each $u_m \in \mH_m$ and $m = 1,\ldots,M$, 
\[ \proj_{ \{ |\cdot |_{\mH_m} \leq \alpha_m \} } (u_m)  =  \frac{u_m}{\max\{1, \alpha_m^{-1}\abs{u_m}_{\mH_m}\}} .\]
To this aim, %
observe that
by definition of the projection in Hilbert spaces,
\begin{eqnarray*}
 \proj_{ \{ |\cdot |_{\mH_m} \leq \alpha_m \} } (u_m) 
& = \argmin_{|\bar u|_{\mH_m} \leq \alpha_m}  |\bar u-u_m|_{\mH_m}
= \argmin_{|\bar u|^2_{\mH_m} \leq \alpha_m^2}  |\bar u -u_m|^2_{\mH_m} \\
& = \argmin_{\bar{u}_1 \in \lspan\{ u_m\}, \ \bar{u}_2 \in \lspan\{ u_m\}^\perp, \atop |\bar{u}_1|^2_{\mH_m} + |\bar{u}_2|^2_{\mH_m} \leq \alpha_m^2}  |\bar{u}_1-u_m|^2_{\mH_m} + |\bar{u}_2|^2_{\mH_m}  \\
& = u_m \Biggl( \argmin_{t\in \RR, \ |t u_m|^2_\mH \leq \alpha_m^2}  |t - 1|^2|u_m|^2_{\mH_m} \Biggr).
\end{eqnarray*}
From this, it is easy to see that the minimum in the last line is achieved for $t = 1$ in case $|u_m|_{\mH_m} \leq \alpha_m$ and $t = \alpha_m /|u_m|_{\mH_m}$ otherwise, from which the explicit form of the projection follows. Considering assertion (v),
we have by Remark~\ref{rem:norm_fenchel_dual} that
\[
  \fl
  F^*(u) = \sum_{m=1}^M (\alpha_m \abs[\mH_m]{\placeholder})^*(u_m)
  = \sum_{m=1}^M \mI_{\sett{\abs[\mH_m]{\placeholder} \leq \alpha_m}}(u_m) =
  \mI_{\sett{\norm[\infty,\alpha^{-1}]{\placeholder} \leq 1}}(u),
\]
so
the statements follow by Lemma \ref{lem:moreau_proximal} and assertion (iv).
\end{proof}
\end{lemma}

Now considering the minimisation problem \eqref{eq:discrete_min_prob_general}, our approach is to rewrite it as a saddle-point problem such that, when applying the iteration \eqref{eq:primal_dual_iterates}, the involved proximal mappings decouple into simple and explicit mappings. To achieve this while allowing for general forward operators $K_h$, we dualise both the regularisation and the data-fidelity term and arrive at the following a saddle-point reformulation of \eqref{eq:discrete_min_prob_general}:
\begin{equation} \label{eq:discrete_saddle_point_problem_general}
\fl \min _{ (u,w) \in U_h \times W_h}  \ \max_{(v,\lambda) \in V_h \times Y_h} \,  \langle D_h(u,w), v \rangle_{V_h} + \langle K_h u, \lambda \rangle_{Y_h}  
- \mI_{\sett{\norm[\infty,\alpha^{-1}]{\placeholder} \leq 1}}(v) - S_{f_h}^* (\lambda),
\end{equation}
recalling that the dual norm of $\norm[1,\alpha]{\placeholder}$ is
denoted by $\norm[\infty,\alpha^{-1}]{\placeholder}$ and that $(\|\placeholder \|_{1,\alpha})^* =  \mI_{\{ \|\placeholder \|_{\infty ,\alpha^{-1}}\leq 1 \}}$, see Remark~\ref{rem:norm_fenchel_dual}.
The following lemma provides some instances of $\norm[1,\alpha]{\placeholder}$ that arise in the context of higher-order TV regularisers and its generalisations. In particular, for these instances, the corresponding dual norm $\norm[\infty,\alpha^{-1}]{\placeholder}$ and the proximal mappings $\prox_{\sigma\mI_{\sett{\norm[\infty,\alpha^{-1}]{\placeholder} \leq 1}}} = \proj_{\sett{\norm[\infty,\alpha^{-1}]{\placeholder} \leq 1}}$ will be provided.
Concrete examples will be discussed in Example \ref{ex:prox_regularizers} below.
\begin{lemma} \label{lem:explicit_prox_for_ell2_space} With $M \in \N$, $\alpha \in {]{0,\infty}[}^M$, $l_1,\ldots,l_M \in \N$, and $\mH_m  \in \{ \tensorspace[l_m]{\RR^2}, \Sym^{l_m}(\RR^2)\}$ for $m=1,\ldots,M$, let
\[
 V_h = \bigtimes _{m=1}^{M} \ell^2(\Omega_h, \mH_m) \quad \text{and}\quad  \|v\|_{1,\alpha} = \sum_{m=1}^M \alpha_{m} \|v_m\|_1,
\]
for $v = (v_1,\ldots,v_M) \in V_h$,
where $V_h$ is equipped with the induced inner product $\langle u,v \rangle _{V_h} = \sum_{m=1}^M \langle u_m,v_m \rangle_{\ell^2(\Omega_h, \mH_m)}$ and norm, and
the $1$-norm on each $\ell^2(\Omega,\mH_M)$ is given in Definition~\ref{def:discrete_spaces}. Then,
the dual norm $\norm[\infty,\alpha^{-1}]{\placeholder}$ satisfies
 \[
 \|v\|_{\infty ,\alpha^{-1}} = \max_{m=1,\ldots,M} \ \alpha_{m}^{-1}\|v_m\|_\infty,
\]
with the $\infty$-norm again according to Definition~\ref{def:discrete_spaces}.
  Further, we have, for $m=1,\ldots,M$, that
 \begin{eqnarray*}
  \prox_{\sigma (\|\placeholder \|_{1,\alpha}) ^*}(v)_m  &=
  \proj_{ \{\|\cdot \|_{\infty,\alpha^{-1}} \leq 1 \}}  (v)_m  %
  = \frac{v_m}{\max\{1,\alpha_{m}^{-1}|v_m|_{\mH_m}\}},
 \end{eqnarray*}
 where the right-hand side has to be interpreted in the pointwise sense, i.e.,
 for $u \in \ell^2(\Omega_h,\mH_m)$ and $(i,j) \in \Omega_h$, it holds that
 $\bigl(\max\{1, \alpha_m^{-1}\abs[\mH_m]{u}\}^{-1} u \bigr)_{i,j} =
 \max\{1, \alpha_m^{-1}\abs[\mH_m]{u_{i,j}}\}^{-1} u_{i,j}$.
 \end{lemma}

 \begin{proof}
   By definition, we have for $u,v \in V_h$ with $\norm[1,\alpha]{u} \leq 1$ that
   \[
     \fl
     \scp[V_h]{v}{u} \leq \sum_{m=1}^M \alpha_m^{-1} \norm[\infty]{v_m}
     \alpha_m \norm[1]{u_m} \leq \max_{m=1,\ldots,M} \alpha_m^{-1} \norm[\infty]{v_m}
     = \max_{m=1,\ldots,M, \atop (i,j) \in \Omega_h} \alpha_m^{-1} \abs[\mH_m]{(v_m)_{i,j}},
   \]
   hence, $\norm[\infty,\alpha^{-1}]{v} \leq \max_{m=1,\ldots,M} \alpha_m^{-1} \norm[\infty]{v_m}$.
   With $(m,i,j)$
   a maximising argument of the right-hand side above,
   equality follows, in case of $v \neq 0$,
   from choosing $u$ according to $(u_m)_{i,j} = \alpha_m^{-1} (v_m)_{i,j}/\abs[\mH_m]{(v_m)_{i,j}}$ and $0$ everywhere else. The case $v = 0$ is trivial.
   Also, since $\Omega_h = \sett{1,\ldots,N_1} \times \sett{1,\ldots,N_2}$ is finite,
   one can interpret each $\ell^2(\Omega,\mH_m)$ as $\mH_m^{N_1N_2}$,
   such that Lemma~\ref{lem:proximal_mapping_examples} (v) applied to $H =V_h
   = \mH_1^{N_1N_2} \times \cdots \times \mH_M^{N_1N_2}$
   immediately yields
   the stated pointwise identity for the proximal mapping/projection.
\end{proof}

Under mild assumptions on $S_{f_h}$, equivalence of the primal problem \eqref{eq:discrete_min_prob_general} and the saddle-point problem \eqref{eq:discrete_saddle_point_problem_general} indeed holds and existence of a solution to both (as well as a corresponding dual problem) can be ensured.
\begin{proposition} \label{prop:duality_result_discrete}
  Under the assumptions stated for problem~\eqref{eq:discrete_min_prob_general},
  there exists a solution.
  Further, if $S_{f_h}$ is such that $Y_h = \bigcup_{t \geq 0} t\bigl(\dom(S_{f_h}) + \range{K_h}\bigr)$, then there exists a solution to the dual problem
\begin{equation}\label{eq:discrete_dual_problem_general}
\fl \max_{(v,\lambda) \in V_h \times Y_h}  -\mI_{ \{0\}}\bigl( (D_h^1)^* v + K_h^* \lambda \bigr) - \mI_{\{0\}} \bigl((D_h^2)^* v \bigr) - \mI_{\sett{\norm[\infty,\alpha^{-1}]{\placeholder} \leq 1}} (v) - S_{f_h}^* (\lambda)
\end{equation}
and to the saddle-point problem \eqref{eq:discrete_saddle_point_problem_general}. Further, strong duality holds and the problems are equivalent in the sense that $((u,w),(v,\lambda))$ is a solution to \eqref{eq:discrete_saddle_point_problem_general} if and only if $(u,w)$ solves \eqref{eq:discrete_min_prob_general} and $(v,\lambda)$ solves \eqref{eq:discrete_dual_problem_general}.
\begin{proof}
  At first note that %
  existence for \eqref{eq:discrete_min_prob_general} can be shown as in Theorem \ref{thm:general_reg_existence_linear}. 
  By virtue of Theorem~\ref{thm:fenchel-rockafellar-duality}, choosing $X = U_h \times W_h$, $Y = V_h \times Y_h$, $F: X \to {]{-\infty,\infty}]}$ as $F=0$, $G: Y \to {]{-\infty,\infty}]}$ as $G(v,\lambda) = \norm[1,\alpha]{v} + S_{f_h}(\lambda)$, and $\Lambda: X \to Y$ as $\Lambda(u,w) = \bigl(D_h(u,w), K_hu\bigr)$, we only need to verify~\eqref{eq:fr-duality-qual} to obtain existence of dual solutions and strong duality. But since $\dom(F) = X$ and $\dom(G) = V_h \times \dom(S_{f_h})$, the latter condition is equivalent to $Y_h = \bigcup_{t \geq 0} t \bigl(\dom(S_{f_h}) + \range{K_h} \bigr)$. Also, since $F^* = \mI_{\sett{0}}$ and $(\norm[1,\alpha]{\placeholder})^* = \mI_{\sett{\norm[\infty,\alpha^{-1}]{\placeholder} \leq 1}}$, see Remark~\ref{rem:norm_fenchel_dual}, the maximisation problem in~\eqref{eq:fr-duality} corresponds to~\eqref{eq:discrete_dual_problem_general}.
  Finally, the equivalence of the saddle-point problem \eqref{eq:discrete_saddle_point_problem_general} to \eqref{eq:discrete_min_prob_general} and~\eqref{eq:discrete_dual_problem_general} then follows from \cite[Proposition III.3.1]{Ekeland_book_mh}.
\end{proof}
\end{proposition}

\begin{remark}
  \label{rem:alternative_duals}
  In some applications, it is beneficial to add an additional penalty term on $u$ in form of $\Phi: U_h \to [0,\infty]$ proper, convex and lower semi-continuous to the energy of \eqref{eq:discrete_min_prob_general}, whereas in other situations when $u \mapsto S_{f_h} (K_hu)$ has a suitable structure, a dualization of the data term is not necessary, see the discussion below. Regarding the former, the differences when extending Proposition \ref{prop:duality_result_discrete} is that existence for the primal problem needs to be shown differently and that the domain of $\Phi$ needs to be taken into account for obtaining strong duality. Existence can, for instance, be proved when assuming that either $\Phi$ is the indicator function of a polyhedral set (see \cite[Proposition 1]{holler15mpeg_mh}), or that $\ker(K_h) \cap \ker(D_h) = \{ 0\}$. Duality is further obtained when $Y_h = \bigcup_{t \geq 0} t \bigl( \dom(S_{f_h}) - K_h \dom(\Phi)\bigr)$. Regarding the latter, a version of Proposition \ref{prop:duality_result_discrete} without the dualization of the data term $S_{f_h}(K_hu)$ holds even without the assumption on the domain of $S_{f_h}$, however, with a different associated dual problem and saddle-point problem.

  In particular, not dualising the data term has impact on the primal-dual optimisation algorithms. In view of the iteration~\eqref{eq:primal_dual_iterates}, the evaluation of the proximal mapping for $u \mapsto S_{f_h}(K_hu)$ then becomes necessary, so this dualization strategy is only practical if the latter proximal mapping can easily be computed. 
  Furthermore, in case of a sufficiently smooth data term, dualization of $S_{f_h}$ can also be avoided by using explicit descent steps for $S_{f_h}$ instead of proximal mappings, where the Lipschitz constant of the derivative of $S_{f_h}$ usually enters in the stepsize bound. See \cite{chambolle2015ergodic_primal_dual_mh} for an extension of the primal-dual algorithm in that direction.
\end{remark}

In view of Proposition \ref{prop:duality_result_discrete}, we now address the numerical solution of the saddle-point problem \eqref{eq:discrete_saddle_point_problem_general}. Applying the iteration \eqref{eq:primal_dual_iterates}, this results in Algorithm \ref{alg:primal_dual_general}, which is given in a general form. A concrete implementation still requires an explicit form of the proximal mapping $\prox_{\sigma S_{f_h}^*}$, a concrete choice of $V_h$, $W_h$ and $D_h$ as well as an estimate on $\|(D_h,K_h)\|$ for the stepsize choice and a suitable stopping criterion. These building blocks will now be provided for different choices of $\mR_\alpha$ and $S_{f_\alpha}$ in a way that they can be combined to a arrive at a concrete, application-specific algorithm. After that, two examples will be discussed.

\begin{algorithm}[t]
  \begin{algorithmic}[1]
\onehalfspacing

\Function{Tikhonov}{$K_h$, $f_h$, $\alpha$}

\State $(u,w,\overline{u},\overline{w}) \gets  (0,0,0,0), (v,\lambda) \gets (0,0)$
\State choose $\sigma, \tau > 0$ such that $\sigma \tau \left \|\left(
  \begin{array}{cc}
    D_h^1 & D_h^2 \\ K_h & 0
  \end{array}
\right) \right\|^2 < 1$
\Repeat
	\State \textbf{Dual updates}
	\State $ v \gets \proj_{ \{\|\cdot \|_{\infty,\alpha^{-1}} \leq 1 \}} \left( v + \sigma (D_h^1 \overline{u} + D_h^2\overline{w}) \right)$
	\State $\lambda \gets \prox_{\sigma S_{f_h}^*} (\lambda + \sigma K_h \overline{u})$
		\State \textbf{Primal updates}
	\State $u_+ \gets u - \tau \bigl ( (D_h^1)^* v + K_h^*\lambda \bigr)$
	\State $w_+ \gets w - \tau \bigl( (D_h^2)^* v \bigr)$
		\State \textbf{Extrapolation and update}
	\State $(\overline{u},\overline{w}) \gets 2(u_+,w_+) - (u,w)$
	\State $(u,w) \gets (u_+,w_+)$
\Until{stopping criterion fulfilled}
\State \Return{$u$}
\EndFunction
\end{algorithmic}
\caption{Primal-dual scheme for the numerical solution of \eqref{eq:discrete_saddle_point_problem_general}}\label{alg:primal_dual_general}
\end{algorithm}

\subsubsection*{Proximal mapping of $S_{f_h}^*$.} Depending on the application of interest, and in particular on the assumption on the underlying measurement noise, different choices of $S_{f_h}$ are reasonable. The one which is probably most relevant in practice is 
\[ S_{f_h}(\lambda) = \frac{1}{2}\| \lambda - f_h \|_2 ^2, \]
which, from a statistical perspective, is the right choice under the assumption of Gaussian noise. In this case, as discussed in Lemma \ref{lem:proximal_mapping_examples}, the proximal mapping of the dual is given as
\[ \prox_{\sigma S_{f_h}^*}(\lambda) = \frac{\lambda - \sigma f_h}{1 + \sigma}.
\]
A second, practically relevant choice is the Kullback-Leibler divergence as in \eqref{eq:kl_definition}. For
discrete data $\bigl((f_h)_i\bigr)_i$ satisfying $(f_h)_i \geq 0$ for each $i$, and
a corresponding discrete signal $(\lambda_i)_i$, this corresponds to 
\begin{equation} \label{eq:kl_divergence_discrete}
\fl S_{f_h} (\lambda) = \KL(\lambda, f_h) =  \left\{\begin{array}{ll}
\sum_i \lambda_i - (f_h)_i - (f_h)_i \log(\frac{\lambda_i}{(f_h)_i}) & \text{if }\lambda_i \in {[{0,\infty}[} \text{ for all }i, \\
\infty & \text{else,}
\end{array} \right.
\end{equation}
where we again use the convention $(f_h)_i \log(0) = -\infty$ for $(f_h)_i>0$ and $0\log(\frac{\lambda_i}{0}) = 0$ whenever $\lambda_i \geq 0$. A direct computation (see for instance \cite{holler16mrpet_mh}) shows that in this case
\[  \prox_{\sigma S_{f_h}^*}(\lambda)_i = \lambda_i - \frac{\lambda_i - 1 + \sqrt{ (\lambda_i- 1)^2 + 4 \sigma (f_h)_i}}{2}.
\]
Another choice that is relevant in the presence of strong data outliers (e.g., due to transmission errors) is
\[ S_{f_h}(\lambda) = \| \lambda - f_h \|_1 \]
in which case
\[
 \prox_{\sigma S_{f_h}^*}(\lambda)_i = \frac{\lambda_i - \sigma (f_h)_i}{\max\{1,|\lambda_i - \sigma (f_h)_i|\}}
  \] can be obtained from Lemmas \ref{lem:moreau_proximal} and \ref{lem:proximal_mapping_examples}. 
  
  As already mentioned in Remark~\ref{rem:alternative_duals}, in case the discrepancies term is not dualised, a corresponding version of the algorithm of \cite{pock2011primaldual_mh} requires the proximal mappings of $\tau S_{f_h}$ which can either be computed directly or obtained from $\prox_{\sigma S_{f_h}^*}$ using Moreau's identity as in Lemma \ref{lem:moreau_proximal}. Further, there are many other choices of $S_{f_h}$ for which the $\prox_{\sigma S_{f_h}^*}$ is simple and explicit, such as, for instance, equality constraints on a subdomain in the case of image inpainting or box constraints in case of dequantization or image decompression.

\subsubsection*{Choice of $\mR_\alpha$ and proximal mapping.} As we show now, the general form $\mR_\alpha(u) = \min _{w \in W_h} \| D_h(u,w)\|_{1,\alpha}$ covers all higher-order regularisation approaches discussed in the previous sections. 

\begin{example}\label{ex:prox_regularizers}\mbox{}
  
\begin{description}
\item[Higher-order total variation.]
   The choice $\mR_\alpha(u) = \alpha \|\nabla ^k u \|_1$, with $k \geq 1$ the order of differentiation, can be realised with
\[ \fl W_{h} = \{ 0\}, \quad V_h = \ell^2(\Omega_h, \tensorspace[k]{\RR^2}), \quad D_h = \nabla_h ^k ,\quad \|v \|_{1,\alpha} = \alpha \|v  \|_1 , \]
which yields, according to Lemma \ref{lem:explicit_prox_for_ell2_space},
for $(i,j) \in \Omega_h$ that
\begin{equation} \label{eq:higher_order_tv_projection_explicit}
  \fl
  \proj_{ \{\|\cdot \|_{\infty,\alpha^{-1}} \leq 1\}}(v)_{i,j} = \proj_{ \{\|\cdot \|_{\infty} \leq \alpha \}}(v)_{i,j}  = \frac{v_{i,j}}{\max\{1,\alpha^{-1}|v|_{i,j}\}}.
  \end{equation}
Here, we used that whenever $W_h = \{ 0\}$, one can ignore the second argument of $D_h: U_h \times W_h \to V_h$ and regard it as operator $D_h: U_h \rightarrow V_h$. %

\item[Sum of higher-order TV functionals.] The choice $\mR_\alpha(u) = \alpha_1 \|\nabla_h ^{k_1} u \|_1 + \alpha_2 \|\nabla_h ^{k_2} u \|_1$, with $k_2 > k_1 \geq 1$ and $\alpha_i > 0$ for $i=1,2$ differentiation orders and weighting parameters, respectively, can be realised with
  \[
    \fl
    \begin{array}l
    W_{h} = \{ 0\}, \quad
    V_h = \ell^2(\Omega_h, \tensorspace[k_1]{\RR^2}) \times \ell^2(\Omega_h, \tensorspace[k_2]{\RR^2})
    \quad D_h  =
    \left( \begin{array}c
\nabla_h ^{k_1} \\ \nabla_h ^{k_2}
           \end{array}\right) , \\
       \| (v_1,v_2) \|_{1,\alpha} = \alpha_1  \|v_1   \|_1 + \alpha_2  \|v_2   \|_1,
    \end{array}
 \]
and yields, according to Lemma \ref{lem:explicit_prox_for_ell2_space},
\begin{equation} \label{eq:sum_of_tv_prox_explicit}
  \fl
  \proj_{ \{\|\cdot \|_{\infty,\alpha^{-1}} \leq 1 \}}(v_1,v_2) = \left( \proj_{ \{\|\cdot \|_{\infty} \leq \alpha_1 \}}(v_1),\ \proj_{ \{\|\cdot \|_{\infty} \leq \alpha_2 \}}(v_2) \right)
\end{equation}
with $\proj_{ \{\|\cdot \|_{\infty} \leq \alpha_i \}}$ as in \eqref{eq:higher_order_tv_projection_explicit}.

\item[Infimal convolution of higher-order TV functionals.] The infimal convolution
  \[ \mR_\alpha(u)=  \min_{w \in \ell^2(\Omega_h)} \alpha_1 \|\nabla _h^{k_1} u - \nabla_h ^{k_1} w  \|_1 + \alpha_2 \|\nabla _h^{k_2} w \|_1  \] can be realised via
  \[
    \fl
    \begin{array}{l}
      W_{h} = \ell^ 2(\Omega_h), \quad
      V_h = \ell^2(\Omega_h, \tensorspace[k_1]{\RR^2}) \times \ell^2(\Omega_h, \tensorspace[k_2]{\RR^2}),\\[\smallskipamount]
      D_h  = (D_h^1 \,|\, D_h^2 ) = \left( \begin{array}{c}
                      \nabla_h ^{k_1} \\0 
                    \end{array}\left|
      \begin{array}{c}
        - \nabla _h^{k_1} \\
        \nabla _h^{k_2}
      \end{array}
      \right)\right. , \quad
      \| (v_1,v_2) \|_{1,\alpha} = \alpha_1  \|v_1   \|_1 + \alpha_2  \|v_2   \|_1,
    \end{array}
  \]
  where $\proj_{ \{\|\cdot \|_{\infty,\alpha^{-1}} \leq 1 \}}$ is given as in \eqref{eq:sum_of_tv_prox_explicit}.
  \item[Second-order total generalised variation.] Let $\alpha_0,\alpha_1 > 0$. The choice
\[ \mR_\alpha(u) = \TGV_\alpha ^2 (u) =  \min_{w \in \ell^2(\Omega_h, \Sym^1(\RR^2)) } \alpha_1 \|\nabla_h  u -  w \|_1 + \alpha_0 \| \symgrad_h  w \|_1  \]
 can be realised via
 \[ \fl
   \begin{array}{l}
     W_{h} = \ell^2(\Omega_h, \Sym^1(\RR^2)), \quad
     V_h = \ell^2(\Omega_h, \Sym^1(\RR^2)) \times \ell^2(\Omega_h, \Sym^2(\RR^2)),
     \\[\smallskipamount]
     D_h  = (D_h^1 \,|\, D_h^2 ) = \left( \begin{array}{c}
                     \nabla _h \\ 0 
                   \end{array}\left|
     \begin{array}{c}
       - \id \\  \symgrad_h
     \end{array}
     
     \right)\right. ,
     \quad
     \| (v_1,v_2) \|_{1,\alpha} = \alpha_1  \|v_1   \|_1 + \alpha_0  \|v_2   \|_1,
   \end{array}
 \]
 where $\proj_{ \{\|\cdot \|_{\infty,\alpha^{-1}} \leq 1 \}}$ is given again as in \eqref{eq:sum_of_tv_prox_explicit} with $\alpha_2$ replaced by $\alpha_0$.
\item[Total generalised variation of order $k$.]
The total generalised variation functional of arbitrary order $k\in \NN$, $k \geq 1$, and weights $\alpha = (\alpha_0,\ldots,\alpha_{k-1}) \in {]{0,\infty}[}^k$, i.e.,
\[ \mR_\alpha(u) = \TGV_\alpha^k(u) = \min_{ w = (w_1,\ldots,w_{k-1})\in W_h \atop w_0 = u, \, w_k = 0} \sum_{m=1}^k \alpha_{k-m}  \| \symgrad_h w_{m-1} - w_m  \|_1 \]
can be realised via
\[
  \fl
  \begin{array}{l}
    \displaystyle W_{h} = \bigtimes _{m=1}^{k-1} \ell^2(\Omega_h, \Sym^m(\RR^2)) ,
    \quad
    V_h = \bigtimes _{m=1}^{k} \ell^2(\Omega_h, \Sym^m(\RR^2)),
    \\[2\medskipamount]
    D_h  = \left(
    \begin{array}{c}
      \nabla _h \\ 0\rlap{\phantom{\vdots}} \\ \vdots \\ \vdots \\ 0
    \end{array}\left|
    \begin{array}{cccc}
       - \id & 0  & \ldots&  0 \\
      \symgrad_h & -\id & \ddots & \vdots \\
        \ddots &\ddots  & \ddots & 0\\
         & \ddots & \symgrad_h & - \id \\
       \cdots  & \cdots & 0 & \symgrad _h
    \end{array}\right)\right. , \quad
    \displaystyle \| v \|_{1,\alpha} =  \sum_{m=1}^k \alpha_{k-m} \|v_m \|_1.  
  \end{array}
\]
In this case,
\[
  \fl
  \bigl(\proj_{ \{\|\cdot \|_{\infty,\alpha^{-1}} \leq 1 \}} (v)\bigr)_m
  = \proj_{ \{\|\cdot \|_{\infty} \leq \alpha_{k-m} \}} (v_m), \qquad
  m=1, \ldots, k,
\]
where $\proj_{ \{\|\cdot \|_{\infty} \leq \alpha_{m} \}}$ is given as in \eqref{eq:higher_order_tv_projection_explicit}.
\end{description}
\end{example}

\subsubsection*{Stepsize choice and stopping rule.}
Algorithm \ref{alg:primal_dual_general} requires to choose stepsizes $\sigma,\tau >0$ such that $ \sigma \tau \|\mK\|^2 < 1$ where
$\mK = \left(
  \begin{array}{cc}
    D_h^1 & D_h^2 \\ K_h & 0 
  \end{array}
\right)$. This, in turn, requires to estimate $\|\mK\|$ which we discuss on the following. The operator $K_h$ is application-dependent and we assume an upper bound for its norm to be given. Regarding the differential operator $D_h = (D_h^1, D_h^2)$, an estimate on the norm of its building blocks $\nabla_h^k$, $\symgrad_h^k$ is provided in Lemma~\ref{lem:norm_estimate_grad_symgrad}. As the following proposition shows, an upper bound on $\|\mK\|$ as well as on the norm of more general block-operators, can then be obtained by computing a simple singular value decomposition of a usually low-dimensional matrix.
\begin{lemma} \label{lem:norm_estimate_vectorized}
Assume that $\mK:\mX \rightarrow \mY$ with $\mX = \mX_1 \times \ldots \times \mX_N$, $\mY = \mY_1 \times \ldots \times \mY_M$ is given as
\[ \mK = \left( \begin{array}{ccc}
\mK_{1,1} & \cdots & \mK_{1,N} \\
\vdots & & \vdots \\
\mK_{M,1} & \cdots&  \mK_{M,N}
\end{array} \right)
\]
and that $\|\mK_{m,n}\| \leq L_{m,n}$ for each $m=1,\ldots,M$, $n=1,\ldots,N$. 
Then, \[ \|\mK \| \leq \sigma_{\max} (
\left( \begin{array}{ccc}
L_{1,1} & \cdots & L_{1,N} \\
\vdots & & \vdots \\
L_{M,1} & \cdots&  L_{M,N}
\end{array} \right))
\]
where $\sigma_{\max}$ denotes the largest singular value of a matrix.
\begin{proof} For $x = (x_1,\ldots,x_N) \in \mX$ we estimate
\begin{eqnarray*}
\| \mK x\|_2 ^2 
& = \sum_{m=1}^M   \left \| \sum_{n=1}^N \mK_{m,n} x_n  \right \|_2 ^2
 \leq  \sum_{m=1}^M \left(  \sum_{n=1}^N L_{m,n} \| x_n\|_2  \right)^2  \\
&  = 
\left\| \left( \begin{array}{ccc}
L_{1,1}& \cdots & L_{1,N} \\
\vdots & & \vdots \\
L_{M,1} & \cdots&  L_{M,N}
\end{array} \right)
 \left( \begin{array}{ccc}
\|x_1\|_2  \\
\vdots \\
\|x_N\|_2
\end{array} \right) \right\|_2 ^2,
\end{eqnarray*}
from which the claimed assertion follows since the matrix norm induced by the $2$-norm corresponds to the largest singular value.
\end{proof}
\end{lemma}
This result can be applied in the
setting~\eqref{eq:discrete_saddle_point_problem_general}, i.e.,
$\mK = \left(
  \begin{array}{cc}
    D_h^1 & D_h^2 \\ K_h & 0 
  \end{array}
\right)$,
leading to
\[
  \fl
  \norm{\mK}^2 \leq \frac{\norm{D_h^1}^2 + \norm{D_h^2}^2 + \norm{K_h}^2
    + \sqrt{\bigl(\norm{D_h^1}^2 + \norm{D_h^2}^2 + \norm{K_h}^2\bigr)^2 - 4
    \norm{D_h^2}^2 \norm{K_h}^2}}{2}
\]
Alternatively, one could use the result when $D_h^1$ or $D_h^2$ have
block structures and a norm estimate is known for each block in addition to an estimate on $\norm{K_h}$.  Two concrete examples will be provided at the end
of this section below.
\begin{remark}
In practice, provided that $L_{m,n}$ is a good upper bound for $\|\mK_{m,n}\|$, the norm estimate of Lemma \ref{lem:norm_estimate_vectorized} is rather tight such that, depending on $\|\mK\|$, the admissible stepsizes can be sufficiently large. Furthermore, the constraint $\sigma \tau \|\mK\|^2 < 1$ still allows to choose an arbitrary positive ratio $\theta = \sigma/\tau$ and, in our experience, often a choice $ \theta \ll 1 $ or $\theta \gg 1$ can accelerate convergence significantly. Finally, we also note that in case no estimate on $\|\mK\|$ can be obtained, or in case an explicit estimate only allows for prohibitively small stepsizes, also an adaptive stepsize choice without prior knowledge of $\|\mK\|$ is possible, see for instance \cite{holler15tgvrec_p2_mh}.
\end{remark}
\begin{remark}
It is worth mentioning that, in case of a uniformly convex functional in the saddle-point formulation~\eqref{eq:discrete_saddle_point_problem_general} (which is not the case in the setting considered here), a further acceleration can be achieved by adaptive stepsize choices, see for instance \cite{pock2011primaldual_mh}.
\end{remark}

\begin{remark}
Regarding a suitable stopping criterion, we note that often, the primal-dual gap, i.e., the gap between the energy of the primal and dual problem \eqref{eq:min_problem_vectorized_general} and \eqref{eq:dual_problem_vectorize_general} evaluated at the current iterates, provides a good measure for optimality. Indeed, with
\[\Gap(x,y) = \mF(\mK x) + \mG(x) + \mG^* (-\mK^*y) + \mF^* (y), \]
$\Gap(x,y) \geq 0$ and $\Gap(x,y) = 0$ if and only if $(x,y)$ is optimal
such that, in principle, the condition $\Gap(x^n,y^n)< \varepsilon$
with $(x^n,y^n)$ the iterates of \eqref{eq:primal_dual_iterates} can
be used as stopping criterion. In case this condition is met, $x^n$ as
well as $y^n$ are both optimal up to an $\varepsilon$-tolerance in
terms of the objective functionals for the primal and dual problem,
respectively.

In the present situation of~\eqref{eq:discrete_saddle_point_problem_general}, however, the primal and dual problem~\eqref{eq:discrete_min_prob_general} and~\eqref{eq:discrete_dual_problem_general} yield
\[
  \fl
  \Gap(u,w,v,\lambda) =
  \left\{
    \begin{array}{ll}
      \begin{minipage}{0.4\linewidth}
        $
        S_{f_h}(K_hu) + \norm[1,\alpha]{D_h(u,w)}$

        \mbox{}\hfill$
        + \mI_{\sett{\norm[\infty,\alpha^{-1}]{\placeholder} \leq 1}}(v) + S_{f_h}^*(\lambda)$
      \end{minipage}
      &
        \begin{minipage}{0.25\linewidth}
          if $(D_h^1)^*v + K_h^*\lambda = 0$

          \vspace*{0.5\smallskipamount}
          \mbox{}\hfill and $(D_h^2)^* v = 0$,
        \end{minipage}
\\[\smallskipamount]
      \infty & \text{else}.
    \end{array}
  \right.
\]
While for the iterates $(u^n, w^n, v^n, \lambda^n)$, we always have
$\norm[\infty,\alpha^{-1}]{v^n} \leq 1$ as well as
$\lambda^n \in \dom(S_{f_h}^*)$,
Algorithm~\ref{alg:primal_dual_general} does not guarantee that
$K_h u^n \in \dom(S_{f_h})$ and, $(D_h^1)v^n + K_h^*\lambda^n = 0$ as
well as $(D_h^2)^*v^n = 0$, such that the primal-dual gap is always
infinite in practice and the stopping criterion is never met.  With
some adaptations, however, it is sometimes still possible to obtain a
primal-dual gap that converges to zero and hence, to deduce a stopping criterion
with optimality guarantees. There are several possibilities for
achieving this. Let us, for simplicity, assume that both $S_{f_h}$ and
$S_{f_h}^*$ are finite everywhere and hence, continuous. This is, for
example, the case for $S_{f_h} = \frac12 \norm{\placeholder -
  f_h}^2$. Next, assume that a-priori norm estimates are available for
all solution pairs $(u^*,w^*)$, say $\norm[\tilde U_h]{u^*} \leq C_u$
and $\norm[\tilde W_h]{w^*} \leq C_w$ for norms
$\norm[\tilde U_h]{\placeholder}$, $\norm[\tilde W_h]{\placeholder}$
on $U_h$, $W_h$ that do not necessarily correspond to the Hilbert
space norms. Such estimates may, for instance, be obtained from the
observation that
$S_{f_h}(u^*) + \norm[1,\alpha]{D_h(u^*,w^*)} \leq S_{f_h}(0)$ and
suitable coercivity estimates, as discussed in
Sections~\ref{sec:higher_order_tv},~\ref{sec:combined_approaches}
and~\ref{sec:tgv}.  Then, the primal problem can, for instance, be
replaced by
\[
  \fl
  \min_{(u, w) \in U_h \times W_h} \ S_{f_h}(K_hu) +
  \norm[1,\alpha]{D_h(u,w)} + \frac12 \bigl(\norm[\tilde U_h]{u} - C_u\bigr)_+^2
  + \frac12 \bigl(\norm[\tilde W_h]{w} - C_w \bigr)_+^2
\]
where $(t)_+ = \max\sett{0,t}$ for $t \in \RR$,
which has, by construction, the same minimizers as the original
problem~\eqref{eq:discrete_min_prob_general}, but a dual problem that
reads as
\[
  \fl
  \begin{array}{rl}
    \displaystyle \max_{(v, \lambda) \in V_h \times Y_h}
    &
      \displaystyle - \frac12 \bigl(
      \norm[\tilde U_h^*]{(D_h^1)^*v + K_h^* \lambda}^2 + C_u \bigr)^2 +
      \frac{C_u^2}2 - \frac12 \bigl( \norm[\tilde W_h^*]{(D_h^2)^*v}^2 +
      C_w \bigr)^2 + \frac{C_w^2}2  \\
    &\qquad -
      \mI_{\sett{\norm[\infty,\alpha^{-1}]{\placeholder} \leq 1}}(v) -
      S_{f_h}^*(\lambda),
  \end{array}
\]
where $\norm[\tilde U_h^*]{\placeholder}$ and
$\norm[\tilde W_h^*]{\placeholder}$ denote the respective dual norms.
By duality and since the minimum of the primal problem did not change,
the modified dual problem also has the same solutions as the original dual
problem.  Now, as the iterates $(v^n, \lambda^n)$ satisfy
$\norm[\infty,\alpha^{-1}]{v^n} \leq 1$ and
$\lambda^n \in \dom S_{f_h}^*$, the dual objective is finite for the
iterates and converges to the maximum as $n \to \infty$. Analogously,
plugging in the sequence $(u^n,w^n)$ into the modified primal problem
yields convergence to the minimum, hence, the respective primal-dual
gap converges to zero for the primal-dual iterates
$(u^n, w^n, v^n, \lambda^n)$. In summary, the functional
\[
  \fl
  \begin{array}{rl}
    \tilde \Gap(u,w,v,\lambda)
    = & 
        \displaystyle
        S_{f_h}(K_hu) + \norm[1,\alpha]{D_h(u,w)}
        + \mI_{\sett{\norm[\infty,\alpha^{-1}]{\placeholder} \leq 1}}(v)
        + S_{f_h}^*(\lambda) \\[\smallskipamount]
      & \displaystyle \quad
        + \frac12 \bigl( \norm[\tilde U_h]{u} - C_u \bigr)_+^2 + \frac12 \bigl( \norm[\tilde W_h]{w} - C_w \bigr)_+^2 \\[\smallskipamount]
    & \displaystyle \quad
        + \frac12 \bigl(\norm[\tilde U_h^*]{(D_h^1)^*v + K_h^* \lambda}^2 + C_u \bigr)^2 - \frac{C_u^2}{2} + \frac12 \bigl( \norm[\tilde W_h^*]{(D_h^2)^*v}^2 +
      C_w \bigr)^2 - \frac{C_w^2}2
  \end{array}
\]
yields the stopping criterion
$\tilde \Gap(u^n, w^n, v^n, \lambda^n) < \varepsilon$ which will be met
for some $n$ and gives $\varepsilon$-optimality of $(u^n,w^n)$ for the
original primal problem~\eqref{eq:discrete_min_prob_general}.

The examples below show how this primal-dual gap reads for specific
applications. For other strategies of modifying the primal-dual gap to
a functional that is positive and finite, converges to zero and
possibly provides an upper bound on optimality of the iterates in
terms of the objective functional, see, for instance
\cite{holler12tvjpeg_mh,holler15tgvrec_p2_mh,bredies2015tvtgvpdr_mh}.
\end{remark}

\subsubsection*{Concrete examples.} \mbox{}

\begin{example}
  \label{ex:discrete_l2_tv2}
As first example, we consider the minimisation problem 
\begin{equation} \label{eq:tv2_l2_example_numerics}
 \min _{u \in U_h} \ \frac12 \|K_h u-f_h \|_2 ^2 +  \alpha \|\nabla ^2_h u\|_1, 
 \end{equation}
i.e., second order-TV regularisation for a linear inverse problem with Gaussian measurement noise.
In this setting, we choose $W_h = \{ 0 \}$, $V_h = \ell^2(\Omega_h, \Sym^2(\RR^2))$,  $D_h = \nabla _h ^2$ and $\|v\|_{1,\alpha} = \alpha \| v\|_1$. 
Assuming that $\|K_h\| \leq 1$ (after possible scaling of $K_h$), Lemma~\ref{lem:norm_estimate_vectorized} together with the estimate $\|\nabla _h^2 \|\leq 8$ from Lemma \ref{lem:norm_estimate_grad_symgrad} yields
\[
  \|\left(
    \begin{array}{c}
      \nabla_h^2 \\ K_h
    \end{array} \right) \| \leq
  \sigma_{\max}( \left( \begin{array}{c}
8 \\ 1
\end{array}
\right) )=  \sqrt{65}.
\]
The resulting concrete realisation of Algorithm \ref{alg:primal_dual_general} can be found in Algorithm \ref{alg:tv2_l2}.
\begin{algorithm}[t]
  \begin{algorithmic}[1]
\onehalfspacing

\Function{$L^2$-$\TV^2$-Tikhonov}{$K_h$, $f_h$, $\alpha$}
\Comment{Requirement: $\norm{K_h} \leq 1$}

\State $(u,\overline{u}) \gets  (0,0), (v, \lambda) \gets (0,0)$
\State choose $\sigma, \tau > 0$ such that $\sigma \tau < \frac1{65}$
\Repeat
	\State \textbf{Dual updates}
	\State $ v \gets \proj_{ \{\|\cdot \|_{\infty} \leq \alpha \}} \left( v + \sigma \nabla_h^2 \overline{u} \right)$
	\State $\lambda \gets  (\lambda + \sigma (K_h \overline{u} - f_h)) /(1 + \sigma)$
		\State \textbf{Primal updates}
	\State $u_+ \gets u - \tau (\divergence _h ^2 v + K_h^* \lambda)$
		\State \textbf{Extrapolation and update}
	\State $\overline{u} \gets 2u_+ - u$
	\State $u \gets u_+$
\Until{stopping criterion fulfilled}
\State \Return{$u$}
\EndFunction
\end{algorithmic}
\caption{Implementation for solving the $L^2$-$\TV^2$ problem \eqref{eq:tv2_l2_example_numerics} \label{alg:tv2_l2} }
\end{algorithm} 
Here, $\proj_{ \{\|\cdot \|_{\infty} \leq \alpha \}}$ is given explicitly in \eqref{eq:higher_order_tv_projection_explicit}, $\divergence_ h^2 = \divergence_h \divergence _h $ is the adjoint of $\nabla_h^2$ and the modified primal-dual gap $\tilde \Gap$ evaluated on the iterates $(u,v, \lambda)$ of the algorithm reduces to
\[
  \fl
  \begin{array}{rl}
    \Gap(u,v,\lambda) = &\displaystyle
                         \frac12 \norm{K_h u - f_h}^2 + \alpha \norm[1]{\grad_h^2 u} +
                     \mI_{\sett{\norm[\infty]{\placeholder} \leq \alpha}}(v) + \frac12\norm{\lambda + f_h}^2 - \frac12\norm{f_h}^2 \\
                   & \displaystyle + \frac12 \bigl( \norm[2]{u} - C_u \bigr)_+^2 + \frac12 \bigl(
                     \norm[2]{\divergence_h^2 v + K_h^* \lambda} + C_u \bigr)^2 - \frac{C_u^2}2
  \end{array}
\]
where $C_u > 0$ is an a-priori bound on $\norm[2]{u^*}$ for solutions
$u^*$ according to~\eqref{eq:tv^k-tikh-hilbert-bound}, for instance.
\end{example}

\begin{example}
As second example, we consider $\TGV_\alpha^2$ regularisation for an inverse problem with Poisson noise and discrete non-negative data $\bigl((f_h)_i\bigr)_i$, which corresponds to solving
\begin{equation} \label{eq:tgv2_kl_example_numerics}
 \min _{u \in U_h }\,  \KL(K_h u,f_h) +  \TGV_\alpha^2(u),
 \end{equation}
with $\KL$ being the discrete Kullback--Leibler divergence as in \eqref{eq:kl_divergence_discrete}.

In this setting, we choose $W_h =  \ell^2 (\Omega_h , \Sym^1 (\RR^2))$,
$V_h = \ell^2(\Omega_h, \Sym^1(\RR^2)) \times \ell^2(\Omega, \Sym^2(\RR^2))$,
$ D_h  = \left( \begin{array}{cc}
\nabla_h  & - I  \\0 &  \symgrad_h
\end{array}\right) $ and $ \| (v_1,v_2) \|_{1,\alpha} = \alpha_1  \|v_1   \|_1 + \alpha_0  \|v_2   \|_1
$.
Setting \[ \mK  = \left( \begin{array}{cc}
\nabla _h & - \id  \\0 &  \symgrad _h\\
K_h & 0 
\end{array}\right) 
\]  and again assuming $\|K_h\| \leq 1$, Lemma \ref{lem:norm_estimate_vectorized} together with the estimates $\|\nabla _h \|\leq \sqrt{8}$ and $\|\symgrad _h \|\leq \sqrt{8}$ from Lemma \ref{lem:norm_estimate_grad_symgrad} yields
\[
\| \mK  \| \leq  \sigma_{\max}( \left( 
\begin{array}{cc}
\sqrt{8} & 1 \\
0 & \sqrt{8} \\
1 & 0
\end{array}
\right) ) = \sqrt{\sqrt{8} + 9} \approx  \sqrt{\frac{71}{6}}.
\]
The resulting, concrete implementation of Algorithm \ref{alg:primal_dual_general} can be found in Algorithm \ref{alg:tgv2_kl}.
\begin{algorithm}[t]
  \begin{algorithmic}[1]
\onehalfspacing

\Function{KL-$\TGV_\alpha^2$-Tikhonov}{$K_h$, $f_h$, $\alpha$}
\Comment{Requirement: $\norm{K_h} \leq 1$}

\State $(u, w, \overline{u},\overline{w}) \gets (0,0,0,0), (v_1, v_2, \lambda) \gets (0, 0, 0)$
\State choose $\sigma, \tau > 0$ such that $\sigma \tau \leq \frac{6}{71}$
\Repeat
	\State \textbf{Dual updates}
	\State $ v_1 \gets \proj_{ \{\|\cdot \|_{\infty} \leq \alpha_1 \}} \left( v_1 + \sigma(\nabla _h \overline{u} - \overline{w})  \right)$
		\State $ v_2 \gets \proj_{ \{\|\cdot \|_{\infty} \leq \alpha_0 \}} \left( v_2 + \sigma \symgrad _h  \overline{w}  \right)$
	\State $\lambda \gets \lambda + \sigma K_h \overline{u}$
	\State $\lambda \gets \lambda - \frac{\lambda - 1 + \sqrt{ (\lambda - 1)^2 + 4 \sigma f_h}}{2}$
		\State \textbf{Primal updates}
	\State $u_+ \gets u + \tau (\divergence _h v_1 - K_h^*\lambda)$
	\State $w_+ \gets w + \tau (v_1 + \divergence_h v_2)$
		\State \textbf{Extrapolation and update}
	\State $(\overline{u},\overline{w}) \gets 2(u_+,w_+) - (u,w)$
	\State $(u,w) \gets (u_+,w_+)$
\Until{stopping criterion fulfilled}
\State \Return{$u$}
\EndFunction
\end{algorithmic}
\caption{Implementation for solving the KL-$\TGV_\alpha^2$ problem \eqref{eq:tgv2_kl_example_numerics} \label{alg:tgv2_kl} }
\end{algorithm}
Here, again $\proj_{ \{\|\cdot \|_{\infty} \leq \alpha_i \}}$ is given explicitly in \eqref{eq:higher_order_tv_projection_explicit}, and, abusing notation, $\divergence_ h $ is both the negative adjoint of $\nabla_h$ and $\symgrad_h$, depending on the input. The modified primal-dual gap $\tilde \Gap$ evaluated on the iterates $(u,w, v_1,v_2, \lambda)$ of the algorithm reduces to
\[
  \fl
  \begin{array}{rl}
    \tilde \Gap(u, w, v_1, v_2, \lambda)
    = & \displaystyle
        \!\!\!\!%
        \KL(K_h u, f_h) + \alpha_1 \norm[1]{\grad_h u - w} + \alpha_0 \norm[1]{\symgrad_h w} + \mI_{\sett{\norm[\infty]{\placeholder} \leq \alpha_1}}(v_1)
    \\[\smallskipamount]
    & \displaystyle \!\!\!\! +
  \mI_{\sett{\norm[\infty]{\placeholder} \leq \alpha_0}}(v_2) +
      \KL^*(\lambda, f_h) + \frac12 \bigl(\norm[2]{u} - C_u\bigr)_+^2
      + \frac12 \bigl( \norm[1]{w} - C_w  \bigr)_+^2 \\[\smallskipamount]
    & \displaystyle \!\!\!\!+ \frac12 \bigl( \norm[2]{K_h^* \lambda - \divergence_h v_1} + C_u \bigr)^2 - \frac{C_u^2}{2} + \frac12 \bigl( \norm[\infty]{v_1 + \divergence_h v_2} + C_w \bigr)^2 - \frac{C_w^2}{2},
  \end{array}
\]
where $\KL^*(\lambda,f_h) = - \sum_i (f_h)_i \log(1 - \lambda_i)$ whenever $\lambda_i \leq 1$ for each $i$ where $(f_h)_i \log(0) = \infty$ for $(f_h)_i > 0$, $0 \log(0) = 0$,  and $\KL^*(\lambda,f_h) = \infty$ else. Further, $C_u$ is an a-priori bound on the $2$-norm of $u^* = w_0^*$ analogous to~\eqref{eq:tv^k-tikh-kl-bound} while $C_w$ is an a-priori bound on the $1$-norm of
$w^* = w_1^*$ according to~\eqref{eq:tgv_tikh_hilbert_bound1}.
\end{example}

We refer to, e.g., \cite{Bredies14_multichannel_mh} for more examples
of primal-dual-based algorithms for TGV regularisation.

\subsection{Implicit and preconditioned optimisation methods}

Let us shortly discuss other proximal algorithms for the solution
of~\eqref{eq:discrete_min_prob_general}. One popular method is the
alternating direction method of multipliers (ADMM) \cite{glowinski1975admm,gabay1983admm} which bases
on augmented Lagrangian formulations
for~\eqref{eq:discrete_min_prob_general}, for instance,
\begin{equation}
  \fl
  \min_{(u,w,v,\lambda) \in U_h \times W_h \times V_h \times Y_h}
  \ S_{f_h}(\lambda) + \norm[1,\alpha]{v} \quad
  \text{subject to} \quad
  \left\{
    \begin{array}{rl}
      \lambda &= K_h u, \\
      v &= D_h^1 u + D_h^2 w,
    \end{array}
  \right.\label{eq:tikhonov_admm_form}
\end{equation}
which results in the augmented Lagrangian
\[
  \fl
  \begin{array}{rl}
    \mL_\tau (u,w,v,\lambda,\bar v, \bar \lambda)
    = &S_{f_h}(\lambda) + \norm[1,\alpha]{v} +
        \scp[Y_h]{K_hu - \lambda}{\bar \lambda} + \scp[V_h]{D_h^1 u + D_h^2 w - v}{\bar v} \\[\smallskipamount]
    &
  \displaystyle + \frac{1}{2\tau} \norm[Y_h]{K_hu - \lambda}^2
  + \frac1{2\tau}\norm[V_h]{D_h^1 u + D_h^2 w - v}^2
  \end{array}
\]
where $\tau > 0$. For~\eqref{eq:tikhonov_admm_form}, the ADMM
algorithm amounts to
\[
  \fl
  \left\{
    \begin{array}{rl}
      (u^{k+1}, w^{k+1})
      & \in
        \argmin_{(u,w) \in U_h \times W_h} \ \mL_\tau(u,w,v^k,\lambda^k, \bar v^k, \bar \lambda^k), \\[\smallskipamount]
      (v^{k+1}, \lambda^{k+1})
      &= \argmin_{(v, \lambda) \in V_h \times Y_h}
        \ \mL_\tau (u^{k+1}, w^{k+1}, v, \lambda, \bar v^k, \bar \lambda^k),
      \\[\smallskipamount]
      \bar v^{k+1} &= \displaystyle \bar v^k + \frac1\tau (D_h^1 u^{k+1} + D_h^2 w^{k+1} - v^{k+1}),
      \\[\smallskipamount]
      \bar \lambda^{k+1} &= \displaystyle \bar \lambda^k + \frac1\tau (K_h u^{k+1} - \lambda^{k+1}).
    \end{array}\right.
\]
Here, the first subproblem amounts to solving a least-squares problem and
the associated normal equation is usually stably solvable since
$D_h^1$ and $D_h^2$ involve discrete differential operators and hence,
the normal equation essentially corresponds to the solution of a
discrete elliptic equation that is perturbed by $K_h^*K_h$. For this
reason, ADMM is usually considered an implicit method. The second step
turns out to be the application of the proximal mappings for $S_{f_h}$ and
$\norm[1,\alpha]{\placeholder}$ while the last update steps have an
explicit form, see Algorithm~\ref{alg:admm_general}. By virtue of
Moreau's identity~\eqref{eq:moreaus_identity} (also see Lemma~\ref{lem:moreau_proximal}), the operators
$\prox_{\tau\norm[1,\alpha]{\placeholder}}$ and $\prox_{\tau S_{f_h}}$
can easily be computed knowing
$\proj_{ \{\|\cdot \|_{\infty,\alpha^{-1}} \leq 1\}}$ and $\prox_{\tau^{-1} S_{f_h}^*}$.  We have, for instance,
\[
  \prox_{\tau \norm[1,\alpha]{\placeholder}}(v) = v - \proj_{ \{\|\cdot \|_{\infty, \alpha^{-1}} \leq \tau\}}(v) = v - \tau \proj_{\sett{\norm[\infty,\alpha^{-1}]{\placeholder} \leq 1}}
  \Bigl( \frac{v}{\tau} \Bigr)
\]
where the projection operator usually has an explicit representation,
see Lemma~\ref{lem:explicit_prox_for_ell2_space} and Example~\ref{ex:prox_regularizers}. Further, for the discrepancies
discussed in Subsection~\ref{subsec:general_saddle_point}, it holds that
\[
  \fl
  \begin{array}{c}
    \displaystyle
    \prox_{\tau \frac12 \norm[2]{\placeholder - f_h}^2}(\lambda)
  = \frac{\lambda + \tau f}{1 + \tau}, \qquad
  \prox_{\tau \norm[1]{\placeholder - f_h}}(\lambda) = \lambda - \tau \frac{\lambda - f_h}{\max\{\tau, \abs{\lambda - f_h}\}}, \\[\smallskipamount]
    \displaystyle
    \prox_{\tau \KL(\placeholder, f_h)}(\lambda) = \frac{\sqrt{(\tau - \lambda)^2 + 4 \tau f_h} + \lambda - \tau}{2}.
  \end{array}
\]
\begin{algorithm}[t]
  \begin{algorithmic}[1]
\onehalfspacing

\Function{Tikhonov-ADMM}{$K_h$, $f_h$, $\alpha$}
\State $(u, w) \gets (0,0), (v, \lambda) \gets (0, 0), (\bar v, \bar \lambda)
\gets (0,0)$
\State choose $\tau > 0$
\Repeat
\State \textbf{Linear subproblem}
\State
$(u,w)$ $\gets$ Solution of
\[ \fl \quad \quad
\left( \begin{array}{cc}
         K_h^*K_h + (D_h^1)^*D_h^1 & (D_h^1)^* D_h^2 \\
         (D_h^2)^* D_h^1 & (D_h^2)^*D_h^2 
\end{array}\right) \left(
\begin{array}{c}
  u \\ w
\end{array}
\right) =
\left(\begin{array}{c}
  K_h^*(\lambda - \tau \bar \lambda) + (D_h^1)^*(v - \tau \bar v) \\
  (D_h^2)^*(v - \tau \bar v)
\end{array}
\right)
\]
\State \textbf{Proximal subproblem}
\State $v \gets \prox_{\tau \norm[1,\alpha]{\placeholder}}(D_h^1 u + D_h^2 w + \tau \bar v)$
\State $\lambda \gets \prox_{\tau S_{f_h}}(K_hu + \tau \bar\lambda)$
\State \textbf{Lagrange multiplier}
\State $\bar v \gets \bar v + \frac1\tau (D_h^1 u + D_h^2 w - v)$
\State $\bar \lambda \gets \bar \lambda + \frac1\tau (K_hu - \lambda)$
\Until{stopping criterion fulfilled}
\State \Return{$u$}
\EndFunction
\end{algorithmic}
\caption{ADMM scheme for the numerical solution of \eqref{eq:tikhonov_admm_form}}\label{alg:admm_general}
\end{algorithm}

While ADMM has the advantage of converging for arbitrary stepsizes
$\tau > 0$ (see, e.g.~\cite{boyd2011admm}), the main drawback is often considered
the linear update step which amounts to solving a linear equation (or,
alternatively, a least-squares problem) which can be computationally
expensive. The latter can be avoided, for instance, with
preconditioning techniques \cite{deng2016gadmm,bredies2017padmm}.  Denoting again by $\mK = \left(
  \begin{array}{cc}
    D_h^1 & D_h^2 \\ K_h & 0
  \end{array}
\right)$, the linear solution step amounts to solving $\mK^* \mK \left(
  \begin{array}{c}
    u \\ w
  \end{array}
\right) = \mK^* \left(
  \begin{array}{c}
    v - \tau \bar v \\ \lambda - \tau \bar\lambda
  \end{array}
\right)$.  Introducing the additional variables
$(u', w') \in U_h \times W_h$ as well as the
constraint $(u', w') = (\rho \id - \mK^*\mK)^{1/2} (u,w)$ for
$\rho > \norm{\mK}^2$, we can consider the problem
\[
  \begin{array}{l}
    \displaystyle
    \min_{(u,w,u',w',v,\lambda) \in (U_h \times W_h)^2 \times V_h \times Y_h} \ S_{f_h}(\lambda) + \norm[1,\alpha]{v} \\[1.5\medskipamount]
    \text{subject to} \quad
    \left\{
    \begin{array}{rl}
      (\rho \id - \mK^*\mK)^{1/2}(u,w) &= (u', w') \\
      \mK (u,w) & = (v, \lambda) 
    \end{array}
                  \right.  
  \end{array}
\]
which is equivalent to~\eqref{eq:tikhonov_admm_form}. The associated
ADMM procedure, however, simplifies. In particular, the linear
subproblem only involves $\rho \id$ whose solution is trivial. Also,
the Lagrange multipliers of the additional constraint are always zero
within the iteration and the evaluation of the square
root $(\rho \id - \mK^*\mK)^{1/2}$ can be avoided. This leads to
the linear subproblem of Algorithm~\ref{alg:admm_general} being
replaced by the linear update step
\[
  \fl
  \left(
  \begin{array}{c}
    u \\ w
  \end{array}\right) \leftarrow
\left(
  \begin{array}{c}
    u + \frac1\rho \bigl( K_h^*(\lambda - \tau \bar \lambda - K_h u) + (D_h^1)^*(v - \tau \bar v - D_h^1u - D_h^2 w ) \bigr)   \\
    w + \frac1\rho(D_h^2)^*(v - \tau \bar v - D_h^1 u - D_h^2 w)
  \end{array}\right).
\]
Also, the procedure then requires, in each iteration, only one
evaluation of $K_h$, $D_h^1$, $D_h^2$ and their respective adjoints as
well as the evaluation of proximal mappings, such that the
computational effort is comparable to
Algorithm~\ref{alg:primal_dual_general}. As a special variant of the
general ADMM algorithm, the above preconditioned version converges for
$\tau > 0$ if $\rho > \norm{\mK}^2$ is satisfied. Thus, an estimate
for $\norm{\mK}$ is required which can, e.g., be obtained by
Lemma~\ref{lem:norm_estimate_vectorized} (also confer the concrete
examples in Subsection~\ref{subsec:general_saddle_point}).  While this
is the most common preconditioning strategy for ADMM, there are many
other possibilities for transforming the original linear subproblem
into a simpler one such that, e.g., the preconditioned problem amounts
to the application of one or more steps of a symmetric Gauss--Seidel
iteration or a symmetric successive over-relaxation (SSOR) procedure
\cite{bredies2017padmm}.

Another class of methods for
solving~\eqref{eq:discrete_min_prob_general} is given by the
Douglas--Rachford iteration \cite{Lions79dr_splitting_mh,eckstein1992maximalmonotoneoperators_mh}, which is an iterative procedure for
solving monotone inclusion problems of the type
\[
  0 \in Az + Bz
\]
in Hilbert space, where $A$, $B$ are maximally monotone operators.
It proceeds as follows:
\[
  \left\{
    \begin{array}{rl}
      z^{k+1} & =(\id + \sigma A)^{-1}(\bar z^k), \\
      \bar z^{k+1} &= \bar z^k + (\id + \sigma B)^{-1}(2 z^{k+1} - \bar z^k) - z^{k+1},
    \end{array}
  \right.
\]
where $\sigma > 0$ is a stepsize parameter. As only the resolvent
operators $(\id + \sigma A)^{-1}$ and $(\id + \sigma B)^{-1}$ are involved,
the Douglas--Rachford iteration is also considered an implicit scheme.
In the context of optimisation problems, the operators $A$ and $B$ are
commonly chosen based on first-order optimality conditions, which are
subgradient inclusions \cite{gabay1983admm,briceno-arias2011monotoneskew}. Here, we choose the saddle-point
formulation~\eqref{eq:discrete_saddle_point_problem_general} and the
associated optimality conditions:
\[
  \left(
    \begin{array}{c}
      0 \\ 0 \\ 0 \\ 0
    \end{array}
  \right)
  \in
  \left(
    \begin{array}{c}
    (D_h^1)^* v + K_h^* \lambda \\
    (D_h^2)^* v \\
    -D_h^1 u - D_h^2 w \\
    -K_hu
    \end{array}
  \right) + \left(
    \begin{array}{c}
      0 \\ 0 \\ \partial \mI_{\sett{\norm[\infty,\alpha^{-1}]{\placeholder} \leq 1}}(v)
      \\ \partial S_{f_h}^* (\lambda).
    \end{array}
  \right)
\]
For instance, choosing $A$ and $B$ as the first and second operator in the above splitting, respectively, leads to the iteration outlined in
Algorithm~\ref{alg:douglas_rachford_general}: Indeed, in terms of
$x = (u,w)$, $y = (v,\lambda)$ and $\mK =
\left(\begin{array}{cc}
  D_h^1 & D_h^2 \\ K_h & 0
\end{array}\right)$,
the resolvent for the linear operator $A$ corresponds to solving the linear system
\[
  \left\{
    \begin{array}{rl}
      x + \sigma \mK^*y
      & = \bar x, \\
      y - \sigma \mK x &= \bar y,
    \end{array}
  \right.
  \quad
  \Leftrightarrow
  \quad
  \left\{
    \begin{array}{rl}
      (\id + \sigma^2 \mK^*\mK) x & = \bar x - \sigma \mK^* \bar y, \\
      y &= \bar y + \sigma \mK x,
    \end{array}
  \right.
\]
which is reflected by the linear subproblem and dual update in Algorithm~\ref{alg:douglas_rachford_general}. The resolvent for $B$ further corresponds to the application of proximal mappings, also see Proposition~\ref{def:prox-mapping}, where the involved proximal operators are the same 
as for the primal-dual iteration in
Algorithm~\ref{alg:primal_dual_general}.
The iteration can be shown to converge for each $\sigma > 0$, see, e.g.,
\cite{briceno-arias2011monotoneskew}.
\begin{algorithm}[t]
  \begin{algorithmic}[1]
\onehalfspacing

\Function{Tikhonov-DR}{$K_h$, $f_h$, $\alpha$}
\State $(u, w) \gets (0,0), (v, \lambda) \gets (0, 0), (\bar v, \bar \lambda)
\gets (0,0)$
\State choose $\sigma > 0$
\Repeat
\State \textbf{Linear subproblem}
\State
$(u,w) \gets 
\left( \begin{array}{cc}
         \id + \sigma^2 \bigl( K_h^*K_h + (D_h^1)^*D_h^1 \bigr) & \sigma^2 (D_h^1)^* D_h^2 \\
         \sigma^2 (D_h^2)^* D_h^1 & \id + \sigma^2 (D_h^2)^*D_h^2 
\end{array}\right)^{-1}$

\mbox{}\hfill$\cdot
\left(\begin{array}{c}
  u - \sigma \bigl( K_h^*\bar \lambda + (D_h^1)^*\bar v \bigr) \\
  w - \sigma (D_h^2)^*\bar v
\end{array}
\right)$
\State \textbf{Dual update}
\State $v \gets \bar v + \sigma (D_h^1 u + D_h^2 w)$
\State $\lambda \gets \bar \lambda + \sigma K_h u$
\State \textbf{Proximal update}
\State $\bar v \gets \bar v + \proj_{\{\|\cdot \|_{\infty,\alpha^{-1}} \leq 1 \}}(2v - \bar v) - v$
\State $\bar\lambda \gets \bar \lambda + \prox_{\sigma S_{f_h}^*}(2\lambda - \bar \lambda) - \lambda$
\Until{stopping criterion fulfilled}
\State \Return{$u$}
\EndFunction
\end{algorithmic}
\caption{Douglas--Rachford scheme for the numerical solution of~\eqref{eq:discrete_saddle_point_problem_general}}\label{alg:douglas_rachford_general}
\end{algorithm}

As for ADMM, the linear subproblem in
Algorithm~\ref{alg:douglas_rachford_general} can be avoided by
preconditioning. Basically, for the above Douglas--Rachford iteration,
the same types of preconditioners can be applied as for ADMM, ranging
from the Richardson-type preconditioner that was discussed in detail
before to symmetric Gauss--Seidel and SSOR-type preconditioners
\cite{bredies2015pdr_mh}. 
In particular, the potential of the latter for TGV-regularised imaging problems was shown in
 \cite{bredies2015tvtgvpdr_mh}.

While all three discussed classes of algorithms, i.e., the primal-dual
method, ADMM, and the Douglas--Rachford iteration can in principle be
used to solve the discrete Tikhonov minimisation problem we are
interested in, experience shows that the primal-dual method is usually
easy to implement as it only involves forward evaluations of the
involved linear operators and simple proximal operators, and thus
suitable for prototyping. It needs, however, norm estimates for the
forward operator and a possible rescaling.  ADMM is, in turn, a very
popular algorithm whose advantage lies, for instance, in its
unconditional convergence (the parameter $\tau > 0$ can be chosen
arbitrarily). Also, in comparison to the primal-dual method, ADMM is
observed to admit, in relevant cases, a more stable convergence behaviour,
meaning less oscillations and faster objective functional reduction in
the first iteration steps. However, ADMM requires the solution of a
linear subproblem in each iteration step which might be expensive or
call for preconditioning.  The same applies to the Douglas--Rachford
iteration which is also unconditionally convergent, comparably stable
and usually involves the solution of a linear subproblem in each
step. In contrast to ADMM it bases, however, on the same saddle-point
formulation as the primal-dual methods such that translating a
prototype primal-dual implementation into a more efficient
Douglas--Rachford implementation with possible preconditioning is more
immediate.

\section{Applications in image processing and computer vision} \label{sec:applications_imaging_computer_vision}

\subsection{Image denoising and deblurring}

\begin{figure}[t]
  \centering
  \begin{tabular}{c@{\ \ \ }c@{\ }c@{\ }c}
    \includegraphics[width=0.235\textwidth]{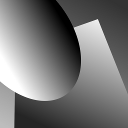} &
    \includegraphics[width=0.235\textwidth]{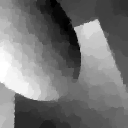} &
    \includegraphics[width=0.235\textwidth]{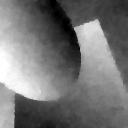} & 
    \includegraphics[width=0.235\textwidth]{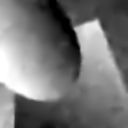} \\
    original image &
    $\alpha\TV$ & 
    $\alpha\TV_\varepsilon$ & 
    $\alpha \TV^2$
    \\[2pt]
    \includegraphics[width=0.235\textwidth]{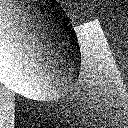} &
    \includegraphics[width=0.235\textwidth]{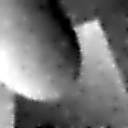} &
    \includegraphics[width=0.235\textwidth]{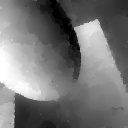} &
    \includegraphics[width=0.235\textwidth]{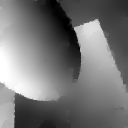} \\
    noisy image &
    $\alpha \norm[\mathcal{M}]{\laplace \placeholder}$ &
    $\alpha_1 \TV \infconv \alpha_2 \TV^2$ &
    $\TGV^2_\alpha$
  \end{tabular}
  
  \caption{Comparison of different first- and second-order
    image models for variational image 
    denoising with $L^2$-discrepancy. Left column: The original
    image (top) and noisy input image (bottom). Columns 2--4: Results for
    variational denoising with different regularisation terms. The 
    parameters were optimised for best PSNR.
    }
  \label{fig:models}
\end{figure}

Image denoising is a simple yet heavily addressed problem in image processing (see for instance \cite{Lebrun12denoising_review_mh} for a review) as it is practically relevant by itself and, in addition, allows to investigate the effect of different smoothing and regularisation approaches independent of particular measurements setups or forward models.
The standard formulation of variational denoising assumes Gaussian noise and, consequently, employs an $L^2$-type data fidelity. Allowing for more general noise models, the denoising problem reads as
\[ \min_{u \in L^p(\Omega)} S_f(u) + \mR_\alpha(u) ,\]
where we assume $S_f:L^p(\Omega) \rightarrow [0,\infty]$, with $p \in [1,\infty]$, to be proper, convex, lower semi-continuous and coercive, and $\mR_\alpha$ to be an appropriate regularisation functional. This setting covers, for instance, Gaussian noise (with $S_f(u) = \frac12\|u-f\|_2 ^2$), impulse noise (with $S_f(u) = \|u-f\|_1$) and Poisson noise (with $S_f(u) = \KL(u,f)$). 
With first- or higher-order $\TV$ regularisation, additive or infimal-convolution-based combinations thereof, or $\TGV$ regularisation, the denoising problem is well-posed for any of the above choices of $S_f$. For $S_f(u) = \frac1q\|u - f\|_q^q$ and $q>1$, also regularisation with $\mR_\alpha(u) = \alpha \|\laplace u\|_\M $ is well-posed and Figure \ref{fig:models} summarises, once again, the result of these different approaches for $q=2$ and Gaussian noise on a piecewise affine test image. It further emphasises again the appropriateness of $\TGV_\alpha^2$ as a regulariser for piecewise smooth images.

In order to visualise difference between different orders of TGV regularisation, Figure \ref{fig:tgv3_test} considers a piecewise smooth image corrupted by Gaussian noise and compares TGV regularisation with orders $k \in \{2,3\}$. It can be seen there that third-order TGV yields a better approximation of smooth structures, resulting in an improved PSNR, while the second-order TGV regularised image has small defects resulting from a piecewise linear approximation of the data.

\begin{figure}[t]
  \centering
  \begin{tabular}{c@{\ }c@{\ }c@{}c}
    \includegraphics[width=0.2\textwidth]{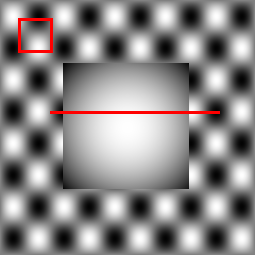} &
    \includegraphics[width=0.2\textwidth]{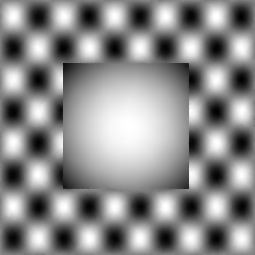} &
    \includegraphics[width=0.3\textwidth,trim=0 6 0 8, clip]{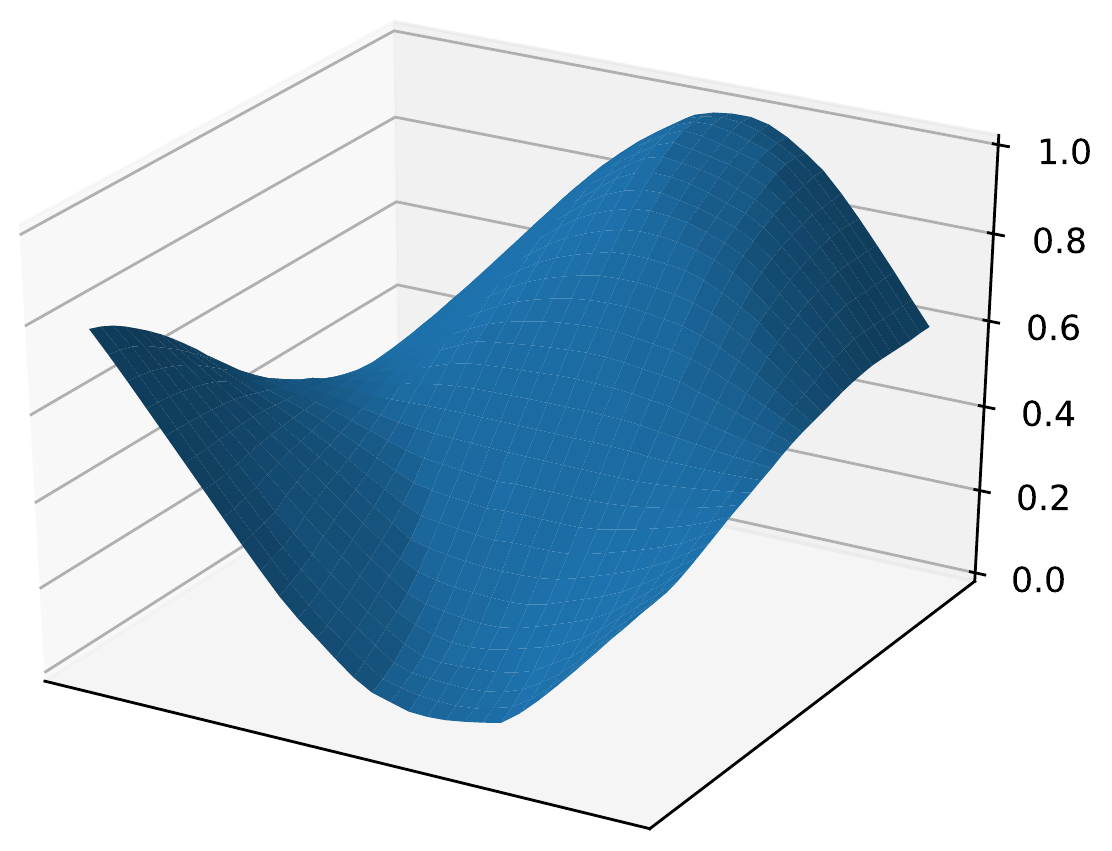} & 
    \includegraphics[width=0.28\textwidth,trim=10 0 10 0, clip]{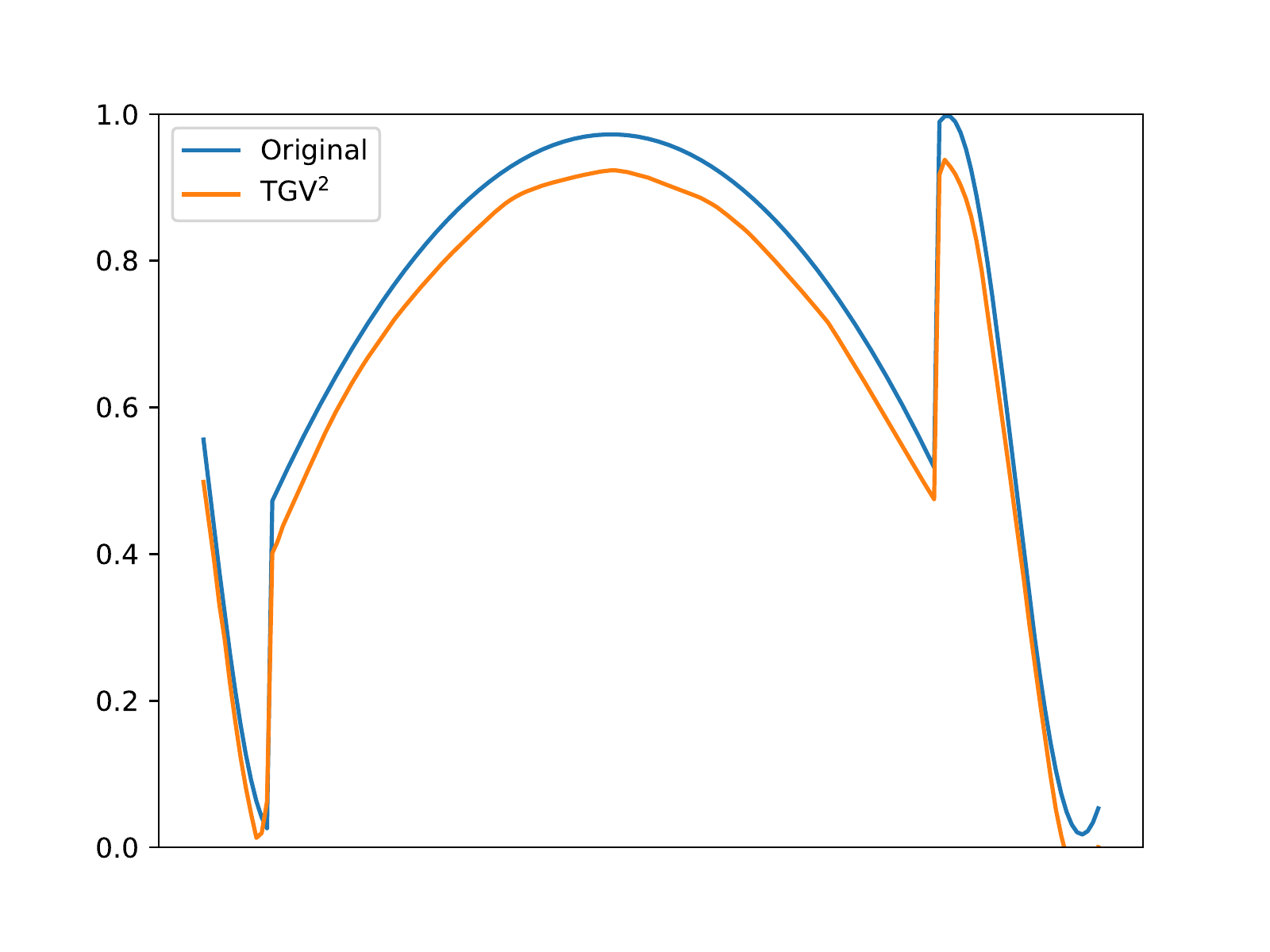} \\
    original image &
    $\TGV_\alpha^2$-denoising & 
    $\TGV_\alpha^2$ surface plot & 
    $\TGV_\alpha^2$ line plot
    \\[2pt]
    \includegraphics[width=0.2\textwidth]{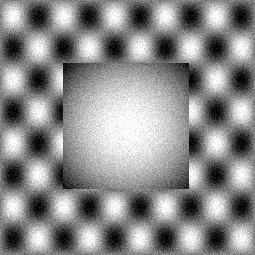} &
    \includegraphics[width=0.2\textwidth]{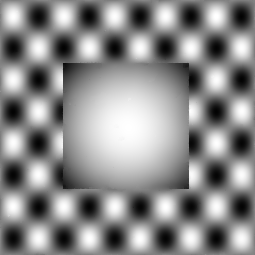} &
    \includegraphics[width=0.3\textwidth,trim=0 6 0 8, clip]{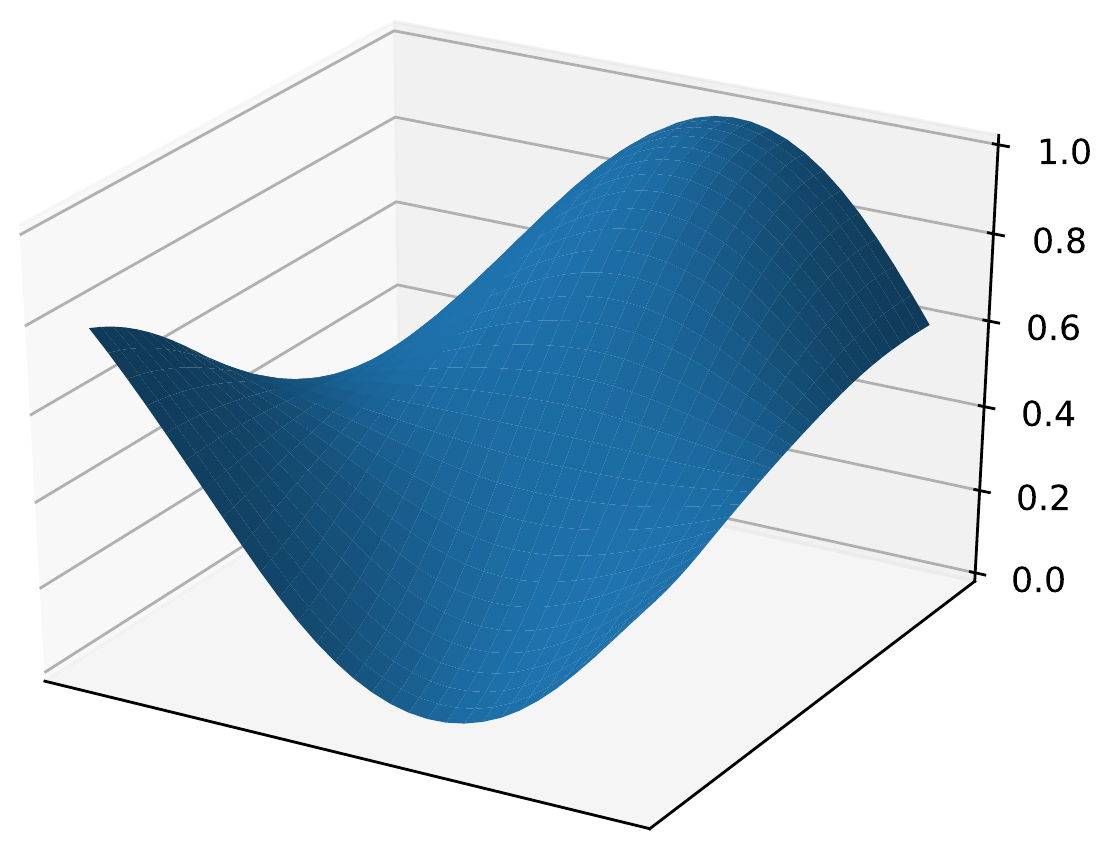} & 
    \includegraphics[width=0.28\textwidth,trim=10 0 10 0, clip]{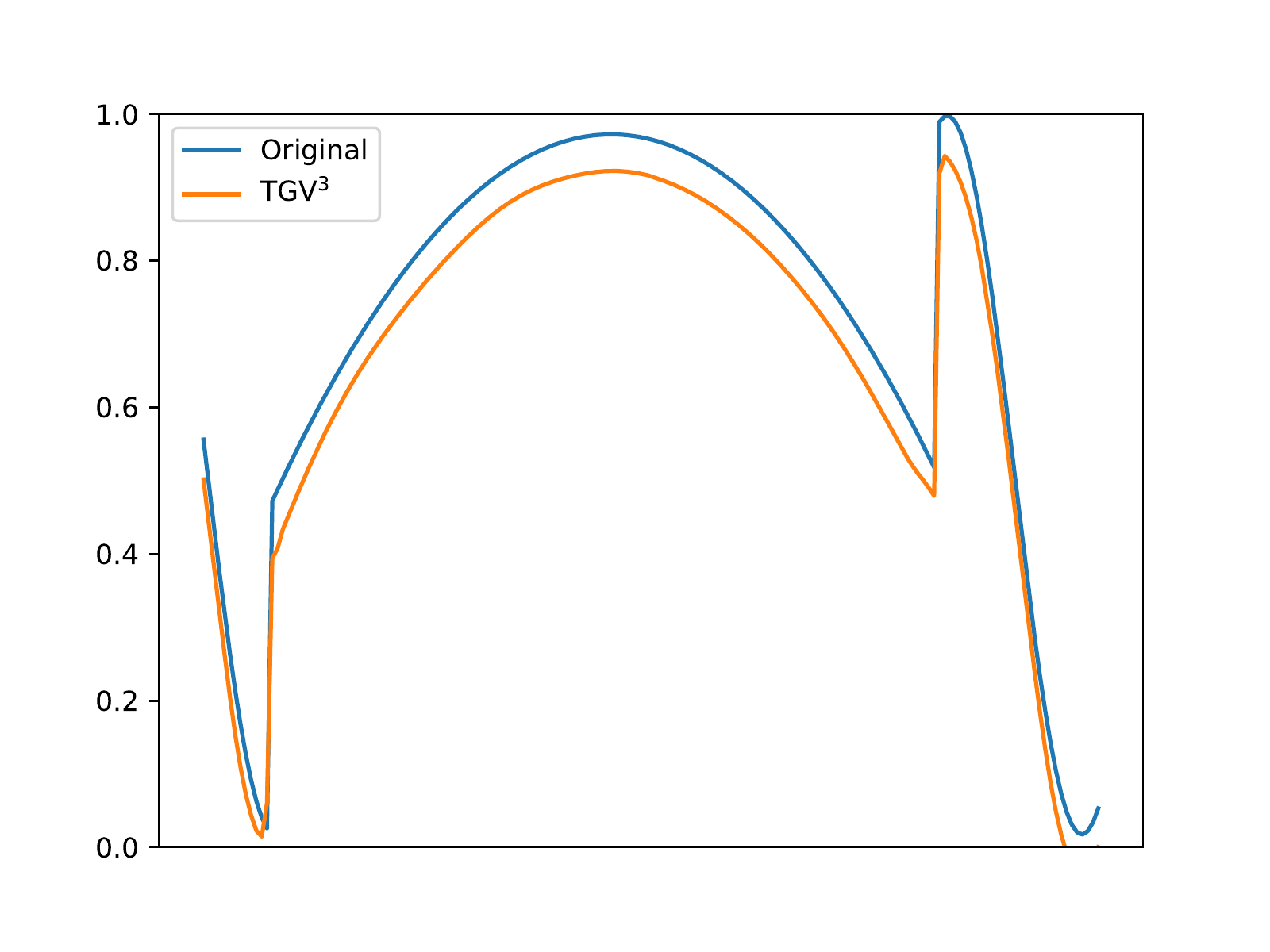} \\
    noisy image &
    $\TGV_\alpha^3$-denoising & 
    $\TGV_\alpha^3$ surface plot & 
    $\TGV_\alpha^3$ line plot
  \end{tabular}
  
  \caption{Comparison of second- and third-order TGV for denoising for a piecewise smooth noisy image (PSNR: $26.0$dB). The red lines in the original image indicate the areas used in the line and surface plots. 
    A close look on these plots reveals piecewise-linearity defects of $\TGV^2_\alpha$, while the $\TGV^3_\alpha$ reconstruction yields a better approximation of smooth structures and an improved PSNR ($\TGV_\alpha^2$: $40.7$dB, $\TGV_\alpha^3$: $42.3$dB). Note that in the line plots, the
    value $0.05$ was subtracted from the $\TGV$-denoising results in order to prevent the respective plots from significantly overlapping with the plots of the original data.  
    }
  \label{fig:tgv3_test}
\end{figure}

Another problem class is 
image deblurring which can be considered as a standard test problem for the ill-posed inversion of linear operators in imaging. Pick a blurring kernel $k \in L^\infty(\Omega_0)$ with bounded domains $\Omega _0,\Omega' \subset \RR^d$  such that $\Omega' - \Omega_0 \subset \Omega$. Then, $K:L^1(\Omega) \rightarrow L^2(\Omega')$ given by
\[
(Ku)(x) = \int_{\Omega_0} u(x-y)k(y) \dd{y}, \quad x \in \Omega'
\]
is well-defined, linear and continuous. Consequently, by Theorems \ref{thm:tv_reg_existence}, \ref{thm:tv_reg_stability} and Proposition \ref{prop:well_posed_tgvk}, 
\[
\min_{u \in \LPspace{p}{\Omega}} \ \frac12 \int_{\Omega'} 
\abs{(u \conv k)(x) - f(x)}^2 \dd{x} + \mR_\alpha(u),
\]
for $1 < p \leq d/(d-1)$ and $\mR_\alpha \in \{ \alpha \TV, \TGV_\alpha^2\}$ admits a solution that stably depends on the data $f \in \LPspace{2}{\Omega'}$, which we assume to be a noise-contaminated image blurred by the convolution operator $K$. A numerical solution can again be obtained with the framework described in Section \ref{sec:numerical_algorithms} and a comparison of the two choices of $\mR_\alpha$ for a test image can be found in Figure \ref{fig:tgv_deconv}. We can observe that both $\TV$ and $\TGV_\alpha^2$ are able to remove noise and blur from the image, however, the $\TV$ reconstruction suffers from staircasing artefacts which are not present with $\TGV_\alpha^2$.

\begin{figure}
  \centering
  \begin{tabular}{c@{\ }c}
    \includegraphics[width=0.45\linewidth]{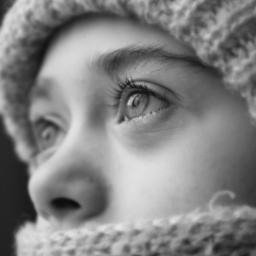} &
    \includegraphics[width=0.45\linewidth]{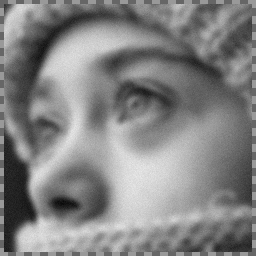} \\[-0.5ex]
    $u_{\mathrm{orig}}$ & $f$ \\[\smallskipamount]
    \includegraphics[width=0.45\linewidth]{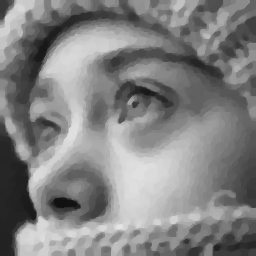} &
    \includegraphics[width=0.45\linewidth]{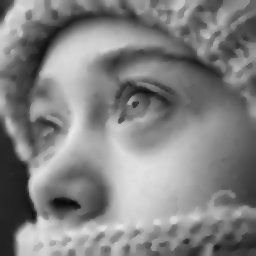} \\[-0.5ex]
    $u_{\TV}$ %
    & $u_{\TGV_\alpha^2}$ %
  \end{tabular}
  \caption{Deconvolution example. The original image $u_\mathrm{orig}$ 
    \cite{alinas_eye} has been blurred and contaminated by 
    noise resulting in $f$. The images $u_{\TV}$ 
    and $u_{\TGV_\alpha^2}$ are the regularised solutions recovered from $f$.}
  \label{fig:tgv_deconv}
\end{figure}

\subsection{Compressed sensing}

The next problem we would like to discuss is compressive sampling with
total variation and total generalised variation
\cite{Bredies14_multichannel_mh}.  More precisely, we aim at
reconstructing a single-channel image from `single-pixel camera' data
\cite{duarte2008compressive_sampling}, an inverse problem with
finite-dimensional data space.  Here, an image is not observed
directly but only the accumulated grey values over finitely many
random pixel patterns are sequentially measured by one sensor, the
`single pixel'. This can be modelled as follows.  For a bounded Lipschitz image
domain $\Omega \subset \RR^2$, let the measurable sets $E_1, \ldots, E_M \subset \Omega$
be the collection of random patterns
where each $E_m$ is associated with the $m$-th measurement. The image
$u$ is then determined by solving the inverse problem
\[
  Ku = f \qquad \text{where} \qquad (Ku)_m = \int_{E_m} u \dd{x} \quad
  \text{for} \quad m = 1,\ldots,M
\]
and $f \in \RR^M$ is the measurement vector, i.e., each $f_m$ is the
output of the sensor for the pattern $E_m$.  As the set of $u$ solving
this inverse problem is an affine space with finite codimension, the
compressive imaging approach assumes that the image $u$ is sparse in a
certain representation which is usually translated into the discrete
total variation $\TV(u)$ being small. A way to reconstruct $u$ from
$f$ is then to solve
\begin{equation}
  \min_{u \in \BV(\Omega)} \ S_f(Ku) + \TV(u), \qquad S_f(v) = \mathcal{I}_{\sett{f}}(v).\label{eq:compressive_imaging_tv}
\end{equation}
In this context, also higher-order regularisers may be used as
sparsity constraint. For instance, in
\cite{Bredies14_multichannel_mh}, total generalised variation of order
$2$ has numerically been tested:
\begin{equation}
  \label{eq:compressive_imaging_tgv}
  \min_{u \in \BV(\Omega)} \ S_f(Ku) + \TGV_\alpha^2(u),
  \qquad S_f(v) = \mathcal{I}_{\sett{f}}(v).
\end{equation}
Figure~\ref{fig:compressive_imaging_example} shows example
reconstructions for real data according to discretized versions
of~\eqref{eq:compressive_imaging_tv}
and~\eqref{eq:compressive_imaging_tgv}. As supported by the theory of
compressed sensing
\cite{candes2006inaccuratemeasurements_mh,candes2006uncertaintyprinciples_mh},
the image can essentially be recovered from a few single-pixel
measurements. Here, TGV-minimisation helps to reconstruct smooth
regions of the image such that in comparison to TV-minimisation, more
features can still be recognised, in particular, when reconstructing
from very few samples. Once again, staircasing artefacts are clearly
visible for the TV-based reconstructions, a fact that recently was
made rigorous in \cite{Carioni18sparsity_mh,Boyer19representer_mh}.

\begin{figure}[t]
  \centering
  \begin{tabular}{c@{\ }c@{\ }c@{\ }c}
    \multicolumn{4}{c}{$\TV$-based compressive imaging reconstruction} \\
    \includegraphics[width=0.235\linewidth]{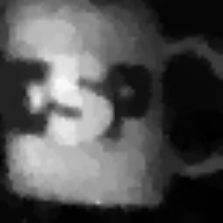} &
    \includegraphics[width=0.235\linewidth]{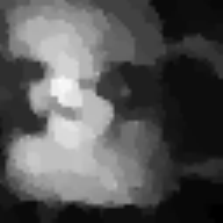} &
    \includegraphics[width=0.235\linewidth]{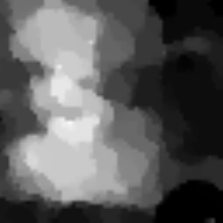} &
    \includegraphics[width=0.235\linewidth]{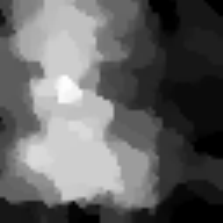}
    \\
    $768$ samples & $384$ samples & $256$ samples & $192$ samples \\[4pt]
    \multicolumn{4}{c}{$\TGV^2$-based compressive imaging reconstruction} \\
    \includegraphics[width=0.235\linewidth]{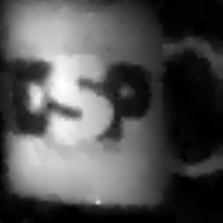} &
    \includegraphics[width=0.235\linewidth]{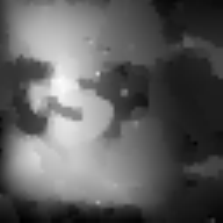} &
    \includegraphics[width=0.235\linewidth]{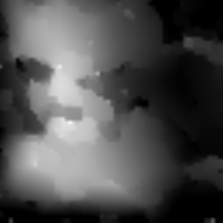} &
    \includegraphics[width=0.235\linewidth]{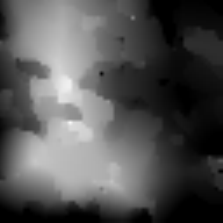}
    \\
    $768$ samples & $384$ samples & $256$ samples & $192$ samples 
  \end{tabular}
  \caption{Example for $\TV$/$\TGV^2$ compressive imaging reconstruction
    for real single-pixel camera data \cite{cs_camera_data}.
    Top: $\TV$-based 
    reconstruction of a $64 \times 64$ image from 18{.}75\%, 9{.}375\%,
    6{.}25\% and 4{.}6875\% of the data (from left to right).
    Bottom: $\TGV^2$-based reconstruction obtained from the same data.
    Figure taken from \cite{Bredies14_multichannel_mh}. Reprinted by permission from Springer Nature.
  }
  \label{fig:compressive_imaging_example}
\end{figure}

\subsection{Optical flow and stereo estimation}

Another important fundamental problem in image processing and computer
vision is the determination of the optical flow \cite{HornSchunck81_optical_flow_mh} of an image
sequence. Here, we consider this task for two consecutive frames $f_0$
and $f_1$ in a sequence of images. This is often modelled by
minimising a possibly joint discrepancy
$S_{f_0,f_1}\bigl(u(0), u(1) \bigr)$ for
$u: [0,1] \times \Omega \to \RR$ subject to the optical flow
constraint $\frac{\partial u}{\partial t} + \grad u \inprod v = 0$,
see, for instance, \cite{borzi2003optimal_mh}.
Here, $v: [0,1] \times \Omega \to \RR^d$ is the optical flow field
that shall be determined. In order to deal with ill-posedness,
ambiguities as well as occlusion, the vector field $v$
needs to be regularised by a penalty term. %
This leads to the PDE-constrained problem
\[
  \min_{u,v} \ S_{f_0,f_1}\bigl(u(0), u(1) \bigr) + \mR_\alpha(v) \qquad
  \text{subject to} \qquad \frac{\partial u}{\partial t} + \grad u \inprod
  v = 0,
\]
where $\mR_\alpha $ is a suitable convex regulariser for vector field
sequences. %
Usually,
$S_{f_0, f_1}$ is chosen such that the initial condition $u(0)$ is
fixed to $f_0$, for instance, $S_{f_0, f_1}(u_0, u_1) = \mI_{\sett{f_0}}(u_0)
+ \frac12 \norm[2]{u_1 - f_1}^2$, see \cite{hinterberger2001opticalflow,borzi2003optimal_mh,keeling2005opticalflow,chen2011opticalflow}.

In many approaches, this problem is reformulated to a correspondence
problem.  This means, on the one hand, %
replacing the optical flow constraint by the displacement introduced
by a vector field $v_0: \Omega \to \RR^2$, i.e., $u(0) = u_0$ and
$u(1) = u_0 \compose (\id + v_0)$. %
The image $u_0: \Omega \to \RR$ is either prespecified or
subject to optimisation. For instance, choosing again
$S_{f_0, f_1}(u_0, u_1) = \mI_{\sett{f_0}}(u_0) + \frac12\norm[2]{u_1
  - f_1}^2$ leads to the classical
correspondence %
problem
\[
  \min_{v_0} \ \frac12\norm[2]{f_0 \compose (\id + v_0)
  - f_1}^2 + \mR_\alpha(v_0),
\]
see, for instance, \cite{HornSchunck81_optical_flow_mh},
which uses the square of the $H^1$-seminorm as a regulariser.
On the other hand, other approaches have been considered for the
discrepancy (and regularisation), see
\cite{brox2004high_mh,zach2007opticalflow}.  In this context, a
popular concept is the \emph{census transform}
\cite{zabih1994censustransform} that describes the local relative
behaviour of an image and is invariant to brightness changes.  For an
image $f: \Omega \to \RR$, measurable patch $\Omega' \subset \RR^2$ and threshold
$\varepsilon > 0$, it is defined as
\[
  \fl
  C_f: \Omega \times \Omega' \to \sett{-1,0,1},
  \quad
  C_f(x,y) =
  \left\{
  \begin{array}{cl}
    \sgn \bigl(f(x+y) - f(x) \bigr)
    &
      \text{if} \ x,x+y \in \Omega \ \text{and} \\
    &
    \ \ \abs{f(x+y) - f(x)} > \varepsilon, \\
    0 & \text{else}.
  \end{array}\right.
\]
Here, one usually
sets $u_0 = f_0$ and $u_1 = f_1$ such that the discrepancy only depends
on the vector field $v_0$, such as, for instance,
\[
  S_{f_0, f_1}(v_0) %
  =
  \int_\Omega \int_{\Omega'} \min \bigl(1, \bigabs{C_{f_0}(x, y)
    - C_{f_1 \compose (\id + v_0)}(x, y)}\bigr)
  \dd{y} \dd{x},
\]
leading to the optical-flow problem
\[
  \min_{v_0} \ S_{f_0, f_1}(v_0) + \mR_\alpha (v_0),
\]
see, for instance, \cite{mueller2011opticalflow,vogel2013evaluation}.
A closely related problem is stereo estimation which can also be
modelled as a correspondence problem. In this context, $f_0$ and $f_1$
constitute a stereo image pair, for instance, $f_0$ being the left
image and $f_1$ being the right image. The stereo information is then
usually reflected by the \emph{disparity} which describes the
displacement of the right image with respect to the left image. This
corresponds to setting the vertical component of the displacement
field $v_0$ to zero, for instance, $(v_0)_2 = 0$. Census-transform based
discrepancies are also used for this task \cite{ranftl2012stereo}, leading to the
stereo-estimation model
\begin{equation}
  \min_{w_0} \ S_{f_0, f_1}\bigl((w_0,0) \bigr) + \mR_\alpha(w_0)\label{eq:stereo_estimation_general}
\end{equation}
with a suitable convex regulariser $\mR_\alpha$ for scalar disparity images.

Both optical flow and stereo estimation are non-convex due to the
non-convex data terms and require dedicated solution techniques. One
possible approach is to smooth the discrepancy functional such that is
becomes (twice) continuously differentiable, and approximate it, for each
$x \in \Omega$, by either first or second-order Taylor expansion.  For
the latter case, if one also projects the pointwise Hessian to the
positive semi-definite cone, one arrives at the convex problem
\[
  \fl
  \min_{v} \ \int_\Omega S(v_0) + \grad S(v_0) \inprod (v - v_0)
  + \frac12 \proj_{S^+}\bigl(\grad^2 S(v_0) \bigr)(v - v_0) \inprod (v - v_0)
  \dd{x} + \alpha \mR(v),
\]
where $v_0$ is the base vector field for the Taylor expansion and
$\proj_{S^+}: S^{2 \times 2} \to S^{2 \times 2}_+$ denotes the
orthogonal projection to the cone of positive semi-definite
matrices $S^{2 \times 2}_+$. Besides classical regularisers such as the $H^1$-seminorm,
the total variation has been chosen \cite{werlberger2012thesis}, i.e.,
$\mR_\alpha = \alpha \TV$, which allows the identification of jumps in
the displacement field associated with object boundaries. The
displacement field is, however, piecewise smooth such that
$\TGV_\alpha^2$ turns out to be advantageous. Further improvements can
be achieved by non-local total generalised variation $\NLTGV^2$, see
\cite{ranftl2014opticalflow}, leading to sharper and more accurate
motion boundaries, see Figure~\ref{fig:optical-flow}.  For stereo
estimation, a similar approach using first-order Taylor expansion and
image-driven total generalised variation $\ITGV_\alpha^2$ also yields
very accurate disparity images \cite{ranftl2012stereo}.

\begin{figure}
  \centering
  \begin{tabular}{c@{\ }c@{\ }c@{\ }c}
    \includegraphics[width=0.3\linewidth]{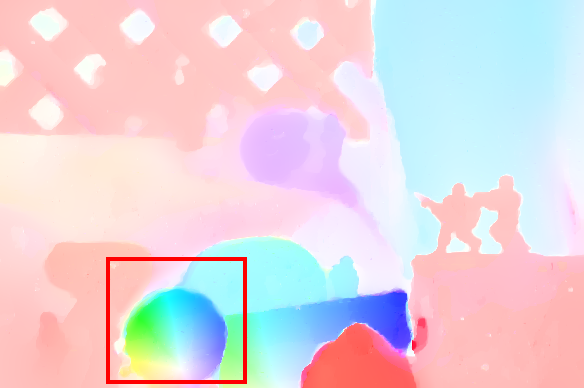}
    & 
    \includegraphics[width=0.15\linewidth]{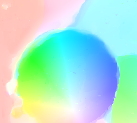}
    &
    \includegraphics[width=0.3\linewidth]{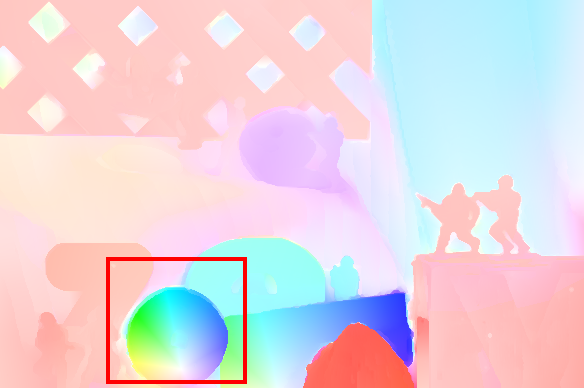}
    &
    \includegraphics[width=0.15\linewidth]{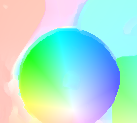}
    \\
    (a) & (b) & (c) & (d)
  \end{tabular}

  \caption{Example for higher-order approaches for optical flow
    determination.  (a) An optical flow field obtained on a sample
    dataset from the Middlebury benchmark \cite{baker2011middlebury} using a
    $\TGV_\alpha^2$ regulariser. (b) An enlarged detail of (a). (c)
    The optical flow field obtained by a $\NLTGV^2$ regulariser.  (d)
    An enlarged detail of (c). Images taken from \cite{ranftl2014opticalflow}. Reprinted by permission from Springer Nature.}
  \label{fig:optical-flow}
\end{figure}

A different concept for solving the non-convex optical flow/stereo
estimation problem is \emph{functional lifting} \cite{DalMaso03_calibration,chambolle2001convex_mh}. For the stereo
estimation problem, this means to recover the characteristic function
of the subgraph of the disparity image, i.e.,
$\chi_{\sett{t \leq w_0}}$. Assume that the discrepancy for the
disparity $w_0$ can be written in integral form, i.e.,
$S_{f_0,f_1}(w_0) = \int_\Omega g\bigl(x, w_0(x)\bigr) \dd{x}$ with a
suitable $g: \Omega \times \RR \to \RR$ that is possibly non-convex
with respect to the second argument.  If $w_0$ is of bounded
variation, then $\chi_{\sett{t \leq w_0}}$ is also of bounded
variation and the weak derivative with respect to $x$ and $t$,
respectively, are Radon measures. Denoting by $v_x$ and $v_t$ the
respective components of a vector, i.e.,
$v = (v_x, v_t) \in \RR^2 \times \RR$, these derivatives satisfy the
identity
$\frac{\partial}{\partial t} \chi_{\sett{t \leq w_0}} = \bigl(
\frac{\grad \chi_{\sett{t \leq w_0}}}{\abs{\grad \chi_{\sett{t \leq w_0}}}}\bigr)_t \abs{\grad \chi_{\sett{t \leq w_0}}}$ as well as
$\grad_x \chi_{\sett{t \leq w_0}} = \bigl( \frac{\grad \chi_{\sett{t \leq w_0}}}{\abs{\grad \chi_{\sett{t \leq w_0}}}}\bigr)_x
\abs{\grad \chi_{\sett{t \leq w_0}}}$.  The discrepancy term
can then be written in the form
\[
  S_{f_0,f_1}\bigl( (w_0, 0) \bigr) =  \int_{\Omega \times \RR}
  g \dd{\Bigabs{\frac{\partial}{\partial t} \chi_{\sett{t \leq w_0}}}}
\]
which is
convex with respect to $\chi_{\sett{t \leq w_0}}$.
In many cases, regularisation functionals
can also be written in terms of $\chi_{\sett{t \leq w_0}}$, for
instance, by the coarea formula,
\[
  \TV(w_0) = \int_{\Omega \times \RR}
  \dd{\abs{\grad_x \chi_{\sett{t \leq w_0}}}},
\]
which is again convex with respect to $\chi_{\sett{t \leq w_0}}$.
As the set of all $\chi_{\sett{t \leq w_0}}$ is still non-convex, this
constraint is usually relaxed to a convex set, for instance, to the conditions
\begin{equation}
  u \in \BV(\Omega \times \RR), \quad 0 \leq u \leq 1,
  \quad
  \lim_{t \to -\infty} u(t, \placeholder) = \ones,
  \quad \lim_{t \to \infty} u(t,\placeholder) = 0,
  \label{eq:functional_lifting_constraint}
\end{equation}
where the limits have to be understood in a suitable sense. Then,
the stereo problem~\eqref{eq:stereo_estimation_general} with total-variation
regularisation can be relaxed to the convex problem
\[
  \min_{u \in \BV(\Omega \times \RR)} \
  \int_{\Omega \times \RR}
  g \dd{\Bigabs{\frac{\partial u}{\partial t}}}
  + \alpha \int_{\Omega \times \RR} \dd{\abs{\grad_x u}}
  \qquad \text{subject to} \ \eqref{eq:functional_lifting_constraint}.
\]
Then, optimal solutions $u^*$ for the above problem yield minimizers
of the original problem when thresholded, i.e., for $s \in {]{0,1}[}$,
the function $\chi_{\sett{s \leq u^*}}$ is the characteristic function
of the subgraph of a $w_0$ that is optimal
for~\eqref{eq:stereo_estimation_general} for the assumed discrepancy
and total-variation regularisation \cite{pock2010convex_relaxation_mh}.

Unfortunately, a straightforward adaptation of this strategy to
higher-order total-variation-type regularisation
functionals is not possible.  For
$\TGV^2$, one can nevertheless benefit from the convexification approach.
Considering the $\TGV_2^\alpha$-regularised problem 
\begin{equation}
  \fl
  \min_{w_0 \in \BV(\Omega), \ w \in \BD(\Omega)}
  \ \int_\Omega g(x, w_0(x) \bigr) \dd{x}
  + \alpha_1 \int_\Omega \dd{\abs{\grad w_0 - w}}  + \alpha_0
  \int_\Omega \dd{\abs{\symgrad w}},\label{eq:stereo_tgv_problem}
\end{equation}
one sees that the problem is convex in $w$ and minimisation with
respect to $w_0$ can still be convexified by functional lifting. For fixed
$w$, the latter
leads to
\[
  \fl
  \begin{array}{rl}
    \displaystyle
    \min_{u \in \BV(\Omega \times \RR)} \ \int_\Omega
    g \dd{\Bigabs{\frac{\partial u}{\partial t}}} + \alpha_1
    \sup \ \Bigl\{ \int_{\Omega \times \RR} u \divergence \varphi \dd{x}
    \ \Bigl| & \displaystyle
              \varphi \in \Ccspace{\infty}{\Omega, \RR^2 \times \RR},
    \abs{\varphi_x(x,t)} \leq 1, \\
    &  \varphi_x(x,t) \inprod w(x) \leq
              \varphi_t(x,t) \ \text{a.e.~in}
      \ \Omega \times \RR \Bigr\} \\
    \text{subject to} \ \eqref{eq:functional_lifting_constraint},
  \end{array}
\]
which is again convex and whose solutions can again be thresholded to
yield a $w_0^*$ that is optimal with respect to $w_0$ for a fixed
$w$. Alternating minimisation then provides a robust solution strategy
for~\eqref{eq:stereo_tgv_problem} based on convex optimisation
\cite{ranftl2013tgvnonconvex}, see Figure~\ref{fig:stereo-est} for
an example.
\begin{figure}
  \centering
  \begin{tabular}{r@{\ }c}
    (a) & \includegraphics[width=0.85\linewidth]{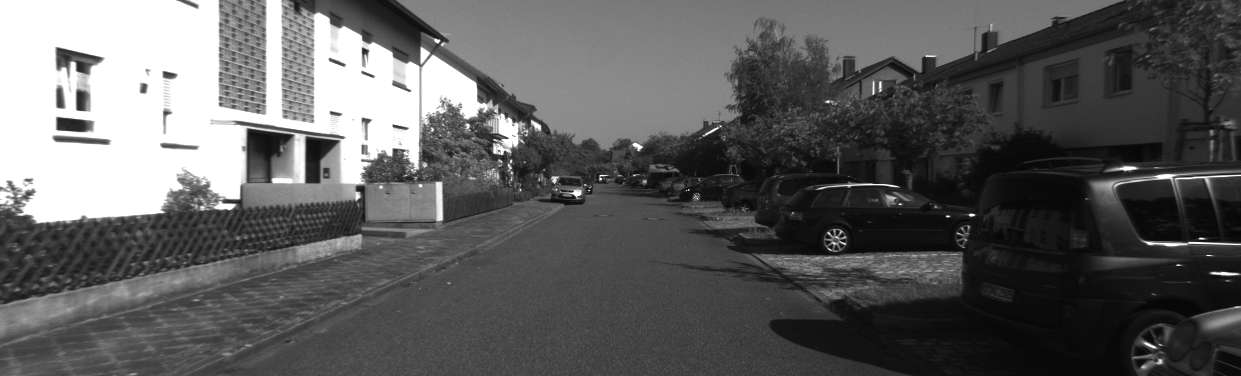} \\
    (b) & \includegraphics[width=0.85\linewidth]{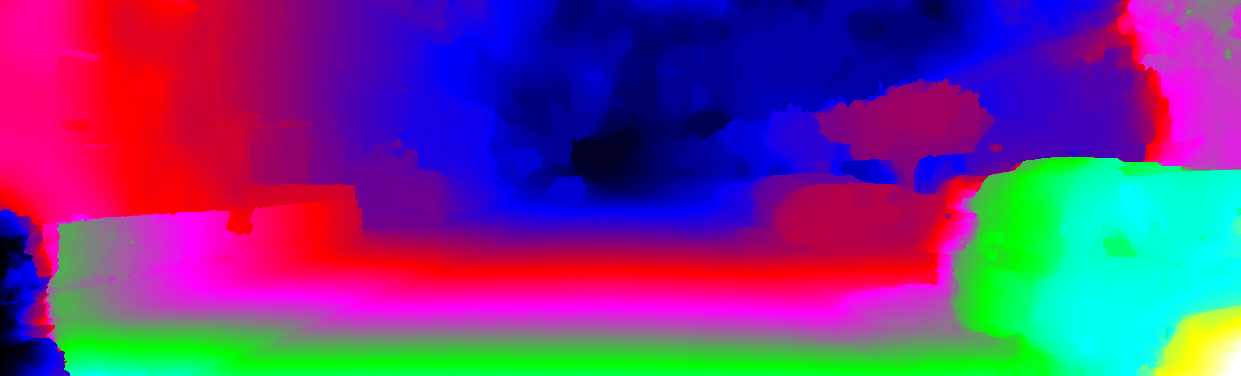} 
  \end{tabular}
  
  \caption{Total-generalised-variation-regularised
    stereo estimation based on functional lifting and convex
    optimisation for an image pair of the KITTI dataset
    \cite{geiger2012kitti}. (a) The reference image.  (b) The
    disparity image obtained with $\TGV^2$-regularisation \cite{ranftl2013tgvnonconvex}. Images
    taken from \cite{tgv_stereo_kitti}.}
  \label{fig:stereo-est}
\end{figure}
In this context, algorithms realising functional lifting strategies
for TV and TGV regularisation have recently further been refined, for
instance, in order to lower the computational complexity associated
with the additional space dimension introduced by the lifting, see,
e.g. \cite{moellenhoff2016sublabel,strecke2019sublabeltgv}.

\subsection{Image and video decompression}

Pixelwise representations of image or image sequence 
data require, on the one hand,
a large amount of digital storage but contain, on the other hand, enough
redundancy to enable compression. Indeed, most digitally stored
images and image sequences, e.g., on cameras, mobile phones or the
world-wide web are compressed. Commonly-used lossy compression standards
such as JPEG, JPEG2000 for images and MPEG for image sequences, however,
suffer from visual artefacts in decompressed data, especially for
high compression rates.

Those artefacts result from errors in the compressed data due to quantisation, which is not accounted for
in the decompression procedure. These errors, however, can be well described using the data
that is available in the compressed file and in particular, precise bounds on the difference of the available data
and the unknown, ground truth data can be obtained. This observation motivates a generic approach for an improved
decompression of such compressed image or video data, which consists of minimising a regularisation functional subject
to these error bounds, see for instance \cite{Zhong97_mh,Alter05_mh,holler12tvjpeg_mh} for TV-based works in this context.
Following this generic approach, we present here a TGV-based reconstruction method (see \cite{holler15tgvrec_p1_mh,holler15tgvrec_p2_mh}) that allows for a variational reconstruction of still
images from compressed data that is directly applicable to the major image compression standards such as JPEG, JPEG2000 or the image compression layer of the DjVu document compression format \cite{holler18djvu}. A further extension of this model to the decompression of MPEG encoded video data will be addressed afterwards.

The underlying principle of a broad class of image and video compression standards, and in particular of JPEG and JPEG 2000 compression, is as follows: First, a linear transformation is used to transform the image data to a different representation where information that is more and less important for visual image quality is well separated. Then, a weighted quantisation of this data (according to its expected importance for visual image quality) is carried out and the quantised data (together with information that allows to obtain the quantisation accuracy) is stored. Thus, defining $K$ to be the linear transformation used in the compression process and $D$ to be a set of admissible, transformed image data that can be obtained using the information available in the compressed file, decompression amounts to finding an image $u$ such that $Ku \in D$. Using the TGV functional to regularise this compression procedure and considering colour images $u:\Omega \rightarrow \R^3$, we arrive at the following minimisation problem:
\begin{equation} \label{eq:cont_problem}
\min_{u\in L^2 (\Omega,\R^3)} \ \mathcal{I}_{U_D} (u) + \TGV_\alpha^k(u), 
\qquad U_D = \set{ u \in L^2 (\Omega,\R^3) }{ Ku \in D }
\end{equation}
 where $ K: L^2(\Omega,\R^3) \rightarrow \ell^2 $ is an analysis operator related to a Riesz basis of $ L^2(\Omega,\R^3)$, and a Frobenius-norm-type coupling of the colour channels is used in $\TGV$, see Subsection \ref{sec:tgv_extensions}. The coefficient dataset 
 $ D \subset \ell ^2 $ reflects interval restrictions on the coefficients, i.e., is defined as $D = \set{ v \in \ell^2 }{ v_n \in J_n \ \text{for all} \ n \in \NN}$ for $\seq{J_n}$ a family of closed intervals. 
 In case $D$ is bounded, well-posedness of this approach can be obtained via a direct extension of Proposition \ref{prop:well_posed_tgvk} to $\R^3$-valued functions, which in particular requires a multi-channel version of the Poincaré inequality for TGV as in Proposition  \ref{prop:tgvk_basic_topological_equivalence}. The latter can straightforwardly be obtained by equivalence of norms in finite dimensions, see for instance
  \cite{Bredies14_multichannel_mh,holler15tgvrec_p1_mh}. 
  Beyond that, existence of a solution to \eqref{eq:cont_problem} can be guaranteed also in case of a non-coercive discrepancy when arbitrarily many of the intervals $J_n$ are unbounded, provided that only finitely many of them are half-bounded, i.e., are the form $J_n = {]{- \infty,c_n}]}$ or $J_n = {[{c_n,\infty}[}$ for $c_n \in \R$, see \cite{holler15tgvrec_p1_mh}. In compression, half-bounded intervals would correspond to knowing only the sign but not the precision of the respective coefficient, a situation which does not occur in JPEG, JPEG2000 and DjVu. %
  Thus, in all relevant applications, all intervals are either bounded or all of $\R$, and hence, solutions exist. %
  Further, under the assumption that all but finitely many intervals have a width that is uniformly bounded from below, again an assumption which holds true in all anticipated applications, optimality conditions for \eqref{eq:cont_problem} can be obtained. 
 
In the application to JPEG decompression, colour images are processed in the YCbCr colour space and the basis transformation operator $K$ corresponds to a colour subsampling followed by a  block- and channel-wise discrete cosine transformation, which together can be expressed as Riesz-basis transform. The interval sequence $\seq{J_n}$ can be obtained using a quantisation matrix that is available in the encoded file and each interval $J_n$ is bounded. 

In the application to JPEG2000 decompression, again the YCbCr colour space is used and $K$ realises a colour-component-wise biorthogonal wavelet transform using Le Gall 5/3 or CDF 9/7 wavelets as defined in \cite[Table 6.1 and Table 6.2]{Cohen92_mh}. Obtaining bounds on the precision of the wavelet coefficients is more involved than with JPEG (see \cite[Section 4.3]{holler15tgvrec_p1_mh}), but can be done by studying the bit truncation scheme of JPEG2000 in detail. As opposed to JPEG, however, the intervals $J_n$ might either be bounded or unbounded.

A third application of the model \eqref{eq:cont_problem} is a variational decompression of the image layers of a DjVu compressed document. DjVu \cite{haffner1998high_mh} is a storage format for digital documents. It encodes document pages via a separation into fore- and background layers as well as a binary switching mask, where the former are encoded using a lossy, transform-based compression and the latter using a dictionary-based compression. While the binary switching mask typically encodes fine details such as written text, the fore- and background layer encode image data, which again suffers from compression artefacts that can be reduced via variational decompression. Here, the extraction of the relevant coefficient data together with error bounds has to account for the particular features of the DjVu compression standard (we refer to \cite{holler18djvu} and its supplementary material for a detailed description and software that extracts the relevant data from DjVu compressed files), but the overall model for the image layers is again similar to the one of JPEG and JPEG2000 decompression. In particular, encoding of the fore- and background layer can be modelled with the operator $K$, in this case corresponding to a colour-component-wise wavelet transformation using the Dubuc--Deslauriers--Lemire (DDL) (4, 4) wavelets \cite{deslauriers1999famille}, and the data intervals $J_n$, which are again either bounded or all of $\R$.

In all of the above applications, a numerical solution of the corresponding particular instance of the minimisation problem \eqref{eq:cont_problem} can be obtained using the primal-dual framework \cite{holler15tgvrec_p2_mh} as described in Section \ref{sec:numerical_algorithms} (see \cite{holler15tgvrec_p2_mh} for details). We refer to Figure \ref{fig:image_compression_examples} for exemplary results using second-order TGV regularisation. Regarding the implementation, relevant differences arise depending on whether the projection onto the dataset $U_D$ can be carried out explicitly or not, the latter requiring a dualization of this constraint and an additional dual variable. Only in the application to JPEG decompression, this projection is explicit due to orthonormality of the cosine transform and the particular structure of the colour subsampling operator. This has the particular advantage that, at any iteration of the algorithm, the solution is feasible and one can, for instance, apply early stopping techniques to obtain already quite improved decompressed images in a computationally cheap way. 

 \paragraph*{Variational MPEG decompression.} The MPEG video compression standard builds on JPEG compression for storing frame-wise image data, but incorporates additional motion prediction and correction steps which can significantly reduce storage size of video data. In MPEG-2 compression, which is a tractable blueprint for the MPEG compression family, video data is processed as subsequent groups of pictures (typically 12-15 frames) which can be handled separately. In each group of pictures, different frame types (I, P and B frames) are defined, and, depending to the frame type, image data is stored by using motion prediction and correction followed by a JPEG-type compression of the corrected data. Similar to JPEG compression, colour images are processed in the YCbCr colour space and additional subsampling of colour components is allowed. 
 
While these are the main features of a typical MPEG video encoder, as usual for most compression standards, the MPEG standard defines the decompression procedure rather than compression. Hence, since compression might differ for different encoders, we build a variational model for MPEG decompression that works with a decoding operator (see \cite{holler15mpeg_mh} for more details on MPEG and the model): Using the information (in particular motion correction vectors and quantisation tables) that is stored in the MPEG compressed files, we can define a linear operator $K$ that maps encoded, motion corrected cosine-transform coefficient data to (colour subsampled) video data. Furthermore, bounds on the coefficient data can be obtained. Using a second operator $S$
to model colour subsampling and choosing a right-inverse $\hat{S}$, MPEG decompression amounts to finding a video $u$ such that
\[ u = s + \hat{S}Kv, \]
where $v \in D$ with $D$ being the admissible set of cosine coefficient data, and $s \in \ker(S)$ compensates for the colour upsampling of $\hat{S}$. Incorporating the infimal-convolution of second-order spatio-temporal TGV functionals as regularisation for video data (see Subsection \ref{sec:tgv_extensions} and \cite{holler15mpeg_mh,holler14ictv}), decompression then amounts to solve
\[ \min_{v \in D, \ s \in \ker(S)}  \ICTGV (s+ \hat{S}Kv) .\]
Again, the minimisation problem can be solved using duality-based convex optimisation methods as described in Section \ref{sec:numerical_algorithms} and we refer to Figure \ref{fig:mpeg_example} for a comparison of standard MPEG-2 decompression and the result obtained with this model.

\begin{figure}
  \centering
  \begin{tabular}{c@{\ }c}
    \includegraphics[width=0.47\linewidth]{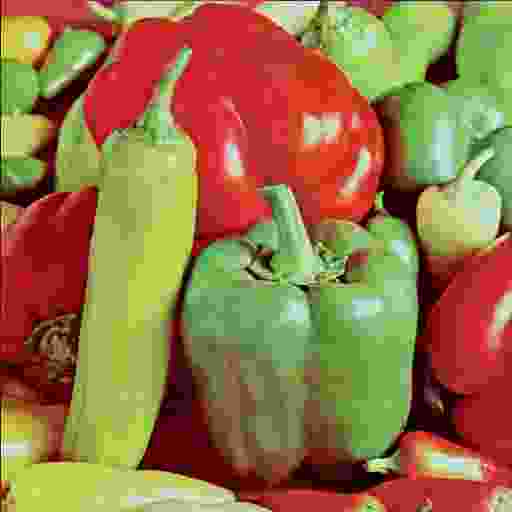} &
    \includegraphics[width=0.47\linewidth]{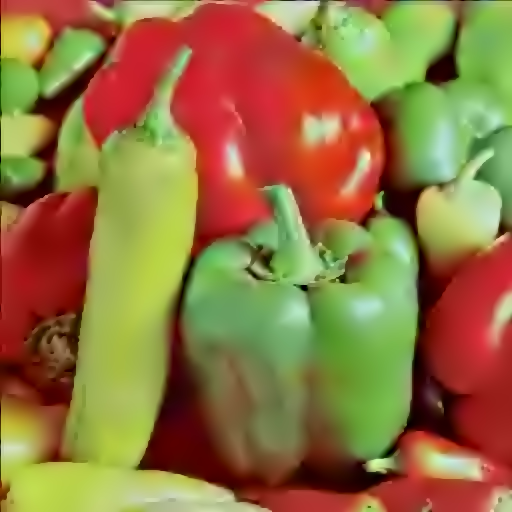} \\
    \includegraphics[width=0.47\linewidth]{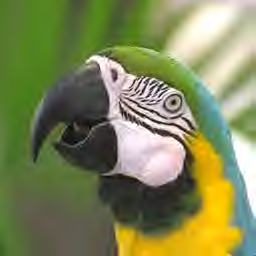} &
    \includegraphics[width=0.47\linewidth]{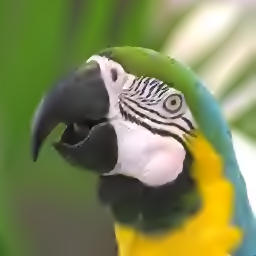} \\
    \scalebox{0.76}{
    	\begin{tikzpicture}
	\draw (0,0) node{\includegraphics[width=4cm]{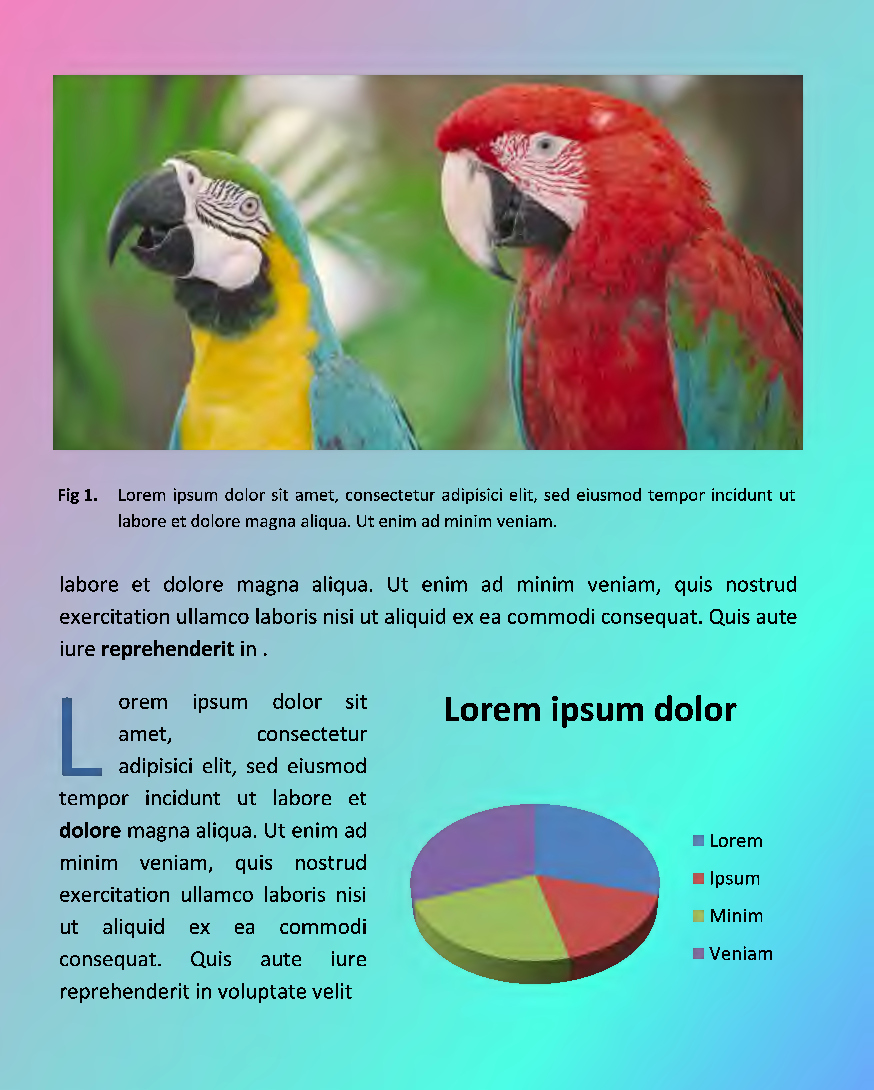}};

	\draw (4.1cm,-1.26) node{\includegraphics[trim = 15cm 3.2cm 8cm 30.4cm, clip, width=4cm]{pics_img_layers_combined_std.jpg}};

	\draw (4.1cm,1.26) node{\includegraphics[trim = 13.5cm 31.4cm 9.5cm 2.2cm, clip, width=4cm]{pics_img_layers_combined_std.jpg}};
	\end{tikzpicture}
	}
	& 
    \scalebox{0.76}{
    	\begin{tikzpicture}
	\draw (0,0) node{\includegraphics[width=4cm]{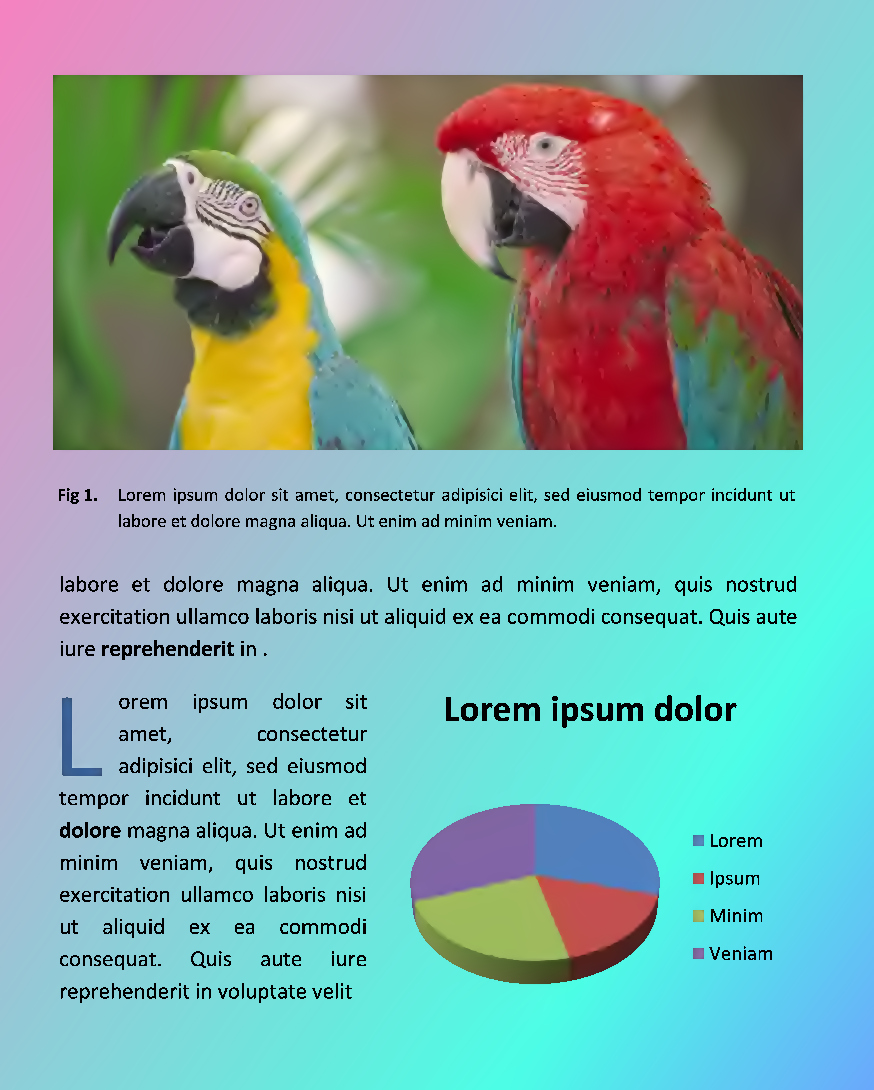}};

	\draw (4.1cm,-1.26) node{\includegraphics[trim = 15cm 3.2cm 8cm 30.4cm, clip, width=4cm]{pics_img_layers_combined_tgv.jpg}};

	\draw (4.1cm,1.26) node{\includegraphics[trim = 13.5cm 31.4cm 9.5cm 2.2cm, clip, width=4cm]{pics_img_layers_combined_tgv.jpg}};
	\end{tikzpicture}
	}
  \end{tabular}
  \caption{\label{fig:image_compression_examples} Example of variational image decompression. Standard (left column) and TGV-based (right column) decompression for a JPEG image (top row) compressed to 0.15 bits-per-pixel (bpp), a JPEG2000 image (middle row) compressed to 0.3 bpp, and a DjVu-compressed document page (bottom row) with close-ups. Results from \cite{holler15tgvrec_p2_mh} (rows 1--2) and \cite{holler18djvu} (bottom row).}
\end{figure}

\begin{figure}[t]
\centering
\includegraphics[scale=0.54]{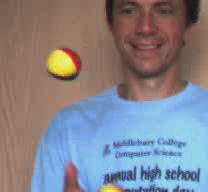}\ \ 
\includegraphics[scale=0.27]{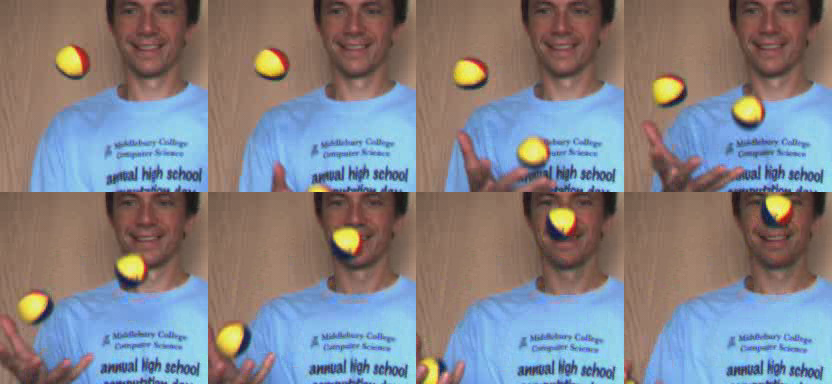}\\[4pt]
\includegraphics[scale=0.54]{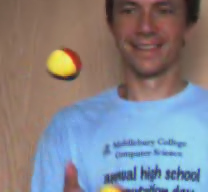}\ \ 
\includegraphics[scale=0.27]{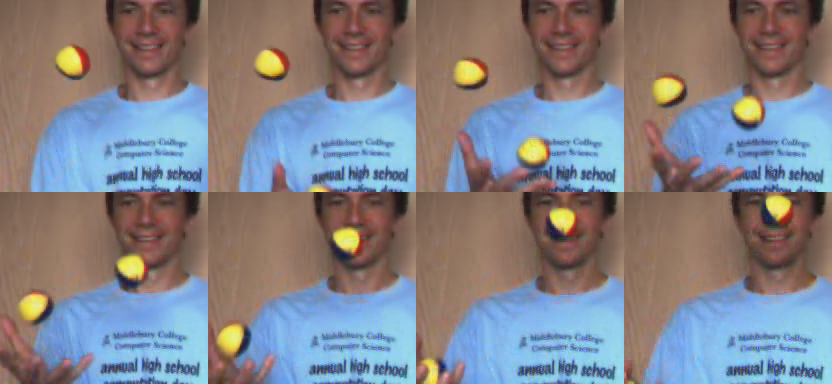}
\caption{\label{fig:mpeg_example} Example of variational MPEG decompression. Standard (top row) and ICTGV-based (bottom row) decompression of the \emph{Juggler} image sequence from \cite{baker2011middlebury}. On the left, the second frame (P-frame) is shown in detail while on the right, all 8 frames are depicted. Figure taken from \cite{holler15mpeg_mh}. Reprinted by permission from Springer Nature.}
\end{figure}

\section{Applications in medical imaging and image reconstruction} \label{sec:applications_medical_imaging_reconstruction}

\subsection{Denoising of dual-energy computed-tomography 
  (CT) data}

Since its development in the 1970s, computed X-ray tomography (CT)
became a standard tool in medical imaging.
As CT bases on X-rays,
the health risks associated with 
ionising radiation is certainly a drawback of this imaging technique. 
Further, the acquired images
do in general not allow to differentiate objects with the same
density.  For the former point, a low radiation dose is an important
goal, being, of course in conflict with the demand of
a high signal-to-noise ratio (SNR).  Regarding the
differentiation of objects with the same density \cite{flohr2006dualenergyct,johnson2007dualenergyct}, a recently developed
approach bases on an additional dataset from a second X-ray source
(typically placed in a 90 degree offset) which possess a different
spectrum (or energy) compared to the standard X-ray emitter in CT, the
\emph{dual-energy CT} device, see Figure~\ref{fig:dual_energy_ct} (a).

Objects of different material having the same response for one X-ray
source may admit a different response for the second source, making
a differentiation possible.
A relevant application of this principle is, for instance, the
quantification of contrast agent concentration. Adjusting a dual-energy
CT device such that normal tissue is insensitive for both X-ray sources
and sensitive for an administered contrast agent allows to infer its
concentration from the difference of the two acquired images, 
see Figure~\ref{fig:dual_energy_ct} (b). This
may be useful, for instance, for recognising perfusion deficits and thus aid 
the diagnosis of, e.g., pulmonary embolism in the lung \cite{lu2010dualenergyct}.
However, due to low doses for the dual-energy CT scan as 
well as a limited sensitivity with respect to the contrast agent, 
the difference image can be noisy and denoising is required in order
to obtain a meaningful interpretation, see 
Figure~\ref{fig:dual_energy_ct} (c).

In the following, a variational denoising approach is derived that
takes the structure of the problem into account.
First, let $A_0$ and
$B_0$ be the noisy CT-reconstructions associated with the respective X-ray 
source. %
Then, as the difference image contains the relevant information, we 
would like to impose regularity on the difference image $A - B$ as well 
as a ``base'' image $B$ instead of penalising each image separately.
As we may assume that the contrast agent concentration as well as the
density is piecewise smooth, is admits a low total generalised variation,
and hence, we choose this functional as a penalty, for instance of second
order. Furthermore, as the results should be usable for a
quantification, we have to account for that and therefore choose an
$L^1$-fidelity term as this is known to possess desirable 
contrast-preservation properties in conjunction with 
$\TV$ and $\TGV$ \cite{chan2005l1tv,bredies2013l1tgv2}. In total,
this leads to the variational problem
\[
\min_{(A,B) \in \LPspace{1}{\Omega}^2}
\ \norm[1]{A - A_0} + \norm[1]{B - B_0} + \TGV^2_\alpha(B)
+ \TGV_{\alpha'}^2(A - B)
\]
where $A_0, B_0 \in \LPspace{1}{\Omega}$ is given and
$\alpha = (\alpha_0,\alpha_1)$ as well as
$\alpha' = (\alpha_0',\alpha_1')$ are positive regularisation
parameters. Having the application in mind, the domain $\Omega$ is typically
a bounded three-dimensional domain with Lipschitz boundary.  Then,
existence of minimizers can be obtained using the tools from
Section~\ref{sec:tgv}, see Proposition~\ref{prop:well_posed_tgvk},
which nevertheless requires some straightforward adaptations. Due to
the lack of strict convexity, however, the solutions might be
non-unique.  Further, numerical algorithms can be developed along the
lines of Section~\ref{sec:numerical_algorithms}, for instance, a
primal-dual algorithm as outlined in
Subsection~\ref{subsec:general_saddle_point}. In case of
non-uniqueness, the minimisation procedure ``chooses'' one solution in
the sense that it converges to one element of the solution set, such
that variational model and optimisation algorithm cannot clearly be
separated and other results might be possible using different
optimisation algorithms.

Figure~\ref{fig:dual_energy_ct} (d) shows denoising results for the
primal-dual algorithm, where a clear improvement of image quality for
the difference image in comparison to Figure~\ref{fig:dual_energy_ct}
(c) can be observed. In particular, the total generalised variation
model is suitable to recover the smooth distribution of the contrast
agent within the lung, including the perfusion deficit region, as well
as the discontinuities induced by bones, vessels,  etc.  Further, one can
see that the dedicated modelling of the problem as a denoising problem
for a difference image based on two datasets turns out to be
beneficial. A denoising procedure that only depends on the noisy
difference image would not allow for such an improvement of image
quality.
 
\begin{figure}[t]
  \centering
  \begin{tabular}{r@{\ }c@{\ \ }c@{\ }l}
    (a) &\includegraphics[width=0.264\textwidth]{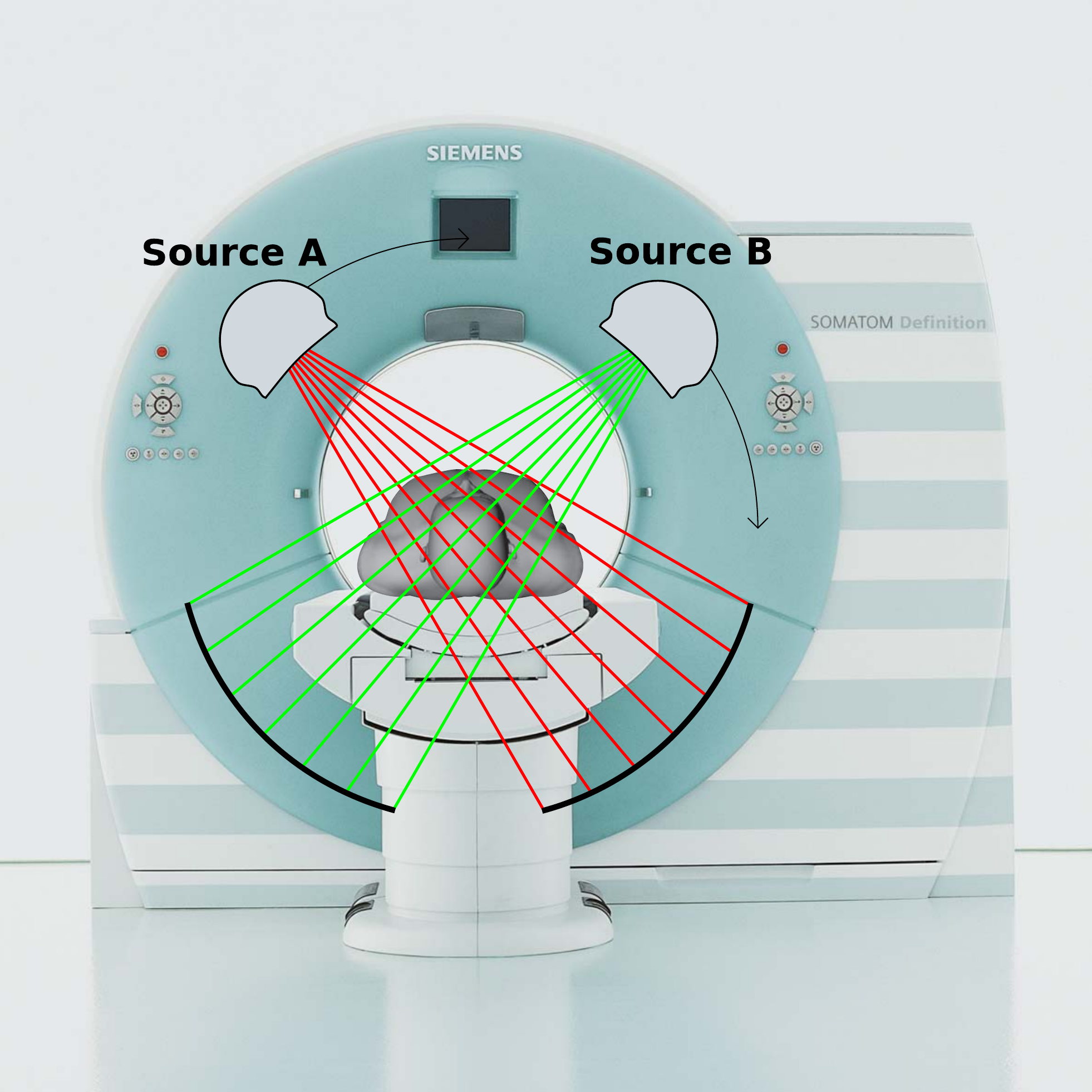}
    & \includegraphics[width=0.264\columnwidth]{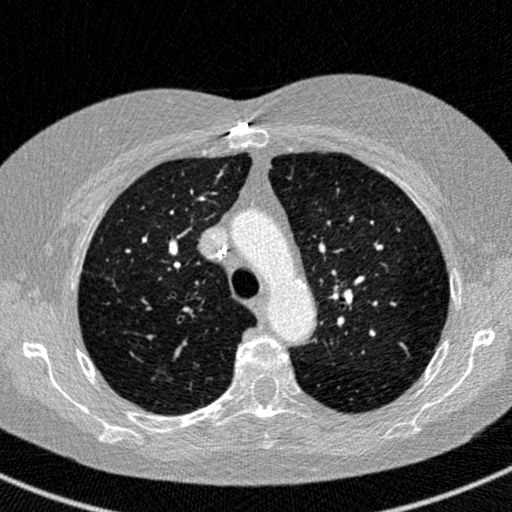}%
      \ \includegraphics[width=0.264\columnwidth]{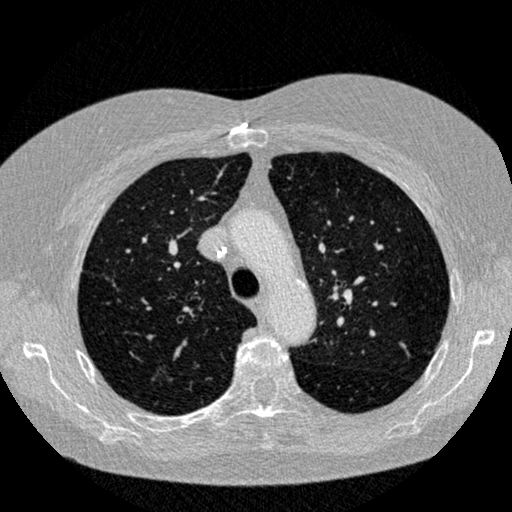}
    & (b)
  \end{tabular}

  \begin{tabular}{r@{\ }c@{\ \ }c@{\ }l}
    (c) &
    \includegraphics[width=0.4\columnwidth]{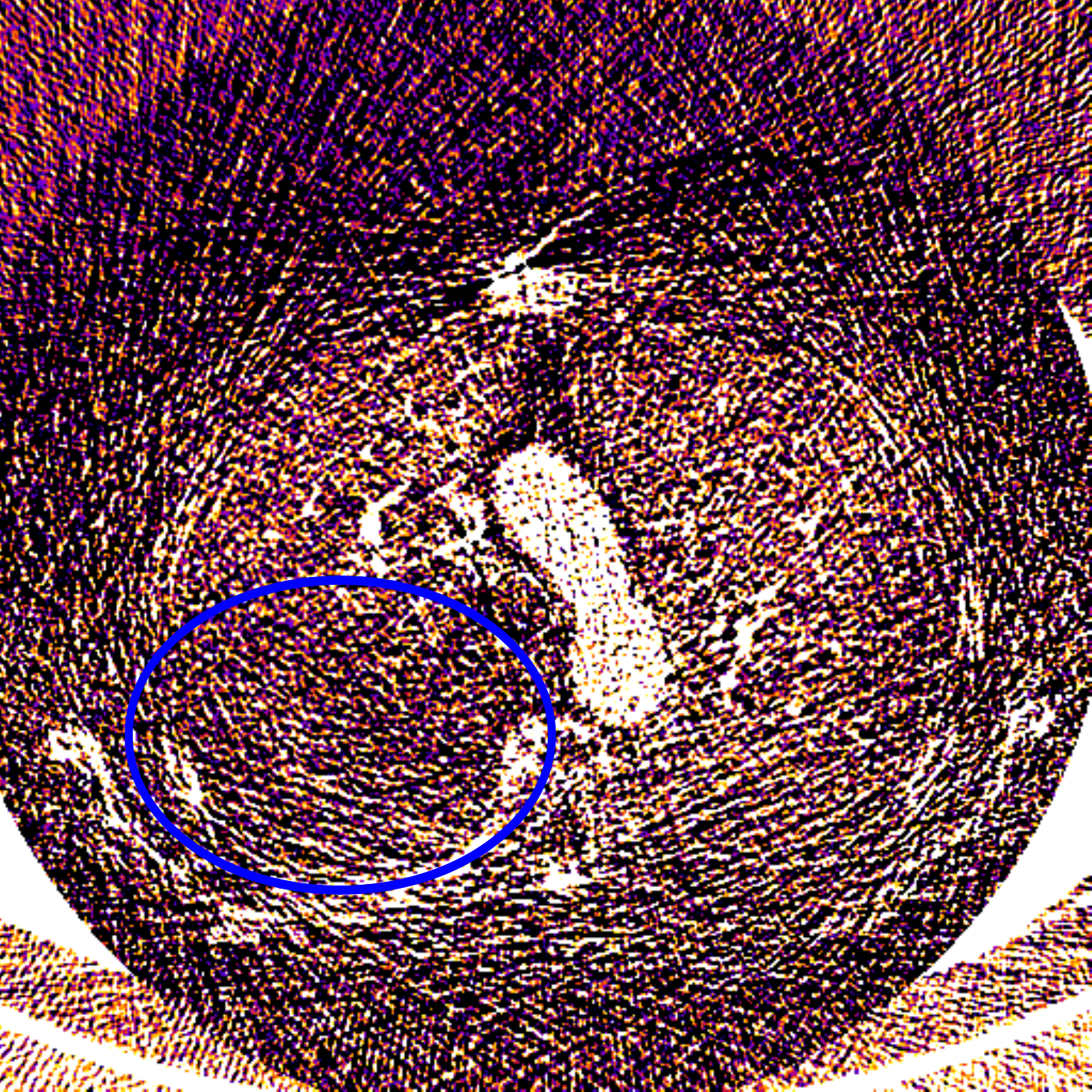}
    &
    \includegraphics[width=0.4\columnwidth]{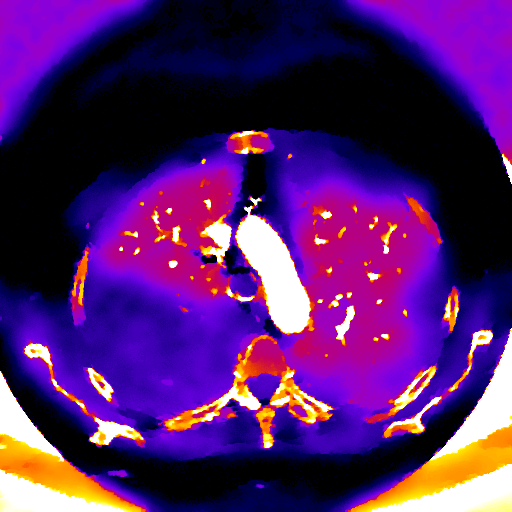}
    & (d)
  \end{tabular}
  \caption{
    Example of $L^1$-$\TGV^2$ denoising for dual energy
    computed tomography.
    (a) A schematic of a dual-energy CT device.
    (b) A pair of (reconstructed) dual-energy CT images.
    (c) A noisy
    difference image with marked perfusion deficit region.
    (d) Difference image of the TGV-denoised dataset (3D denoising,
    only one slice is shown).
  }
  \label{fig:dual_energy_ct}
\end{figure}

\subsection{Parallel reconstruction in 
  magnetic resonance imaging} \label{sec:parallel_mri}

Magnetic resonance imaging (MRI) is a tomographic imaging technique that is heavily used in medical imaging and beyond. It builds on an interplay of magnetic fields and radio-frequency pulses, which allows for localised excitation and, via induction of current in receiver coils, for a subsequent measurement of the proton density inside the object of interest \cite{Brown14_mri_book_mh}.
In the standard setting, MRI delivers qualitative images visualising the density of hydrogen protons, e.g., inside the human body. Its usefulness is in particular due to an excellent soft tissue contrast (as opposed to computed tomography) and a high spatial resolution of MR images. The trade-off, in particular for the latter, is the long measurement time, which comes with obvious drawbacks such as patient discomfort, limitations on patient throughput and imaging artefacts resulting from temporally inconsistent data due to patient motion.

Subsampled data acquisition and parallel imaging \cite{Sodickson1997_multicoil_mr,Pruessman1999_sense_mh,Griswold2002_grappa_mh} (combined with appropriate reconstruction methods) are nowadays standard techniques to accelerate MRI measurements. As the data in an MR experiment is acquired sequentially, a reduced number of measurements directly implies a reduced measurement time, however, in order to maintain the same image resolution, the resulting lack of data needs to be compensated for by other means. Parallel imaging achieves this to some extent by using not a single but multiple measurement coils and combining the corresponding measured signals for image reconstruction. On top of that, advanced mathematical reconstruction methods such as compressed sensing techniques \cite{block2007multiplecoils_mh,lustig2007sparsemri_mh}  or, more general, variational reconstruction have been shown to allow for a further, significant reduction of measurement time with a negligible loss of image quality.

In this context, transform-based regularisation techniques \cite{lustig2007sparsemri_mh,Kutyniok18shearlet_3d_mri} and derivative-based techniques \cite{block2007multiplecoils_mh,knoll2011tgv2mri_mh} are among the most popular approaches. More recently, also learning-based methods building on the structure of variational approaches have become very popular \cite{Hammernik2018learning_mri_mh,rueckert19_dmri_learning}.
Here, we focus on variational regularisation approaches with first- and higher-order derivatives. To this aim, we first deal with the forward model of parallel, static MR imaging.

In a standard setting, the MR measurement process can be modelled as measuring Fourier coefficients of the unknown image. In order to include measurements from multiple coils, spatially varying sensitivity profiles of these coils also need to be included in the forward model via a pointwise multiplication in image space. Subsampled data acquisition then corresponds to measuring the Fourier coefficients only on a certain measurement domain in Fourier space, which is defined by a subsampling pattern. 
Let $c_1,\ldots,c_k \in \mathcal{C}_0(\R^d,\CC)$ be functions modelling some fixed coil sensitivity profiles for $k$ receiver coils, let $\sigma$ be a positive, finite Radon measure on $\RR^d$ that defines the sampling pattern, and let $\Omega \subset \RR^d$ be a bounded Lipschitz domain that represents the image domain. Then, following the lines of \cite{Bredies2019optimal_tp_dynamic_mh}, we define, for $p \in [1,\infty]$, the MR measurement operator $K: \LPspace{p}{\Omega, \CC} \rightarrow L^2_\sigma(\R^d,\CC)^k$ as 
\begin{equation} \label{eq:parallel_mr_forward_operator}
(Ku)_i(\xi) = \widehat{c_iu}(\xi) = \frac{1}{(2\pi)^{d/2}}\int_{\RR^d} c_i(x) u(x) \e^{-\imag\xi \cdot x } \dd x, \quad \xi \in \RR^d,
\end{equation} 
where we %
extend $u$ by zero to $\RR^d$. Note that for each $u \in \LPspace{p}{\Omega,\CC}$,
$Ku$ as a function on $\RR^d$ is bounded and continuous which follows from %
\[ |(Ku)_i(\xi)| \leq \frac{\|c_i\|_\infty}{(2\pi)^{d/2}} \|u\|_{1} \leq C\frac{\|c_i\|_\infty}{(2\pi)^{d/2}} \|u\|_p. \]
Thus, since $\sigma$ is finite, $K$ indeed linearly and continuously maps into $\LPspace[\sigma]{2}{\RR^d,\CC}$.

While here, we assume the coil sensitivities to be known (such that the forward model is linear), obtaining them prior to image reconstruction is non-trivial and we refer to \cite{block2007multiplecoils_mh,schloegl15_coil_mh,Walsh00coil_recon_mh,Uecker14espirit_mh} for some existing methods. In the experiments discussed below, we followed the approach of \cite{block2007multiplecoils_mh} and employed, for each individual coil, a variational reconstruction with a quadratic regularisation on the derivative ($H^1$-regularisation) followed by the convolution with a smoothing kernel. For each coil, the sensitivity profile was then obtained by division with the \emph{sum-of-squares} image (which is, despite its name, the pointwise square root of the sum of the squared modulus of the individual coil images).

A regularised reconstruction from MR measurement data $f \in \LPspace[\sigma]{2}{\RR^d,\CC}^k$ can be obtained by solving 
\begin{equation} \label{eq:mri_energy_general}
  \min_{u \in \LPspace{p}{\Omega,\CC}} \ \frac12 \sum_{i=1}^k \|(Ku)_i - f_i\|_2%
  ^2 + \mR_\alpha(u),
\end{equation}
where we test with both $\mR_\alpha = \alpha\TV$ and $\mR_\alpha = \TGV_\alpha^2$, in which case we can choose $1 < p \leq d/(d-1)$. Note that well-posedness for \eqref{eq:mri_energy_general} follows from Theorems \ref{thm:tv_reg_existence}, \ref{thm:tv_reg_stability} in the case of TV and from Proposition \ref{prop:well_posed_tgvk} in the case of $\TGV_\alpha^2$ (where straightforward adaptions are necessary to include complex-valued functions).
Numerically, the optimisation problem can be solved using the algorithmic framework described in Section \ref{sec:numerical_algorithms}, where again, some modifications are necessary to deal with complex-valued images.

Figure \ref{fig:recons_good_sens} compares the results between these two choices of regularisation functionals and a conventional reconstruction based on direct Fourier inversion using non-uniform fast Fourier 
transform (NUFFT)~\cite{fessler2003nufft} for different subsampling factors and a dataset for which a fully sampled ground truth is available.
Undersampled 2D radial spin-echo measurements of the human brain were 
performed with a clinical 3T scanner
using a receive-only 12 channel head coil.
Sequence parameters 
were: $T_R = $ \unit[2500]{ms}, $T_E = $ \unit[50]{ms}, matrix size 
$256 \times 256$, slice thickness \unit[2]{mm}, in-plane resolution 
\unit[0.78]{mm}$\times$\unit[0.78]{mm}.
The sampling direction of every second spoke was reversed to 
reduce artefacts from off-resonances~\cite{Block08thesis},
and numerical experiments were performed using
96, 48 and 24 projections.
As 
$\frac{\pi}{2} N$ projections ($402$ for $N=256$ in our case) have to be 
acquired to obtain a fully sampled dataset in line with the Nyquist 
criterion~\cite{bernstein2004handbook_mri_mh}, this corresponds to 
undersampling factors of approximately 4, 8 and 16. The raw data was 
exported from the scanner, and image reconstruction was performed offline.

It can be seen that in particular at higher subsampling factors, variational, derivative-based reconstruction reduces artefacts stemming from limited Fourier measurements. Both TV and TGV perform well, while a closer look reveals that staircasing artefacts present with TV can be avoided using second-order TGV regularisation.

\begin{figure}
  \centering
  \begin{tabular}{r@{\ }c@{\ }c@{\ }c@{\ }c}
    & NUFFT & TV & TGV \\
    & \includegraphics[width=0.30\textwidth]{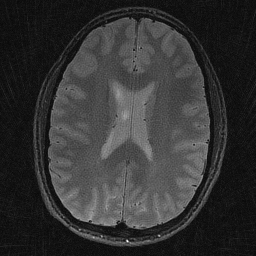}
            & \includegraphics[width=0.30\textwidth]{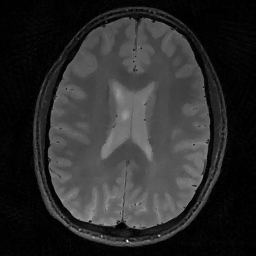}
                 & \includegraphics[width=0.30\textwidth]{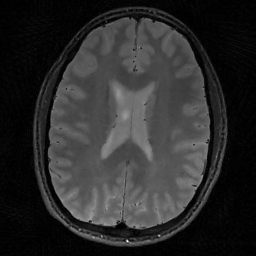}
    \\[-1.5\smallskipamount]
    96 & \includegraphics[trim=140 40 60 190,clip,width=0.30\textwidth]{pics_recon_96_nufft.png}
            & \includegraphics[trim=140 40 60 190,clip,width=0.30\textwidth]{pics_recon_96_tv.png}
              & \includegraphics[trim=140 40 60 190,clip,width=0.30\textwidth]{pics_recon_96_tgv.png}
      \\
    & \includegraphics[width=0.30\textwidth]{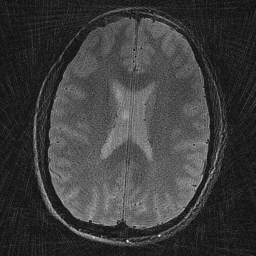}
            & \includegraphics[width=0.30\textwidth]{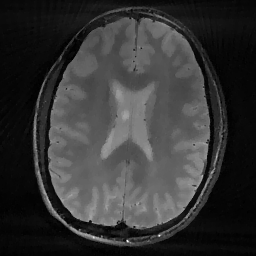}
                 & \includegraphics[width=0.30\textwidth]{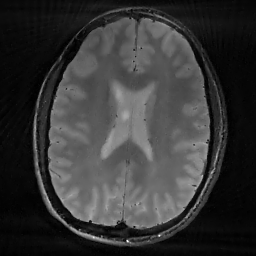}
    \\[-1.5\smallskipamount]
    48 & \includegraphics[trim=140 40 60 190,clip,width=0.30\textwidth]{pics_recon_48_nufft.png}
            & \includegraphics[trim=140 40 60 190,clip,width=0.30\textwidth]{pics_recon_48_tv.png}
              & \includegraphics[trim=140 40 60 190,clip,width=0.30\textwidth]{pics_recon_48_tgv.png}
    \\
     & \includegraphics[width=0.30\textwidth]{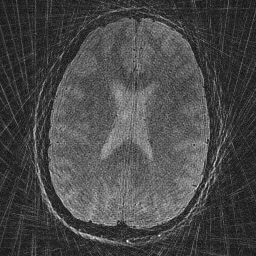}
            & \includegraphics[width=0.30\textwidth]{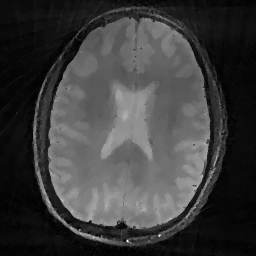}
                 & \includegraphics[width=0.30\textwidth]{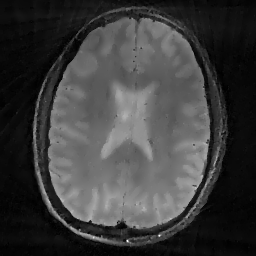}
    \\[-1.5\smallskipamount]
    24 & \includegraphics[trim=140 40 60 190,clip,width=0.30\textwidth]{pics_recon_24_nufft.png}
            & \includegraphics[trim=140 40 60 190,clip,width=0.30\textwidth]{pics_recon_24_tv.png}
              & \includegraphics[trim=140 40 60 190,clip,width=0.30\textwidth]{pics_recon_24_tgv.png}
  \end{tabular}
  
  \caption{Parallel undersampling MRI of the human brain (256$\times$256 pixels) from 96, 48 and 24 radial projections (top, middle, bottom row). Left column: Conventional NUFFT reconstruction. Middle column: Reconstruction with TV regularisation.
    Right column: Reconstruction with $\TGV^2$ regularisation. All reconstructed images
  are shown with a closeup of the lower right brain region.}
  \label{fig:recons_good_sens}
\end{figure}

\subsection{Diffusion tensor imaging}

Magnetic resonance imaging offers, apart from obtaining morphological
images as outlined in Subsection~\ref{sec:parallel_mri}, many other
possibilities to acquire information about the imaged objects.
Among these possibilities, diffusion tensor imaging (DTI) is one of the more
recent developments. It aims at measuring the diffusion directions of
water protons in each spatial point. The physical background is given
by the Bloch--Torrey equation which describes the spatio-temporal
evolution of magnetisation vector taking diffusion processes into
account \cite{torrey1956diffusion}.
Based on this, \emph{diffusion-weighted imaging} can be performed
which uses dedicated MR sequences depending on a direction vector
$q \in \RR^3$ in order to obtain displacement information associated
with that direction.

This leads to the following model. Assume that
$\rho_0: \RR^{3} \to \RR$ is the proton density to recover and
$\rho_t: \RR^3 \times \RR^3 \to \RR$ is the function such that for
each $x, x' \in \RR^3$, the value $\rho_t(x,x')$ represents the
probability of a proton moving from $x$ to $x'$ during the time
$t > 0$. By applying a diffusion-sensitive sequence (such as, e.g.,
a pulsed-gradient spin echo \cite{stejskal1965pgse}) associated with the vector $q \in \RR^3$, one is able to
measure in $k$-space as follows:
\[
  S(k, q) = \frac{1}{(2\pi)^3} \int_{\RR^3} \rho_0(x) \expE^{- \imag k
    \inprod x} \int_{\RR^3} \rho_t(x, x') \expE^{- \imag q \inprod (x'
    - x)} \dd{x'} \dd{x},
\]
where $k \in \RR^3$, see \cite{callaghan1991principlesmri,cory1990displacementnmr}.  Note that in practice, also the
coil sensitivity profile would influence the measurement as
outlined in Subsection~\ref{sec:parallel_mri}, however, for the sake
of simplicity, we neglect this aspect in the following.  Now, sampling
$q$ across $\RR^3$ would then, in principle, allow to recover the
6-dimensional function $u: (x,x') \mapsto \rho_0(x) \rho_t(x,x')$ by
inverse Fourier transform, since $S(k,q) = (\fourier u)(k - q, q)$ for
each $k,q \in \RR^3$. The 6D-space spanned by the coordinates $k$ and
$q$ is called \emph{$kq$-space}.  Assuming that for a fixed
$q \in \RR^3$, the $k$-space is fully sampled then allows to recover
$f_q: \RR^3 \to \CC$ by inverse Fourier transform, where
\[
  f_q(x) = \rho_0(x) \frac1{(2\pi)^{3/2}} \int_{\RR^3} \rho_t(x,
  x') \expE^{-\imag q \inprod (x' - x)} \dd{x'}.
\]
Obtaining and analysing $f_q$ for a coverage of the $q$-space is
called \emph{$q$-space imaging} which also is the basis of
orientation-based analysis such as \emph{$q$-ball imaging}
\cite{tuch2004qball}. However, as these techniques require too much
measurement time in practice, one usually makes assumptions about the
structure of $\rho_t$ in order to avoid the measurement of $f_q$ for
too many $q$.

Along this line, the probably simplest model is to assume that for
each $x$, $\rho_t(x, \placeholder)$ follows a Gaussian distribution
centred around $x$ with symmetric positive definite covariance matrix
$2t D(x) \in S^{3 \times 3}$, i.e.,
\[
  \rho_t(x,x') = \frac{1}{\sqrt{(4\pi t)^3 \abs{\det D(x)}}}
  \expE^{-\frac{1}{4t} (x' - x) \inprod D(x)^{-1} (x' - x)}.
\]
For fixed $x \in \RR^3$, this can be interpreted as the fundamental
solution of the diffusion equation
\[
  \frac{\partial \rho}{\partial t}
  - \divergence \bigl(D(x) \grad_{x'} \rho \bigr) = 0
  \quad
  \text{in} \quad  {]{0,\infty}[}
\]
shifted by $x$ and evaluated at time $t$. The model for $\rho_t$
thus indeed reflects linear diffusion
through a homogeneous medium. This makes sense as diffusion during the
measurement process is usually orders of magnitudes smaller than the
spatial scale one is interested in, but the homogeneity assumption
might also be violated in case when microstructures are present.
Nevertheless, with this assumption, in the above case of full
$k$-space sampling, one gets
\begin{equation}
  f_q(x) = \rho_0(x) \expE^{-t q \inprod D(x)q}.\label{eq:dti_model}
\end{equation}
Clearly, for $q = 0$, we have $f_0 = \rho_0$, and assuming $\rho_0 > 0$ almost
everywhere leads to the following pointwise equation that is linear in
$D$:
\[
  D \inprod (q \tensor q) = -\frac1t \log \Bigl( \frac{f_q}{f_0}
  \Bigr).
\]
Hence, one can recover by $D$ by measuring $f_0$ and
$f_{q_1}, \ldots f_{q_m}$ for $q_1,\ldots,q_m \in \RR^3$ suitably
chosen, i.e., such that in particular, $D$ is uniquely determinable
from $D \inprod (q_i \tensor q_i)$ for $i=1,\ldots,m$. In particular,
one requires that the symmetric tensors
$q_1 \tensor q_1,\ldots, q_m \tensor q_m$ span the space
$\Sym^2(\RR^3)$, meaning that $m$ must be at least $6$. Note that
according to~\eqref{eq:dti_model}, $f_q$ must be real and
non-negative, such that in practice, it suffices to reconstruct the
absolute value of $f_q$, for instance, by computing the sum-of-squares
image.

The inverse problem for $D$ can then be described as
follows. Restricting the considerations to a bounded domain
$\Omega \subset \RR^3$ and letting $p \in [1,\infty]$ such that
$g_1, \ldots, g_m \in \LPspace{p}{\Omega}$ where
$g_i = -\frac1t \log(f_{q_i}/f_0)$, we aim at solving
\begin{equation}
  D \inprod (q_i \tensor q_i) = g_i \quad
  \text{for} \quad i = 1,\ldots,m,\label{eq:dti_direct}
\end{equation}
for $D \in \LPspace{p}{\Omega, \Sym^2(\RR^3)}$. It is easy to see that
this problem is well-posed, but regularisation is still necessary in
practice as the measurements and the reconstruction are usually very
noisy. To do so, one can, in the case $p=2$, minimizer a Tikhonov
functional with quadratic discrepancy term and positive
semi-definiteness constraints:
\begin{equation}
  \min_{D \in \LPspace{2}{\Omega,\Sym^2(\RR^3)}} \ \frac12 \sum_{i =
    1}^m \norm[2]{D \inprod (q_i \tensor q_i) - g_i}^2 + \mR_\alpha(D) +
  \mI_{\sett{D \geq 0}}(D),\label{eq:dti_tikh}
\end{equation}
Here, $\sett{D \geq 0}$ denotes the set of symmetric tensor fields
that are positive semi-definite almost everywhere in $\Omega$.
Further, the regulariser $\mR_\alpha$ is preferably tailored to the
structure of symmetric tensor fields. Since the $D$ to recover can be
assumed to admit discontinuities, for instance, at tissue borders, the
total deformation $\TD$ for $\Sym^2(\RR^3)$-valued functions as
described in Subsection~\ref{subsec:higher_order_tv}, constitutes a
meaningful regulariser. In this context, higher-order regularisation
via $\TGV_\alpha^2$ for $\Sym^2(\RR^3)$-valued functions according to
Definition~\ref{def:tgv} can be beneficial as, e.g., principal
diffusion directions might smoothly vary within the same tissue type
\cite{valkonen2013tgvdti_mh,valkonen2013tgvdticompare_mh}.
In both cases, problem~\eqref{eq:dti_tikh} is well-posed, admits a
unique solution which can, once discretized, be found numerically by
the algorithms outlined in Section~\ref{sec:numerical_algorithms}.

Once the diffusion tensor field $D$ is obtained, one can use it to
visualise some of its properties. For instance, in the context of
medical imaging, the eigenvectors and eigenvalues of $D$ play a role
in interpreting DTI data. Based on the \emph{fractional anisotropy}
\cite{basser1996quantitativedti}, which is defined as
\[
  \mathrm{FA}_D = \sqrt{\frac{1}{2}} \frac{\sqrt{(\lambda_1 - \lambda_2)^2 + (\lambda_2 - \lambda_3)^2 + (\lambda_3 - \lambda_1)^2}}{\sqrt{\lambda_1^2 + \lambda_2^2 + \lambda_3^2}}
\]
where $\lambda_1,\lambda_2,\lambda_3: \Omega \to {[{0,\infty}[}$ are
the eigenvalues of $D$ as a function in $\Omega$, one is able to
identify isotropic regions ($\mathrm{FA}_D \approx 0$) as well as
regions where diffusion only takes place in one direction
($\mathrm{FA}_D \approx 1$). The latter case indicates the presence of
fibres whose orientation then corresponds to a principal eigenvalue of
$D$. Figure~\ref{fig:dti-recon} shows an example of DTI reconstruction
from noisy data using $\TD$ and $\TGV^2$ regularisation for symmetric
tensor fields, using principal-direction/fractional-anisotropy-based
visualisation. It turns out that also here, higher-order
regularisation is beneficial for image reconstruction
\cite{valkonen2013tgvdti_mh}. In particular, the faithful recovery of
piecewise smooth fibre orientation fields may improve advanced
visualisation techniques such as DTI-based tractography.
\begin{figure}
  \centering
  \begin{tabular}{c@{\ }c@{\ }c}
    \includegraphics[width=0.25\linewidth]{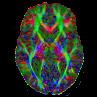} 
    &
    \includegraphics[width=0.25\linewidth]{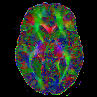} 
    &
    \includegraphics[width=0.25\linewidth]{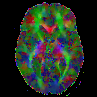} \\[-0.4ex]
    (a) %
    & (b) %
    & (c) %
    \\[0.6ex]
    \includegraphics[width=0.25\linewidth]{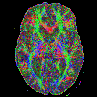} 
    &
    \includegraphics[width=0.25\linewidth]{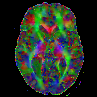} 
    &
    \includegraphics[width=0.25\linewidth]{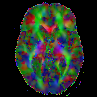} \\[-0.4ex]
    (d) %
    & (e) %
    & (f) %
  \end{tabular}
  \newlength{\colw}
  \newlength{\w}
  \setlength{\colw}{0.235\textwidth}
  \setlength{\w}{0.5\colw}
  \begin{tikzpicture}
    \scriptsize
    
    \pgftext[at=\pgfpoint{-0.25\w}{-0.25\w},left,bottom]{%
      \includegraphics[height=0.5\w]{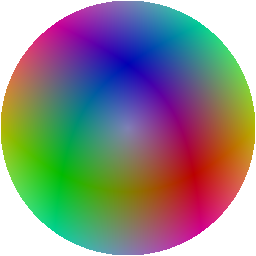}%
    }
    
    \draw[->,color=white] (0,0)--node[above] {$x$} (0.1414\w,-0.1\w);
    \draw[->,color=white] (0,0)--node[below] {$y$} (-0.1414\w,-0.1\w);
    \draw[->,color=white] (0,0)--node[left] {$z$} (0,0.1414\w);
    
    \def\mp#1#2{\begin{minipage}{#1}
        #2
        Colour-coding of a principal eigenvector %
        of $D$.
        The intensity is the function
        \mbox{$\min\{1,\mbox{FA}_D+1/3\}$}
        of the fractional anisotropy $\mbox{FA}_D$.
      \end{minipage}}

    \node[below] at (0,-0.4\w) {\mp{0.65\colw}{\raggedright}};
  \end{tikzpicture}
  \caption{
    Example of $\TD$- and $\TGV^2$-regularised diffusion tensor imaging reconstruction. (a) Ground truth, (d) direct reconstruction from noisy data, (b) direct inversion of~\eqref{eq:dti_direct} followed by $\TD$-denoising, (c) Tikhonov regularisation according to~\eqref{eq:dti_tikh} with $\TD$-regulariser, (e) direction inversion of~\eqref{eq:dti_direct} followed by $\TGV^2$-denoising, (f) Tikhonov regularisation according to~\eqref{eq:dti_tikh} with $\TGV^2$-regulariser. All images visualise one slice of the respective 3D tensor fields.
  }
  \label{fig:dti-recon}
\end{figure}
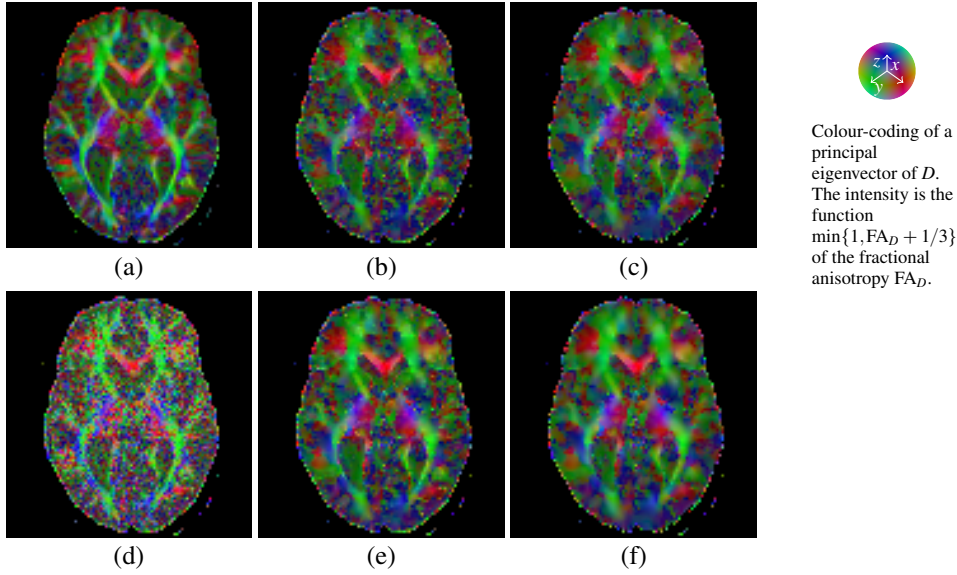

\subsection{Quantitative susceptibility mapping}

Magnetic resonance imaging also has capabilities for the
quantification of certain material properties. One of these properties
is the \emph{magnetic susceptibility} which quantifies the ability of
a material to magnetise in a magnetic field such as the static field
that is used in MRI. Recovering the susceptibility distribution of an
object is called \emph{quantitative susceptibility mapping} (QSM) \cite{shmueli2009qsm,deistung2017qsmoverview}.

Assuming that the static field is
aligned with the $z$-axis of a three-dimensional coordinate system,
this susceptibility can be related to the $z$-component of the static
field inhomogeneity $\delta B_0: \RR^3 \to \RR$ that is caused by the
material, which in turn induces a shift in resonance frequency and,
consequently, a phase shift in the complex image data. For instance,
if $\varphi_t: \RR^3 \to {[{-\pi,\pi}[}$ denotes the phase of an MR
image acquired with a gradient echo (GRE) sequence with echo time
$t > 0$, the relation between $\delta B_0$ and $\varphi_0$ can be
stated as:
\[
  \varphi_t = \varphi_0 + 2\pi \gamma t (\delta B_0) \quad
  \mod \quad 2\pi
\]
where $\varphi_0: \RR^3 \to {[{-\pi,\pi}[}$ is the time-independent
phase offset induced by a single measurement coil and $\gamma$ is the
gyromagnetic ratio. Using multiple coils, the phase offset $\varphi_0$
can be recovered \cite{robinson2017phase} such that we may assume, in
the following, that $\varphi_0 = 0$.  Pursuing a Lorentzian sphere
approach and assuming that in the near field, the magnetic dipoles
moments that cause the magnetisation are randomly distributed, one is able
to relate $\delta B_0$ with the magnetic susceptibility $\chi$
associated with the static field orientation approximately as follows
\cite{schweser2016qsm}:
\[
  \delta B_0 = B_0 (\chi \conv d),
\]
where $B_0$ is the static field strength and $d: \RR^3\without\sett{0} \to \RR$ is the
dipole kernel according to
\[
  d(x,y,z) =  \frac{1}{4\pi} \frac{2z^2 - x^2 - y^2}{(x^2 + y^2 + z^2)^{5/2}}.
\]
Assuming further that the susceptibility is isotropic, i.e., does not
depend on the orientation of the static field, it may be recovered
from the phase data $\varphi_t$. However, the phase image $\varphi_t$
is only well-defined where the magnitude of the MR image is non-zero
(or above a certain threshold). Denoting by $\Omega \subset \RR^3$ a
Lipschitz domain that describes where $\varphi_t$ is available,
recovering $\chi$ then amounts to solving
\[
  2\pi \gamma t B_0 (\chi \conv d) = \varphi_t \quad
  \mod \quad 2\pi \qquad \text{in} \qquad \Omega
\]
for $\chi: \RR^3 \to \RR$. This problem poses several
challenges. First, the values on the left-hand side are only available
up to integer multiples of $2\pi$, such that \emph{phase unwrapping}
becomes necessary. There is a plethora of methods available for doing
this for discrete data \cite{robinson2017phase}, however, in regions
of fast phase change, these methods might not correctly resolve the
ambiguities introduced by phase wrapping. Consequently, the unwrapped
phase image $\varphi_t^{\unwrap}$ might be inaccurate.

With unwrapped phase data being available, the next challenge is to
obtain $\chi$ on the whole space from a noisy version of
$\chi \conv d$ on $\Omega$, which is an underdetermined problem. The
usual approach for this challenge is to split $\chi$ into its
contributions on $\Omega$ and $\RR^3 \without \Omega$ and only aim at
reconstructing $\chi$ on $\Omega$. Now, as the dipole kernel $d$ is
harmonic on $\RR^3 \without \sett{0}$, the function
$\chi|_{\RR^3 \without \Omega} \conv d$ is harmonic in $\Omega$.
Thus, one can write
\begin{equation}
  \left\{
    \begin{array}{rll}
      2 \pi \gamma t B_0(\chi|_\Omega \conv d) + \psi
      & = \varphi_t^{\unwrap} & \text{in} \ \Omega,\\
      \laplace \psi &= 0 & \text{in} \ \Omega,\\
    \end{array}
  \right.\label{eq:qsm_unwrap}
\end{equation}
and solve this equation instead. For QSM, one often estimates $\psi$
first and subtracts this estimate from the data. This step is called
\emph{background field removal} in this context and there are many
different approaches for that \cite{schweser2017background}.
Depending on the accuracy of the background field estimate, this step
may introduce further errors into the data.  Nevertheless, the
procedure results in a foreground field estimate $\varphi_t^{\fg}$
for which only the deconvolution problem
\[
  2\pi \gamma t \chi|_\Omega \conv d = \varphi_t^{\fg} \quad \text{in}
  \quad \Omega
\]
has to be solved. As this problem is ill-posed, it needs to be
regularised. A Tikhonov regularisation approach can then be phrased
as follows:
\[
  \min_{\chi \in \LPspace{p}{\Omega}} \ \frac1p \int_\Omega \abs{\chi \conv d_t - \varphi_t^{\fg}}^p \dd{x} + \mR_\alpha(\chi)
\]
for $1 < p < \infty$, $d_t = 2\pi \gamma t d$ and $\mR_\alpha$ a
regularisation functional on $\LPspace{p}{\Omega}$. As the convolution
with $d_t$ results in a singular integral, the operation
$\chi \mapsto \chi \conv d_t$ is only continuous
$\LPspace{p}{\Omega} \to \LPspace{p}{\Omega}$ by the Calderón--Zygmund
inequality %
\cite{calderon1952singularintegral}, i.e., does not increase
regularity. In this context, first-order regularisers ($H^1$ and
$\TV$) have been used \cite{bilgic2013l1qsm}, but also $\TGV^2$ has
been employed \cite{chatnuntawech2017qsmtvtgv}. Note that in these
approaches, one usually considers $p=2$ which might cause problems
regarding well-posedness for $\TV$ and $\TGV^2$ as in 3D, coercivity
only holds in $\LPspace{3/2}{\Omega}$. This problem can for instance
be avoided by setting $d_t$ to zero in a small ball around zero; a
strategy that also seems consistent with the modelling of the forward
problem \cite{schweser2016qsm}. A numerical solution for $\chi$ then
finally gives a susceptibility map of the region of interest
$\Omega$. However, since the overall procedure involves three
sequential steps, each possibly introducing an error that propagates,
an integrative variational model that essentially only depends on the
original wrapped phase data $\varphi_t$ is desirable.

Such a model can indeed be derived. First, observe that
in case of sufficient regularity, the Laplacian of the unwrapped phase
can easily and directly be obtained from $\varphi_t$:
\begin{equation}
  \laplace \varphi_t^{\unwrap} = \Imag \bigl( (\laplace \expE^{\imag
    \varphi_t}) \expE^{-\imag \varphi_t} \bigr),\label{eq:laplace_phase_unwrap}
\end{equation}
such that $\varphi_t^{\unwrap}$ is known up to an additive harmonic
contribution. Indeed, this is the concept behind \emph{Laplacian phase
  unwrapping} \cite{schofield2003phaseunwrap}. Further, introducing
the wave-type operator 
\[
\wave = \frac13 \Bigl ( \frac{\partial^2}{\partial x^2}
+ \frac{\partial^2}{\partial y^2}\Bigr) 
- \frac23 \frac{\partial^2}{\partial z^2}
\]
and noticing that $d = \wave \Gamma$, where
$\Gamma: \RR^3 \without \sett{0} \to \RR$ is the fundamental solution
of the Laplace equation, i.e.,
$\Gamma(x,y,z) = \frac1{4\pi} (x^2 + y^2 + z^2)^{-1/2}$, it follows
from~\eqref{eq:qsm_unwrap} that
\begin{equation}
  2\pi \gamma t \wave \chi = \laplace \varphi_t^{\unwrap} \qquad
  \text{in} \qquad \Omega.\label{eq:qsm_pde}
\end{equation}
In particular, the harmonic contribution from the background field
vanishes and the data obtained in~\eqref{eq:laplace_phase_unwrap} can
directly be used on the right-hand side. Thus, only a wave-type
partial differential equation has to be solved in which there is no
longer the need for background field correction. The equation is, however,
missing boundary conditions such that one cannot expect to recover
$\chi$ in all circumstances. Under a-priori assumptions on $\chi$, the
lack of boundary conditions can be mitigated by the introduction of a
regularisation functional. Indeed, assuming that $\chi$ is piecewise
constant and of bounded variation, the minimisation of $\TV$ subject
to~\eqref{eq:qsm_pde} recovers $\chi$ up to an additive constant
\cite{bredies2018pdetv}.

Since the data $\varphi_t$ might be noisy, the variational model
should also account for errors on the right-hand side
of~\eqref{eq:qsm_pde} and introduce a suitable discrepancy
term. Assuming Gaussian noise for $\varphi_t$, the right-hand side
$\laplace \varphi_t^{\unwrap}$ is perturbed by noise in
$H^{-2}(\Omega)$ which suggests a $H^{-2}$-discrepancy term
for~\eqref{eq:qsm_pde}. The latter can be realised by requiring
$\laplace \psi = 2\pi \gamma t \wave \chi - \laplace
\varphi_t^{\unwrap}$ for a $\psi \in \LPspace{2}{\Omega}$ and
measuring the $L^2$-norm of $\psi$. In total, this leads to
\begin{equation}
  \label{eq:qsm_all_in_one}
  \left\{
    \begin{array}{c}
      \displaystyle
      \min_{\chi \in \LPspace{p}{\Omega}, \psi \in \LPspace{2}{\Omega}} \ \frac12 \int_{\Omega} |\psi|^2 \ \mathrm{d}x
      + \mR_\alpha(\chi) \\[\medskipamount]
      \text{subject to} \quad \laplace \psi = 2\pi \gamma t\wave \chi - \laplace \varphi_t^{\unwrap} \quad
      \text{in} \quad \Omega,
    \end{array}\right.
\end{equation}
where $1 \leq p < \infty$, the constraint has to be understood in the
distributional sense, and $\mR_\alpha$ is a regularisation functional
on $\LPspace{p}{\Omega}$ realising a-priori assumptions on $\chi$ that
compensate for the lack of boundary conditions. In
\cite{langkammer2015qsmtgv_mh}, the choice
$\mR_\alpha = \TGV_\alpha^2$ was proposed and studied. Choosing
$p = \frac32$, the functional in~\eqref{eq:qsm_all_in_one} is coercive
up to finite dimensions and the linear PDE-constraint is closed, so one can
easily see that an optimal solution always exists and yields finite
values once there is a pair
$(\chi,\psi) \in \LPspace{2}{\Omega} \times \BV(\Omega)$ that
satisfies the constraints.

A numerical algorithm for the discrete solution
of~\eqref{eq:qsm_all_in_one} with $\TGV^2$-regularisation can easily
be derived by employing the tools of
Section~\ref{sec:numerical_algorithms} and, e.g., finite-difference
discretizations of the operator $\wave$. In
\cite{langkammer2015qsmtgv_mh}, a primal-dual algorithm has been
implemented and tested for synthetic as well as real-life data. It
turns out that the integrative approach~\eqref{eq:qsm_all_in_one} is
very robust to noise and can in particular be employed for fast 3D
MRI-acquisition schemes that may yield low signal-to noise ratio such
as 3D echo-planar imaging (EPI) \cite{poser2010epi3d}.
It has been tested on raw phase data, see
Figure~\ref{fig:integrative_qsm}, where the benefits of higher-order
regularisation also become apparent. Due to the short scan time that
is possible by this approach as well as its robustness, it might
additionally contribute to advance QSM further towards clinical
applications.

\begin{figure}[t]
  \centering
  \begin{tabular}{c@{\ }c@{\ }c}
    \includegraphics[width=0.31\columnwidth]{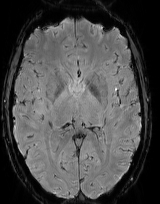} &
    \includegraphics[width=0.31\columnwidth]{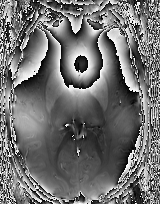} &
    \includegraphics[width=0.31\columnwidth]{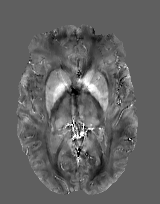}  \\
    (a) & (b) & (c)
  \end{tabular}%
  \caption{Example for integrative
    TGV-regularised susceptibility reconstruction from wrapped
    phase data. (a) Magnitude image (for brain mask extraction).
    (b) Input phase image $\varphi_t$ 
    (single gradient echo, echo time: 27ms, field strength: 3T). 
    (c) Result of the integrative approach~\eqref{eq:qsm_all_in_one} (scale 
    from -0.15 to 0.25ppm).
    All images visualise one slice of the respective 3D image.}
  \label{fig:integrative_qsm}
\end{figure}

\subsection{Dynamic MRI reconstruction} \label{sec:dynamic_mri}
As mentioned in Subsection \ref{sec:parallel_mri}, data acquisition in MR imaging is relatively slow. 
This can be compensated by subsampling and variational reconstruction techniques such that in controlled environments, as for instance with brain or knee imaging, a good reconstruction quality can be obtained. The situation is more difficult when imaging parts of the body that are affected, for instance, by breathing motion, or when one aims to image certain dynamics such as with dynamic contrast enhanced MRI or heart imaging.
Regarding unwanted motion, there exists a large amount of literature on motion correction techniques (see \cite{Zaitsev2015motion_mri_review_mh} for a review) which can be separated into prospective and retrospective motion correction and which often rely on additional measurements to estimate and correct for unwanted motion. In contrast to that, dynamic MRI aims to capture certain dynamic processes such as heartbeats, the flow of blood or contrast agent. 
Here, the approach is often to acquire highly subsampled data, possibly combined with gating techniques, such that motion consistency can be assumed for each single frame of a time series of measurements.
The severe lack of data for each frame can then only be mitigated by exploiting temporal correspondences between different measurement times. One way to achieve this is via Tikhonov regularisation of the dynamic inverse problem, which, for instance, amounts to
\begin{equation} \label{eq:dynamic_mri_general}
 \min _{u \in \LPspace{p}{{]{0,T}[} \times \Omega}} \ \frac12 \sum_{i=1}^k  \int_0^T \|(K_tu_t)_i - (f_t)_i\|_2^2  \dd{t} + \mR_\alpha(u),
\end{equation}
where  $p \in [1,\infty]$, $T > 0$, %
$\Omega \subset \RR^d$ is the image domain, and for almost every $t \in {]{0,T}[}$, $\sigma_t$ is a positive, finite Radon measure on $\RR^d$ that represents the possibly time-dependent Fourier sampling pattern at time $t$, such that $K_t: \LPspace{p}{\Omega} \to \LPspace[\sigma_t]{2}{\RR^d, \CC}^k$ according to~\eqref{eq:parallel_mr_forward_operator} models the MR forward operator, and $f_t \in \LPspace[\sigma_t]{2}{\RR^d, \CC}^k$ represents the associated measurement data. Further, $u_t$ denotes the evaluation of $u$ at time $t$ which is almost everywhere a function in $\LPspace{p}{\Omega}$.
As usual, $\mR_\alpha$ corresponds to the regularisation functional that can be used to enforce additional regularity constraints. Note that in order to obtain a well-defined formulation for the time-dependent integral, the $\sigma_t$ have to vary in a measurable way with $t$ such that the associated $K_t$ and the data $f_t$ are also measurable in a suitable sense.
We refer to \cite{Bredies2019optimal_tp_dynamic_mh} for %
details on the necessary notion and spaces,
and an analysis of the above problem 
in the context of optimal-transport-based regularisation.

In the context of clinical MR applications, temporal Fourier transforms \cite{Jung09_kt_focuss_mh}, temporal derivatives \cite{Adluru07_dmri_tv_mh} or combinations thereof \cite{Feng13_kt_sparse_sense_mh} have, for instance, been proposed for temporal regularisation. More recently, methods that build on motion-dependent additive decomposition of the dynamic image data into different components have been successful. The work \cite{Otazo14_l_plus_s_mri} achieves this in a discrete setting via low-rank and sparse decomposition which, for the low-rank component, penalises the singular values of the matrix containing the vectorised frames in each column. In contrast to that, by employing the $\ICTGV$ functional presented in Subsection \ref{sec:tgv_extensions}, the work \cite{holler17ictgvmri_mh} achieves an additive decomposition and adaptive regularisation of the dynamic data via penalising differently weighted spatio-temporal derivatives.
There, problem \eqref{eq:dynamic_mri_general} is solved for the choice 
\[\mR_\alpha(u) = \ICTGV_\alpha^2(u) = \inf_{w \in \BV({]{0,T}[} \times \Omega)} \ \TGV^2_{\beta_1} (u-w) + \TGV_{\beta_2}^2 (w),\]
where the $\TGV_{\beta_i}^2$ are second-order spatio-temporal TGV functionals that employ different weightings of the components of the spatio-temporal derivatives in such a way that for $\TGV_{\beta_1}^2$, changes in time are penalised stronger than changes in space while $\TGV_{\beta_2}^2$ acts the other way around. The numerical solution of \eqref{eq:dynamic_mri_general} can again be obtained within the algorithmic framework presented in Section \ref{sec:numerical_algorithms} and we refer to \cite{ictgv_mri_code_mh} for a GPU-accelerated open source implementation and demo scripts.

Figure \ref{fig:ictgv_mri_cine} shows the result of ICTGV-regularised reconstruction of a multi-coil cardiac cine dataset (subsampled with factor $\approx{}11$) and compares to the straightforward sum-of-squares (SOS) reconstruction. Since the SOS reconstruction does not account for temporal correspondences, it is not able to obtain a useful result for a high subsampling factor while the ICTGV-based reconstruction resolves fine details as well as motion dynamics rather well. 
Figure \ref{fig:ictgv_mri_cine_ls} shows a comparison to the low-rank and sparse (L+S) method of \cite{Otazo14_l_plus_s_mri} for a second cine dataset with a different view. Here, the parameters for the L+S method where optimised for each experiment separately using the (in practice unknown) ground truth while for ICTGV, the parameters were trained a-priori on a different dataset and fixed afterwards. It can be seen in Figure \ref{fig:ictgv_mri_cine_ls} that both methods perform rather well up the high subsampling factors, where the ICTGV-based is able to recover fine details (highlighted by arrows) that are lost with L+S reconstruction.

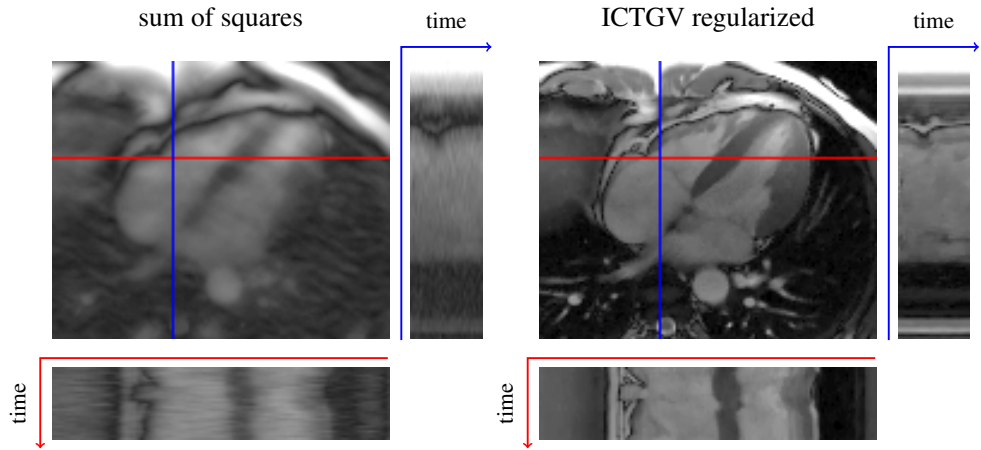
\begin{figure}
  \centering
  
\scalebox{0.60}{
\begin{tikzpicture}
         \pgfmathsetmacro{\off}{-1.0}
         \pgfmathsetmacro{\offv}{0}
		 
         \node at (0,4) {\LARGE sum of squares};
         \draw (0,0) node{ \includegraphics[scale=1.5]{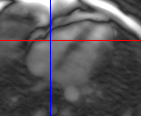} };
         \draw (0,-4.5) node{\includegraphics[scale=1.5]{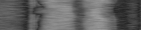} };
         \draw[red,very thick,<-] (-4,-5.5) -- (-4,-3.5) -- (3.7,-3.5);
         \node[rotate=90] at (-4.5,-4.5) {\Large time};
         \draw (5,0) node{\includegraphics[scale=1.5]{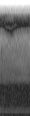} };
         \draw[blue,very thick,->] (4,-3.1) -- (4,3.4) -- (6,3.4);
         \node at (5,4) {\Large time};
         
         \pgfmathsetmacro{\off}{10.8}
         \node at (0+\off,4) {\LARGE ICTGV regularized};
         \draw (0+\off,0) node{ \includegraphics[scale=1.5]{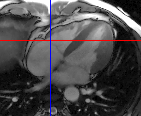} };
         \draw (0+\off,-4.5) node{\includegraphics[scale=1.5]{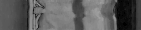} };
         \draw[red,very thick,<-] (-4+\off,-5.5) -- (-4+\off,-3.5) -- (3.7+\off,-3.5);
         \node[rotate=90] at (-4.5+\off,-4.5) {\Large time};
         \draw (5+\off,0) node{\includegraphics[scale=1.5]{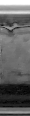} };
         \draw[blue,very thick,->] (4+\off,-3.1) -- (4+\off,3.4) -- (6+\off,3.4);
         \node at (5+\off,4) {\Large time};
         
         \pgfmathsetmacro{\offb}{10.4}
         \pgfmathsetmacro{\voff}{-0.5}
       \end{tikzpicture}
       }  
       \caption{\label{fig:ictgv_mri_cine} Comparison of straightforward sum-of-squares (left) and ICTGV-regularised (right) reconstruction for a dynamic MR dataset with subsampling factor $\approx 11$.
       Each image shows one frame of the reconstructed image sequence along with the temporal evolution of one horizontal and vertical cross section indicated by the red and blue line, respectively.}
\end{figure}

\begin{figure}
\includegraphics[width=\textwidth]{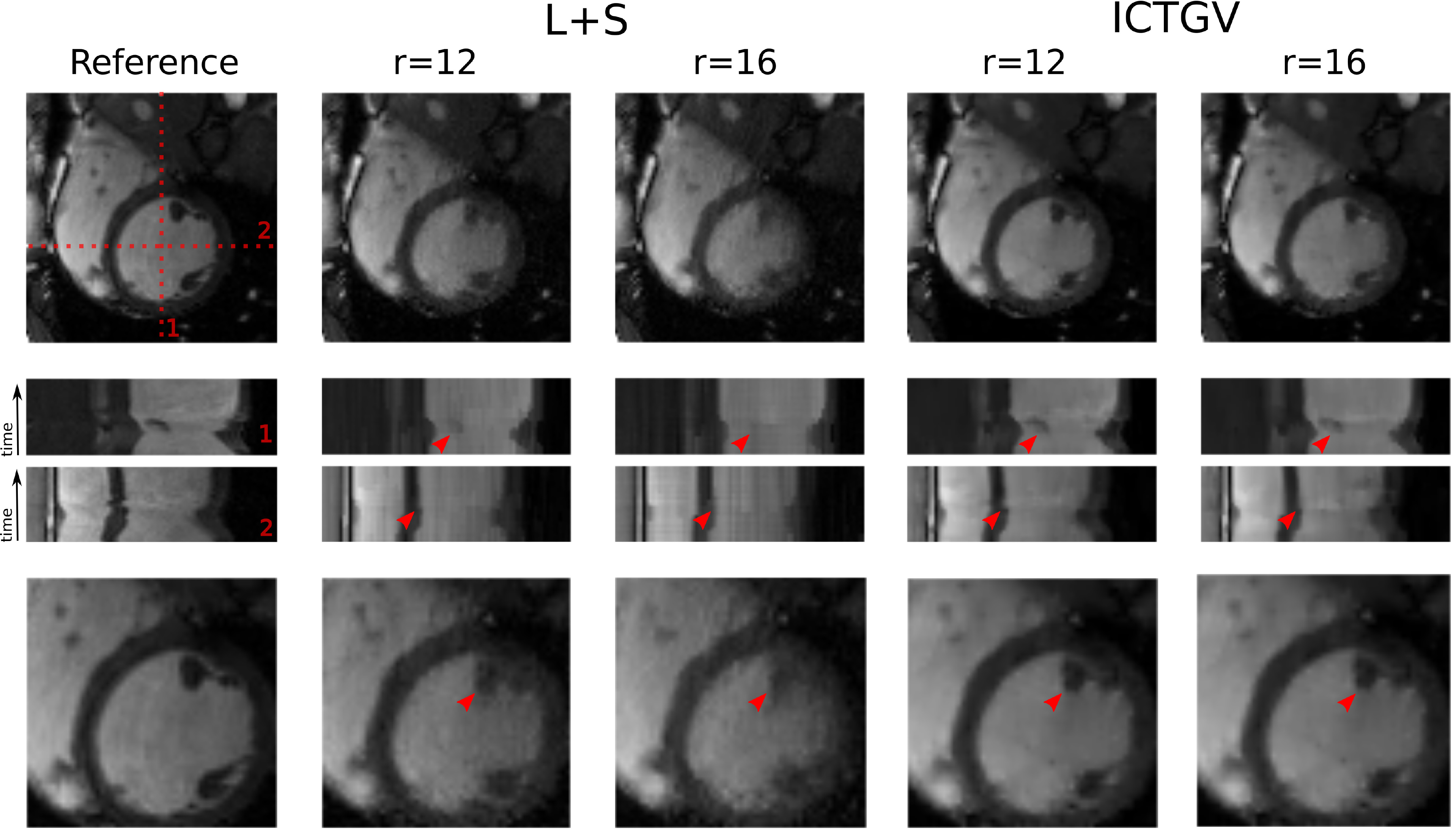}
\caption{\label{fig:ictgv_mri_cine_ls} Comparison of L+S- and ICTGV-regularised dynamic MR reconstruction. The first column shows, from top to bottom, a frame of the ground truth image sequence along with the temporal evolution of a vertical and horizontal cross section (indicated by red dotted lines) as well as a close up. Columns 2--3 depict the reconstruction results for L+S regularisation, while columns 4--5 depict the corresponding results for ICTGV regularisation (subsampling factors $r=12$ and $r=16$). The red arrows indicate details that are lost by L+S regularisation but maintained with ICTGV regularisation. 
  Figure taken from \cite{holler17ictgvmri_mh}. Reprinted by permission from John Wiley and Sons.}
\end{figure}

\subsection{Joint MR-PET reconstruction} \label{sec:mri_pet_application}
We have seen in Subsections \ref{sec:parallel_mri} and \ref{sec:dynamic_mri} that image reconstruction from parallel, subsampled MRI data is non-trivial and can greatly be improved with variational regularisation. Beyond MRI and CT, a further medical imaging modality of high clinical relevance is positron emission tomography (PET). As opposed to standard MR imaging, PET imaging is quantitative and builds on reconstructing the spatial distribution of a radioactive tracer that is injected into the patient prior to the measurement. The forward model in PET imaging is  the X-ray transform (often combined with resolution modelling) and, since measurements correspond to photon counts, the noise in PET imaging is typically assumed to be Poisson distributed. Reconstructing images from PET measurement data is a non-trivial inverse problem, where difficulties arise, for instance, from high Poisson noise due to limited data acquisition time, dosage restrictions for the radioactive tracer, 
as well as from limited measurement resolution due to finite detector size and photon acollinearity. As a result, variational reconstruction methods and in particular TV regularisation are employed also in PET imaging to improve reconstruction (see, for instance, \cite{Jonsson1998tv_pet_mh,sawatzky2013_em_tv}). 

In a clinical workflow, often both MR and PET images are acquired, which provides two complementary sources of information for diagnoses.
This can also be exploited for reconstruction and in particular, for MR-prior-based PET reconstruction methods, which incorporate structural information from the MR image for PET reconstruction, are now well established in theory \cite{holler18structural_mh} and in practice \cite{Vunckx2012_pet_prior_eval,Ehrhardt_PET_prior_mh,holler17pls_mh}. While those methods regard an a-priori reconstructed MR image as fixed, anatomical prior for PET, also joint, synergistic reconstruction is possible and recently became more popular due to the availability of joint MR-PET scanners \cite{holler16mrpet_mh,Ehrhardt_PET_MRI_mh}. An advantage of the latter is that neither of the two images is fixed a-priori and, in principle, a mutual benefit for both modalities due to joint reconstruction is possible. To this aim, the regularisation term needs to incorporate an appropriate coupling of the two modalities, and here we discuss the coupled TGV-based approach of \cite{holler16mrpet_mh} that allows to achieve this.

At first, we consider the forward model for PET imaging, which consists of a convolution followed by an attenuated the X-ray transform and additive corrections. With $\Omega \subset \R^d$ the image domain such that $ \Omega \subset B_R(0)$ for some $R>0$, the X-ray transform can be defined as linear operator $P:L^p(\Omega) \rightarrow L^1_\mu(\Sigma)$, where $p \in [1,\infty]$, $\Sigma \subset  \set{(\vartheta,x)}{\vartheta \in \mathcal{S}^{d-1}, \, x \in \sett{\vartheta}^\perp,\, \|x\|< R}$, $\Sigma$  is a non-empty and open subset of the tangent bundle to $\mathcal{S}^{d-1}$, via
 \[ Pu(\vartheta,x)= \int_\R u(x+t\vartheta) \dd t .\]
Note that here, $u$ is extended by zero outside $\Omega$ and the measure $\mu$ on $\Sigma$ is induced by the functional
\[ \varphi \mapsto \int_{\mS} \int_{\sett{\vartheta}^\perp} \varphi(\vartheta,x) \dd{\hausdorff{d-1}(x)} \dd{\hausdorff{d-1}(\vartheta)}, \]
see \cite[Section 3.4]{Markoe06_mh} for details. We further denote by $k\in L^1(B_r(0))$ a convolution kernel with width $r>0$ that models physical limitations in PET imaging, for instance, due to finite detector size and photon acollinearity, see \cite{rahmim2013resolution_pet_mh}. The PET forward model is then defined as $\kpet :L^p(\Omega) \rightarrow L^1_\mu(\Sigma)$
\[u \mapsto   \kpet u + c, \quad \kpet u(\vartheta,x) = a(\vartheta,x) P(u\ast k)(\vartheta,x) ,
\]
where $u \ast k$ denotes the convolution of $u$ and $k$ (using again zero extension), $a \in L^\infty_\mu(\Sigma)$ with $a >0$ a.e.~includes a correction for attenuation and detector sensitivities and $c \in L^1_\mu(\Sigma)$ with $c \geq 0$ a.e.~accounts for additive errors due to random and scattered events. Assuming the noise in PET to be Poisson distributed, we use the Kullback--Leibler divergence as defined in \eqref{eq:kl_definition} for data fidelity. 

For the MR forward model, we use again the parallel MR operator $\kmr$ according to~\eqref{eq:parallel_mr_forward_operator} in Subsection \ref{sec:parallel_mri} that includes coil sensitivity profiles and a measurement trajectory defined via a finite, positive Radon measure $\sigma$ on $\RR^d$.

For regularisation, we use an extension of second-order TGV to multi-channel data as discussed in Subsection \ref{sec:tgv_extensions}, which, as in parallel MR reconstruction, is adapted to complex-valued data.
 That is, we define $\TGV_\alpha^2$ for $u = (u_1,u_2) \in \LPlocspace{1}{\Omega,\CC}$ similar to \eqref{eq:tgv_color_def}, where we use the spectral norm as pointwise dual norm on $\Sym^1(\CC^d)^2$ and the Frobenius norm on $\Sym^2(\CC^d)^2$. In the primal version of TGV analogous to~\eqref{eq:tgv2_primal}, this results in particular in a pointwise nuclear-norm penalisation of the first-order derivative information $\nabla u -w$ and is motivated by the goal of enforcing pointwise rank one of $\nabla u - w$ in a discretized setting and hence an alignment of level sets.
 
 With these building blocks, a variational model for coupled MR-PET reconstruction can be written as
 \begin{equation*} %
   \fl
   \min_ {u = (u_1,u_2) \in L^ p(\Omega,\CC)^2} \ \lambda_1\sum_{i=1}^ k \| (\kmr u_1)_i - (f_1)_i\|_{2,\sigma} +  \lambda _2 \KL(\kpet u_2 + c, f_2)
   +  \TGV_\alpha^ 2(u)
\end{equation*}
where $(f_1)_1, \ldots, (f_1)_k \in \LPspace[\sigma]{2}{\RR^d,\CC}$ and $f_2 \in \LPspace[\mu]{1}{\Sigma}$, $f_2 \geq 0$ almost everywhere is the given measurement data for MR and PET, respectively, and $\lambda_1,\lambda_2>0$ are the weights for the different data terms. Well-posedness for this model follows again by a straightforward adaptation of Proposition \ref{prop:well_posed_tgvk} to the multi-channel setting. Regarding the regularisation parameters $\lambda_1,\lambda_2$, this is a particular case of coupled multi-discrepancy regularisation and we refer to \cite{holler18coupled_mh} for results on convergence and parameter choice for vanishing noise.  A numerical solution can again be obtained with the techniques described in Section \ref{sec:numerical_algorithms}, where the discrete forward operator $K_h$ is vectorised as $K_h = \diag(\kmrh,\kpeth)$ with discretized operators $\kmrh$ and $\kpeth$, and the discrepancy $S_{f_h}$ is the component-wise sum of the two discrepancies above in a discrete version.

Numerical results for 3D in-vivo data using this method, together with a comparison to a standard method, can be found in Figure \ref{fig:mr_pet_example} (see also \cite{holler16mrpet_mh} for a more detailed evaluation). As can be seen there, the coupling of the two modalities yield improved reconstruction results in particular for the PET channels, making sharp features and details more visible.  

\begin{figure}
\begin{tabular}{c@{\ }c@{\ }c}
    \includegraphics[width=0.31\columnwidth]{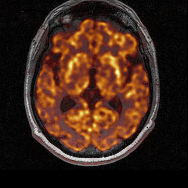} &
    \includegraphics[width=0.31\columnwidth]{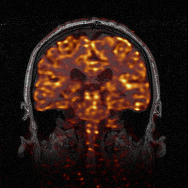} &
    \includegraphics[width=0.31\columnwidth]{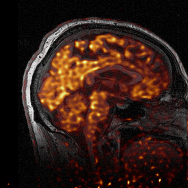} \\
    
    \includegraphics[width=0.31\columnwidth]{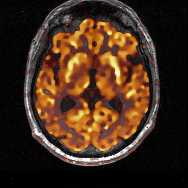} &
    \includegraphics[width=0.31\columnwidth]{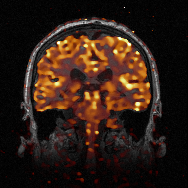} &
    \includegraphics[width=0.31\columnwidth]{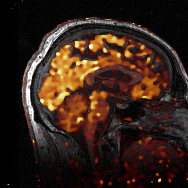} 

  \end{tabular}%
\caption{\label{fig:mr_pet_example} Slices of fused, reconstructed 3D in-vivo MR and PET images from an MPRAGE contrast with 2-fold subsampling and a 10min fluorodeoxyglucose (FDG) PET head scan, respectively (left to right: transversal view, coronal view, sagittal view). Top row: Standard methods (CG sense \cite{Pruessmann2001_cgsense} for MR, expectation maximisation \cite{Shepp1982_em_pet_mh} for PET). Bottom row: Nuclear-norm-TGV-based variational reconstruction.}
\end{figure}

\subsection{Radon inversion for multi-channel electron microscopy}

Similar to joint MR-PET reconstruction, coupled higher-order regularisation can also be used in multi-channel electron microscopy imaging for improving reconstruction quality.
As a particular technique in electron microscopy, scanning transmission electron microscopy (STEM) allows for a three-dimensional imaging of nanomaterials down to atomic resolution and is heavily used in material sciences and nanotechnology, e.g. for quality control and troubleshooting in the production of microchips. Beyond providing pure density images, spectroscopy methods in STEM imaging also allow to image the 3D elemental and chemical make-up of a sample. 

Standard techniques for density and spectroscopy imaging in STEM are high-angle annular dark-field (HAADF) imaging and energy-dispersive X-ray spectroscopy (EDXS), respectively. For both imaging methods, measurement data can be acquired simultaneously while raster-scanning the material sample with a focused electron beam. HAADF imaging records the number of electrons scattered
to a specific annular range while EDXS allows to record characteristic X-rays for specific elements which are emitted when electrons change their shell position. 
For each position of the electron beam, both HAADF and EDXS measurements correspond to measuring (approximately) the density of a weighted sum of all elements and single element, respectively, integrated along the line of the electron beam that intersects the sample. Scanning over the entire sample orthogonal to an imaging plane, the acquired signals hence correspond to a slice-wise Radon transform of an overall density image (HAADF) and different elemental maps (EDXS). 

Due to physical restrictions in the imaging system, the number of available projections (i.e., measurement angles) as well as the signal-to-noise ratio, in particular for EDXS, is limited and volumetric images obtained with standard image reconstruction methods, such as the simultaneous iterative reconstruction technique (SIRT) \cite{GILBERT1972105}, suffer from artefacts and noise. 
As a result, regularised reconstruction is increasingly used also for electron tomography, with total-variation-based methods being a popular example \cite{goris_electron_2012_mh}. While TV regularisation works well for piecewise-constant density distributions with sharp interfaces, the presence of gradual changes between different sample regions, e.g., due to diffusion at interfaces, motivates the usage of higher-order regularisation approaches for electron tomography \cite{holler18tgvtem_mh}, also in a single-channel setting \cite{Al-aleef_electron_2015_mh}.
In a multi-channel setting as discussed here, an additional coupling of different measurement channels is very beneficial in particular for the reconstruction of elemental maps and has been carried out with first-order TV regularisation in \cite{ZHONG201834,ZHONG2018133} and second-order TGV regularisation in \cite{holler18tgvtem_mh}. In the following, we discuss the TGV-based approach of \cite{holler18tgvtem_mh} in more detail and provide experimental results.

With $\Sigma  = \Sigma_1 \times \Sigma_2$,
$\Sigma_1 \subset \mS^1 \times {]{-R,R}[}$  %
non-empty, open, $\Sigma_2 \subset {]{-R,R}[}$ non-empty, open,
for some $R>0$ and $\mu = (\mathcal{H}^1\restricted \mS^1 \times \mathcal{L}^1) \times \mathcal{L}^1$, we define, for $\Omega = B_R(0) \times {]{-R,R}[}$ where $B_R(0) \subset \RR^2$ and $p \in [1,\infty]$, the forward operator for electron tomography as
$\ktem:L^p(\Omega) \rightarrow L^1_\mu(\Sigma)$ via
\[ \ktem u (\vartheta,s,z) = \int_{\smallset{x \in B_R(0)}{x \inprod \vartheta = s}} u(x,z) \dd{\hausdorff{1}(x)} \]
which corresponds to a slice-wise 2D Radon transform. By continuity of the Radon transform from $L^1(B_R(0))$ to $L^1(\mS^1 \times {]{-R,R}[})$ (see \cite[Section 3.4]{Markoe06_mh}) there is a $C>0$  such that, for every $u \in L^p(B_R(0))$ and  almost every $z \in {]{-R,R}[}$,
\[ \int_{\Sigma_1 } | \ktem u (\vartheta,s,z)|  \dd{(\mathcal{H}^1\times  \mathcal{L}^1)(\vartheta,s)} \leq C \int_{B_R(0)} |u(x,z)| \dd{x}.
\]
Integrating over $\Sigma_2$, it follows that $\ktem$ is bounded from $L^p(\Omega)$  to $ L^1_\mu(\Sigma)$.
Now assume $f_1,\ldots,f_n$ to be given, multi-channel measurement data and the forward model for the $i$-th %
measurement channel to be described with $(\ktem)_i \in \mathcal{L}(L^p(\Omega),L^1_\mu(\Sigma))$ for $i=1,\ldots,n$. In the example considered below, $f = (f_{\mathop{\rm HAADF}},f_{\mathop{\rm Yb}},f_{\mathop{\rm Al}},f_{\mathop{\rm Si}})$, with $f_{\mathop{\rm HAADF}}$ the HAADF data, and $(f_{\mathop{\rm Yb}},f_{\mathop{\rm Al}},f_{\mathop{\rm Si}})$ the EDXS data for Ytterbium, Aluminum and Silicon, respectively. With $\TGV_\alpha^2$ the multi-channel extension of second-order TGV as discussed in Subsection \ref{sec:tgv_extensions}, using a Frobenius-norm coupling of the different channels, we consider
\begin{equation*} 
  \min_ {u \in L^ p(\Omega)^n }  \ \sum_{i=1}^ {n} \lambda_i \KL(\ktem^i u_i,f_i)
  + \TGV_\alpha^ 2(u)
\end{equation*}
for the reconstruction of multi-channel image data, for which well-posedness again results from a multi-channel extension of Proposition \ref{prop:well_posed_tgvk}.
A numerical solution can be obtained using the framework as described in Section \ref{sec:numerical_algorithms} where again the discrete forward operator $K_h$ and the discrepancy term $S_{f_h}$ are vectorised accordingly, similar as in Subsection \ref{sec:mri_pet_application}.

Experimental results for this setting and a comparison to other methods can be found in Figure \ref{fig:tgv_tem_results}, where in particular,
separate TGV regularisation of each channel %
is compared to the
Frobenius-norm-based coupling as mentioned above. It can be seen in Figure \ref{fig:tgv_tem_results} that using TGV regularisation significantly improves upon the standard SIRT method. Also, a coupling of the different channels is very beneficial in particular for the elemental maps, making material inclusions visible that can hardly be seen with an uncoupled reconstruction. We refer to \cite{holler18tgvtem_mh} for a more detailed evaluation (and comparison to TV-based regularisation).

\begin{figure}
\includegraphics[width=\textwidth]{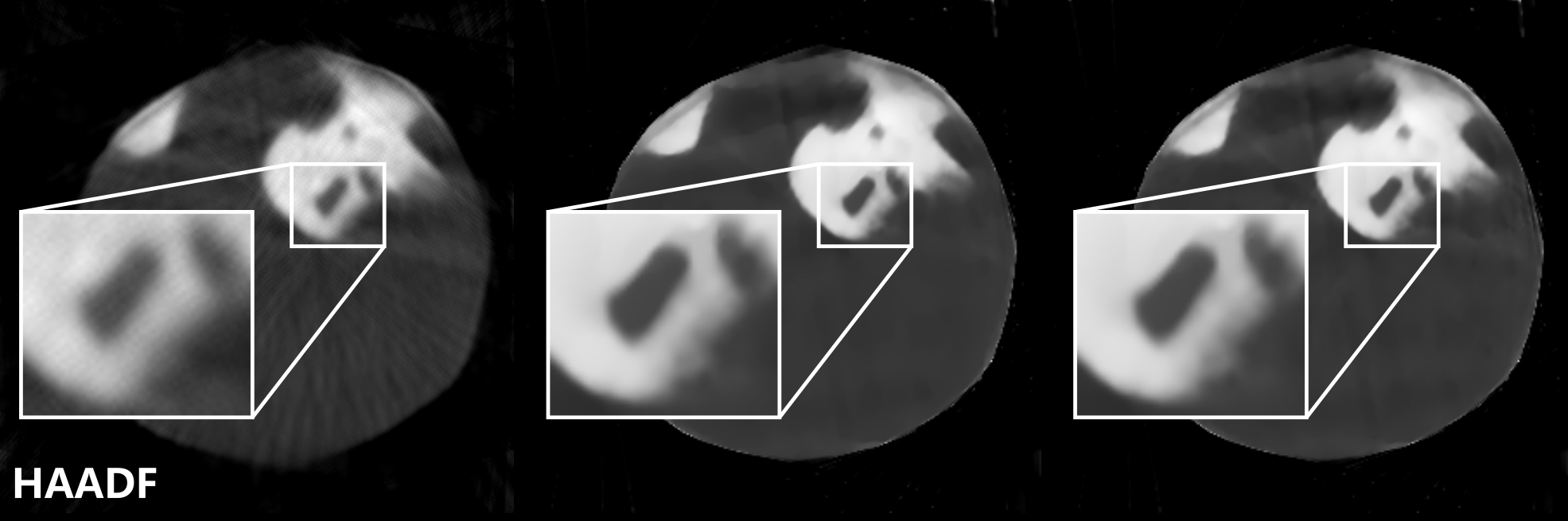}
\includegraphics[width=\textwidth]{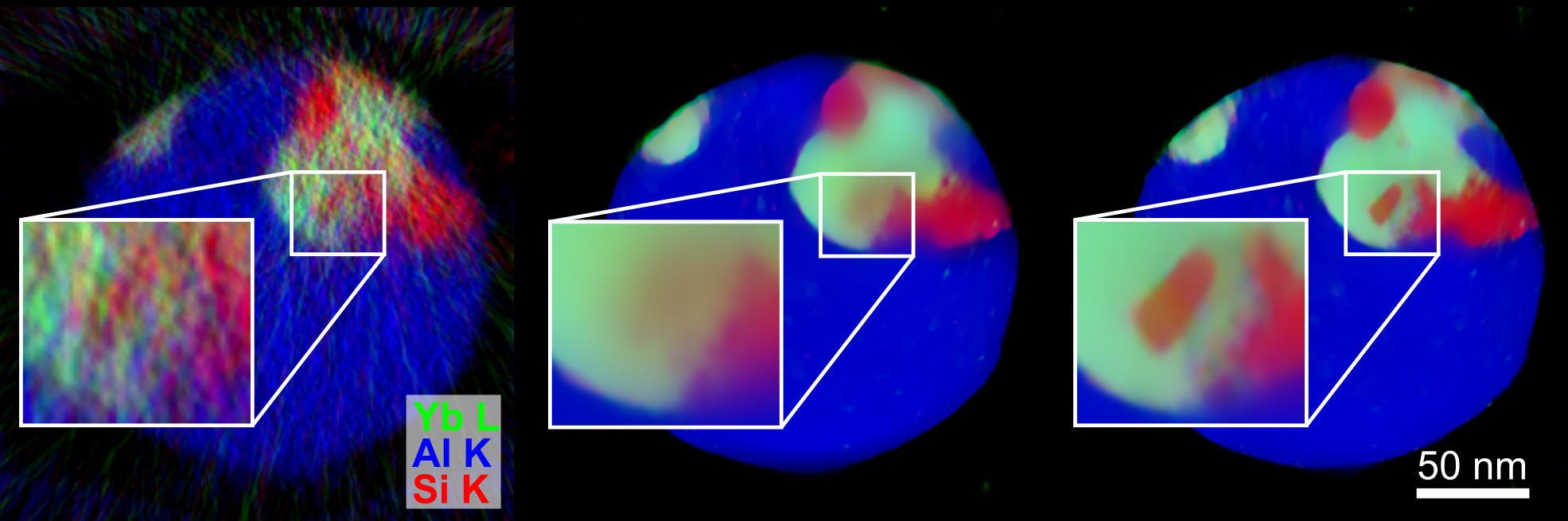}
\caption{\label{fig:tgv_tem_results}
  Example of multi-channel electron tomography. 
  The images show density maps (top row) and elemental maps (bottom row) of one slice of a 3D multi-channel electron tomography reconstruction using different reconstruction strategies \cite{holler18tgvtem_mh}. Left column: SIRT \cite{GILBERT1972105} method. Middle column: Uncoupled TGV-based regularisation. Right column: Coupled TGV-based regularisation. Images taken from \cite{holler18tgvtem_mh}.}
\end{figure}

\section{Conclusions}

The higher-order total variation strategies and application examples
discussed in this review show once again that while regularisation
makes it possible to solve ill-posed inverse problems in the first
place, the actual choice of the regularisation strategy has a
tremendous impact on the qualitative properties of the regularised
solutions and can be decisive on whether the inverse problem is
considered being solved in practice. In the considered context of
Tikhonov regularisation, convex regularisation functionals offer a
great flexibility in terms of functional-analytic properties and
a-priori assumptions on the solutions. With the total variation being
an established regulariser sharing desirable properties such as the
ability to recover discontinuities, higher-order total variation
regularisers offer additional possibilities, mainly the efficient
modelling of piecewise smooth regions in which the derivative of some
order may jump. We have seen in this paper that a regularisation
theory for these functionals can be established and the overall theory is now
sufficiently advanced such that the favourable properties of both
first- and higher-order TV can be obtained with suitable functionals,
for instance by infimal convolution. The underlying concepts are in
particular suitable for various generalisations. The total generalised
variation, for instance, bases on TV-type penalties for a
multiple-order differentiation cascade and thus enables to realise the
a-priori assumption of piecewise smoothness with jump
discontinuities. Further, due to higher-order derivatives being
intrinsically connected to symmetric tensor fields, a generalisation
to dedicated regularisation approaches for the latter is immediate.
All these approaches and generalisations are indeed beneficial for
applications and the solution of concrete inverse problems. This is in
particular the case for inverse problems in medical imaging.

Of course, there are still several directions of future research, open
questions and topics that have not been covered by this review. For
instance, one of the major differences between first-order TV and
higher-order approaches is the availability of a co-area formula that
can be used to describe the total variation of scalar function in
terms of its sublevel-sets. Generalisations to vector-valued functions
or higher-order derivatives do either not exist or are not practical
from the view of regularisation theory for inverse problems. As the
co-area formula allows, for instance, to obtain geometrical properties
for TV-regularisation
\cite{chambolle2016geometrictv,iglesias2018tvinverse}, it would be
interesting to bridge the gap to higher-order TV approaches such that
similar statements can be made. Some recent progress in this
direction might be the connection between the solutions of certain
linear inverse problems and the extremal points of the sublevel sets
of the regulariser \cite{Carioni18sparsity_mh,Boyer19representer_mh},
since the extremal points of the TV-balls are essentially characteristic
functions. However, for higher-order TV and the generalisations
discussed in this paper, a characterisation of its extremal points is
not known to date. Further, in the context of TV-regularisation, only
natural orders of differentiation have been considered in detail so
far, with regularisation theory for fractional-order TV just emerging
\cite{zhang2015fractionaltv,williams2016fractionaltv,davoli2018fractionaltgv}.
Indeed, there are many open questions for fractional-order TV
regularisation ranging from the properties of the fractional
derivative operators and their underlying spaces to optimal selection
of the fractional differentiation parameter as well as the
construction of efficient numerical algorithms. Finally, with all the
possibilities of combining distributional differentiation and
Radon-norm-penalisation, which are the essential building blocks of
the regularisers discussed in this paper, the question arises whether
their structure, parameters and differential operators can also be
learned by data-driven optimisation. Some results in this direction
can already be found in the literature
\cite{Calatroni16bilevel_mh,DelosReyes2016_bilevel_parameter_mh}, and
we expect that more will follow in the future.

\section*{References}
\addcontentsline{toc}{section}{References}

\bibliographystyle{iopart-num}
\bibliography{mh_lit_dat,paper} %

\appendix
\renewcommand{\thesection}{\Alph{section}} 

\section{Additional proofs}

\begin{lemma} \label{lem:kl_basic_properties}
  With $\Omega' \subset \RR^d$ measurable, in accordance with Equation \eqref{eq:kl_definition}, let the %
  functional $\KL$ on $L^1(\Omega')^2$ be given as
\[(v,f) \mapsto  \KL(v,f) = 
\left\{\begin{array}{rl}
\int_{\Omega'} f \Bigl( 
\frac{v}{f} - \log\Bigl( \frac{v}{f} \Bigr) - 1 \Bigr) \dd{x} 
& \text{if } f\geq 0,\, v \geq 0 \text{ a.e.},\\ 
\infty & \text{else,}
\end{array}\right.
\]
where we set the integrand to $v$ where $f = 0$ and
to $\infty$ where $v = 0$ and $f > 0$.
Then, $\KL$ is well-defined, non-negative, convex and lower semi-continuous. In case $f\geq 0$ a.e., it holds that $\KL(v,f) = 0 $ if and only if $v=f$. Further, for all $v,f \in L^1(\Omega')$,
\begin{equation} \label{eq:kl_l1_estimate}
 \|v-f\|_1^2 \leq \left( \frac{2}{3}\|f\|_1 + \frac{4}{3}\|v\|_1 \right) \KL(v,f),
 \end{equation}
 and in particular,
 \begin{equation} \label{eq:kl_coercivity_estimate} \|v\|_1 \leq 2 \bigl(\KL(v,f) + \|f\|_1\bigr) \quad  \text{and} \quad \|f\|_1 \leq 2 \bigl(\KL(v,f) + \|v\|_1\bigr)
 \end{equation}
 for all $f,v \in L^1(\Omega')$. 
\begin{proof}
At first note that, in case $f,v\geq 0$, $\KL$ is given by integrating
$\kl:[0,\infty[^2 \rightarrow [0,\infty]$ with $\kl(x,y) = x-y - y\log(\frac{x}{y})$ for $x,y \in {[{0,\infty}[}$, where we use the conventions $0\log(\frac{v}{0}) = 0$ for $v \geq 0$ and $-f\log(\frac{0}{f}) = \infty$ for $f >0$. It is easy to see that $\kl$ is non-negative, convex and lower semi-continuous, hence $\KL$ is well-defined, non-negative, convex, lower semi-continuous and, in case $f\geq 0$ a.e., $\KL(v,f) = 0 $ if and only if $v=f$.
Also, a simple computation (see \cite{Borwein91_mh}) shows that for all $x,y \in [0,\infty[$,
\[(x-y)^2 \leq (\frac{2y}{3} + \frac{4x}{3}) \kl(x,y).\]
from which the estimate \eqref{eq:kl_l1_estimate} follows with the Cauchy--Schwarz inequality applied to the square root of the above estimate. 
Now, for the first estimate in \eqref{eq:kl_coercivity_estimate}, we take $f,v \in L^1(\Omega')$ and note that in case $\|v\|_1 \leq \|f\|_1$, the estimate holds trivially. In the other case, $v \neq 0$ and we observe that \eqref{eq:kl_l1_estimate} implies 
\[ \|v\|_1 ^2 - 2 \|v\|_1 \|f\|_1 \leq \frac{2}{3}\|f\|_1\KL(v,f) + \frac{4}{3}\|v\|_1 \KL(v,f) \]
from which the claimed estimate follows from rearranging, dividing by $\|v\|_1$ and noting that $\|f\|_1 /\|v\|_1 \leq 1$. The second estimate in \eqref{eq:kl_coercivity_estimate} follows analogously.
\end{proof}
\end{lemma}

\begin{lemma} \label{lem:kl_lim_inf_sup_estimates}
  For $\seq{f^n}$ and $f$ in $L^1(\Omega')$, let $\KL(f,f^n) \rightarrow 0$. Then, %
  $\|f-f^n\|_1 \rightarrow 0$ and %
  for each sequence
$\seq{v^n}$ in
  $\LPspace{1}{\Omega'}$ with $v^n \wrightarrow v$ as $n \to \infty$ for
  $v \in \LPspace{1}{\Omega'}$, it holds that
  \[
  \KL(v,f) \leq \liminf_{n \to \infty} \ \KL(v^n,f^n).
  \]
  If, in addition, $f^n \leq Cf$ a.e. in $\Omega'$ for all $n$ and some $C>0$, then for all $v \in L^1(\Omega')$, we have
  \[\limsup_{n \to \infty} \ \KL(v,f^n) \leq \KL(v,f).\]
\end{lemma}

\begin{proof} Assume that $\KL(f,f^n) \rightarrow 0$. It follows from the second estimate in \eqref{eq:kl_coercivity_estimate} that $\seq{\KL(f,f^n)}$ bounded implies $\seq{\|f^n\|_1}$ bounded which, using \eqref{eq:kl_l1_estimate}, yields that $f^n \rightarrow f$ in $L^1 (\Omega')$.
The $\liminf$ estimate then follows from lower semi-continuity as in Lemma \ref{lem:kl_basic_properties}.
 Now assume that additionally, $f^n \leq C f$ a.e.~in $\Omega'$ for all $n$ and some $C > 0$. By $L^1$-convergence we can
 take a  subsequence $\seq{f^{n_k}}$ such that $f^{n_k} \to f$ pointwise
  a.e.,  and $\lim_{k\rightarrow \infty}\KL(v,f^{n_k}) = \limsup_{n \to \infty} \KL(v, f^n)$.   
  As $\KL(f, f^{n_k}) \to 0$ as $k\to \infty$, we have
  $\int_{\Omega'} f^{n_k} \log(f/f^{n_k}) \dd{x} \to 0$. Also, since
  $f \log(v/f) \in \LPspace{1}{\Omega'}$ and $f^{n_k}/f$, where we set $f^{n_k}/f=0$ where $f=0$,
  is bounded a.e.~uniformly with respect to $k$, we have
  $\int_{\Omega'} (f^{n_k}/f) f \log(v/f) \dd{x} \to \int_{\Omega'} f
  \log(v/f) \dd{x}$ by virtue of Lebesgue's theorem.
  Together, we get
  \begin{eqnarray*}
    \fl
    \limsup_{n \to \infty} \KL(v,f^n) &= \lim_{k \to \infty} \KL(v,
    f^{n_k}) \\
    & = \lim_{k \to \infty} \int_{\Omega'} v \dd{x} -
    \int_{\Omega'} \frac{f^{n_k}}{f} f \log \Bigl( \frac{v}{f} \Bigr)
    \dd{x} - \int_{\Omega'} f^{n_k} \log \Bigl( \frac{f}{f^{n_k}} \Bigr)
    \dd{x} - \int_{\Omega'} f^{n_k} \dd{x} \\
    &= \int_{\Omega'} v - f \log \Bigl( \frac{v}{f} \Bigr) - f \dd{x} 
      = \KL(v,f),
  \end{eqnarray*}
  which is what we wanted to show.  
\end{proof}

\begin{lemma}
  \label{lem:kernel_symgrad}
  Let $k\geq 1$, $l\geq 0$ and $u: \Omega \to \Sym^l(\RR^d)$ be
  $(k+l)$-times continuously differentiable such that $\symgrad^ku = 0$
  in $\Omega$. Then, $\grad^{k+l} \tensor u = 0$ in $\Omega$.
\end{lemma}

\begin{proof}
  The statement is a slight generalisation of \cite[Proposition
  3.1]{bredies2013boundeddeformation_mh} and its proof is
  analogous. We present it for the sake of completeness.  Choose
  $a_1,\ldots,a_{2l+k} \in \RR^d$. We show that
  $(\grad^{k+l} \tensor u)(x)(a_1,\ldots,a_{2l+k}) = 0$ for each
  $x \in \Omega$. For this purpose, let
  $L \subset \sett{1,\ldots, 2l+k}$ with $\abs{L} = l$ and denote,
  dropping the dependence on $x$, by
  \[
    u_L = u(a_{\pi(1)}, \ldots, a_{\pi(l)})
  \]
  for some bijective $\pi: \sett{1,\ldots,l} \to L$, giving a
  $(k+l)$-times differentiable $u_L: \Omega \to \RR$. Observe that by symmetry,
  $u_L$ does not depend on the choice of $\pi$ but indeed only on $L$.
  Likewise, denote by
  \[
    \frac{\partial^{k+l} u_L}{\partial a_{\complement L}} =
    (\grad^{k+l} \tensor u_L)(a_{\sigma(1)}, \ldots, a_{\sigma(k+l)})
  \]
  for some bijective $\sigma: \sett{1,\ldots,k+l} \to \complement L$.
  By symmetry of the
  derivative,
  $\frac{\partial^{k+l} u_L}{\partial a_{\complement L}}: \Omega \to
  \RR$ only depends on $L$. We also introduce an analogous notation for
  the symmetrised derivative $\symgrad^k$:
  \[
    (\symgrad^ku)_{\complement L} =
    \frac1{(k + l)!} \sum_{%
      \begin{array}{c}
        \scriptstyle \sigma: \sett{1,\ldots,k+l} \to \complement L, \\[-0.4em]
        \scriptstyle \sigma \ {\rm bijective}
      \end{array}}
    (\grad^k \tensor u)(a_{\sigma(1)}, \ldots, a_{\sigma(k+l)}),
  \]
  and, for some $\pi: \sett{1,\ldots,l} \to L$ bijective,
  \[
    \fl
    \frac{\partial^l (\symgrad^ku)_{\complement L}}{\partial a_L}
    = \frac{1}{(k+l)!} \sum_{%
      \begin{array}{c}
        \scriptstyle \sigma: \sett{1,\ldots,k+l} \to \complement L, \\[-0.4em]
        \scriptstyle \sigma \ {\rm bijective}
      \end{array}}
    (\grad^{k+l} \tensor u)(a_{\pi(1)}, \ldots, a_{\pi(l)}, a_{\sigma(1)},
    \ldots, a_{\sigma(k+l)}).
  \]
  Now, as for $\pi: \sett{1,\ldots,l} \to L$ bijective, the definitions as
  well as symmetry yield
  \[
    \fl
    \begin{array}{rl}
      \displaystyle \frac{\partial^l (\symgrad^ku)_{\complement L}}{\partial a_L}
      \!\!\!
      &
        \displaystyle = \frac{1}{(k+l)!}
        \sum_{%
        \begin{array}{c}
          \scriptstyle K \subset \complement L,
          \\[-0.4em] \scriptstyle \abs{K} = k
        \end{array}} \\[0.4em]
      & \displaystyle
        \sum_{%
        \begin{array}{c}
          \scriptstyle \sigma: \sett{1,\ldots,k+l} \to \complement L, \\[-0.4em]
        \scriptstyle \sigma(\sett{1,\ldots,k}) = K, \\[-0.4em]
        \scriptstyle \sigma(\sett{k+1,\ldots,k+l}) = \complement L\setminus K
        \end{array}}
      \hspace*{-3.2em}
      (\grad^{k+l} \tensor u)(a_{\pi(1)}, \ldots, a_{\pi(l)}, a_{\sigma(1)}, \ldots, a_{\sigma(k)}, a_{\sigma(k+1)}, \ldots, a_{\sigma(k+l)}) \\
      & \displaystyle
        = \frac{k!l!}{(k+l)!}
        \sum_{%
        \begin{array}{c}
          \scriptstyle K \subset \complement L,
          \\[-0.4em] \scriptstyle \abs{K} = k
        \end{array}}
      \frac{\partial^{k+l} u_{\complement L \setminus K}}{\partial a_{\complement (\complement L \setminus K)}}
        = \binom{k+l}{l}^{-1}
      \sum_{%
      \begin{array}{c}
        \scriptstyle M \subset \sett{1,\ldots,2l+k},
        \\[-0.4em] \scriptstyle \abs{M}=l, \ L \cap M = \emptyset
      \end{array}}
      \frac{\partial^{k+l} u_M}{\partial a_{\complement M}},
    \end{array}
  \]
  we see that each
  $\frac{\partial^l (\symgrad^ku)_{\complement L}}{\partial a_L}$ can
  be written as a linear combination of
  $\frac{\partial^{k+l} u_M}{\partial a_{\complement M}}$. Up to the
  factor $\binom{k+l}{l}^{-1}$, the linear mapping that takes the
  formal vector
  $\bigl(\frac{\partial^{k+l} u_M}{\partial a_{\complement M}}
  \bigr)_M$ indexed by all $M \subset \sett{1,\ldots,2l+k}$,
  $\abs{M} = l$ to the formal vector
  $\bigl(\frac{\partial^l (\symgrad^ku)_{\complement L}}{\partial
    a_L}\bigr)_L$ indexed by all $L \subset \sett{1,\ldots,2l+k}$,
  $\abs{L} = l$ is corresponding to the multiplication with the
  adjacency matrix of the Kneser graph $K_{2l+k,l}$, see, for
  instance, \cite{biggs1993algebraic} for a definition. The latter is
  regular, which can, for instance, be seen by looking at its
  eigenvalues which are known to be
  \[
    \lambda_m = (-1)^m \binom{k+l-m}{l-m}, \qquad
    m = 0,\ldots, l,
  \]
  see again \cite{biggs1993algebraic}.
  Thus, we can find real numbers $(c_L)_L$ indexed by all $L \subset \sett{1,\ldots,2k+l}$, $\abs{L} = l$ and independent from $u$,
  and $a_1,\ldots,a_{2l+k}$ such that for $M = \sett{k+l+1,\ldots,2l+k}$,
  the identity 
  \[
    \frac{\partial^{k+l} u_M}{\partial a_{\complement M}} =
    \sum_{%
      \begin{array}{c}
        \scriptstyle L \subset \sett{1,\ldots,2l+k},
        \\[-0.4em] \scriptstyle \abs{L}=l,\ L \cap M = \emptyset
      \end{array}}
    c_L \frac{\partial^l (\symgrad^k u)_{\complement L}}{\partial a_L}
  \]
  holds. If $\symgrad^k u = 0$, then the right-hand side is $0$ while
  the left-hand side corresponds to
  $\grad^{k+l} u(a_1,\ldots,a_{2l+k})$. This completes the proof.
\end{proof}

\begin{lemma} \label{lem:bvk_strict_convergence}
  Let $k \geq 1$ and $\Omega \subset \RR^d$ be a bounded Lipschitz domain.
  For each $u \in \BV^k(\Omega)$ and $\delta > 0$, there exists a $u^\delta \in \BV^k(\Omega) \cap C^\infty(\Omega)$ such that for $\delta \rightarrow 0 $,
\[ \|u^\delta - u \|_1 \rightarrow 0  \quad \text{ and } \quad \|\nabla^m  u^\delta \|_\M \rightarrow \|\nabla^m u  \|_\M  \quad  \text{for}\quad m=1,\ldots,k,\]
i.e., $\seq{u^\delta}$ converges strictly in $\BV^k(\Omega)$ to $u$ as $\delta \to 0$.
\end{lemma}
\begin{proof}
The proof builds on the result \cite[Lemma 5.4]{holler14inversetgv_mh} and techniques from  \cite{Ambrosio_mh,evans1992measure_mh}.
Choose a sequence of open sets $ \seq{\Omega_n} $ such that 
  $ \Omega = \bigcup _{n \in \NN} \Omega_n$, 
  $ \overline{\Omega}_n \compactin \Omega $ for all 
  $ n \in \NN $ and any point of $ \Omega $ belongs to at most four sets 
  $ \Omega_n $ (cf. \cite[Theorem 3.9]{Ambrosio_mh} 
  for a construction 
  of such sets). Further, let $ \seq{\varphi^n} $ be a partition 
  of unity relative to $ \seq{\Omega_n} $, i.e.,
  $ \varphi^n \in \Ccspace{\infty}{\Omega_n}$ with $ \varphi^n \geq 0$ 
  for all $ n \in \NN $ and $ \sum _{n=1}^\infty \varphi^n = 1$ pointwise in $\Omega$. 
  Finally, let $\rho \in \Ccspace{\infty}{\RR^d}$ be a standard 
  mollifier, i.e., $\rho$ is radially symmetric,
  non-negative and satisfies $\int_{\RR^d} \rho \dd{x} = 1$.
  Denote by $\rho_{\epsilon}$ the function given by 
  $\rho_\epsilon(x) = \epsilon^{-d} \rho(x/\epsilon)$ for $\epsilon > 0$.

  As $\rho$ is a mollifier and $\varphi^n$ has compact support in $\Omega_n$,
  we can find, for any $n \in \NN$, 
  an $\epsilon_n > 0$ such that 
  $\supp((v\varphi^n) \conv \rho_{\epsilon_n}) \subset \Omega_n$ for any $v \in \BD(\Omega,\Sym^l(\Omega))$, $l\in \N$.
  Further, as shown in \cite[Lemma 5.4]{holler14inversetgv_mh}, for any $v\in \BD(\Omega,\Sym^l(\Omega))$ fixed, for any $\delta > 0 $ we can 
  pick a sequence $\seq{\epsilon_n^\delta}$ with each $\epsilon_n^\delta$ being small enough such that with $v^\delta = \sum_{n=1}^\infty (v \varphi^n ) \ast \rho _{\epsilon_n^\delta}$, we have
  \[ \|v^\delta - v \|_1 \leq \delta \quad \text{ and } \quad \|\symgrad v^\delta \|_\M \leq \|\symgrad v  \|_\M + \delta. \]
In particular, for $u \in \BV^k(\Omega) $ fixed and $v_l = \nabla^l u \in \BD^{k-l}(\Omega,\Sym^l(\RR^d))$ for $l=0,\ldots,k-1$, we can pick a sequence $\seq{\epsilon_n^\delta}$ with each component small enough such that
\begin{equation} \label{eq:bvk_strict_conv_up_to_km1}
  \fl
 \|v_l^\delta - v_l \|_1 \leq \delta \quad \text{ and } \quad \|\symgrad v_l^\delta \|_\M \leq \|\nabla \tensor v_l \|_\M + \delta \quad \text{for} \quad l=0,\ldots,k-1,
\end{equation}
since $\grad \tensor v_l = \symgrad v_l$.
Further we note that, as additional consequence of the Sobolev--Korn inequality of Theorem \ref{thm:sobolev_korn}, $v_l \in \Hspace{k-1-l,1}{\Omega, \Sym^l(\RR^d)}$ and hence by the product rule $\symgrad (v_{l-1}\varphi^n) = (\symgrad v_{l-1}) \varphi^n + \interleave (v_{l-1} \tensor \grad \varphi^n)$, we get 
\[ \fl \symgrad v_{l-1}^\delta - v_l^\delta   = \sum_{n=1}^\infty \symgrad (v_{l-1} \varphi^n) \ast \rho_{\epsilon^\delta_n} -(v_l  \varphi^n) \ast \rho_{\epsilon^\delta_n} = \sum_{n=1}^\infty  \interleave (v_{l-1} \otimes \nabla \varphi^n) \ast \rho_{\epsilon^\delta_n} . \]
In addition, 
\[ \sum_{n=1}^\infty  \interleave (v_{l-1} \otimes \nabla \varphi^n)  = \interleave\Bigl(  v_{l-1} \otimes \nabla \Bigl(\sum_{n=1}^\infty \varphi^n \Bigr) \Bigr) = 0 . \]
Since each $\interleave (v_{l-1} \otimes \nabla \varphi^n) \in \Hspace{k-l,1}{\Omega, \Sym^l(\RR^d)}$, by adaptation of standard mollification results \cite[Theorem 5.2.2]{evans1992measure_mh} we can further reduce any $\epsilon^\delta_n$ to be small enough such that for each $m = 1,\ldots,k$,
\[
\fl
  \Bignorm[1]{\symgrad^{m-l}  \left( \interleave(v_{l-1} \otimes \nabla \varphi^n) \ast \rho_{\epsilon^\delta_n} - \interleave (v_{l-1} \otimes \nabla \varphi^n)  \right) } \leq 2^{-n} \delta
\quad \text{for} \quad l = 1,\ldots,m-1.
\]
and consequently,
\[ \fl
  \bignorm[1]{\symgrad^{m-l} \Bigl(\symgrad v_{l-1}^\delta -  v_l^\delta  \Bigr) } = \Bignorm[1]{\symgrad^{m-l}  \Bigl( \sum_{n=1}^\infty \interleave(v_{l-1} \otimes \nabla \varphi^n) \ast \rho_{\epsilon^\delta_n} - \interleave (v_{l-1} \otimes \nabla \varphi^n) \Bigr) }\leq \delta . \]
Now, setting $u^\delta = v_0^\delta$,  we estimate for $m = 1,\ldots,k$, using
the second estimate in \eqref{eq:bvk_strict_conv_up_to_km1} and that $\symgrad v_{m-1} = \grad^m u$ as well as $\symgrad^m u^\delta = \nabla ^m u^\delta$,
\begin{eqnarray*}
  \|\nabla ^{m} u^\delta \|_\M
  &= \Bignorm[\radon]{\Bigl(\sum_{l=1}^{m-1} \symgrad^{m-l}(\symgrad v_{l-1}^\delta - v_l^\delta) \Bigr) + \symgrad v_{m-1}^\delta} \\ 
  & \leq \Bigl( \sum_{l=1}^{m-1} \| \symgrad^{m-l} ( \symgrad v_{l-1}^\delta - v_l^\delta) \|_ 1 \Bigr) + \| \symgrad v_{m-1}^\delta \|_\M  \\
 & \leq m \delta  +  \| \nabla^{m} u \|_\M.
\end{eqnarray*}
This shows in particular that $u^\delta \in \BV^k(\Omega)$ and by construction, $u^\delta \in \Cspace{\infty}{\Omega}$.
 Taking the limit $\delta \rightarrow 0$ and using the lower semi-continuity of $\TV^{m}$, we finally obtain
 \[  \|\nabla ^{m} u^\delta \|_\M  \rightarrow  \|\nabla ^{m} u \|_\M \]
 for $m= 1,\ldots,k$ which, together with the first estimate in \eqref{eq:bvk_strict_conv_up_to_km1}, implies the assertion.
\end{proof}

\end{document}